\documentclass[11pt,oneside,leqno]{amsart}

\usepackage{amsaddr}

\usepackage{eurosym}
\usepackage{amsmath}
\usepackage{amsfonts}
\usepackage{amsthm}
\usepackage{color}
\usepackage{graphicx}
\usepackage{hyperref}
\usepackage{anysize}

\usepackage{url}

\usepackage{epstopdf}
\usepackage{algpseudocode}
\usepackage{algorithm}
\usepackage{placeins}

\marginsize{2cm}{2cm}{2.2cm}{2cm}

\newtheorem{theorem}{Theorem}
\newtheorem{proposition}{Proposition}
\newtheorem{lemma}{Lemma}

\newtheorem{remark}{Remark}

\newcommand{\R}{\mathbb{R}}
\newcommand{\uu}{\underline{U}}
\newcommand{\zetau}{\underline{\zeta}}
\newcommand{\hu}{\underline{h}}
\newcommand{\vu}{\underline{V}}

\newcommand{\mt}{\mathcal{T}}

\newcommand{\tu}{\underline{\mathcal{T}}}
\newcommand{\hm}{h_{min}}
\newcommand{\ls}{\Lambda^s}

\definecolor{Gray}{gray}{0.9}

\begin{document}

\title[Moving bottom detection]{Optimal control approach for moving bottom detection in one-dimensional shallow waters by surface measurements}

\author[R. Lecaros]{R. Lecaros}
\address[R. Lecaros]{Departamento de Matem\'atica, Universidad T\'ecnica Federico Santa Mar\'ia,  Santiago, Chile.}
\email{rodrigo.lecaros@usm.cl}

\author[J. L\'opez-R\'ios]{J. L\'opez-R\'ios}
\address[J. L\'opez-R\'ios]{(Corresponding Author)Universidad Industrial de Santander, Escuela de Matemáticas, A.A. 678, Bucaramanga, Colombia}
\email{jclopezr@uis.edu.co}

\author[G. I. Montecinos]{G. I. Montecinos}
\address[G. I. Montecinos]{Department of Mathematical Engineering, Universidad de La Frontera, Temuco, Chile}
\email{gino.montecinos@ufrontera.cl}

\author[E. Zuazua]{E. Zuazua}
\address[E. Zuazua]{Chair in Dynamics, Control and Numerics / Alexander von Humboldt-Professorship \\
	Department of Data Science \\
	Friedrich-Alexander-Universit\"at Erlangen-N\"urnberg \\
	91058 Erlangen, Germany \\
	and \\
	Chair of Computational Mathematics \\
	Fundaci\'on Deusto \\
	Av. de las Universidades, 24 \\ 48007 Bilbao, Basque Country, Spain \\
	and \\
	Departamento de Matem\'aticas \\ Universidad Aut\'onoma de Madrid \\
	28049 Madrid, Spain}
\email{enrique.zuazua@fau.de}

\date{\today}

\thanks{R. Lecaros, J. L\'opez-R\'ios and G. Montecinos have been partially supported by the Math-Amsud project
CIPIF 22-MATH-01. R. Lecaros was partially supported by FONDECYT (Chile) Grant 1221892. J. L\'opez-R\'ios acknowledges support by Vicerrectoría de Investigación y Extensión of Universidad Industrial de Santander, project 3752.
E. Zuazua has been funded by the Alexander von Humboldt-Professorship program, the Transregio 154 Project ``Mathematical Modelling, Simulation and Optimization Using the Example of Gas Networks" of the DFG, the ModConFlex Marie Curie Action, HORIZON-MSCA-2021-$d$N-01, the COST Action MAT-DYN-NET, grants PID2020-112617GB-C22 and TED2021-131390B-I00 of MINECO (Spain), and by the Madrid Goverment -- UAM Agreement for the Excellence of the University Research Staff in the context of the V PRICIT (Regional Programme of Research and Technological Innovation).}


\begin{abstract}
We consider the Boussinesq-Peregrine (\ref{BP}) system as described by Lannes [Lannes, D. (2013). The water waves problem: mathematical analysis and asymptotics (Vol. 188). American Mathematical Soc.], within the shallow water regime, and study the inverse problem of determining the time and space variations of the channel bottom profile, from measurements of the wave profile and its velocity on the free surface. A well-posedness result within a Sobolev framework for (\ref{BP}), considering a time dependent bottom, is presented. Then, the inverse problem is reformulated as a nonlinear PDE-constrained optimization one. An existence result of the minimum, under constraints on the admissible set of bottoms, is presented. Moreover, an implementation of the gradient descent approach, via the adjoint method, is considered. For solving numerically both, the forward (\ref{BP}) and its adjoint system, we derive a universal and low-dissipation scheme, which contains non-conservative products. The scheme is based on the FORCE-$\alpha$ method proposed in [Toro, E. F., Saggiorato, B., Tokareva, S., and Hidalgo, A. (2020). Low-dissipation centred schemes for hyperbolic equations in conservative and non-conservative form. Journal of Computational Physics, 416, 109545]. Finally, we implement this methodology to recover three different bottom profiles; a smooth bottom, a discontinuous one, and a continuous profile with a large gradient. We compare with two classical discretizations for (\ref{BP}) and the adjoint system. These results corroborate the effectiveness of the proposed methodology to recover bottom profiles.
\end{abstract}

\maketitle

\section{Introduction}
The problem of wave generation and the forces that originate them comprise a wide area of study, involving modeling through fundamental laws, numerical resolution of equations and laboratory scale tests. In the particular case of water in the ocean, it is of interest to study those waves produced by displacements and changes in the seabed, a task that has been developed from many points of view in the classical literature \cite{constantin2011nonlinear,dutykh2007water,iguchi2011mathematical,Nersisyan01082015} and is currently a very active area of research \cite{beizel2012simulation,lecaros2020stability,qi2017numerical,shen2022interference}.

In this paper we investigate, theoretically and numerically, the effect of ocean bottom motions on waves. More precisely, within the context of shallow water models, we analyze the inverse problem of finding a single bottom, which depends on the time variable, and which generates a specific wave. These types of problems have been considered before, in the context of unidirectional shallow water approximating models \cite{baudouin2014determination,Nersisyan01082015} and in a broader sense, through wave generating devices (wave-makers) such as those considered in \cite{mottelet2000controllability,su2020stabilizability,su2021strong} for the water-waves equation.

The problem of generating waves through seafloor displacements is directly related to the generation of tsunamis by submarine earthquakes and plate displacements (see \cite{dutykh2007water,iguchi2011mathematical,nosov2001nonlinear}). Since it is important to quantify both types of displacements, vertical and horizontal, in this work we follow an approach that allows considering both of them, as reported in \cite{dutykh2012contribution,iguchi2011mathematical,nosov2009method}. Moreover, from the point of view of applications, there are man-made facilities for surfing away from the ocean \cite{webpage}, where a wave with specific characteristics is created by the unidirectional motion of an underwater rigid object, moving in a preset direction.

We put ourselves in the context of wave propagation in a shallow water channel, so that the dynamics correspond to the laws of conservation of momentum and mass \cite{lannes2013water}. In this general framework, the interaction between the rigid boundary, represented by the bottom of the channel, with the free surface can be well understood and characterized through terms representing a boundary condition or an external source in the Partial Differential Equation. In particular, we will use the shallow water approximation of the general system of water-waves, with a time-varying bottom, known as the Boussinesq-Peregrine (\ref{BP}) model with varying bottom, as derived in \cite{lannes2013water} (see details in section \ref{S2}). 

The main objective of this work is to provide a general theoretical framework, based on conservation laws, to analyze the inverse problem of detecting a seafloor moving in time. Moreover, we propose an efficient computational procedure to fully determine the variations of the channel (ocean) bottom in time and space.
For that, we chose to study the (\ref{BP}) system, which is general enough and contains the main characteristics of shallow water systems; it also includes terms corresponding to a time-varying bottom and in this context has not been studied before. We study three main aspects: first, the existence of solutions of the system in the framework of Sobolev spaces. Second, the theoretical identification of the bottom, as the solution to a minimization problem, involving the state equation and its adjoint system. Third, the implementation of a descent algorithm and a finite volume scheme discretization of the equation and its adjoint system.

Existence and uniqueness results for the (\ref{BP}) system have been obtained in Sobolev spaces \cite{israwi2011large,lannes2013water}, for a steady bottom. In particular, a classical energy approach has been employed in \cite{israwi2011large} applied to a symmetrizable hyperbolic system. We provide a well-posedness theorem for the local existence of solutions to (\ref{BP}), when the bottom is time-dependent, which has not been reported so far. On the extensive literature concerning well posedness for the general system of water-waves, we mention Nalimov  \cite{nalimov1974cauchy}, Yosihara \cite{yosihara1983capillary}, Craig \cite{craig1985existence}, S.Wu \cite{wu1997well,wu1999well} and Lannes \cite{lannes2005well}.

Optimal control problems on water-waves have been studied in \cite{fontelos2017bottom}, in the context of the inverse problem of bottom identification through surface measurements, where the authors addressed the identifiability and set the problem of finding a unique bottom minimizing the $L^2$ norm in the horizontal variable of the Dirichlet-Neumann operator. Concerning the (\ref{BP}) system and the bathymetry detection, we mention the work by Dutykh et al. \cite{Nersisyan01082015}, where the authors reduced the problem to the unidirectional wave propagation to get a Benjami-Bona-Mahony (BBM)-type equation and studied the optimization problem of generating the largest possible wave in the $L^2$ sense within a given bounded interval, by a constant velocity moving bottom.

Our numerical procedure to find the bottom profiles involves two hyperbolic systems: the system (\ref{BP}) for modeling the wave problem, and its adjoint which contains non-conservative products. We propose a numerical scheme, for both (\ref{BP}) and the adjoint system, which is universal in the sense that it applies for both conservative and non-conservative systems.  We consider two additional schemes aimed at evaluating how a classical combination of conservative and non-conservative schemes behaves for solving the problem.  We refer to them as reference schemes  since they are viable discretizations that users can employ. The first one consists of using the conservative scheme of Rusanov \cite{rusanov1961} (also called Local Lax-Friedrich flux)  for (\ref{BP}) and the finite difference discretization for the adjoint system. The second scheme consists of the conservative scheme of Rusanov for (\ref{BP}) and the non-conservative one, for the adjoint system.

A related formulation to the one we are studying is the controllability problem by a source, for the (\ref{BP}) system. Alazard et al. \cite{alazard2018control} have provided a local exact controllability result, obtained for the full water-waves system by controlling a localized portion of the free surface. In \cite{fontelos2023controllability}, the authors proved the interior exact controllability of the 2d-Euler's system by injection of jet fluids through a rigid boundary. We can also mention the works on KdV \cite{rosier2000exact}, BBM \cite{micu2001controllability,zhang2003unique}, and some Boussinesq--type systems \cite{micu2009control}, concerning some literature about controllability of asymptotic regimes and dispersive wave equations. Finally, even though the computation of the adjoint state in the optimal control problem approach follow classical ideas, it is worth to mention that this has not been done before in this context of moving bottoms.

The paper is organized as follows. In section \ref{S2}, we present the (\ref{BP}) system and the formulation of the optimal control problem. In section \ref{S4}, we set the optimal control problem and prove the existence of a minimum. In section \ref{S5}, a particular formulation for recovering a moving bottom is settled. In section \ref{S6}, the coupled formulation approach and the numerical method for solving the inverse problem, are presented. In section \ref{S7}, numerical results are reported. Finally, in section \ref{S8}, conclusions are drawn.



\section{The mathematical framework}\label{S2}

We consider an ideal incompressible fluid in a two dimensional domain, (see Figure \ref{Figure1}). We denote the horizontal independent variable by $x$ and the vertical variable by $y$, and assume that the line $y=1$ corresponds to the still water level and $H_0=1$ is a constant reference depth. Let $\zeta,b:[0,T]\times\R\rightarrow\R$ ($T>0$), be the surface and bottom parameterizations respectively. Let $\mu$ be a small parameter to be specified later and related to the shallow water regime and $\epsilon=O(\mu)$ which represents the small variations of the bottom and free surface, as explained in \cite{lannes2013water}. We assume that the total depth $h\equiv h(t,x)=1+\epsilon(\zeta-b)$, remains positive at all times $t$:
\begin{equation}
\label{ES34}
\exists h_{min}>0, \ \inf_{x\in\R}h\ge \hm,
\end{equation}
which is a necessary condition for the system (\ref{BP}) to be valid.

For a given bottom $b(t,x)$, we are going to consider, on $[0,T]\times\R$, the one dimensional Boussinesq-Peregrine system with moving bottom for $(\zeta,V)$
\begin{equation}
\label{BP}
\tag{BP}
\left\{
\begin{aligned}
&\zeta_t+(hV)_x=b_t, \\
&\left(1-\frac{\mu}{3h}\partial_x(h^3\partial_x\cdot)\right)V_t+\zeta_x+\epsilon VV_x=-\frac{\epsilon}{2}b_{ttx}, \\
&\zeta(0,\cdot)=\zeta_0(\cdot), \ V(0,\cdot)=V_0(\cdot), \ \text{in }\R,
\end{aligned}
\right.
\end{equation}
with $V:[0,T]\times\R\rightarrow\R$ being the total depth averaged velocity. From now on, this will be our model to understand the wave-bottom interaction. Recall that $(\zeta,V)$ are the unknown free surface elevation and velocity, respectively, and $b$ is a given function, representing the topography of the moving bottom. Finally, $\mu$ and $\epsilon$ are dimensionless, small parameters, representing the shallowness of the channel. 

\begin{figure}
\centering
\includegraphics[scale=0.9]{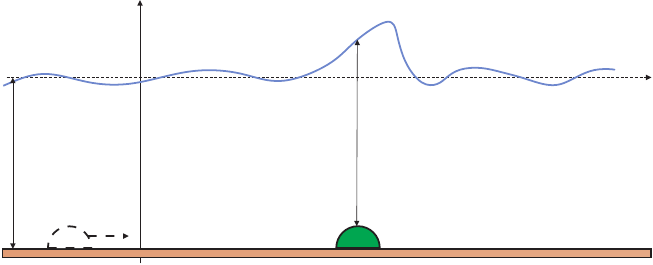}
\put(-230,83){$1$}
\put(-275,39){$H_0=1$}
\put(-65,88){$\zeta(t,x)$}
\put(-116,12){$b(t,x)$}
\put(-125,44){$h$}

\caption{Sketch of the physical domain and main notation}\label{Figure1}
\end{figure}
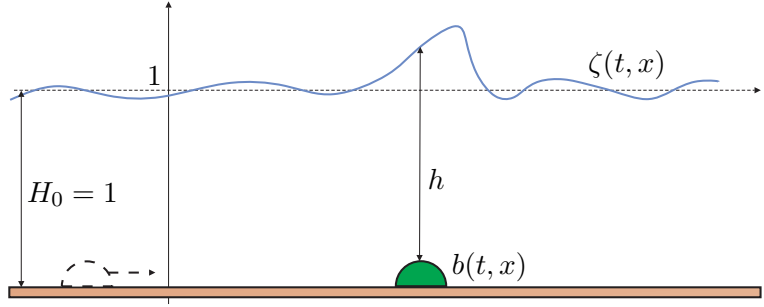

System \eqref{BP} can be derived from the general water waves equations under the shallow water assumption ($\mu\ll1$), assuming $\epsilon=O(\mu)$. See \cite[page 145]{lannes2013water} for a two dimensional version of (\ref{BP}), and \cite{lannes2009derivation}, where many of the shallow water models are deduced from the general water waves system.

We now turn our attention to the optimization problem for system (\ref{BP}). We assume that the bottom can change during a time interval over the entire domain; namely, the time-dependent bathymetry is given by $b=b(t,x)$. The detection of the bottom will be obtained as the solution to an optimization problem. That is, we will find, theoretically and numerically, the function $b$ to watch a prescribed wave and velocity $\overline{\zeta}$, $\overline{V}$, at some fixed time $T>0$ (see sections \ref{S4} and \ref{S5}). Mathematically, for fixed $\zeta_0$ and $V_0$, we minimize the functional
\begin{equation}
	J(b)=\frac{1}{2}\int_0^T|\zeta-\bar{\zeta}|_{L^2}^2dt+\frac{1}{2}\int_0^T|V-\bar{V}|_{L^2}^2dt,
\end{equation}
where $\bar{\zeta},\bar{V}\in L^2((0,T);L^2(\R))$ are given, $(\zeta,V)$ is a solution to \eqref{BP} in $C([0,T];H^s\times H^{s+1})$, and the Sobolev space $H^s=H^s(\R)$, $s\ge 3/2$ (see \ref{S3} for details). 

Note the explicit way the time derivatives of $b$ appear in (\ref{BP}), at the right-hand side in the $b_t,b_{ttx}$ terms. We start by formulating an existence and uniqueness theorem for (\ref{BP}), so that the regularity of the bottom be explicit. This is needed when formulating the problem of finding $b$ by the optimal control approach.

\begin{theorem}
	\label{Theorem Existencia0}
	Let $s>3/2$ and $(\zeta_0,V_0)\in \mathbf{X}^s$. Let $b\in W^{2,\infty}([0,\infty);H^{s+1})$ and assume \eqref{ES34} is valid. Then, there exists $T_{BP}>0$, uniformly bounded from below with respect to $\epsilon$, such that system \eqref{ES32} admits a unique solution $(\zeta,V)^T\in C([0,T_{BP}/\epsilon];\mathbf{X}^s)$ with initial condition $(\zeta_0,V_0)^T$.
\end{theorem}

Here we have introduced the space $\mathbf{X}^s$, defined as
\[ \mathbf{X}^s=H^s(\R)\times H^{s+1}(\R),  \]
endowed with the norm
\[ \forall \mathbf{U}\in \mathbf{X}^s, \ |\mathbf{U}|_{\mathbf{X}^s}^2=|\zeta|_{H^s}^2+|V|_{H^s}^2+\mu|V_x|_{H^s}^2. \]

For the sake of clearness in the presentation, we leave the details of the proof for \ref{S3}.

\begin{remark}
When the water waves equations are nondimensionalized, a rough analysis of their linearization around the rest state is performed. This shows the relevance of various dimensionless parameters, namely, the amplitude parameter $\epsilon$, the shallowness parameter $\mu$, and the topography parameter $\beta$ (see \cite{lannes2013water} for details). With the relevant physical dimensionless parameters introduced, asymptotic regimes are identified (the shallow water regime for instance) as conditions on these dimensionless parameters (e.g., $\mu<<1$ for the shallow water regime).

The asymptotic regime that corresponds to the system (\ref{BP}) is precisely
\[ \{(\epsilon,\beta\,\mu):0\le\mu\le\mu_0, \ 0\le\beta\le\epsilon\le\mu\}. \]
That is the reason why we decided to keep the two parameters $\mu$ and $\epsilon$ in the formulation of (\ref{BP}) above.
\end{remark}

\section{The optimal control approach}
\label{S4}

Following the explanations in the previous section, the process of minimizing the functional $J$ will allow us to design the bottom $b(t,x)$, given a specific surface profile and its velocity, $(\overline{\zeta},\overline{V})$. Therefore we will address two practical questions. The first is tsunami detection with surface measurements, where we also include the possibility of horizontal and vertical displacements caused by plate motions and underwater earthquakes \cite{dutykh2007water}. Second, the design of moving underwater structures to generate specific waves for an entertainment purpose \cite{Nersisyan01082015}. With this in mind, we present the following mathematical formulation, and the sufficient conditions for a bottom $b$ to exist, minimizing functional $J$.

The first thing we need is an existence and uniqueness of solutions theorem, for system (\ref{BP}). Because the proof strategy is classical and follows arguments similar to those presented in \cite{israwi2011large,lannes2013water} for the nonlinear shallow water and green-naghdi systems, respectively, we postpone the proof to an appendix and concentrate on the optimization problem.

We consider the space $\mathbf{X}_T^s:=C([0,T/\epsilon];\mathbf{X}^s)$ endowed with its canonical norm. Let $\mathcal{F}\subset W^{2,\infty}(0,T;H^{s+1})$ be a nonempty, closed and convex set. Given $d_0>\frac{1}{2}$, $s\ge d_0+1$ and $(\zeta_0,V_0)\in \mathbf{X}^s$, let the following minimization problem related to \eqref{BP}:
\begin{equation}
	\label{P}
	\tag{P}
	\left\{
	\begin{aligned}
		&\text{Find $(\zeta,V,b)\in \mathbf{X}^s\times\mathcal{F}$ such that the functional} \\
		& J(b)=\frac{1}{2}\int_0^T|\zeta-\bar{\zeta}|_{L^2}^2dt+\frac{1}{2}\int_0^T|V-\bar{V}|_{L^2}^2dt \\
		&\text{is minimized, subject to } \ (\zeta,V,b) \ \text{satisfies } \eqref{BP}.
	\end{aligned}
	\right.
\end{equation}
Here $(\bar{\zeta},\bar{V})\in L^2(0,T;L^2)\times L^2(0,T;L^2)$ represents the desired state.


The set of optimal solutions of \eqref{P} is defined by
\[ \mathcal{S}_{ad}=\{\omega=(\zeta,V,b)\in X^s\times\mathcal{F}: \omega \text{ is a solution of }\eqref{BP}\}. \]

Note that $b$ is not acting on the right--hand side of \eqref{BP} only (through $b_t$, $b_{ttx}$), but the spatial derivatives of the terms involving $h$. That is, $b$ acts on the source and coefficients.

Thanks to the remark \ref{r2} in the appendix, we can take $\epsilon$ sufficiently small, so that the existence time of the solutions is guaranteed in a time interval containing $[0,T]$. This is important when passing to the limit in a minimizing sequence for proving the existence of a minimum for the functional $J$, as the next theorem shows.

\begin{theorem}
\label{Theorem Existencia Minimo}
Let $(\zeta_0,V_0)^T\in \mathbf{X}^s$ with $s\ge 1+\delta$, $\delta>0$. Let $T>0$ such that $\mathcal{F}$ is bounded in $W^{2,\infty}(0,T;H^{s+1})$. Then the optimal control problem \eqref{P} has at least one global optimal solution $(\zeta^*,V^*,b^*)\in\mathcal{S}_{ad}$.
\end{theorem}
\begin{proof}
From Theorem \ref{Theorem Existencia} in the appendix, we deduce that $\mathcal{S}_{ad}\ne\emptyset$. Let $\{\omega^n\}_{n\in\mathbb{N}}=\{(\zeta^n,V^n,b^n)\}_{n\in\mathbb{N}}\subset\mathcal{S}_{ad}$ a minimizing sequence of $J$. Let $\mathbf{U}^n(t,x)$ be the corresponding solution of problem \eqref{BP} with bottom $b^n(t,x)$. From the definition of $J$ and the assumption $\mathcal{F}$ bounded in $W^{2,\infty}(0,T;H^{s+1})$, 
\[ \{b^n\} \text{ is bounded in } W^{2,\infty}(0,T;H^{s+1}). \]

We know $\mathcal{F}$ is weakly closed in $W^{2,\infty}(H^{s+1})$. Then, from Theorem \ref{Theorem Existencia} and the corresponding estimate as in \eqref{ES35_1}, there exists $\omega^*=(\zeta^*,V^*,b^*)\in\mathcal{S}_{ad}$ such that, for some subsequence of $\{\omega^n\}$, still denoted $\{\omega^n\}_{n\in\mathbb{N}}$, we have weak convergences in $\mathbf{X}^s\times W^{2,\infty}(0,T;H^{s+1})$.

If $\mathbf{U}^{n,m}:=\mathbf{U}^n-\mathbf{U}^m$ we have (see Remark \ref{remark1})
\begin{equation*}
	\begin{pmatrix}
		1 & 0 \\ 0 & \mt\cdot
	\end{pmatrix}
	\begin{pmatrix}
		\zeta^{n,m}_t \\ V^{n,m}_t
	\end{pmatrix}
	+
	\begin{pmatrix}
		\epsilon V^n & h^n \\ h^n & \mt(\epsilon V^n\cdot)
	\end{pmatrix}
	\begin{pmatrix}
		\zeta^{n,m}_x \\ V^{n,m}_x
	\end{pmatrix}
	=
	\begin{pmatrix}
		b^{n,m}_t+\epsilon b^{n,m}_xV^m+(\zeta_x^m-b_x^n)V^{n,m} \\ -\frac{\epsilon}{2}h^nb^{n,m}_{ttx}-\frac{\epsilon}{2}h^{n,m}b^m_{ttx}-\zeta_x^mh^{n,m}
	\end{pmatrix}.
\end{equation*}

Then, multiplying last equation in $L^2$ by $\mathbf{U}^{n,m}$, integrating by parts and applying the Gronwall inequality (as is done in the proof of Proposition \ref{Energy estimate} in the appendix), we obtain that $\mathbf{U}^{n,m}$ is a Cauchy sequence in $H^1$. Hence $\omega^*$ is solution of \eqref{BP}, pointwisely; that is, $\omega^*\in\mathcal{S}_{ad}$. Therefore,
\[ \lim_{n\to+\infty}J(\omega_n)=\inf_{\omega\in\mathcal{S}_{ad}}J(\omega_n)\le J(\omega^*). \]

On the other hand, since $J$ is lower semicontinuous on $\mathcal{S}_{ad}$, we have $J(\omega^*)\le\liminf_{n\to+\infty}J(\omega_n)$.
\end{proof}

Finally to implement the descent strategy, let us rewrite \eqref{BP} as
\begin{equation*}
	\left\{
	\begin{aligned}
		&r_t+(hV)_x=0, \\
		& V_t-\frac{\mu}{3h}\partial_x(h^3V_{tx})+r_x+\epsilon VV_x=-b_x-\frac{\epsilon}{2}b_{ttx},
	\end{aligned}
	\right.
\end{equation*}
with $r=\zeta-b$, $h=1+\epsilon r$. Then, the adjoint system is given by (see \ref{a1})
\begin{equation*}
	\left\{
	\begin{aligned}
		&p_t+\epsilon Vp_x+q_x=\bar{\zeta}-\zeta, \\
		&q_t-\frac{\mu}{3}\partial_x(h^2q_{tx})+hp_x+\epsilon Vq_x=\bar{V}-V, \\
		&p(T)=0, \quad q(T)=0,
	\end{aligned}
	\right.
\end{equation*}
and the descent direction is
\[ -\nabla J(b)=-q_x-\frac{\epsilon}{2}q_{ttx}+(\bar{\zeta}-\zeta), \]
accordingly (see Theorem \ref{t3} in the Appendix).

\begin{remark}\label{remark1}
	Because we are in the shallow water regime and $\mu<<1$, let us neglect the term $\mu\partial_x(h^3\partial_xV_t)$.
	
\end{remark}

\section{Numerical reconstruction  of a time dependent bathymetry}\label{sec:test-case}\label{S5}

Now we will consider the problem of numerically reconstructing the bathymetry $b(t,x)$ from measurements of $\zeta$ and $V$ at the surface. We will do so in the shallow water regime described in the previous section, in which the term involving $V_{txx}$ is neglected (since $\mu<<1$). As it was described above, we are looking for a function $b$ such that
\begin{eqnarray}
	\label{eq:dJ}
	\begin{array}{c}
		\displaystyle
		J(b) = \frac{1}{ 2} \int_{0}^T |\zeta - \bar{\zeta}|_{L^2}^2 + \frac{1}{ 2} \int_{0}^T|V - \bar{V}|_{L^2}^2 \;
	\end{array}
\end{eqnarray}
is minimized, where $\zeta,V$ are constrained to the system
\begin{eqnarray}
	\label{eq:direct-system}
	\begin{array}{c}
		\partial_t 
		\left[
		\begin{array}{c}
			r  \\
			V
		\end{array}\right]
		+ \partial_x
		\left[
		\begin{array}{c}
			hV    \\
			r + \epsilon \frac{ V^2}{2}
		\end{array}\right]
		=
		\left[
		\begin{array}{c}
			0  \\
			- b_x - \frac{ \epsilon}{2} b_{ttx}
		\end{array}\right]
		\;
	\end{array},
\end{eqnarray}
with $ r = \zeta - b$, $h = 1 + \epsilon r$. 

In (\ref{eq:dJ}), $\bar{\zeta}(0,x)$ is a given function. $\bar{V}$ is obtained after solving (\ref{eq:direct-system}) over the interval  $[0,L]$ with appropriate boundary conditions and a given bathymetry $\bar{b}$ and initial condition $[r(0,x), V(0,x)]^T = [\bar{\zeta}(0,x)- \bar{b}(0,x), V_0]^T $ for $V_0$ a given initial velocity.  That is, we use synthetic data for assessing the methodology.

The problem of finding the bathymetry is stated as a PDE-constraint optimization problem which is studied through the adjoint method. That is, a minimizing sequence $\{ b^k \}$ is generated via
\begin{eqnarray}
	\begin{array}{c}
		b^{k+1} = b^k - \lambda_b \cdot \nabla J(b^k) \;,
	\end{array}
\end{eqnarray}
where $\lambda_b$ is a constant parameter and
\begin{eqnarray}
	\label{eq:NablaJ}
	\begin{array}{c}
		\nabla J = q_x + \frac{1}{2}  \epsilon q_{xtt} - (\bar{\zeta} - \zeta) \;,
	\end{array}
\end{eqnarray}
with $p$ and $q$ adjoint variables related through the adjoint system
\begin{eqnarray}
	\label{eq:dual-system}
	\begin{array}{c}
		\partial_t 
		\left[
		\begin{array}{c}
			p  \\
			q
		\end{array}\right]
		+
		\left[
		\begin{array}{cc}
			\epsilon V   & 1  \\
			h     & \epsilon V
		\end{array}\right]
		\partial_x 
		\left[
		\begin{array}{c}
			p  \\
			q
		\end{array}\right]
		=
		\left[
		\begin{array}{c}
			\bar{\zeta}-\zeta  \\
			\bar{V}-V
		\end{array}\right]
		\;,
	\end{array}
\end{eqnarray}
on $[0,L]$, endowed with $ p(T,x) = q(T,x) = 0$ and boundary conditions chosen consistently with that of (\ref{eq:direct-system}), namely transmissive or periodic boundary conditions.   Notice that, this system is solved back in time from $t = T$ up to $t = 0$.  Here $ \zeta = r + b^k \;,$  $r$ and $V$ satisfy the system
\begin{eqnarray}
	\begin{array}{c}
		\partial_t 
		\left[
		\begin{array}{c}
			r  \\
			V
		\end{array}\right]
		+ \partial_x
		\left[
		\begin{array}{c}
			hV    \\
			r + \epsilon \frac{V^2}{2}
		\end{array}\right]
		=
		\left[
		\begin{array}{c}
			0  \\
			- b^k_x - \frac{ \epsilon}{2} b^k_{ttx}
		\end{array}\right] 
		\;,
	\end{array}
\end{eqnarray}
with $[r(0,x),V(0,x)]^T = [\bar{\zeta}(0,x) - \bar{b}(0,x),V_0]^T $ on $[0, L]$ and transmissible boundary condition.

\section{Numerical discretization}\label{S6}
System (\ref{eq:direct-system}) can be expressed as
\begin{eqnarray}
	\label{eq:general-state}
	\begin{array}{c}
		\partial_t \mathbf{U} + \partial_x \mathbf{F}(\mathbf{U}) =  \mathbf{B}(x, \mathbf{U}) \;, 
		\\
		\mathbf{U} = \mathbf{U}_0(x) \;,
	\end{array}
\end{eqnarray}
where $\mathbf{U}_0(x)$ is the prescribed initial condition function. The system in its quasilinear version takes the form
\begin{eqnarray}
	\label{eq:general-state-quasi}
	\begin{array}{c}
		
		\partial_t \mathbf{U} +  \mathbf{A}(\mathbf{U}) \partial_x \mathbf{U} =  \mathbf{B}(x, \mathbf{U}) \;,

		\\
		\mathbf{U} = \mathbf{U}_0(x) \;,
	\end{array}
\end{eqnarray}
where $ \mathbf{A}(\mathbf{U} )  $  is the Jacobian matrix of $ \mathbf{F}(\mathbf{U})$ with respect to $\mathbf{U}$. 

The adjoint system (\ref{eq:dual-system}) can be written as
\begin{eqnarray}
	\label{eq:general-dual}
	\begin{array}{c}
		\partial_t \mathbf{P} +\mathbf{A}^T(\mathbf{U})  \partial_x  \mathbf{P} =  \mathbf{R}(x, \mathbf{U}, \mathbf{P}) \;,
		\\
		\mathbf{P}(T,x) = \mathbf{0} \;,
	\end{array}
\end{eqnarray}
where $\mathbf{P}$ is the adjoint state. Both systems can be written in a unified way, in the sense of \cite{montecinos2019numerical}, for the state $\mathbf{W} = [\mathbf{Q}, \mathbf{P}]^T$ given by
\begin{eqnarray}
	\label{eq:unified-system}
	\begin{array}{c}
		\partial_t \mathbf{W}
		+
		\mathbf{A_W}
		\partial_x \mathbf{W}
		= 
		\mathbf{S_W}
		\;,
	\end{array}
\end{eqnarray}
where
\begin{eqnarray}
	\label{eq:unified-system-matrix}
	\begin{array}{c}
		\mathbf{A_W} = 
		\left[
		\begin{array}{cc}
			\mathbf{A}(\mathbf{U}) & \mathbf{0} \\
			\mathbf{0} &   \mathbf{A}^T(\mathbf{U})
		\end{array}
		\right]
		
		\;, \ 
		\mathbf{S_W}
		= 
		\left[
		\begin{array}{c}
			\mathbf{B}(x, \mathbf{U})
			\\
			\mathbf{R}(x, \mathbf{U}, \mathbf{P})
		\end{array}
		\right]
		\;.
	\end{array}
\end{eqnarray}

We derive a numerical scheme for system (\ref{eq:unified-system}). In this type of applications, the numerical diffusion can increase due to small CFL coefficients which are used in the context of inverse problems, in contrast to the usual applications (direct problems) to conservation laws. In this paper, we explore the novel FORCE-$\alpha$ scheme reported in \cite{toro2020low} for first order of accuracy and extended to higher orders in \cite{montecinos2021universal}. This scheme has the advantage of improving the numerical diffusion. The numerical scheme has the form
\begin{eqnarray}
	\label{eq:one-step-unified-dual}
	\begin{array}{c}
		\mathbf{W}_i^{n+1} = \mathbf{W}_i^n - \omega \frac{\Delta t}{ \Delta x} ( \mathbf{D}_{i+\frac{1}{2} }^{-} +  \mathbf{D}_{ i-\frac{1}{2} }^{+}) + \omega \Delta t \mathbf{S_W}_i \;,
	\end{array}
\end{eqnarray}
where
\begin{eqnarray}
	\begin{array}{c}
		\mathbf{D}_{i+\frac{1}{2}}^{\pm} 
		=
		\omega \mathbf{A}^\pm_{i+\frac{1}{2}}
		
		\cdot (\mathbf{W}_{i+1}^n - \mathbf{W}_{i}^n ) \;,
	\end{array}
\end{eqnarray}
here $ \omega $ is a parameter that controls forward and backward propagation in time, $ \omega = 1$ and $\omega= -1$, respectively. The matrix $\mathbf{A}^\pm_{i+\frac{1}{2}}$ are given by
\begin{eqnarray}
	\begin{array}{c}
		\mathbf{A}^\pm_{i+\frac{1}{2}}
		= \frac{1}{2}\hat{\mathbf{A}}_{i+\frac{1}{2}} \pm \frac{1}{4} \frac{\alpha_F \Delta t}{\Delta x}(  \hat{\mathbf{A}}_{i+\frac{1}{2}}^2 + (\frac{ \Delta t}{\alpha \Delta t} )^2 \mathbf{I}   )  \;,
		
	\end{array}
\end{eqnarray}
where $\mathbf{I}$ is the identity matrix, and
\begin{eqnarray}
	\begin{array}{c}
		\hat{\mathbf{A}}_{i+\frac{1}{2}} = \int_{0}^{1} \mathbf{A_W}( \Phi(s,\mathbf{W}_i^n, \mathbf{W}_{i+1}^n)) ds  
		\approx \sum_{k=1}^{3} w_1 \cdot \hat{\mathbf{A}}_{i+\frac{1}{2}} = \int_{0}^{1} \mathbf{A_W}( \Phi(\xi_k,\mathbf{W}_i^n, \mathbf{W}_{i+1}^n))  \;,
	\end{array}
\end{eqnarray}
where  $  \Phi(s,\mathbf{W}_i^n, \mathbf{W}_{i+1}^n) = \mathbf{W}_{i}^{n} + (\omega s-\frac{(\omega-1)}{2}) ( \mathbf{W}_{i+1}^{n} - \mathbf{W}_{i}^{n})$.  The last term is approximated via the Gauss-Legendre quadrature rule characterized by  $ \xi_1 = \frac{1}{2} (1 - \sqrt{\frac{3}{5}})  $, $ \xi_2 = \frac{1}{2}  $, $ \xi_3 = \frac{1}{2} (1 + \sqrt{ \frac{3}{5} } )  $, $ w_1 = \frac{5}{18}  $, $ w_2 = \frac{8}{18}  $ and $ w_3 = \frac{5}{18}$. 
The source term is discretized as
\begin{eqnarray}
	\label{eq:source-term-unified}
	\begin{array}{c}
		\mathbf{S_W}_{i}
		= 
		\displaystyle
		\left[
		\begin{array}{c}
			\displaystyle
			0 \\
			\\
			
			\displaystyle
			-  b^n_{x,i} 
			
			- \frac{\epsilon}{2} \biggl(
			\frac{ ( b_{x,i}^{n+1}-2b_{x,i}^n+b_{x,i}^{n-1} )  }{\Delta t^2}
			\biggr)

			\\
			
			\zeta_i^n -  \bar{\zeta}(x_i,t^n) \\
			
			V_i^n-\bar{V}_i^n  
			
		\end{array}
		\right]
		\;,
	\end{array}
\end{eqnarray}

with 
\begin{eqnarray}
	\label{eq:bx}
	\begin{array}{c}
		b^n_{x,i} = \frac{ b_{i+1}^n - b_{i-1}^n  }{ 2 \Delta x} \;.
	\end{array}
\end{eqnarray}
Notice that for shallow water applications, the approximation of spatial derivatives, for a $b(t,x)$ known, can be replaced by $ b^n_{x,i} = \frac{ b(x_i-\frac{\Delta x}{2}, t^n) - b(x_i+\frac{\Delta x}{2}, t^n)  }{ \Delta x} \;.$ In this way the well-balanced property is ensured. However, in this work $b$ is only known in a discrete spatial and temporal location at each iteration, so we are limited to use (\ref{eq:source-term-unified}).

In (\ref{eq:one-step-unified-dual}), $\omega = 1$ evolves the system forward in time and thus the first components of $\mathbf{W}$ solve (\ref{eq:general-state}), whereas  $\omega = -1$ solves (\ref{eq:one-step-unified-dual}) backward in time, so the last components of $\mathbf{W}$ solve (\ref{eq:general-dual}).  For solving the system backwards, we set to zero the variables of $\mathbf{W}$ associated to the adjoint state and froze the values of $\mathbf{W}$ associated to $\mathbf{U}$, that is, we keep the values obtained in the forward evaluation. This is needed because the hyperbolic system is not reversible in time, that means wave patterns in backward evolution may be different from forward ones. Since in the forward evolution, variables associated to the adjoint system do not influence the state variables we also set them to be zero. At each iteration of the present strategy, we extract components $r$, $p$ and $q$ to form $\nabla J$ given by (\ref{eq:NablaJ}). We remark that the coupled formulation is only required for building the numerical scheme; in practice this still works as a classical solver.


Although the particular case of finding $b(t,x)$ is addressed, the framework is general enough for any source $b(t,x)$ as well.  As stated in \cite{montecinos2021universal}, this method is universal, since hyperbolic problems written in conservative and non-conservative form are solved with the scheme without any modification of the code. Since the present scheme is derived from the non-conservative {\bf FORCE-$\alpha$} method applied to both direct and adjoint systems through the {\bf C}ouple {\bf S}ystem {\bf F}ormulation, from now on, we referee this scheme {\bf FORCE-$\alpha$+CSF}.

\section{Numerical results}\label{S7}

In this section, we solve the following three test problems; smooth bottom profiles, discontinuous bottom profiles and smooth profiles with a large gradient.  In order to compare the current scheme, we use the following reference schemes. The first one, presented in \ref{sec:ref-scheme},  is characterized by the use of the conservative Rusanov finite volume method for solving (\ref{eq:direct-system}) and finite difference approximation for handling (\ref{eq:dual-system}), from now on called {\bf Rusanov+FD}. The second scheme, consists of the non-conservative FORCE-$\alpha$ scheme applied for solving (\ref{eq:direct-system}) only, and finite difference approximation, as in \ref{sec:ref-scheme}, for the adjoint system (\ref{eq:dual-system}), from now on {\bf FORCE-$\alpha$+FD}. The choice of these reference schemes attempts to reproduce the discretization that users usually implement as a first choice when performing this type of tests.

The reference schemes do not consider strategies based on other types of discretizations. For instance, those in \cite{lellouche1994boundary}, where finite difference methods for both (\ref{BP}) and adjoint systems are implemented. Neither those in \cite{ghidaglia2001numerical}, where  a semidiscrete  scheme is employed to solve (\ref{BP}), which consists of a conservative flux for spatial discretization \cite{Nersisyan01082015} and Runge-Kutta scheme for time discretization. They are used in combination with the Matlab Optimization Toolbox  {\it fmincon} for solving the constrained optimization problem. The comparison with some of these methods is material for a future work.

The first reference scheme must be sensitive to numerical diffusion due to small CFL coefficients, whereas the second one must control this by implementing the FORCE-$\alpha$ on the state system.

To illustrate the performance of the present scheme and reference ones, the sequence of time $t=0.25$, $t = 0.5$, $t=0.75$  and iterations $0,1,2,4,8,$ are depicted. In all simulations, we set $L=20$, $\bar{\zeta} = 1$, $\Delta t = 0.01$, $100$ cells, $\alpha_F = 2$, $V_0 = 1.5$, $\varepsilon = 0.001$, $b^0 = 0.01$ and $\lambda_b = 0.71$.  

\subsection{Smooth bottom profile}
This test aims at recovering the smooth bottom profile
\begin{eqnarray}
	\label{eq:b-smooth}
	\begin{array}{c}
		
		\bar{b}(t,x) = 0.1(1 +  t \cdot exp(  - (x-10- 2.5 t)^2) )\;.
	\end{array}
\end{eqnarray}
Systems (\ref{eq:direct-system}) and (\ref{eq:dual-system}) are solved with transmissible boundary conditions. Note that this profile consists of a soliton that moves to the right and the amplitude of the wave increases with the time.

Figure \ref{fig:b-for-iter-and-times:test-0-ConsRusanov} shows the results for the {\bf Rusanov+FD} scheme. We observe some oscillations in the first iterations. We note that they do not disappear but reduce as the iterations increase.  Figure \ref{fig:b-for-iter-and-times:test-0:ForceAlphaCons} shows the results for the {\bf FORCE-$\alpha$+FD} scheme. We still observe some oscillations on the first iterations, but they almost disappear as the iterations increase.  Figure \ref{fig:b-for-iter-and-times:test-0:NonCons-Force-alpha} shows the results for the {\bf FORCE-$\alpha$+CSF} scheme. We note that oscillations are reduced at the initial iterations and disappear as the iterations increase.  Figure \ref{fig:comp-error:test-0} shows the $L_\infty$ norm of $\nabla J$ against the number of iterations, to facilitate the visualization the plot is depicted in logarithmic scale. So, this measures the error between $b^k $ and $b^{k+1}$, which is the empirical convergence of the global algorithm. These results show that in order to obtain convergence, it is not only important how the state system is discretized, but also how the adjoint system is does so.  This test reveals that a low-dissipation scheme is beneficial in the PDE-constraint optimization context and smooth variables.   

\begin{figure}[h]
	\begin{center}                  
		\includegraphics[scale=0.45]{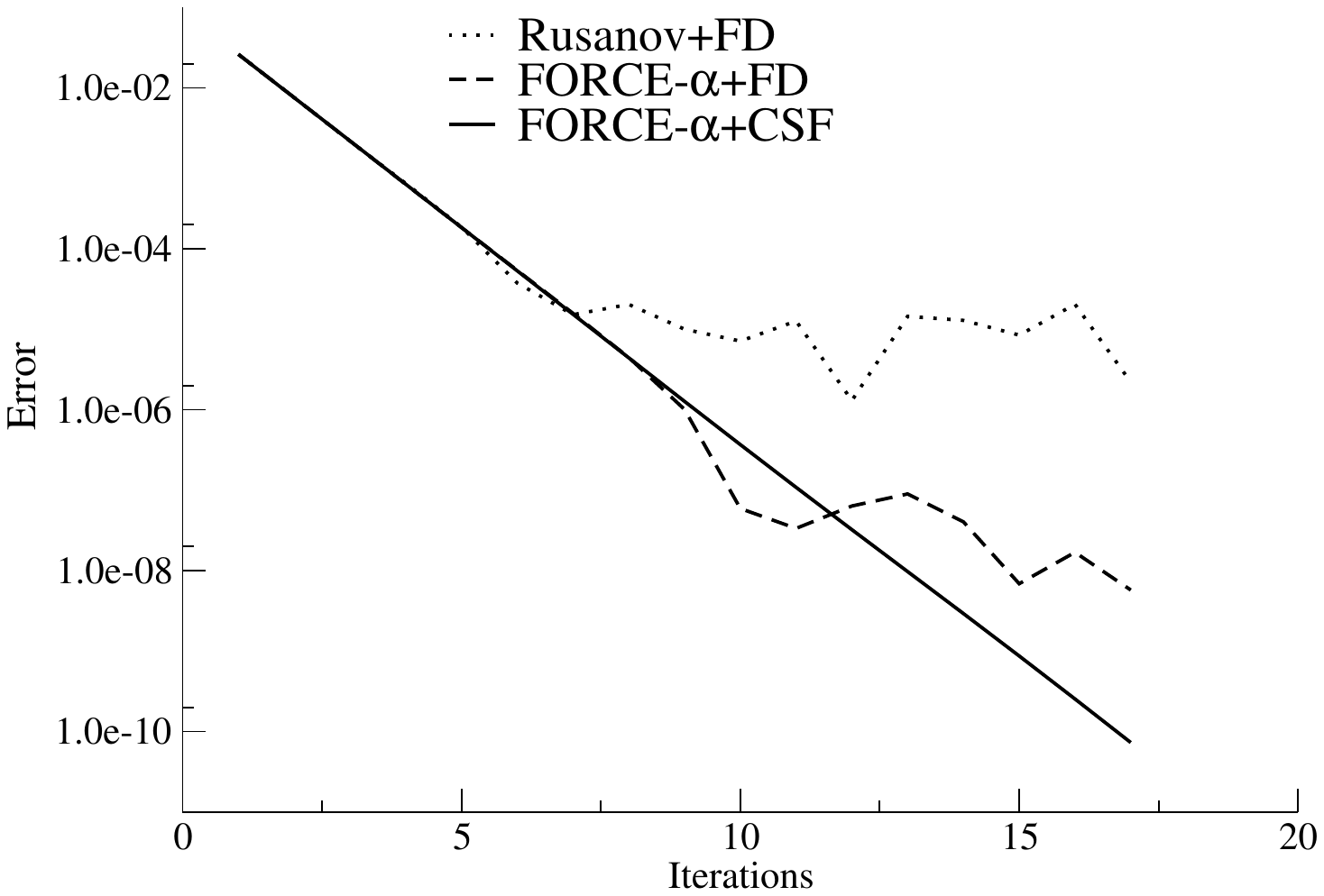}	
	\end{center}
	\caption{Smooth bottom profile (\ref{eq:b-smooth}): The $L_\infty$ norm of $\nabla J $ for $\varepsilon = 0.001$, $\lambda_b = 0.71$, $100$ cells at $t = 1$, $\alpha_F = 2$. 
		(Dot line) {\bf Rusanov+FD} scheme.
		(Dash line) {\bf FORCE-$\alpha$+FD} scheme. 
		(Full line) {\bf FORCE-$\alpha$+CSF} scheme.
	}\label{fig:comp-error:test-0}
\end{figure}


\begin{figure}   
	\centering
	\includegraphics[width=0.3\textwidth]{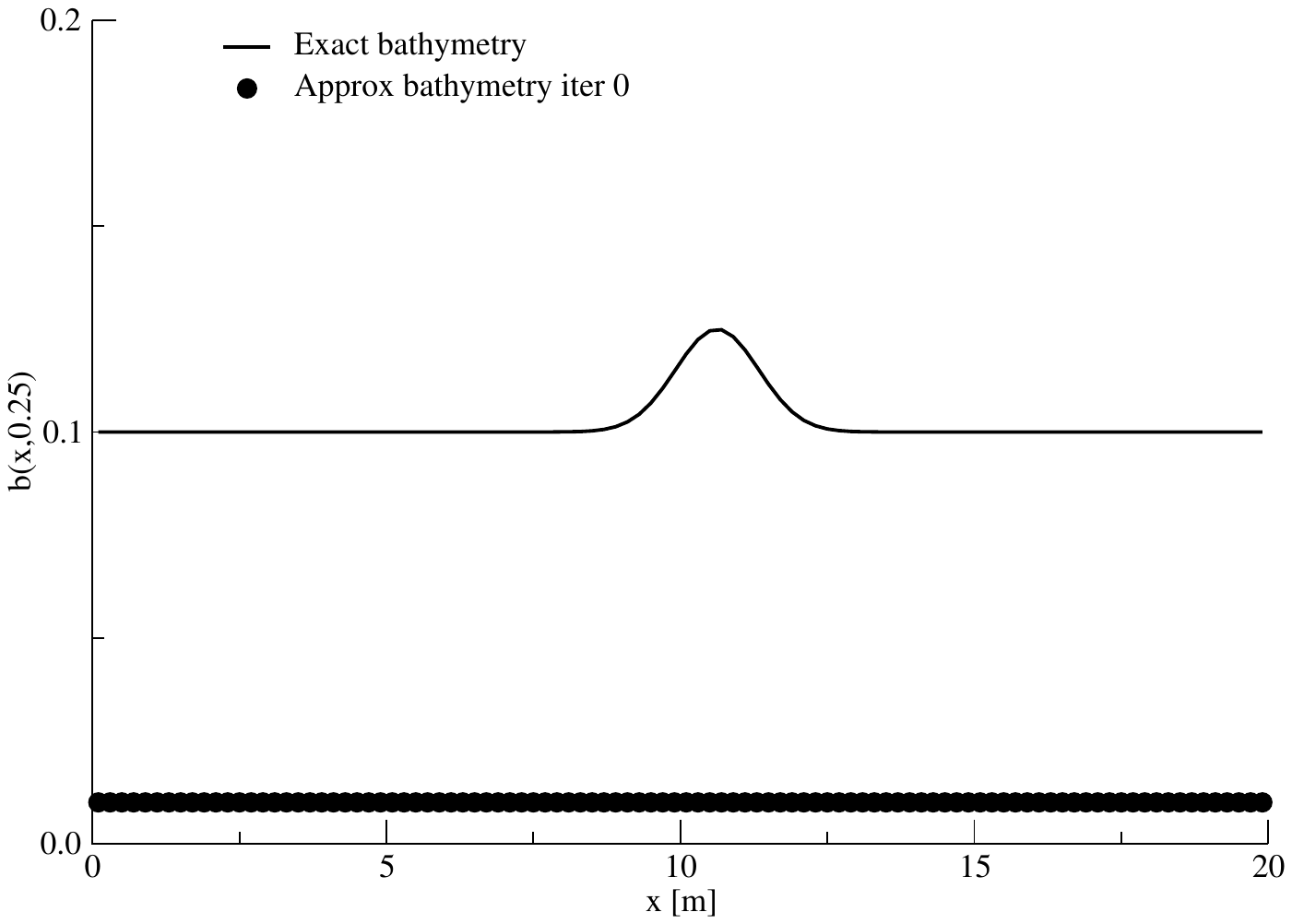} \quad
	\includegraphics[width=0.3\textwidth]{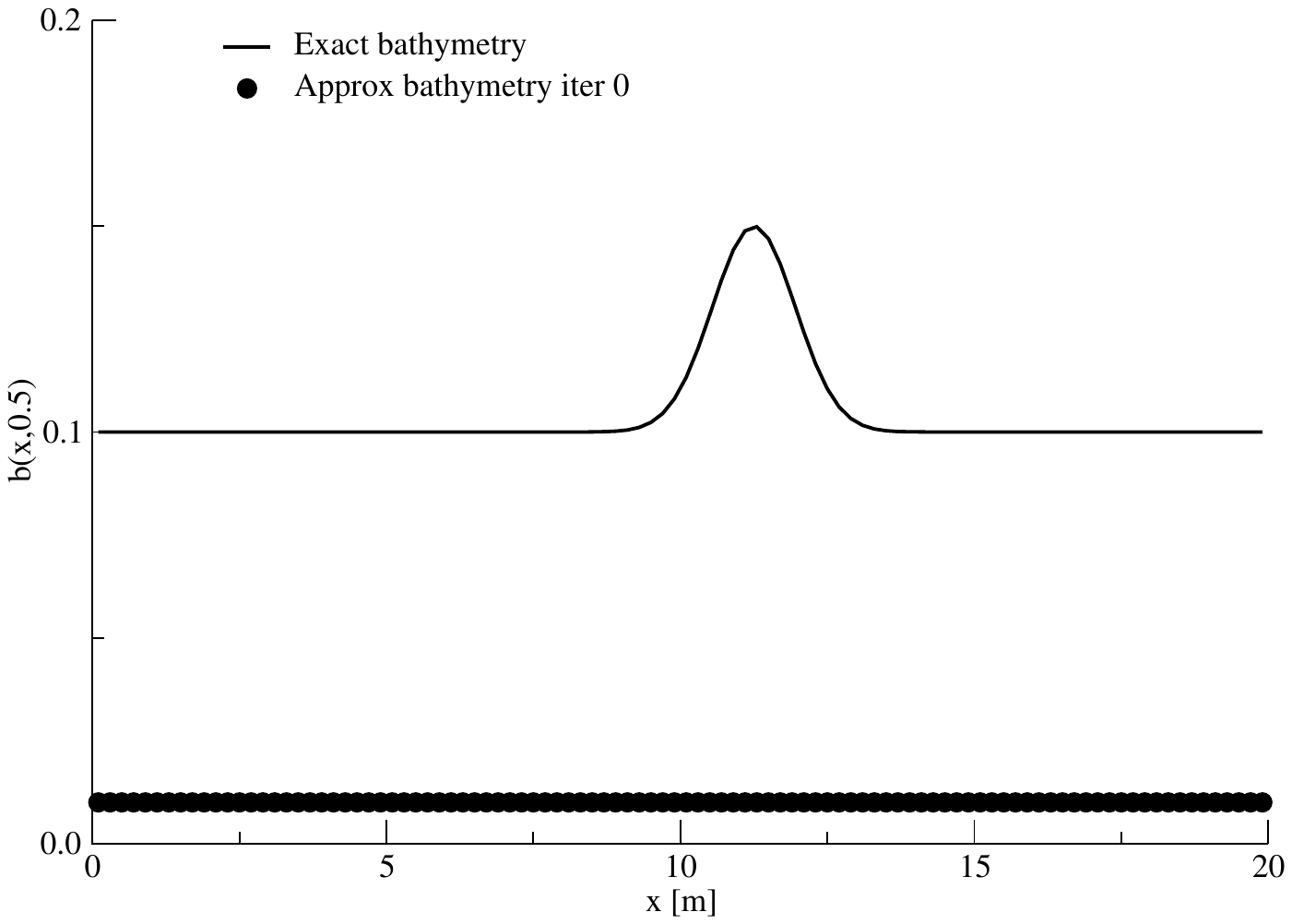} \quad
	\includegraphics[width=0.3\textwidth]{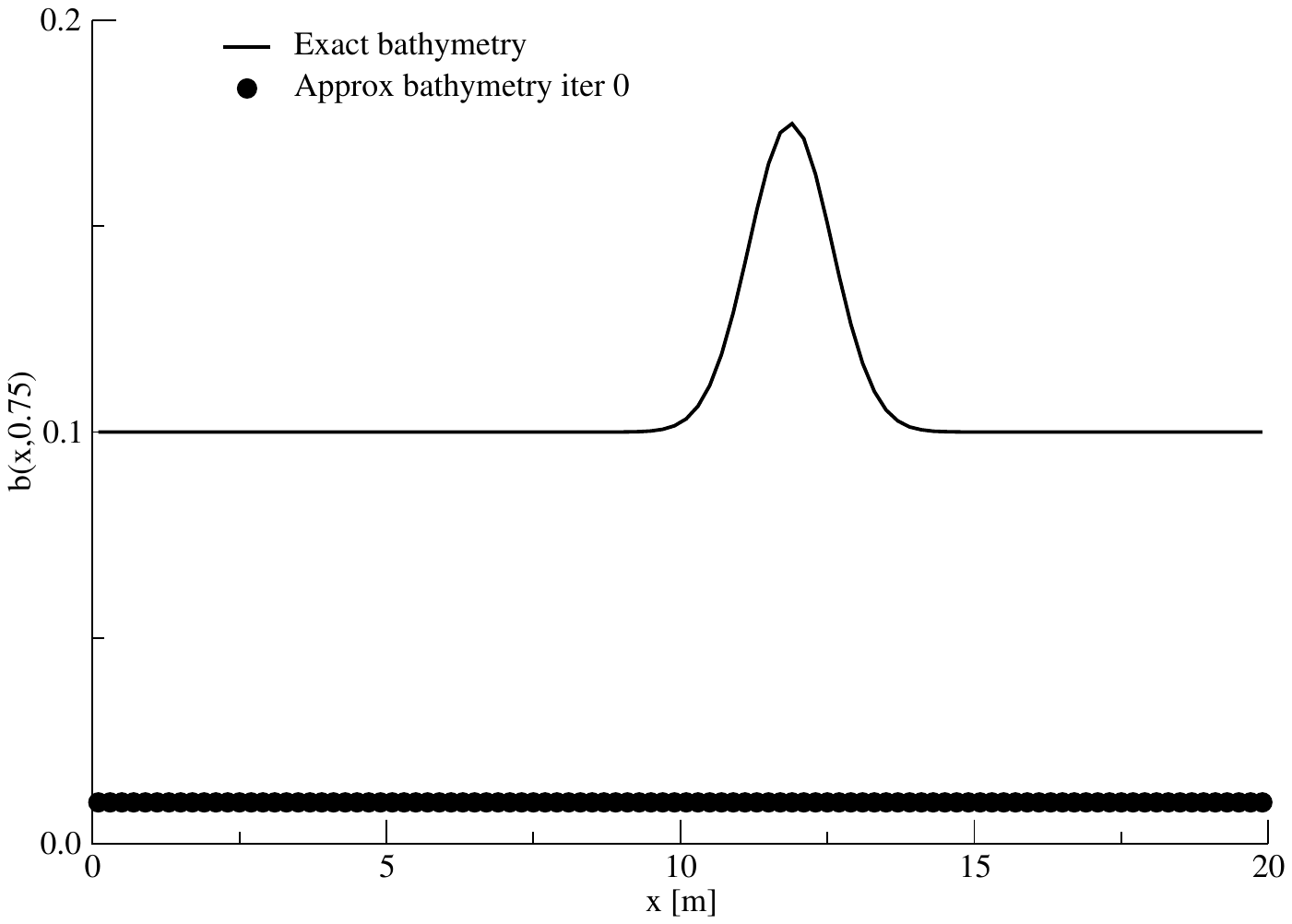} 
	\\
	
	\includegraphics[width=0.3\textwidth]{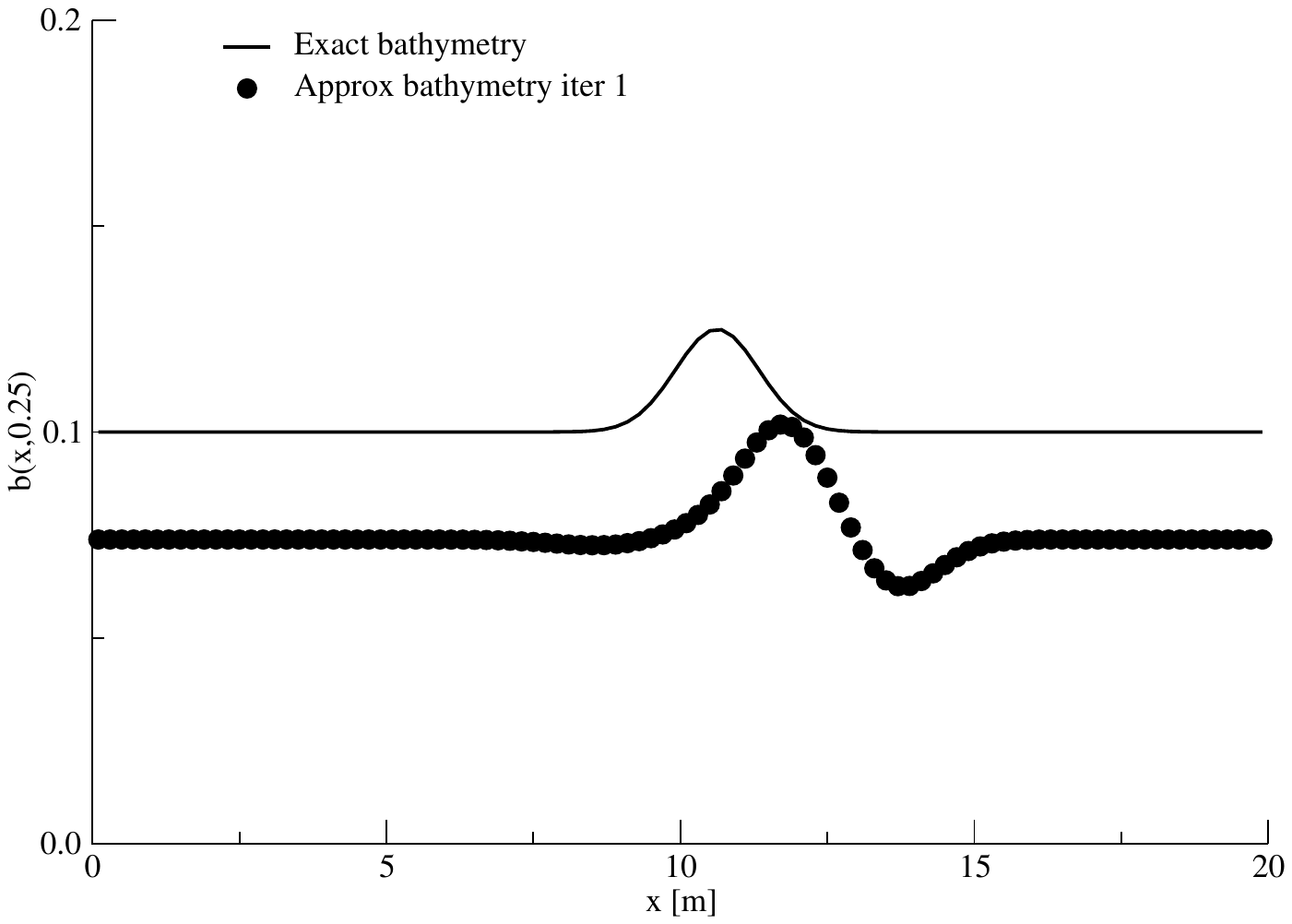} \quad
	\includegraphics[width=0.3\textwidth]{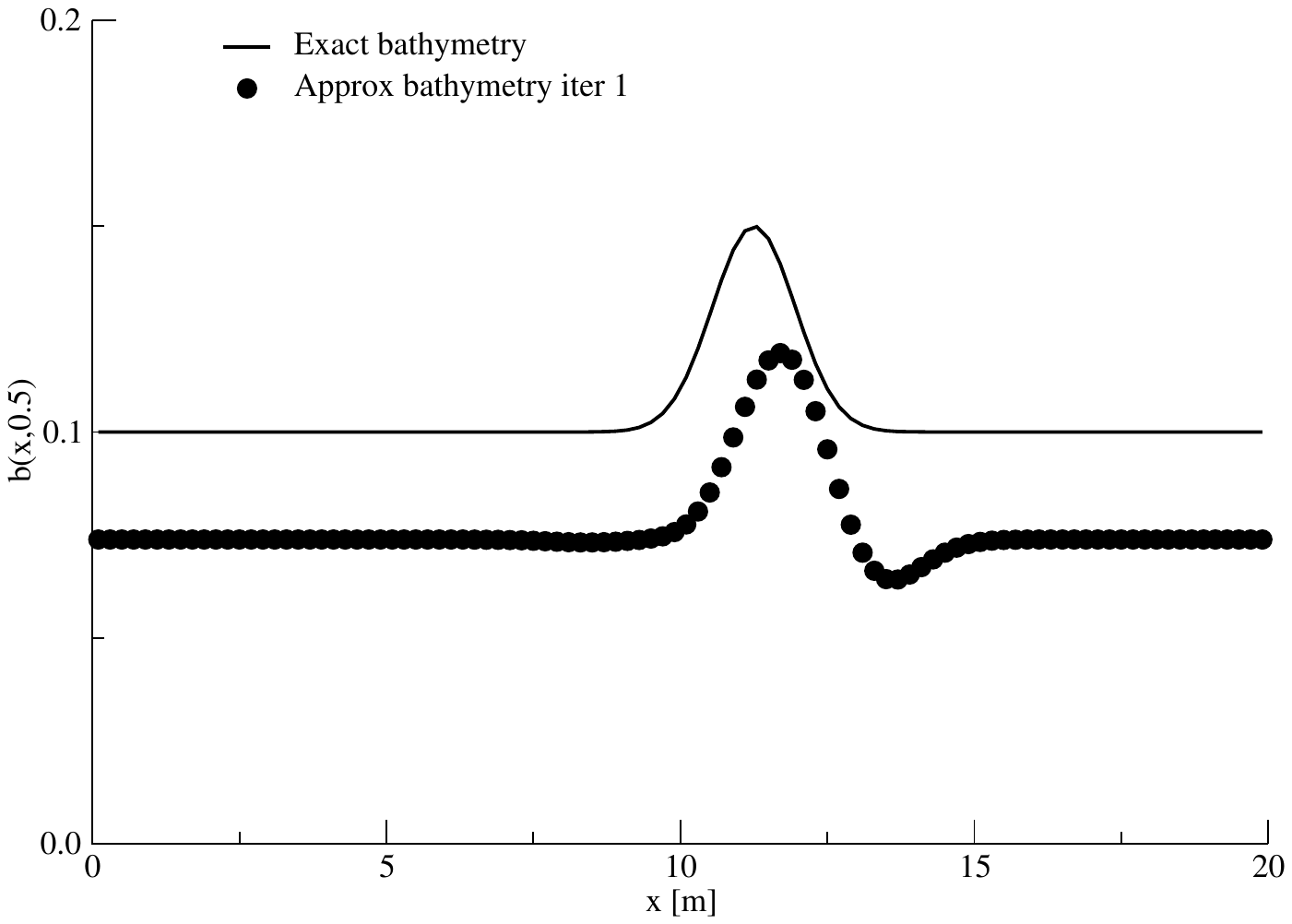} \quad
	\includegraphics[width=0.3\textwidth]{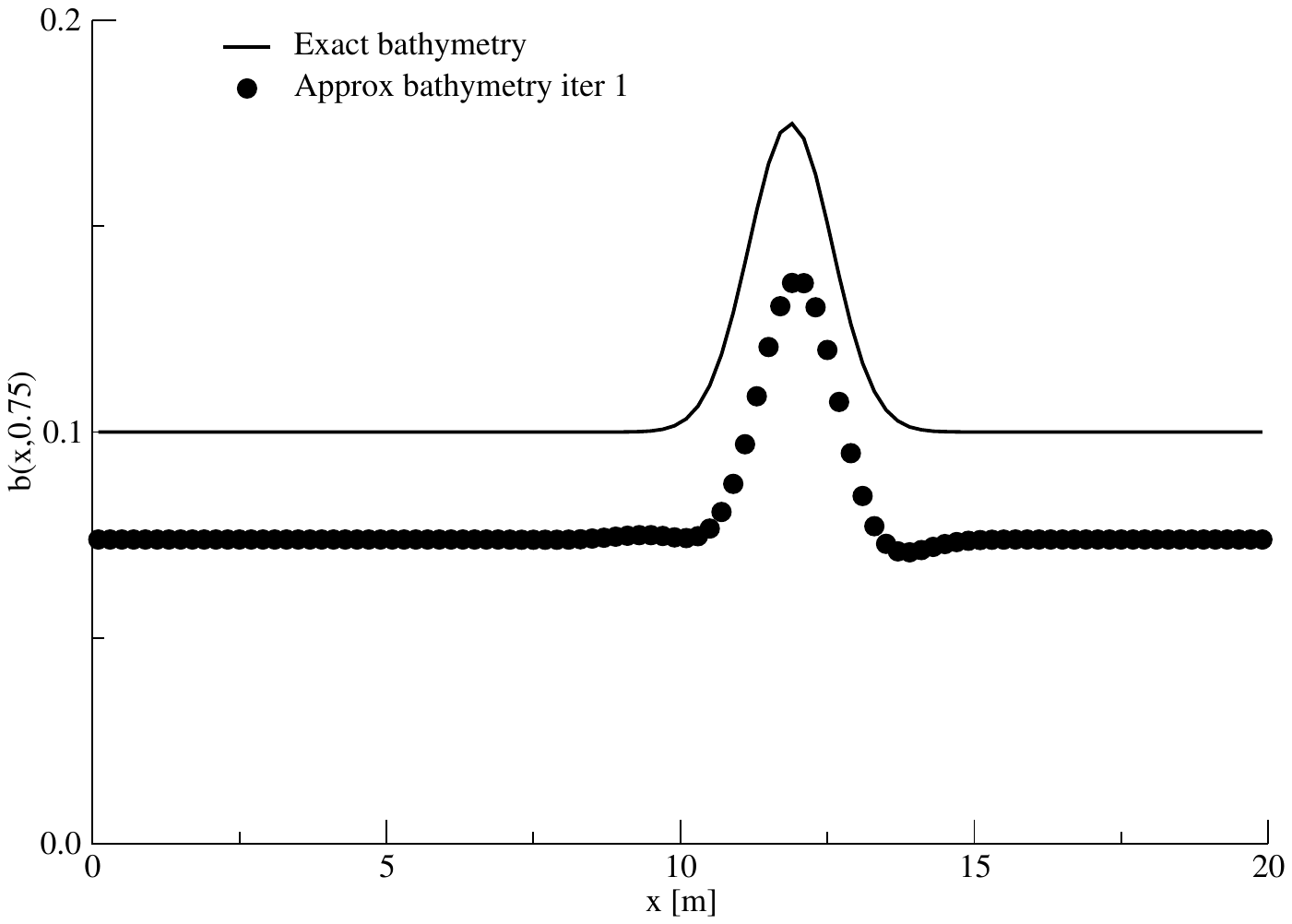} \\
	
	\includegraphics[width=0.3\textwidth]{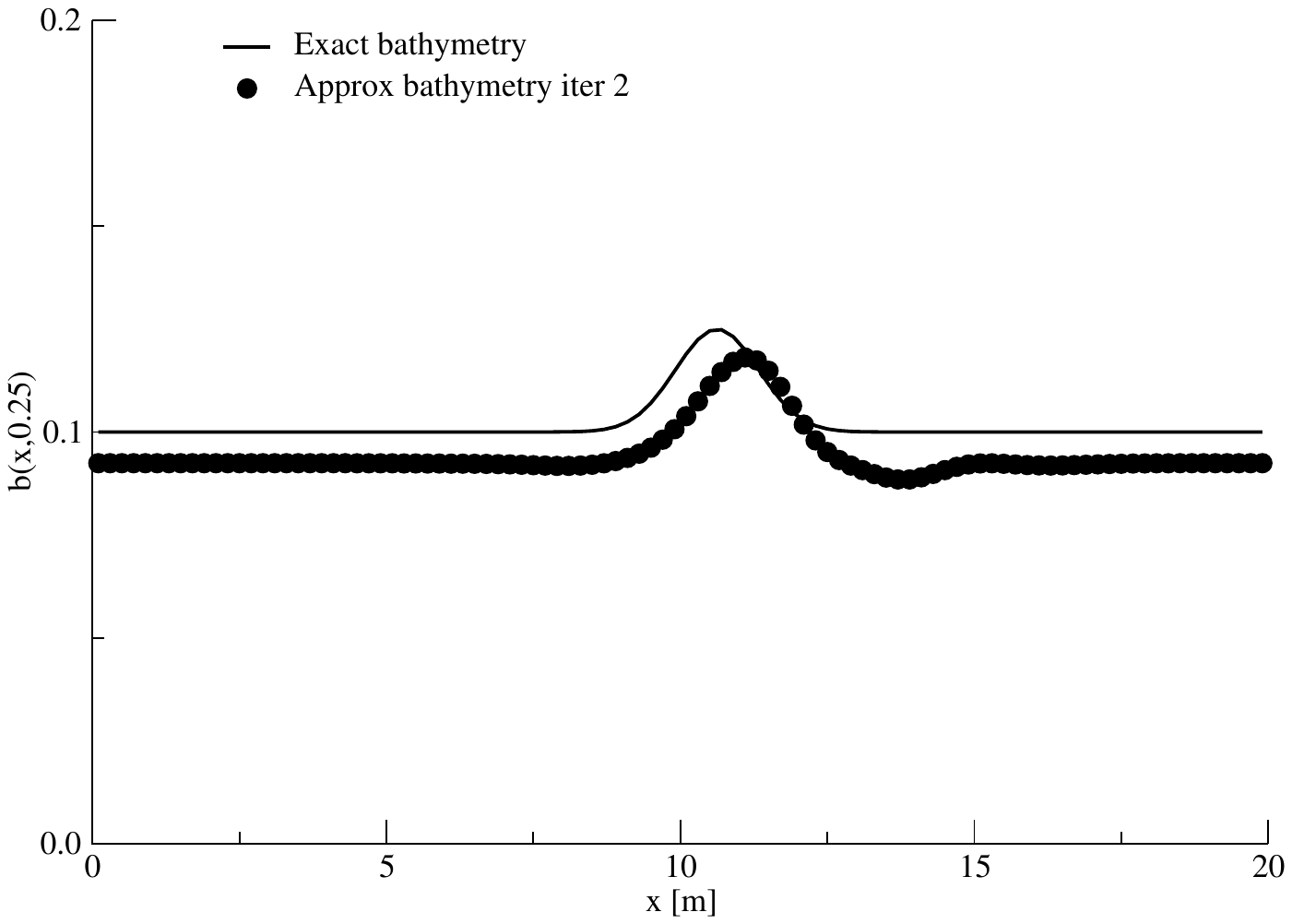} \quad
	\includegraphics[width=0.3\textwidth]{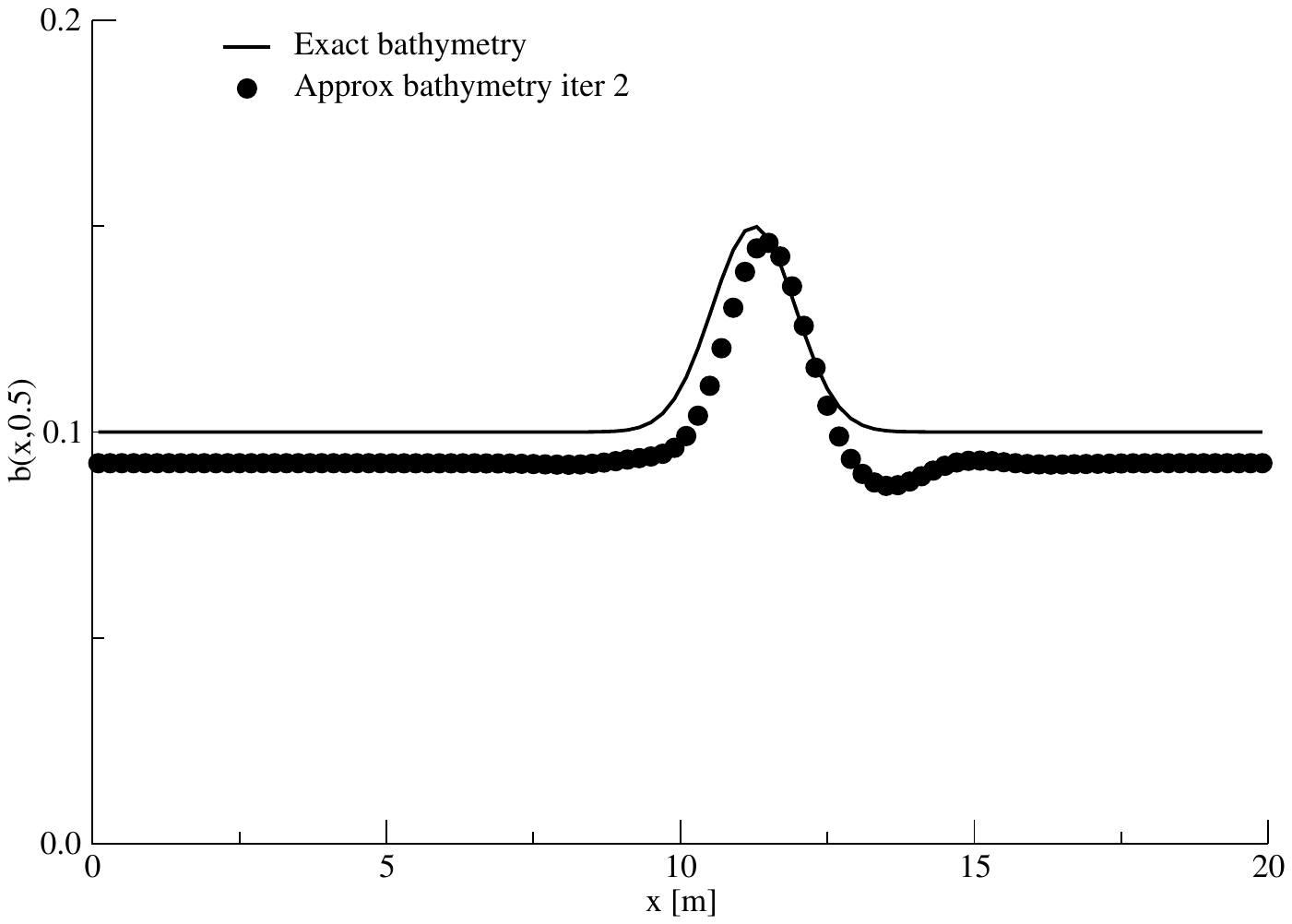} \quad
	\includegraphics[width=0.3\textwidth]{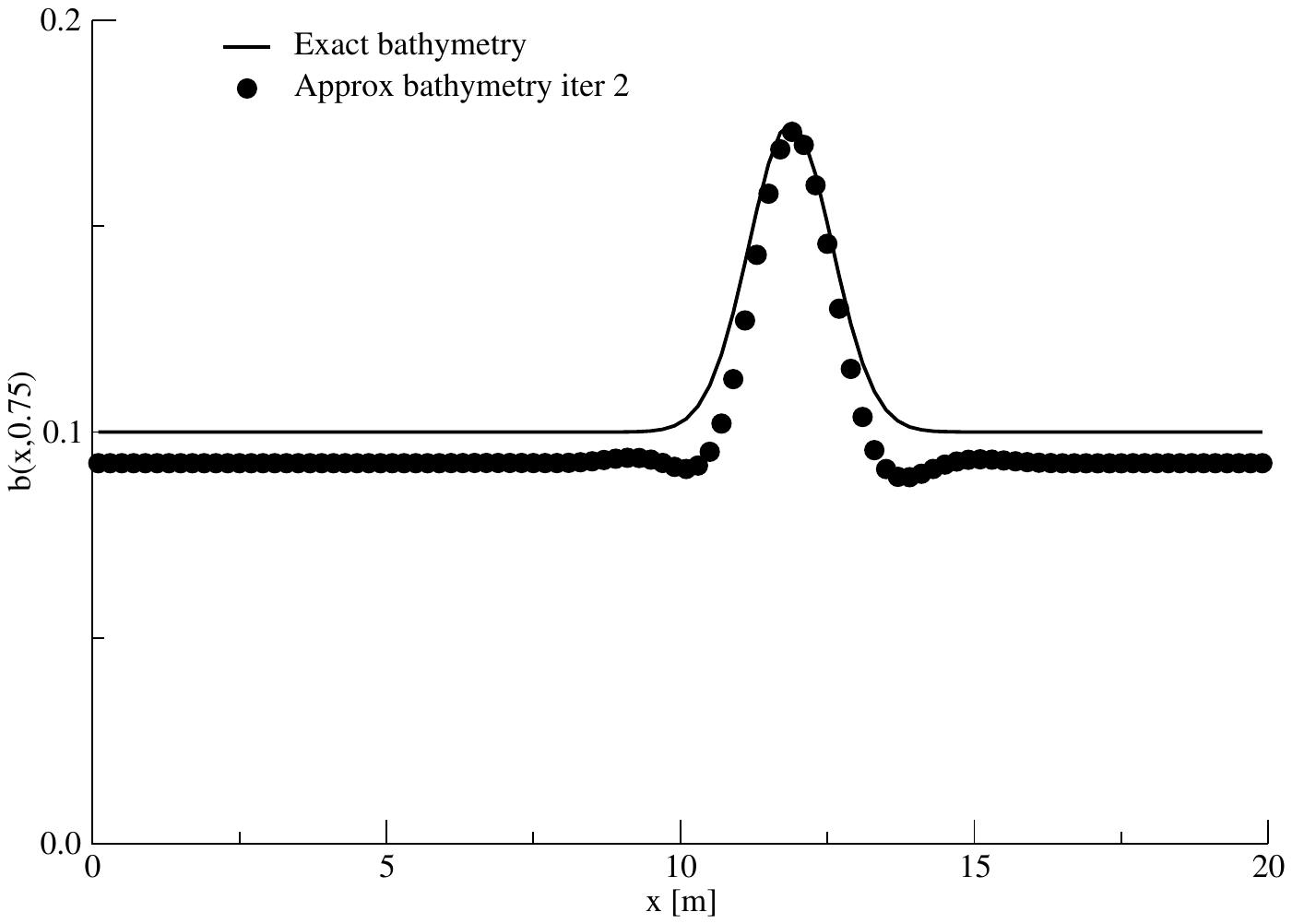} \\

	\includegraphics[width=0.3\textwidth]{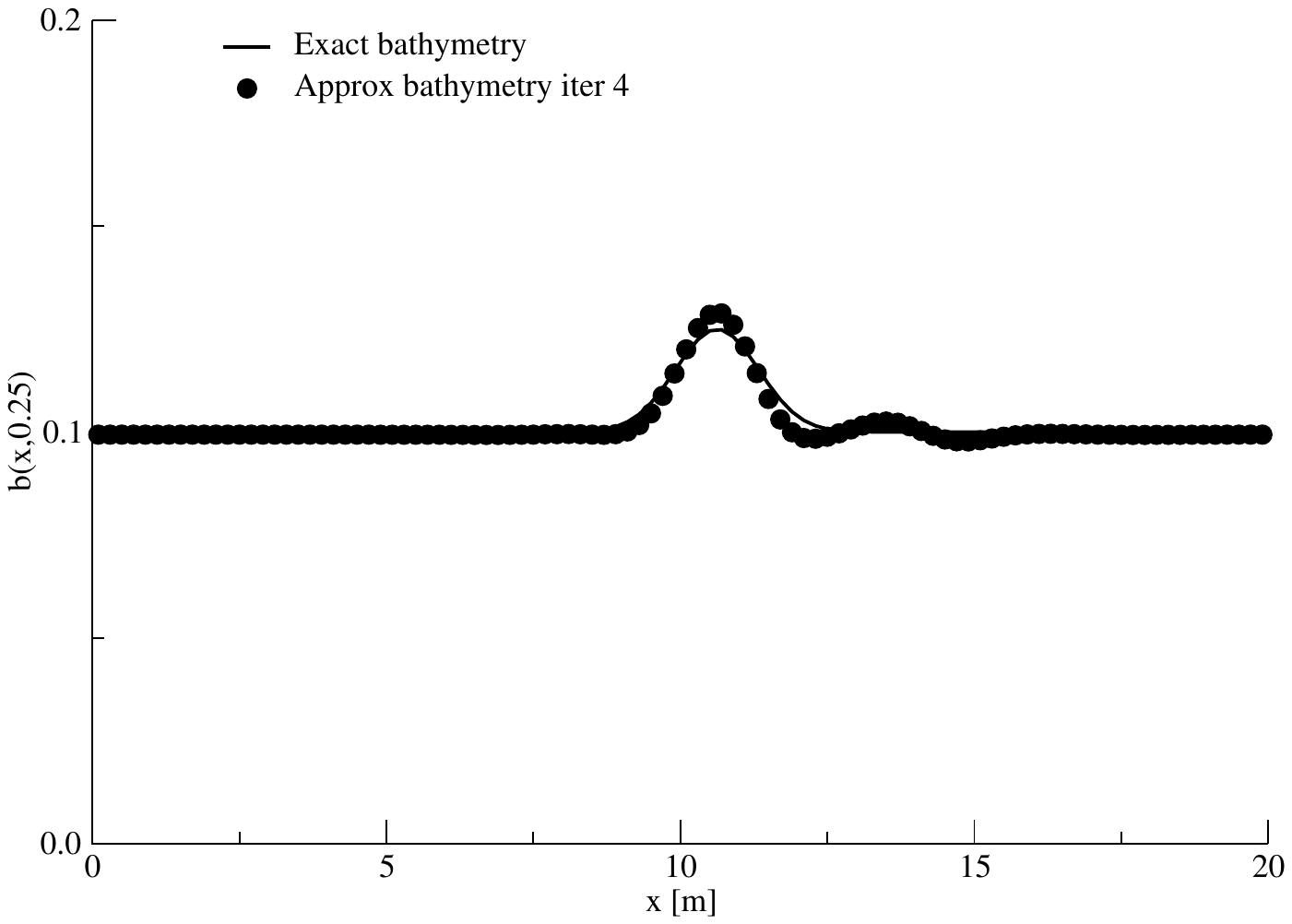} \quad
	\includegraphics[width=0.3\textwidth]{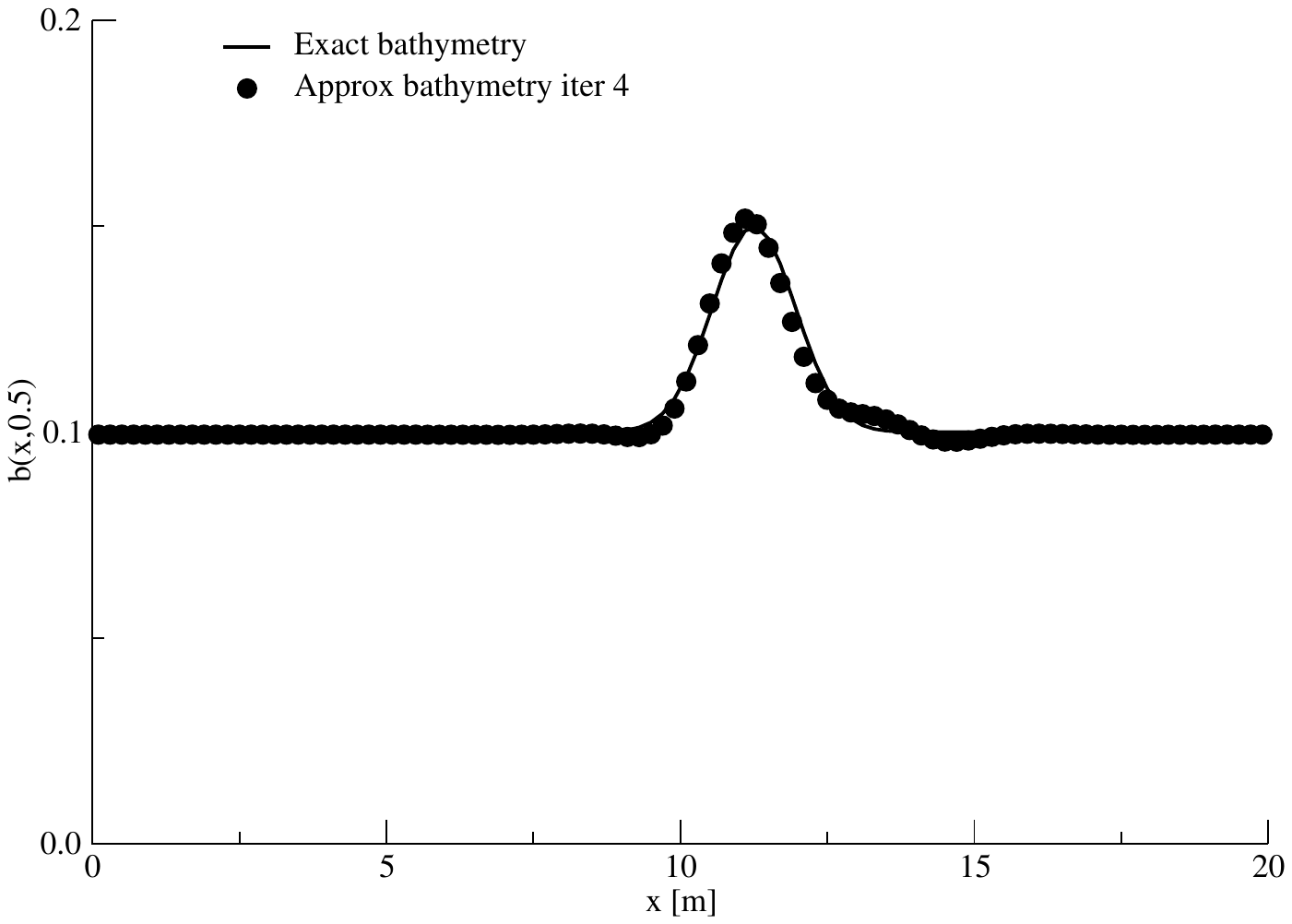} \quad
	\includegraphics[width=0.3\textwidth]{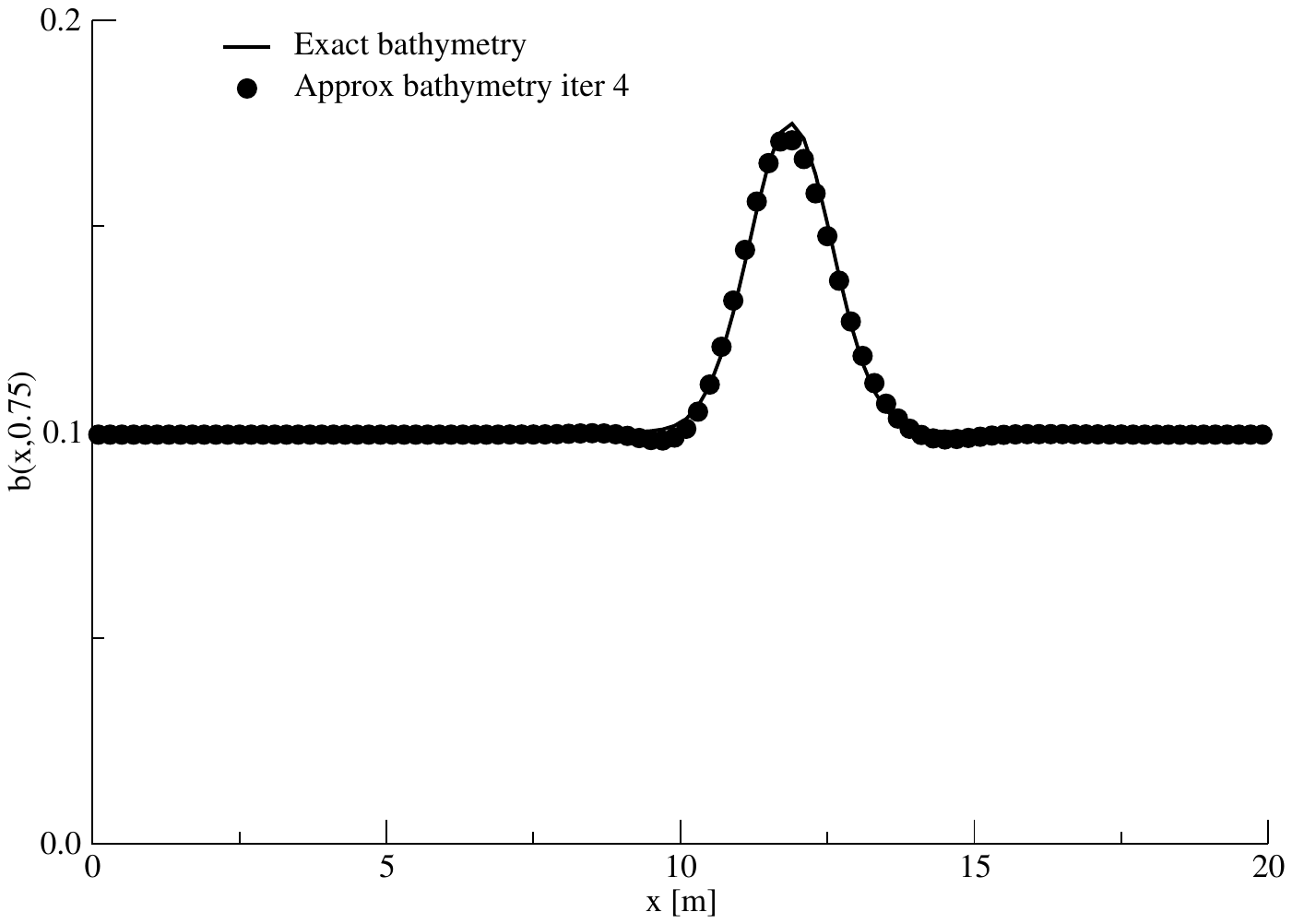} \\
	
	\includegraphics[width=0.3\textwidth]{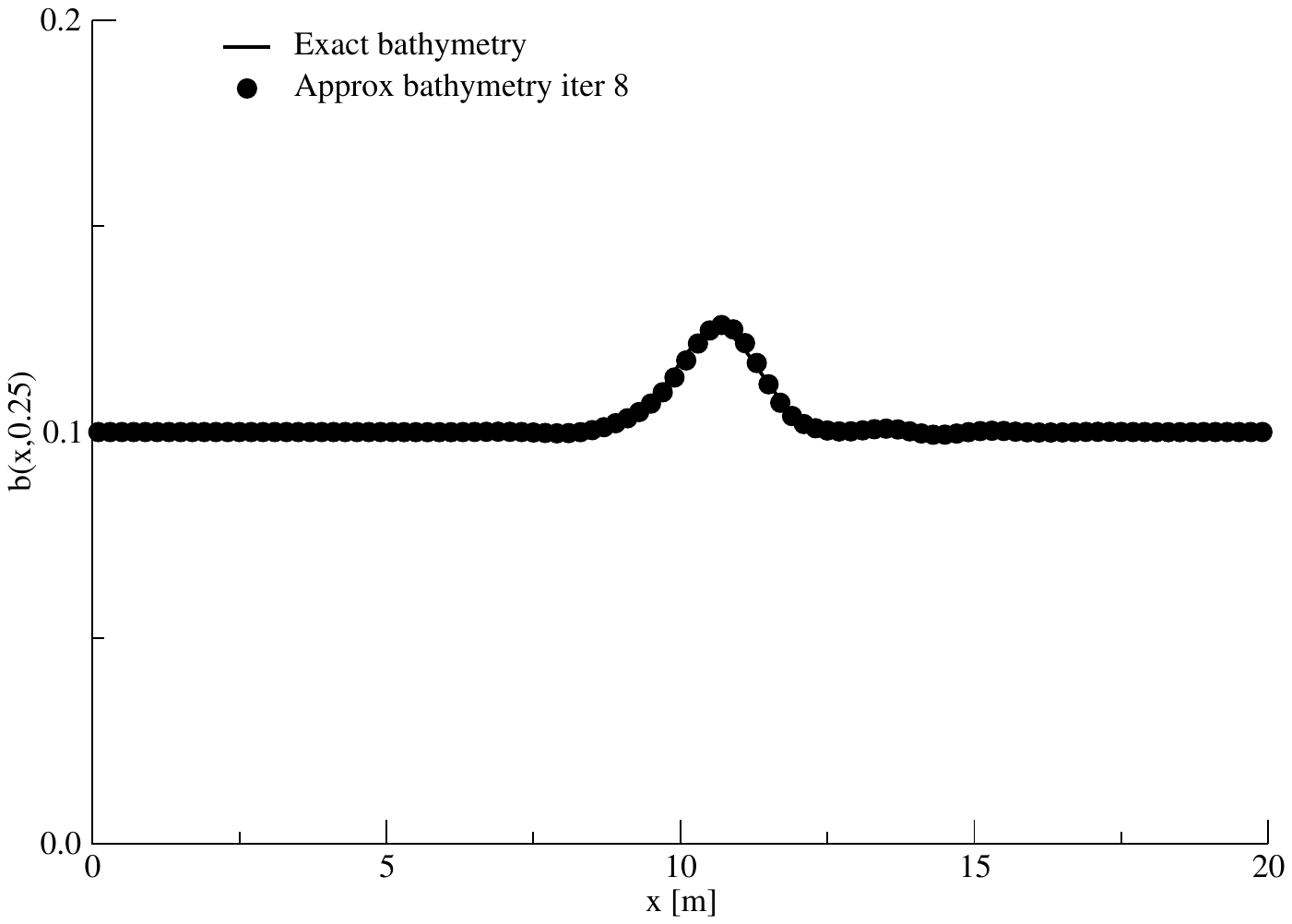} \quad
	\includegraphics[width=0.3\textwidth]{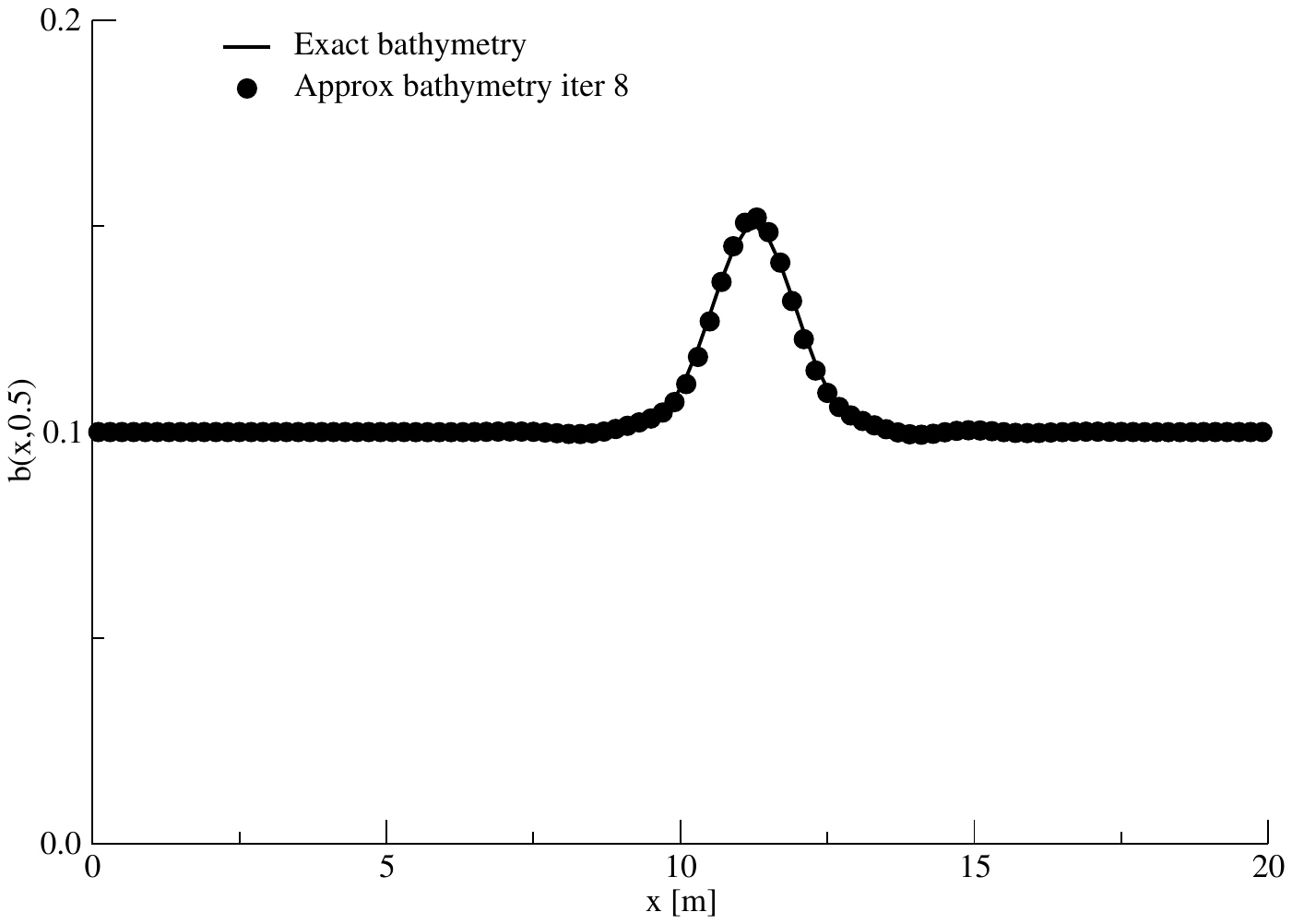} \quad
	\includegraphics[width=0.3\textwidth]{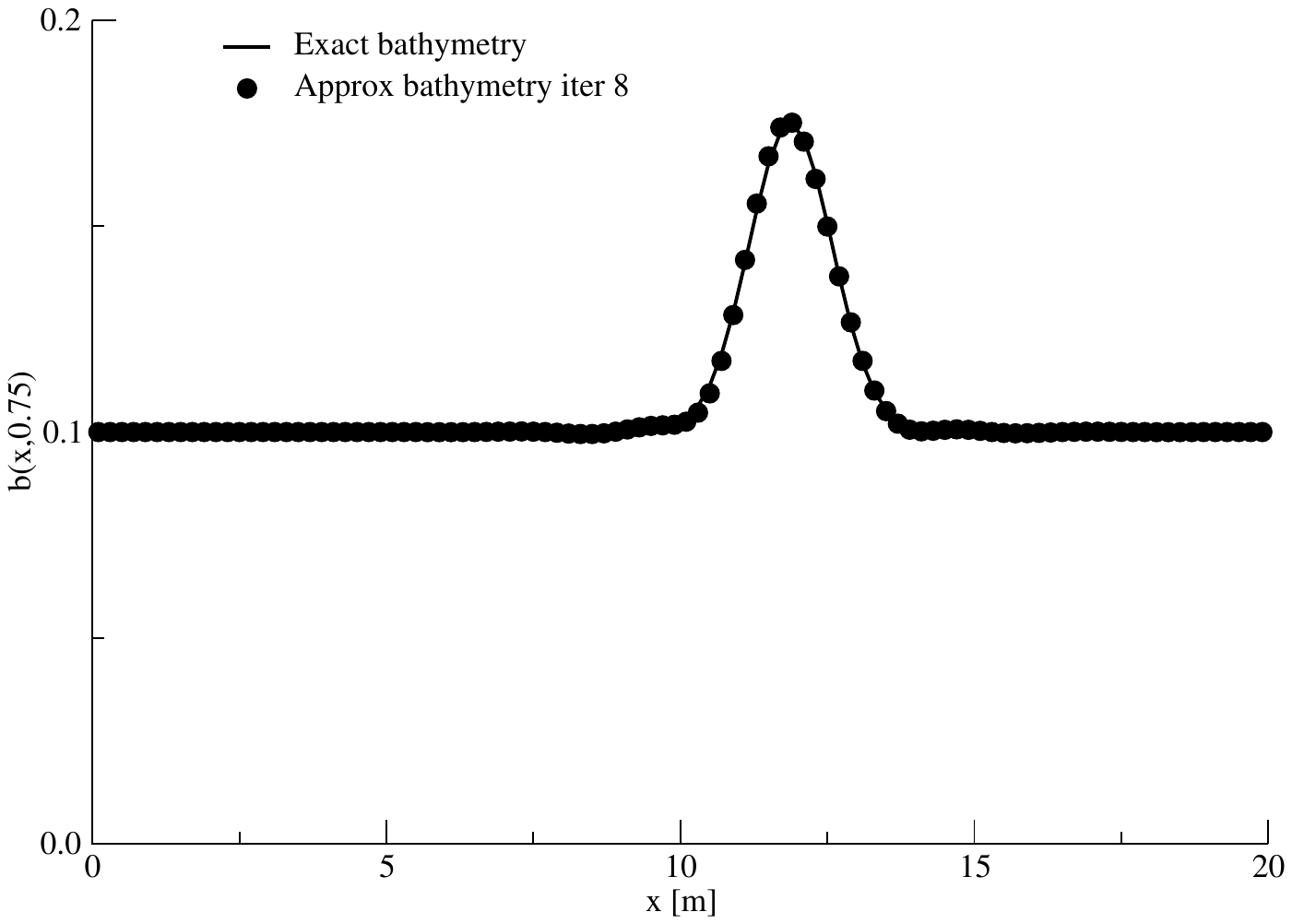} \\
	
	\caption{Smooth bottom profile (\ref{eq:b-smooth}): Result for the reconstruction procedure resulting from {\bf Rusanov+FD}. Parameters $\Delta t = 0.01$,  $\alpha_F = 2$, $\varepsilon = 0.001$, $\lambda_b = 0.71$, $100$ cells.
		{\bf Feft:} $t=0.25$, {\bf centered} $t=0.5$, {\bf right:} $t=0.75$.}
	\label{fig:b-for-iter-and-times:test-0-ConsRusanov}
\end{figure}


\begin{figure}   
	\centering
	\includegraphics[width=0.3\textwidth]{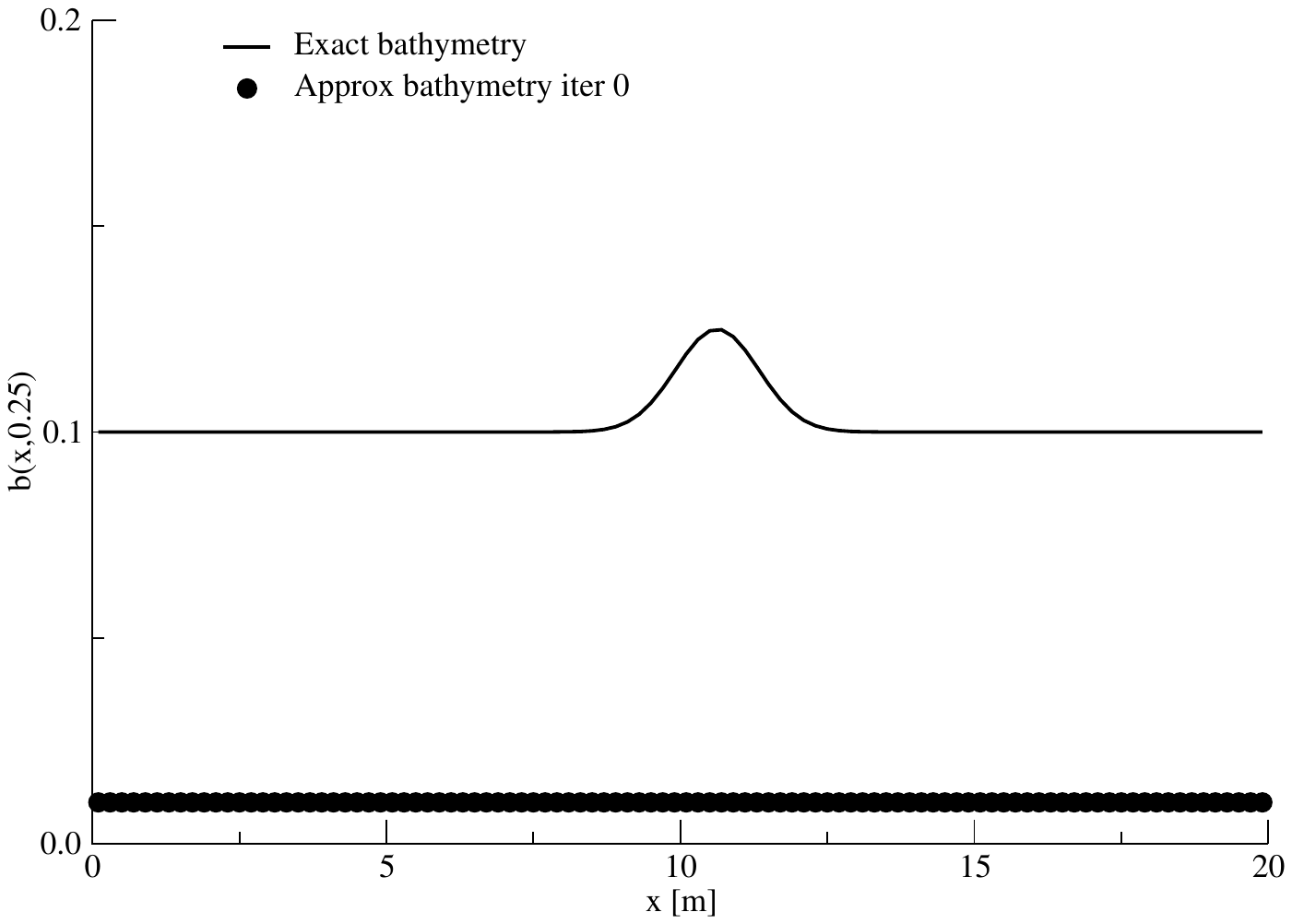} \quad
	\includegraphics[width=0.3\textwidth]{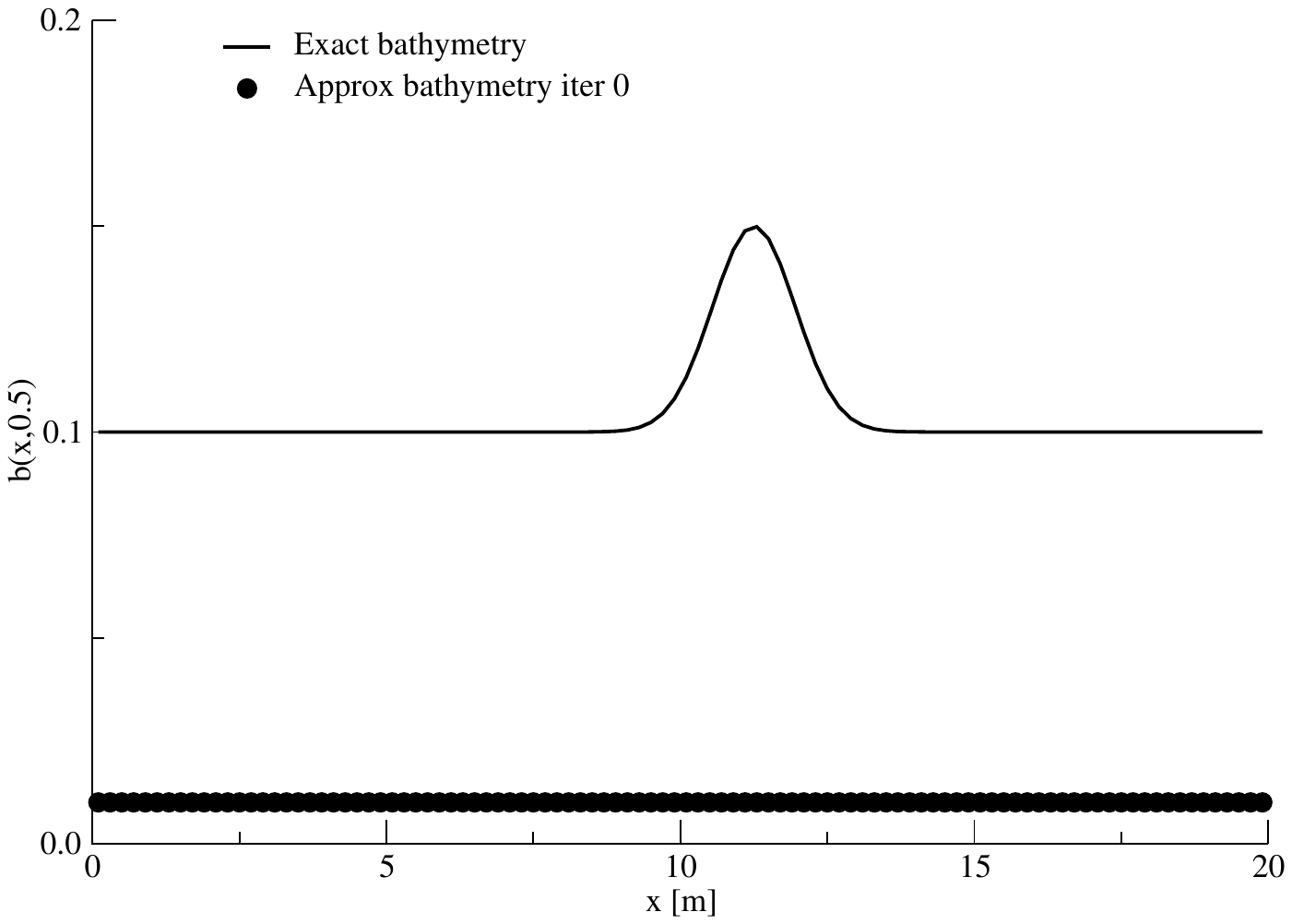} \quad
	\includegraphics[width=0.3\textwidth]{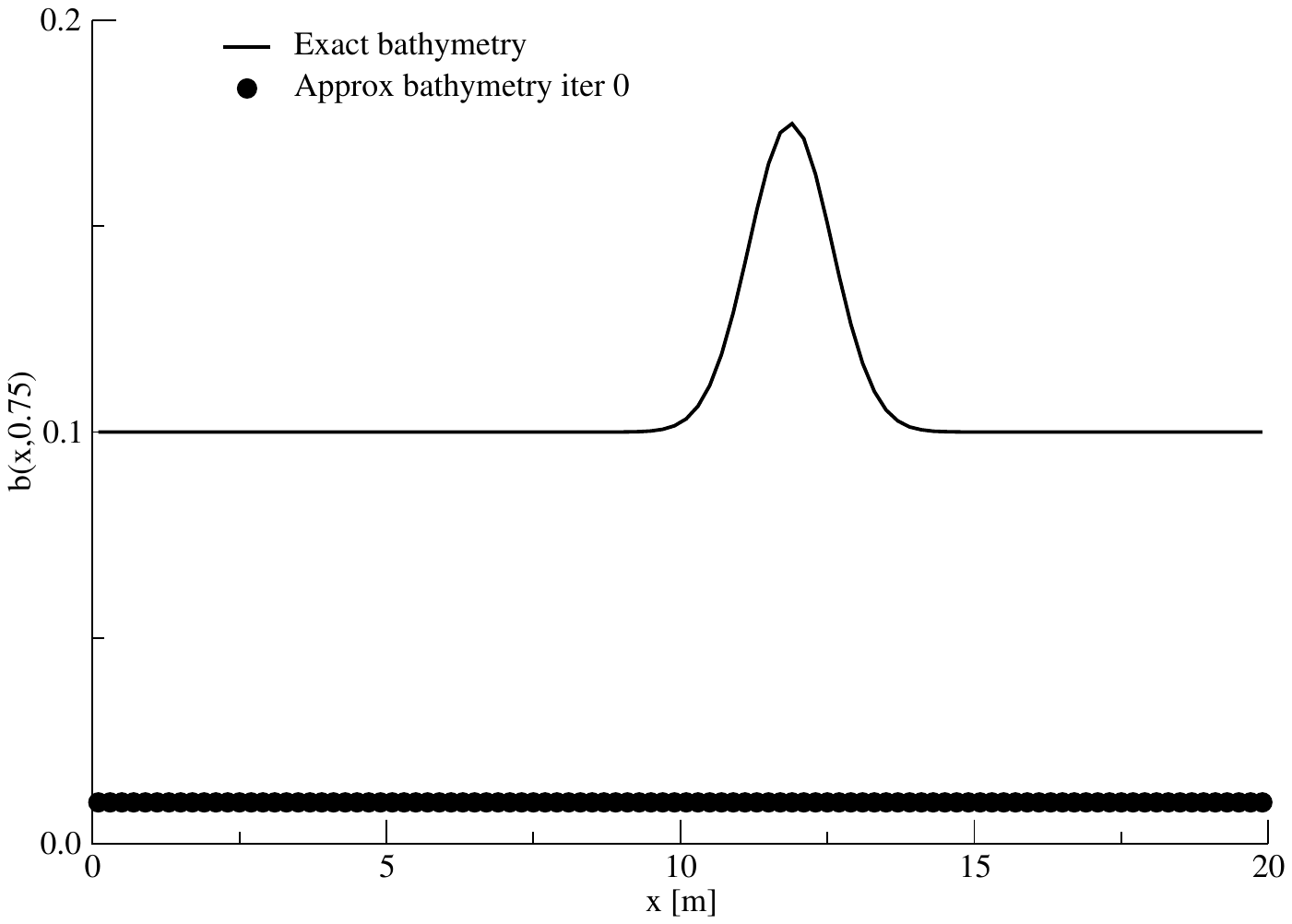} 
	\\
	
	\includegraphics[width=0.3\textwidth]{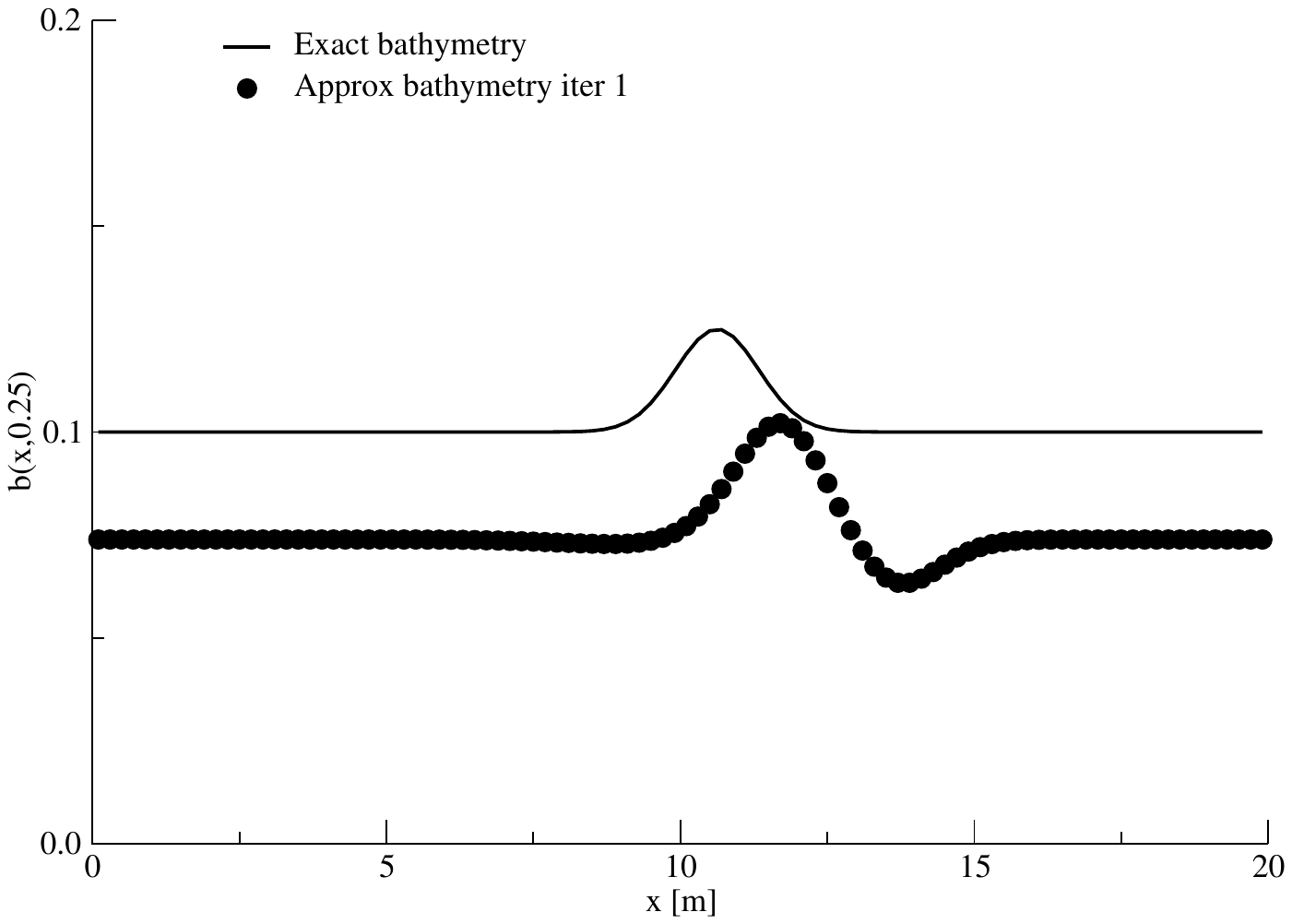} \quad
	\includegraphics[width=0.3\textwidth]{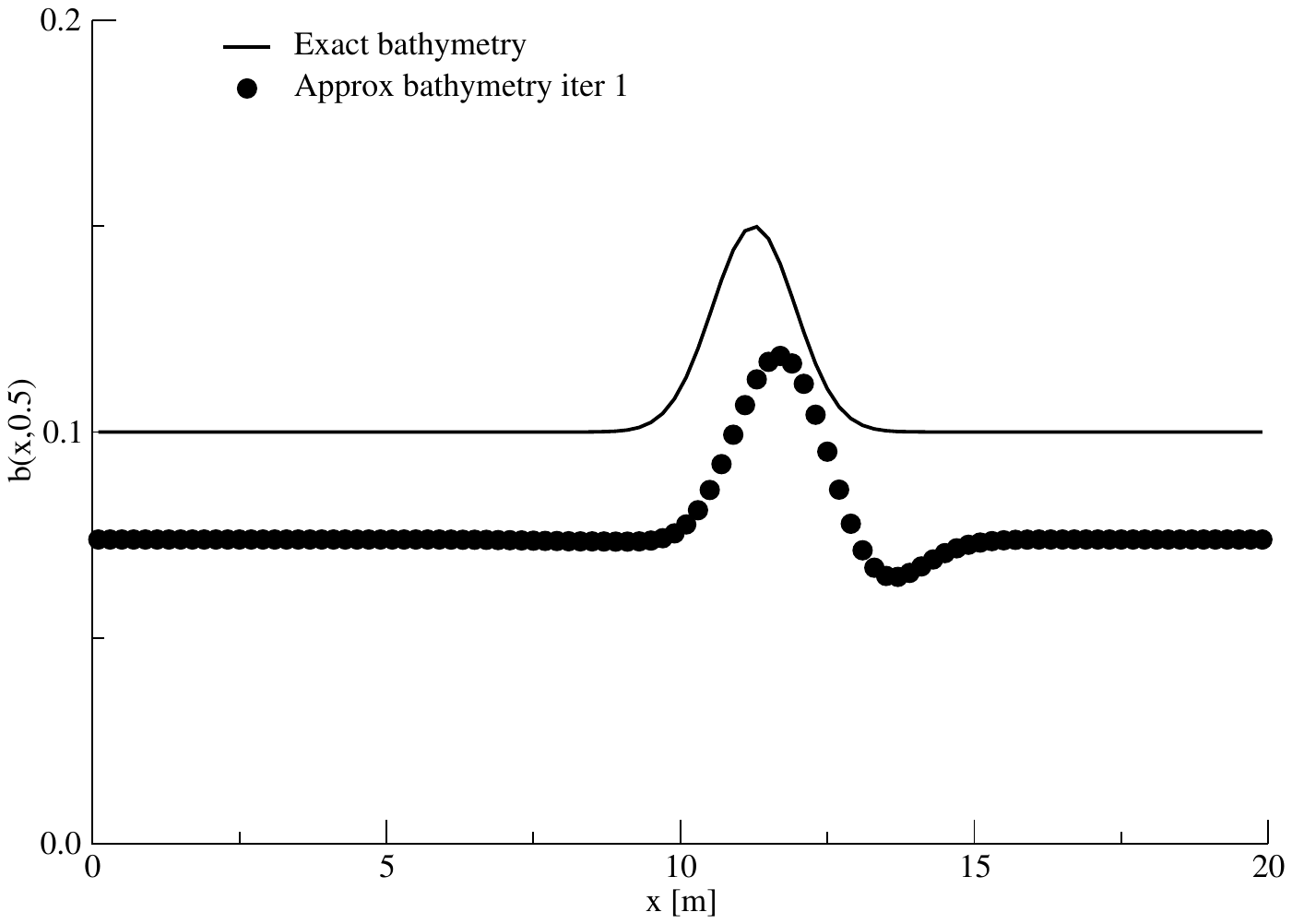} \quad
	\includegraphics[width=0.3\textwidth]{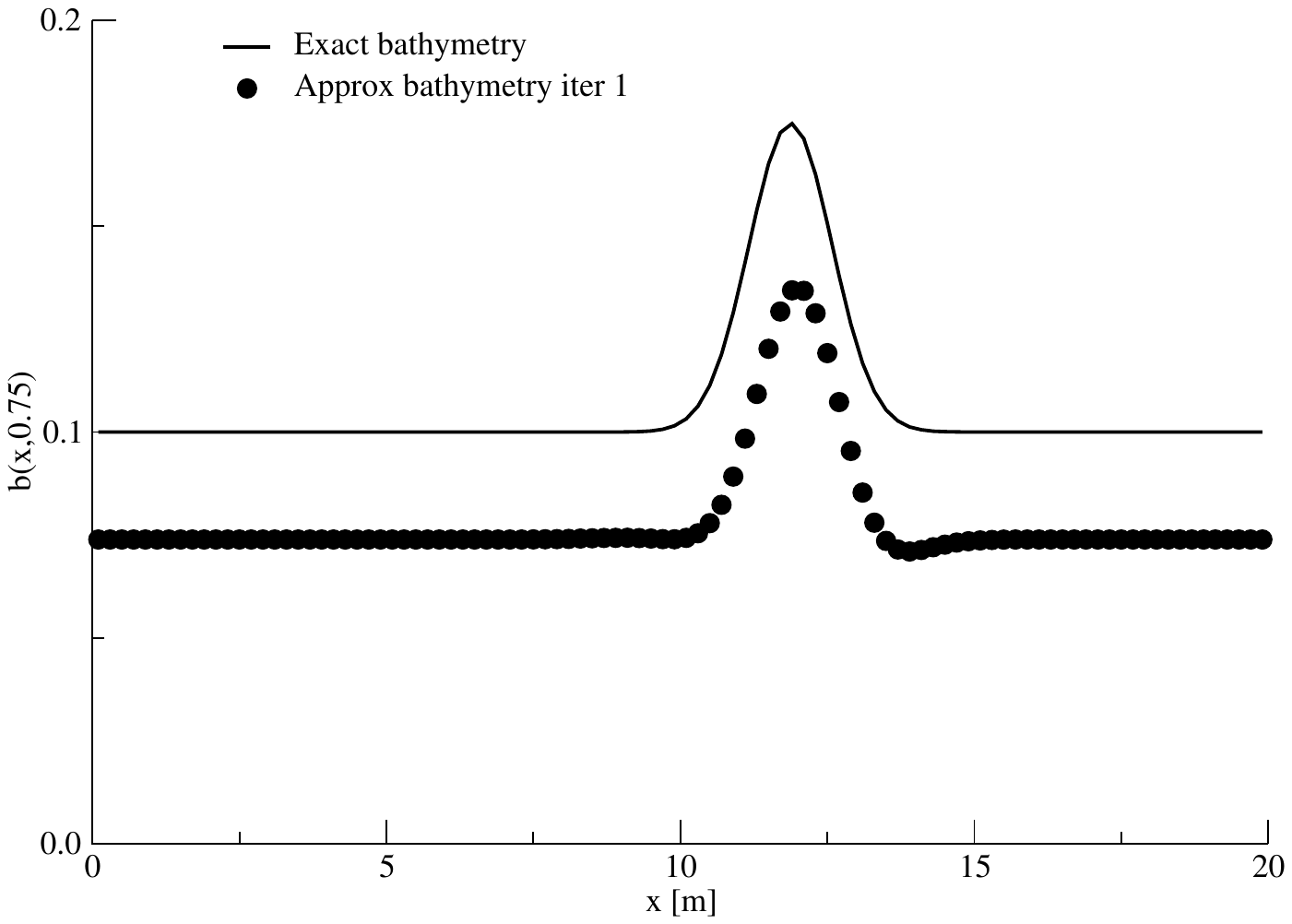} \\
	
	\includegraphics[width=0.3\textwidth]{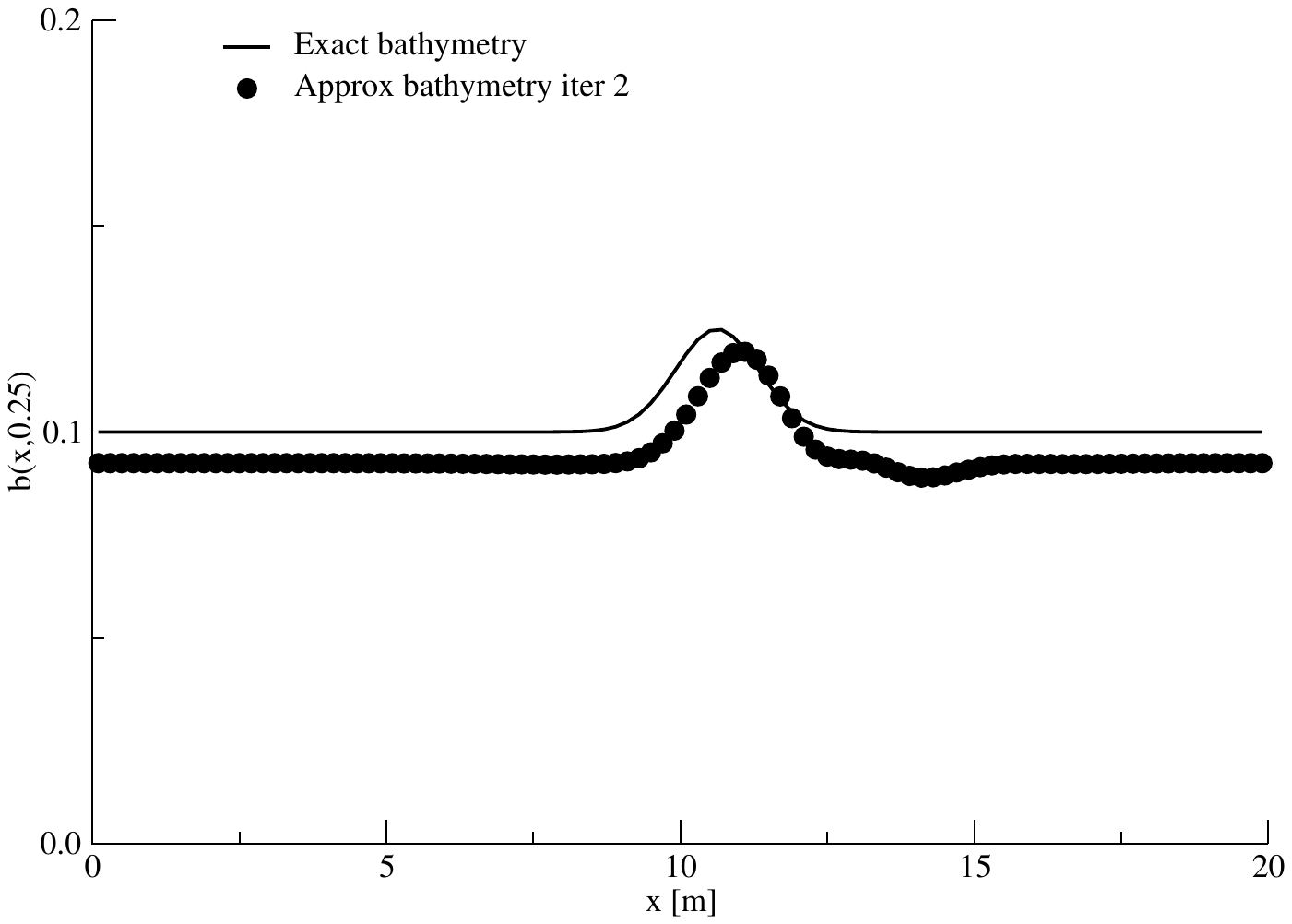} \quad
	\includegraphics[width=0.3\textwidth]{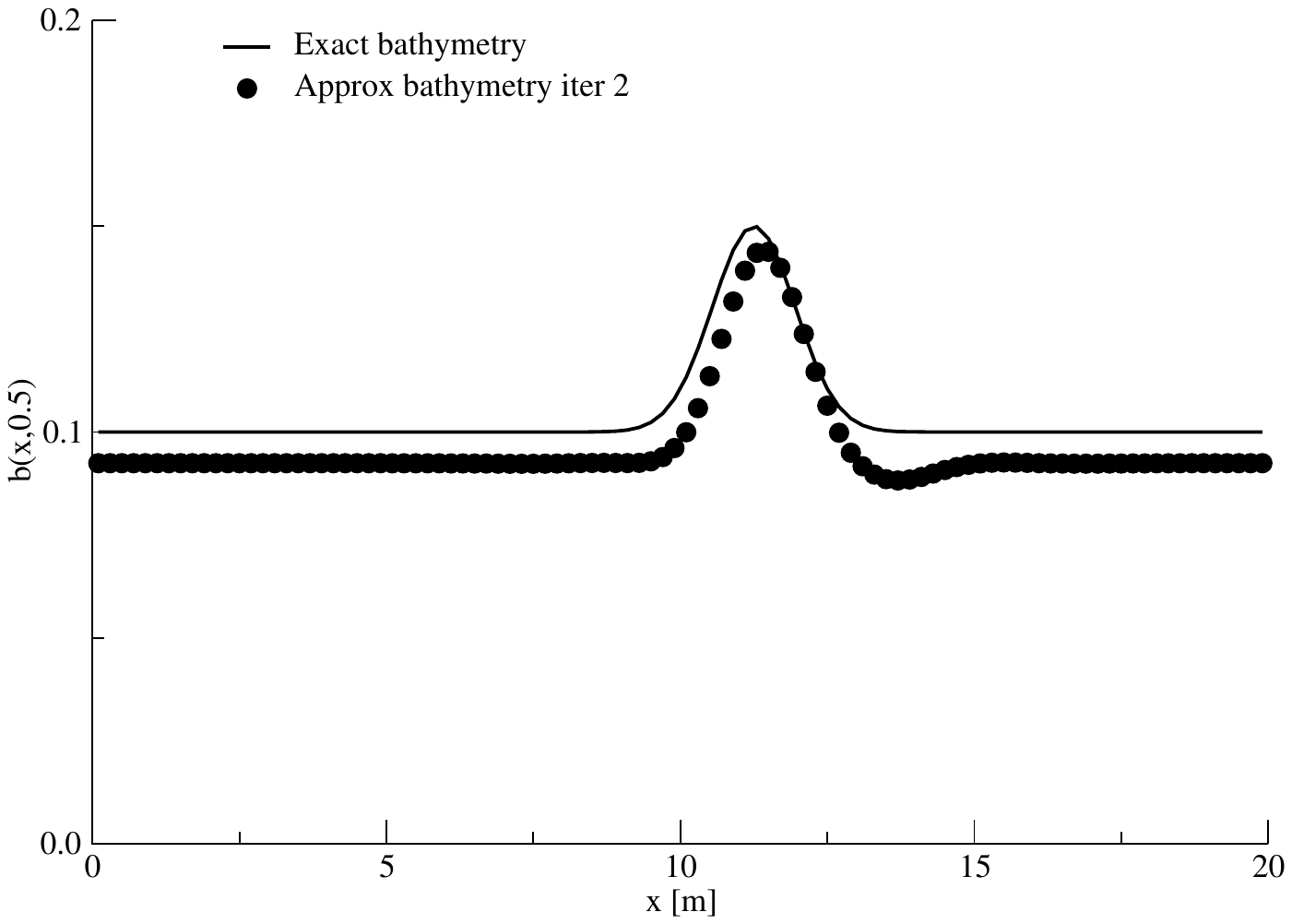} \quad
	\includegraphics[width=0.3\textwidth]{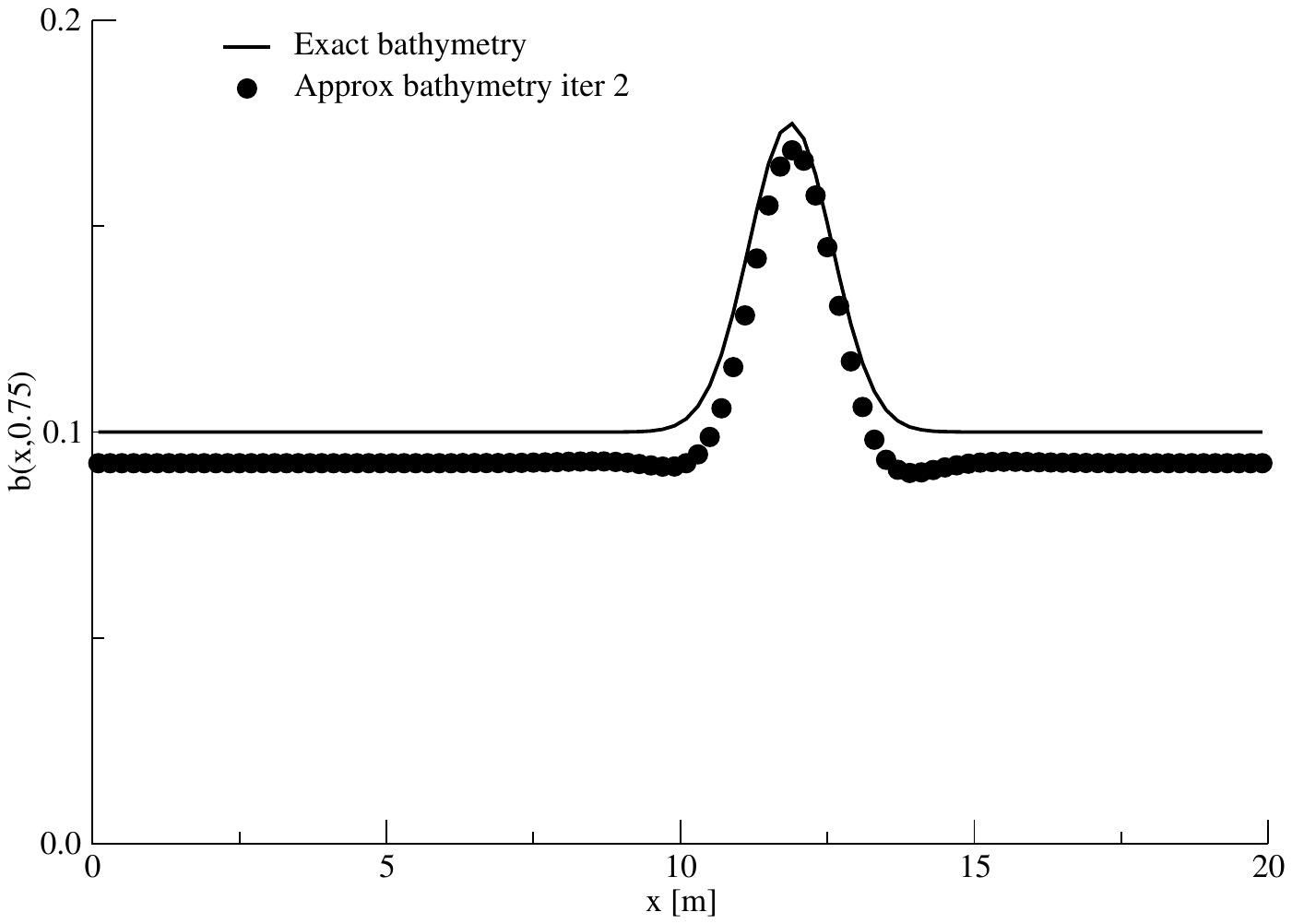} \\

	\includegraphics[width=0.3\textwidth]{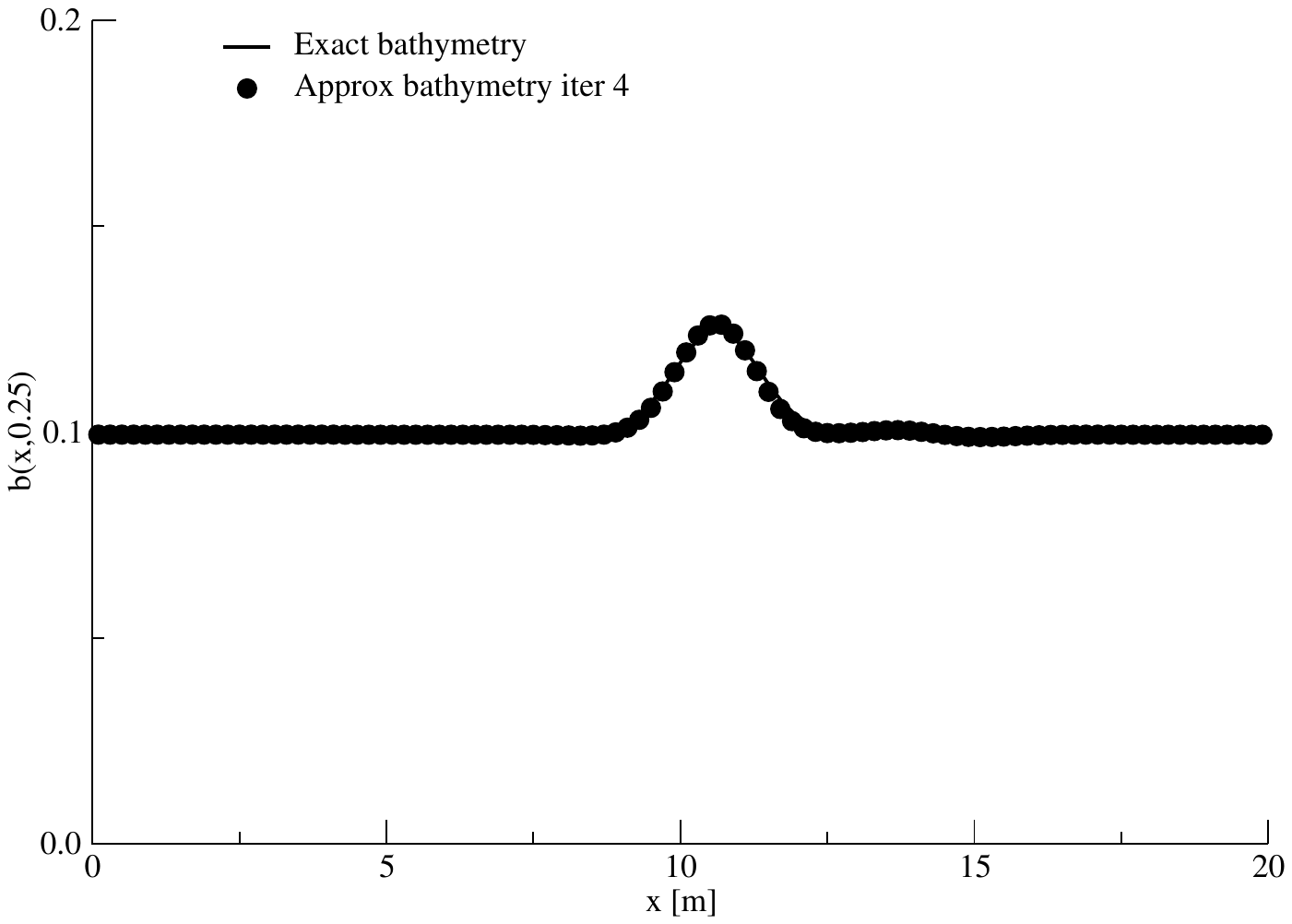} \quad
	\includegraphics[width=0.3\textwidth]{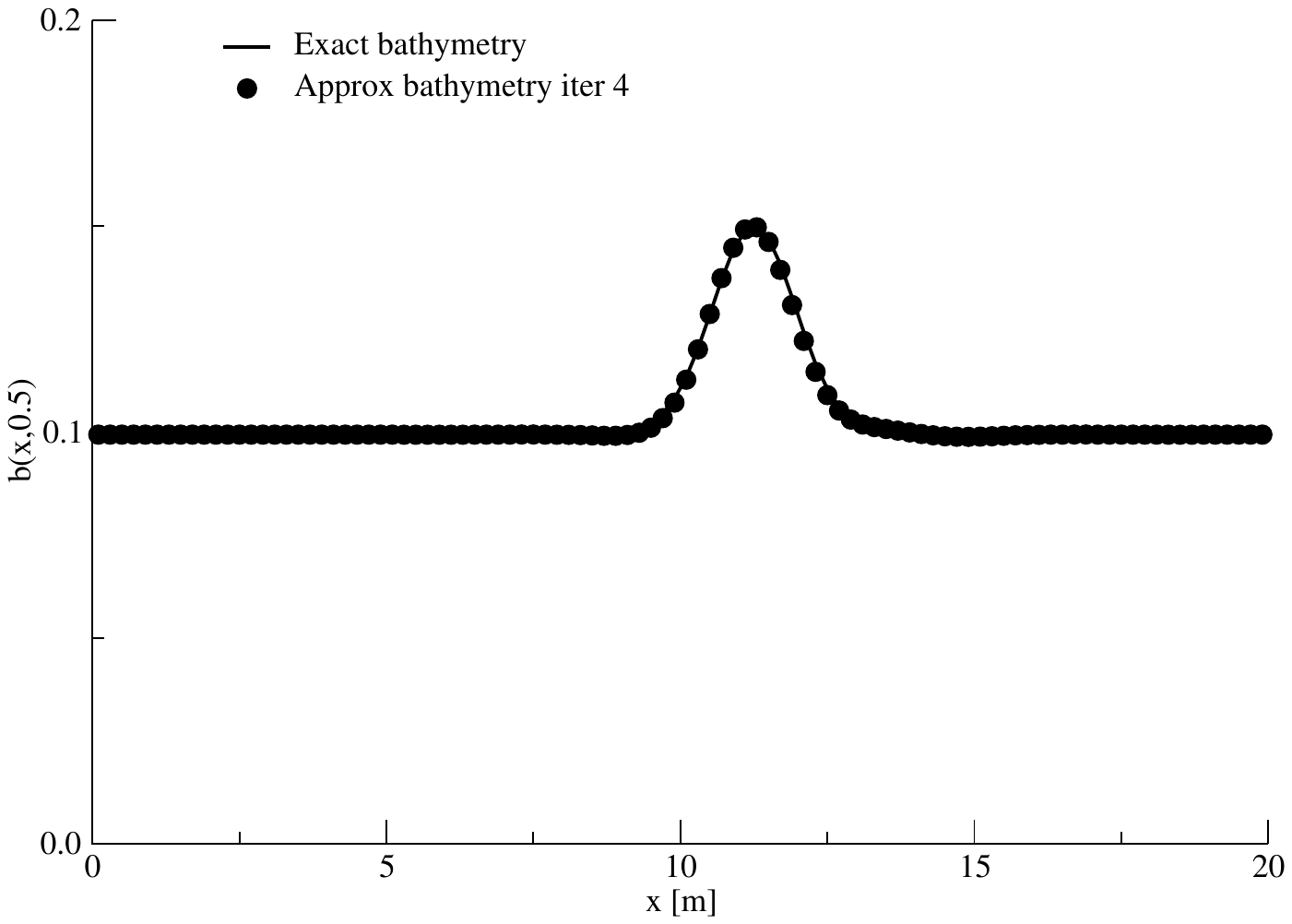} \quad
	\includegraphics[width=0.3\textwidth]{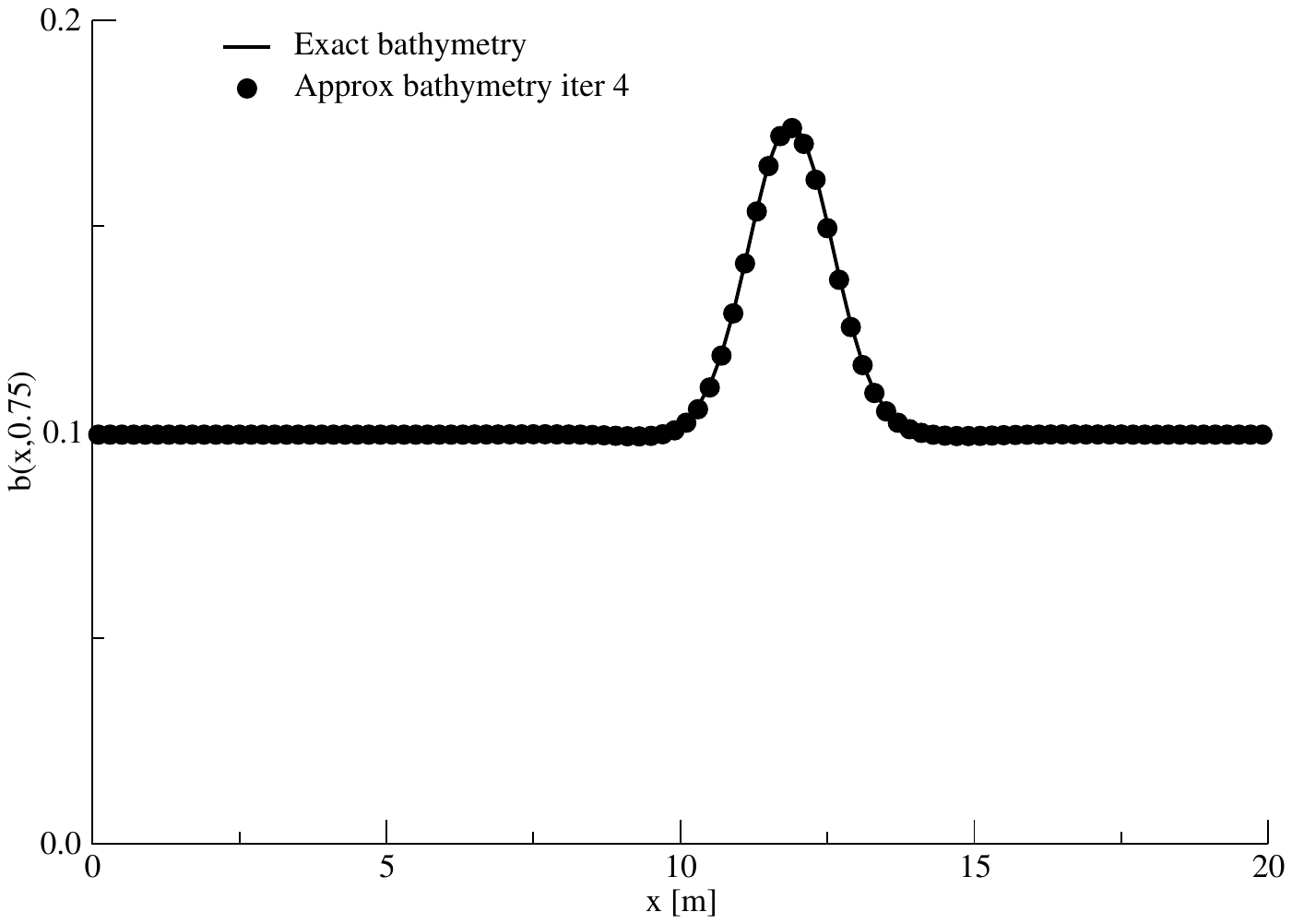} \\
	
	\includegraphics[width=0.3\textwidth]{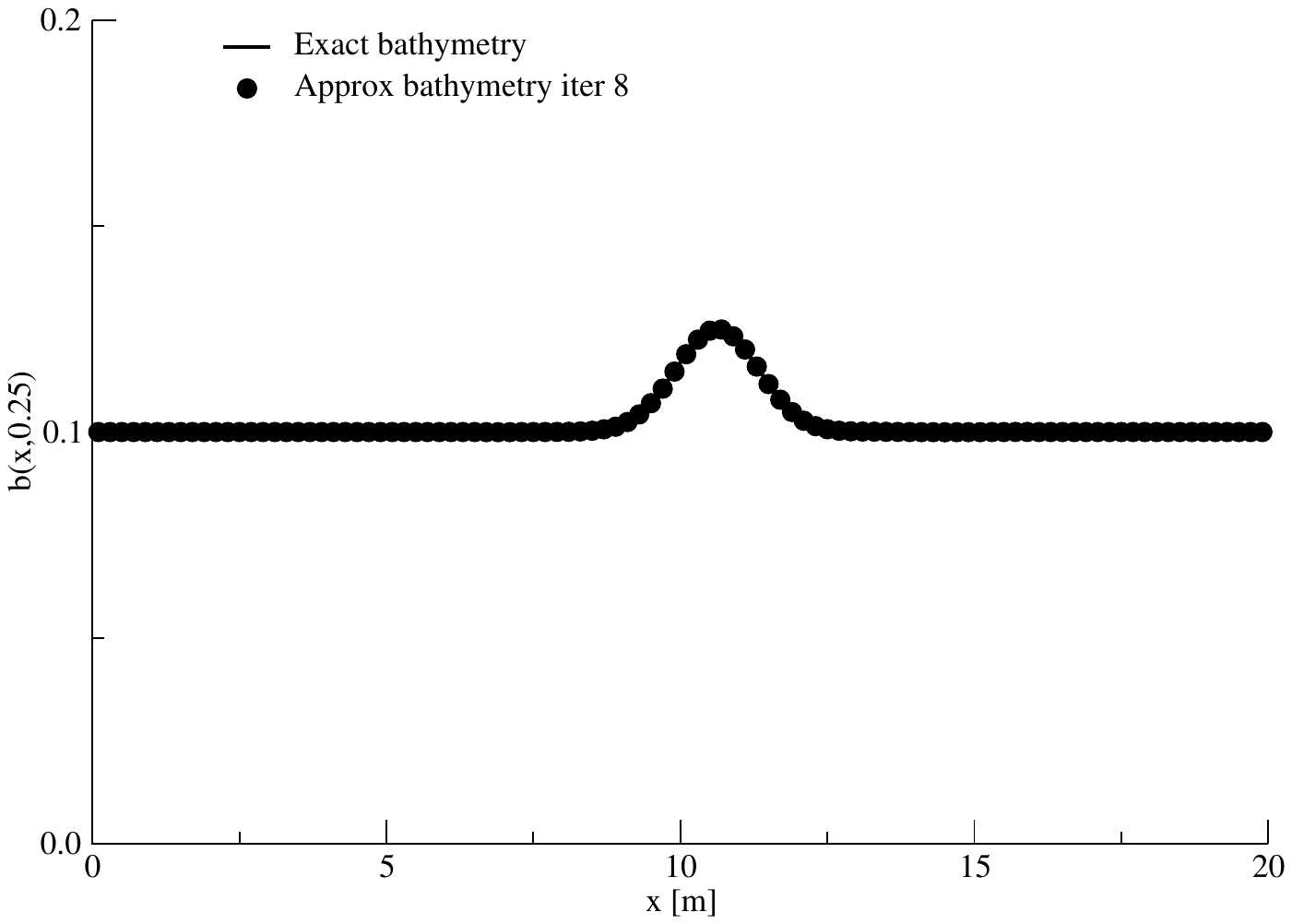} \quad
	\includegraphics[width=0.3\textwidth]{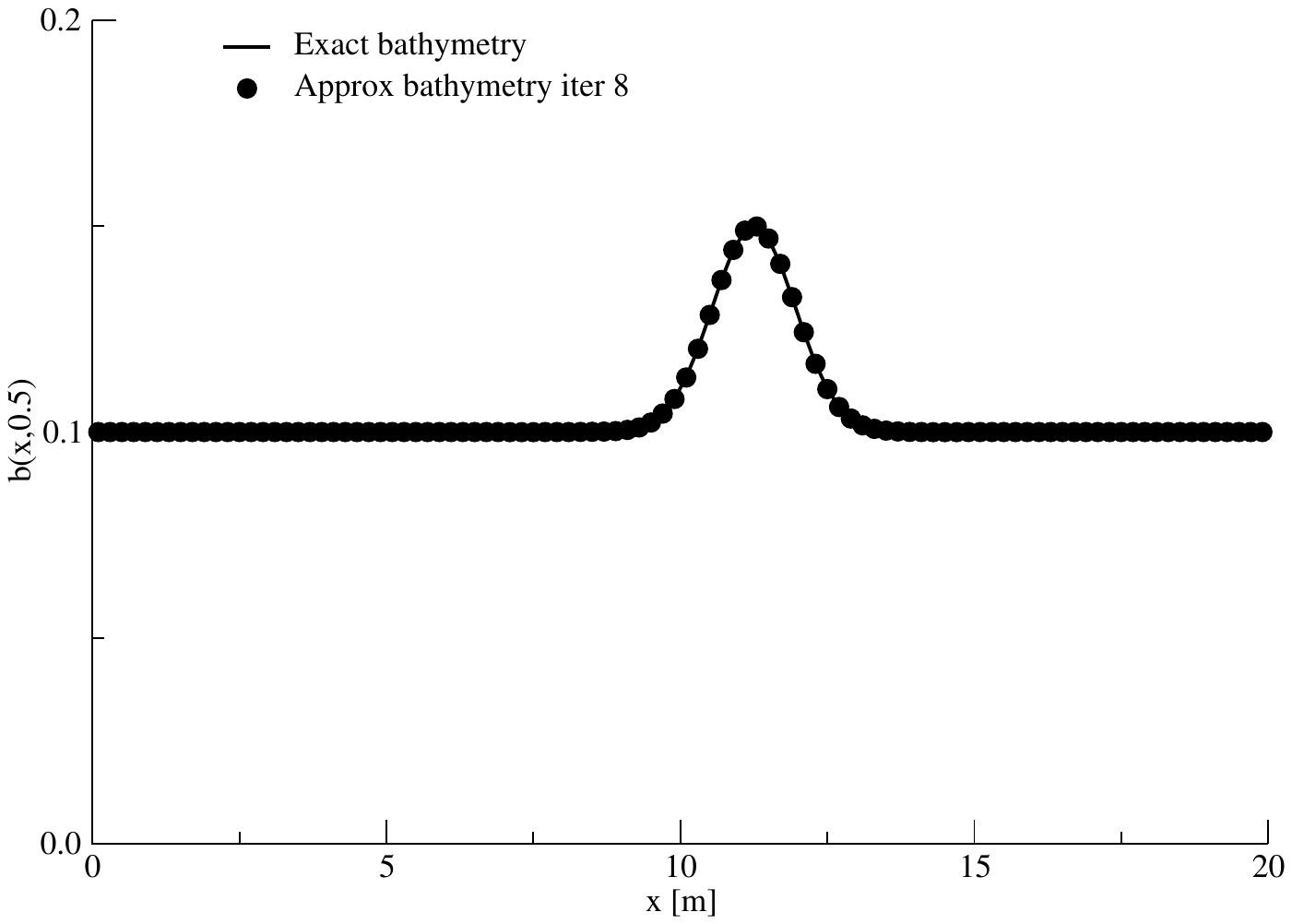} \quad
	\includegraphics[width=0.3\textwidth]{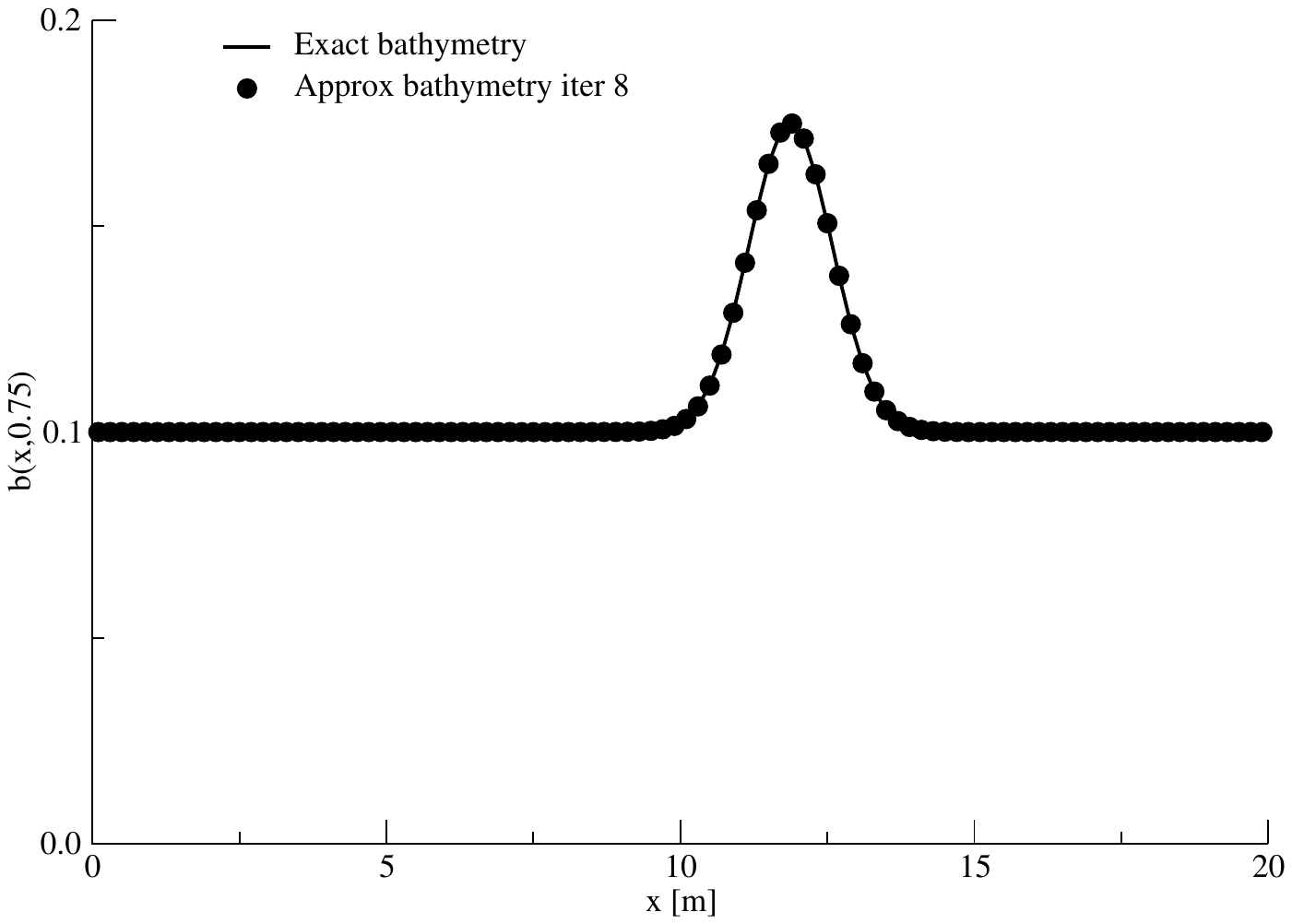} \\
	
	\caption{Smooth bottom profile (\ref{eq:b-smooth}): Result for the reconstruction procedure resulting from non-conservative {\bf FORCE-$\alpha$+FD}. Parameters $\Delta t = 0.01$,  $\alpha_F = 2$, $\varepsilon = 0.001$, $\lambda_b = 0.71$, $100$ cells.
		{\bf Feft:} $t=0.25$, {\bf centered} $t=0.5$, {\bf right:} $t=0.75$.}
	\label{fig:b-for-iter-and-times:test-0:ForceAlphaCons}
\end{figure}


\begin{figure}   
	\centering
	\includegraphics[width=0.3\textwidth]{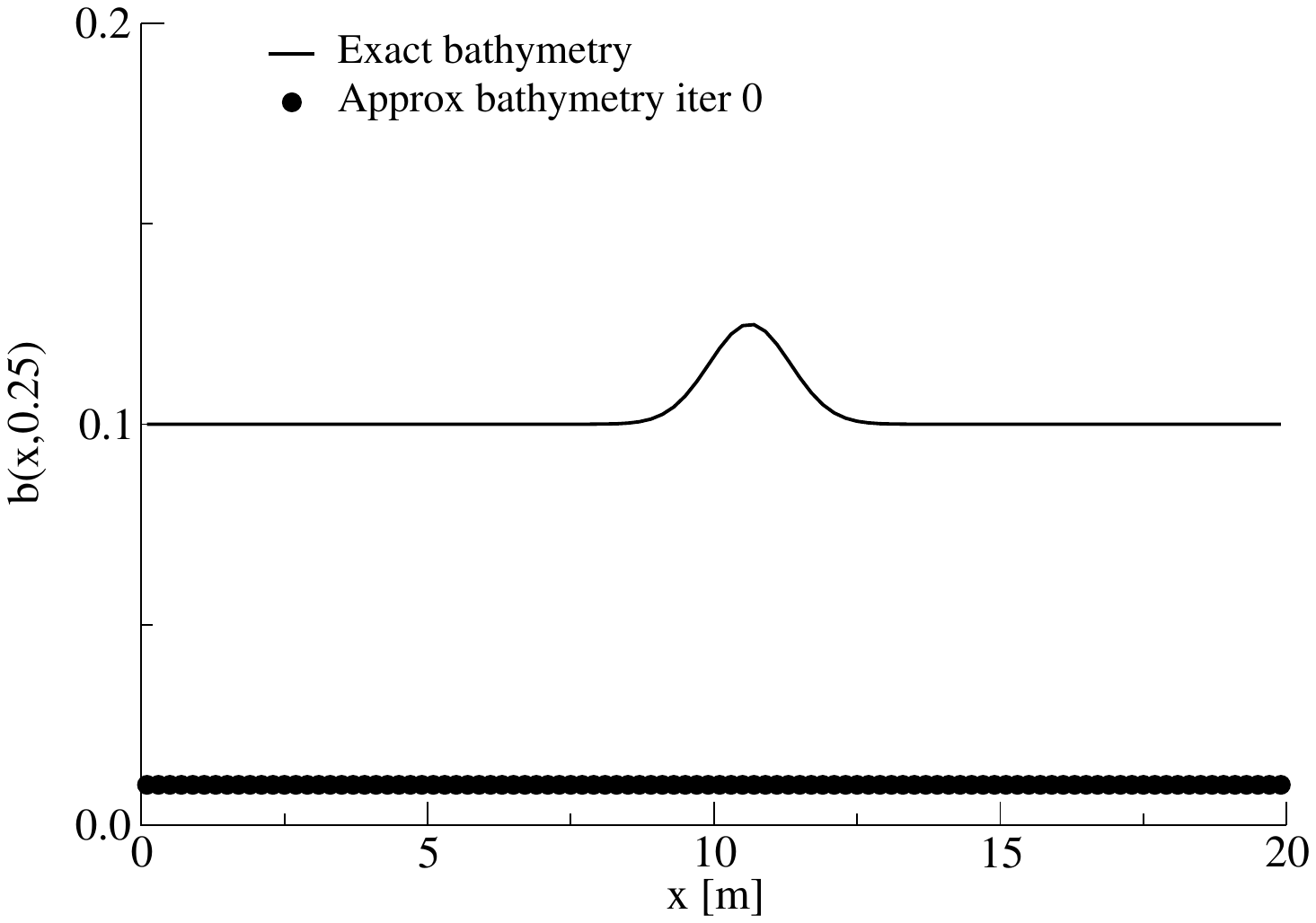} \quad
	\includegraphics[width=0.3\textwidth]{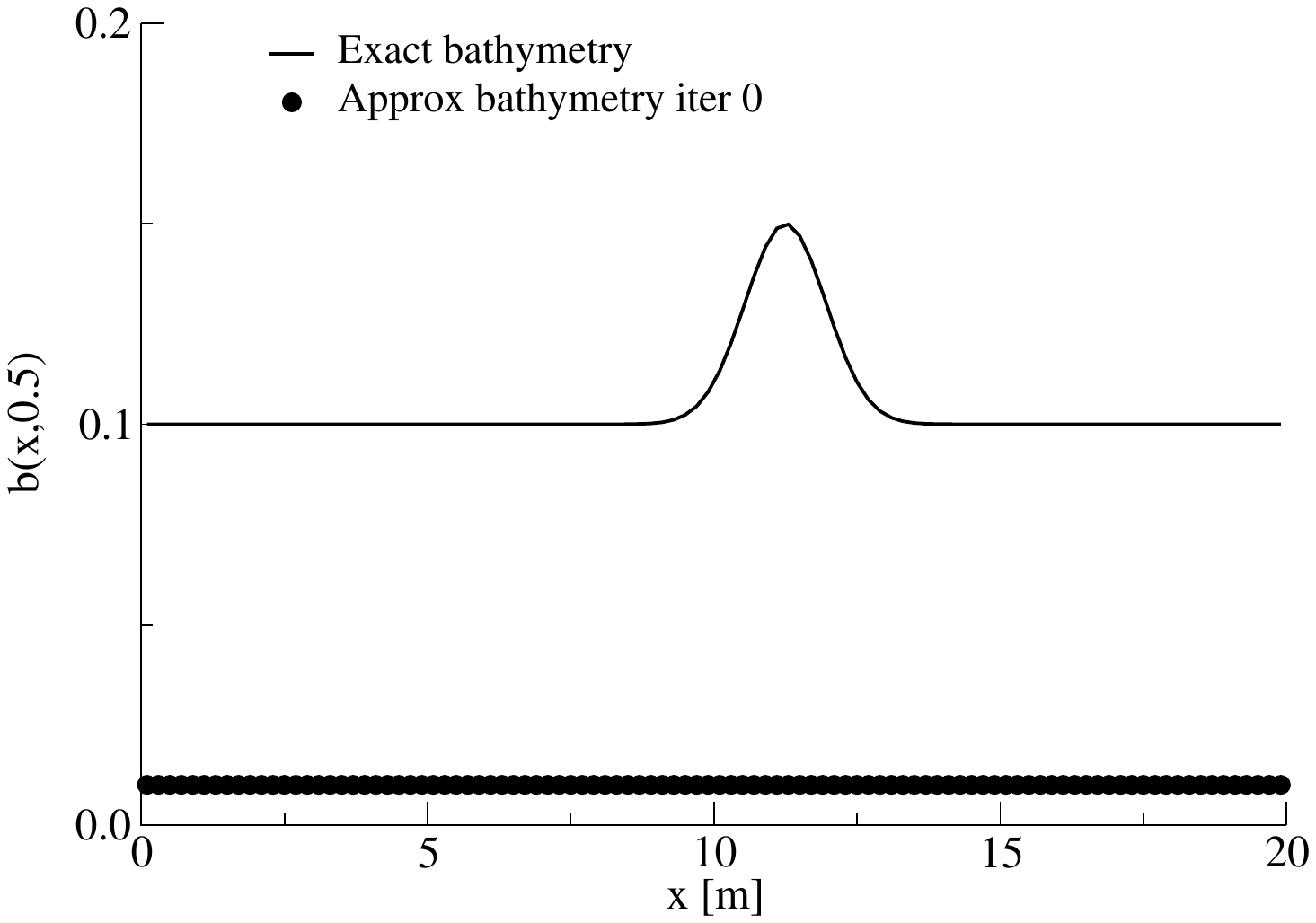} \quad
	\includegraphics[width=0.3\textwidth]{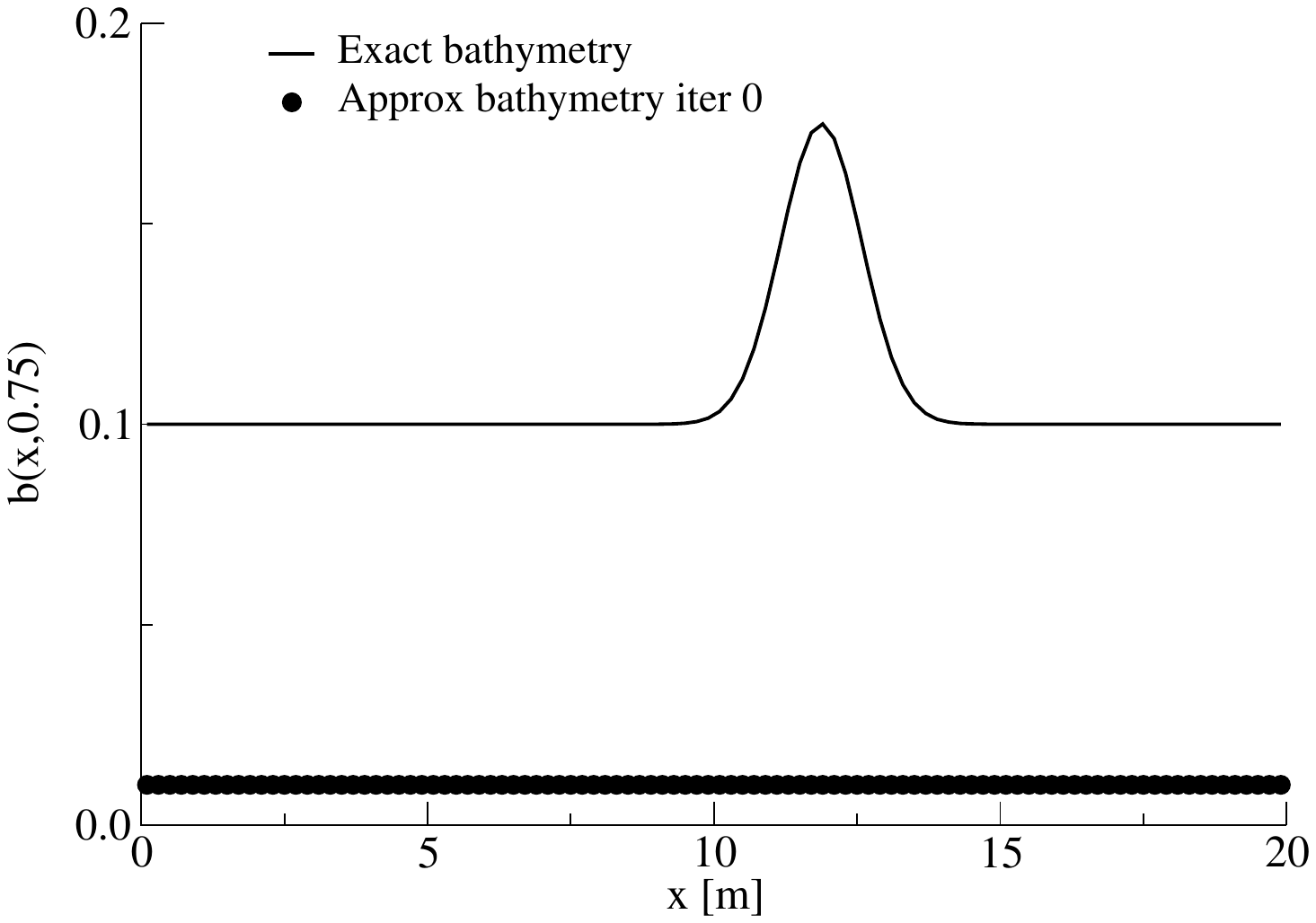} 
	\\
	
	\includegraphics[width=0.3\textwidth]{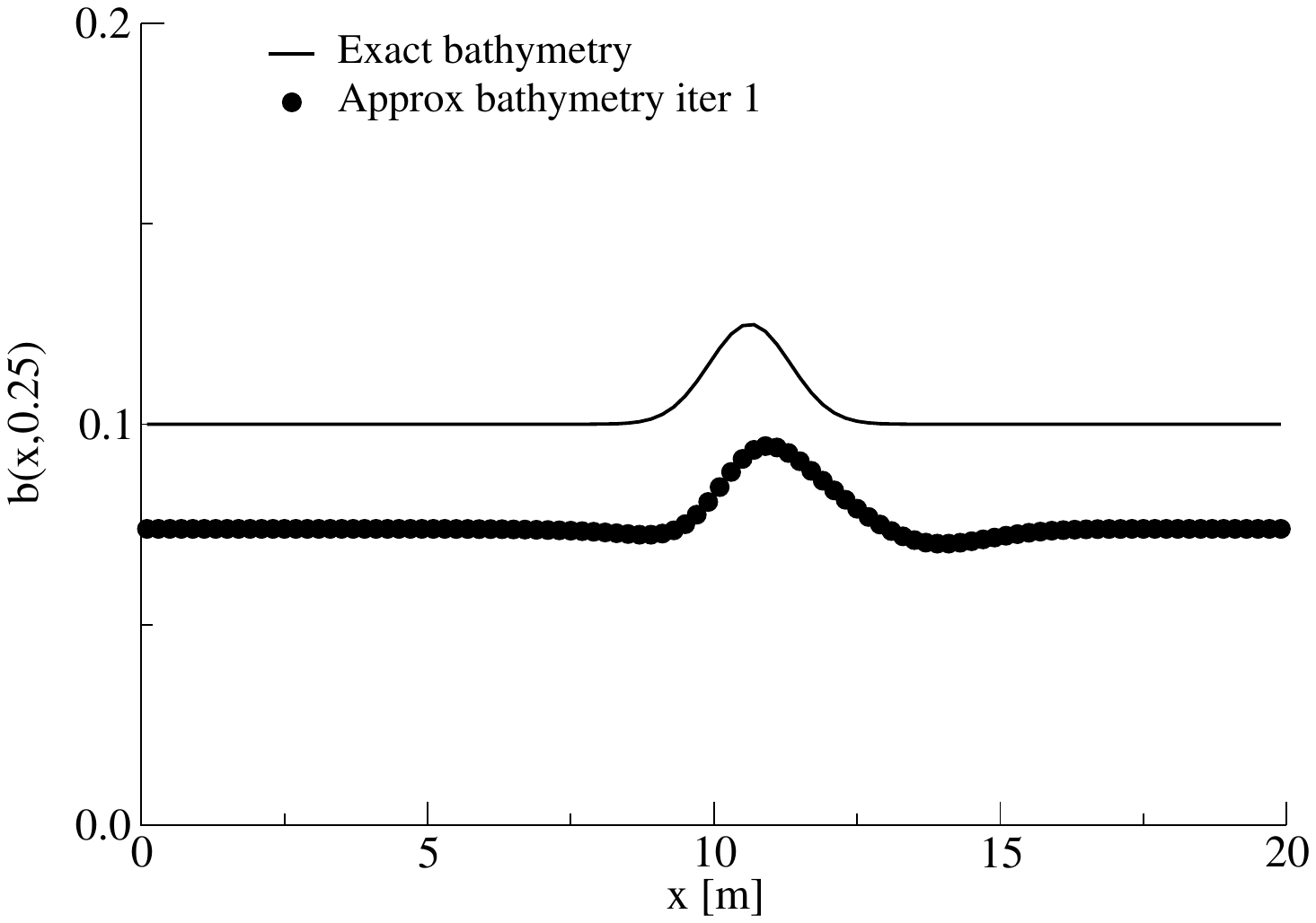} \quad
	\includegraphics[width=0.3\textwidth]{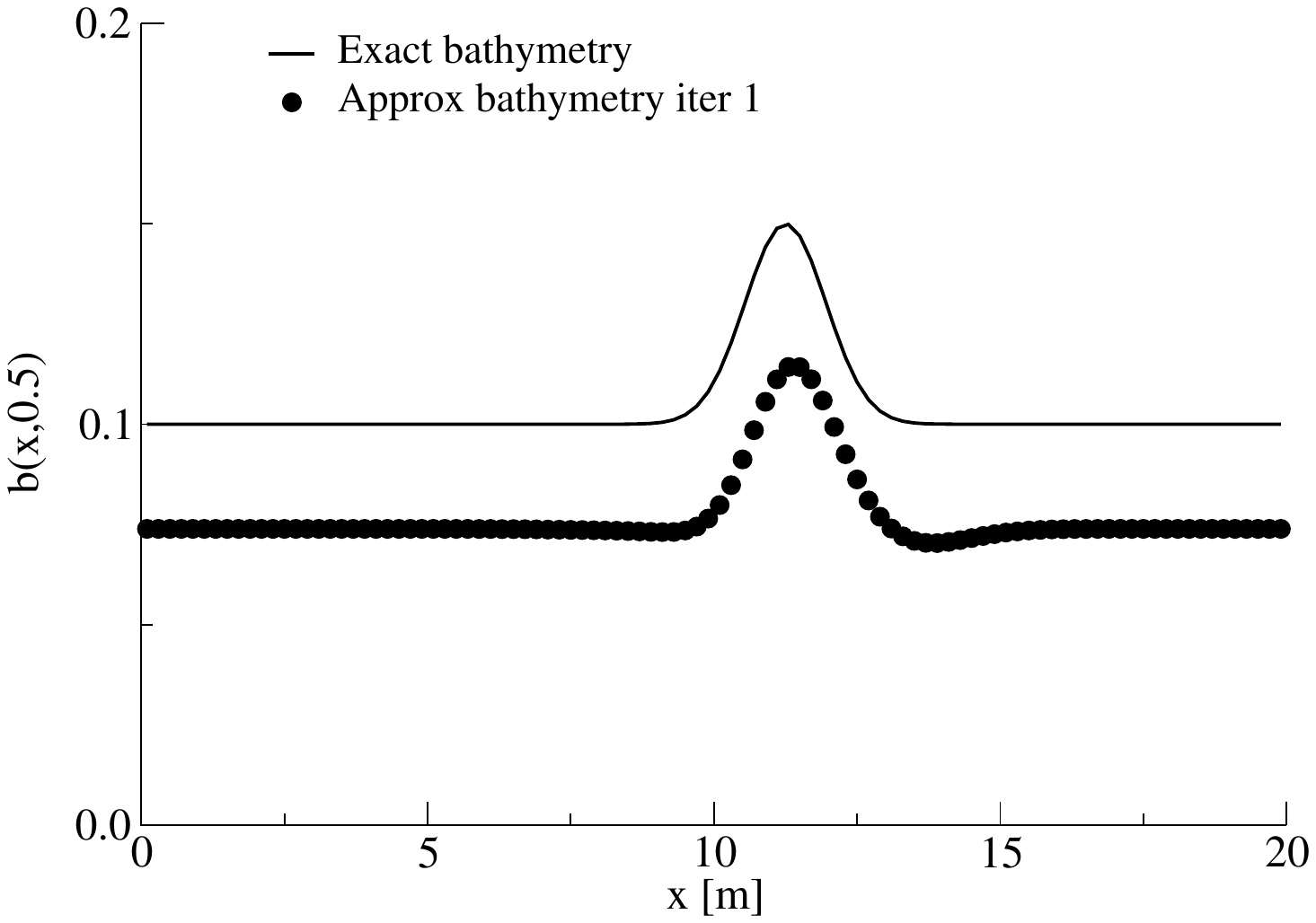} \quad
	\includegraphics[width=0.3\textwidth]{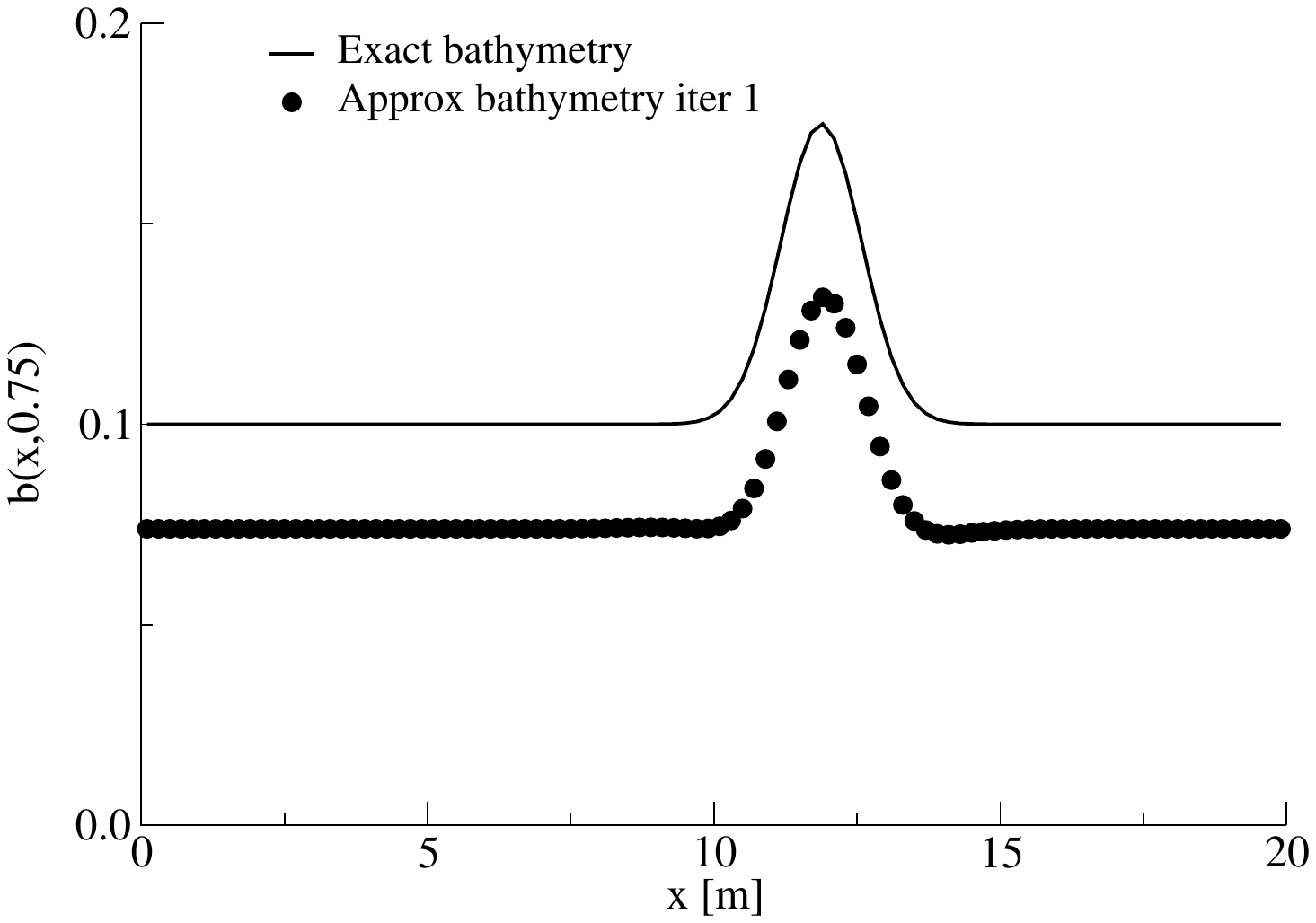} \\
	
	\includegraphics[width=0.3\textwidth]{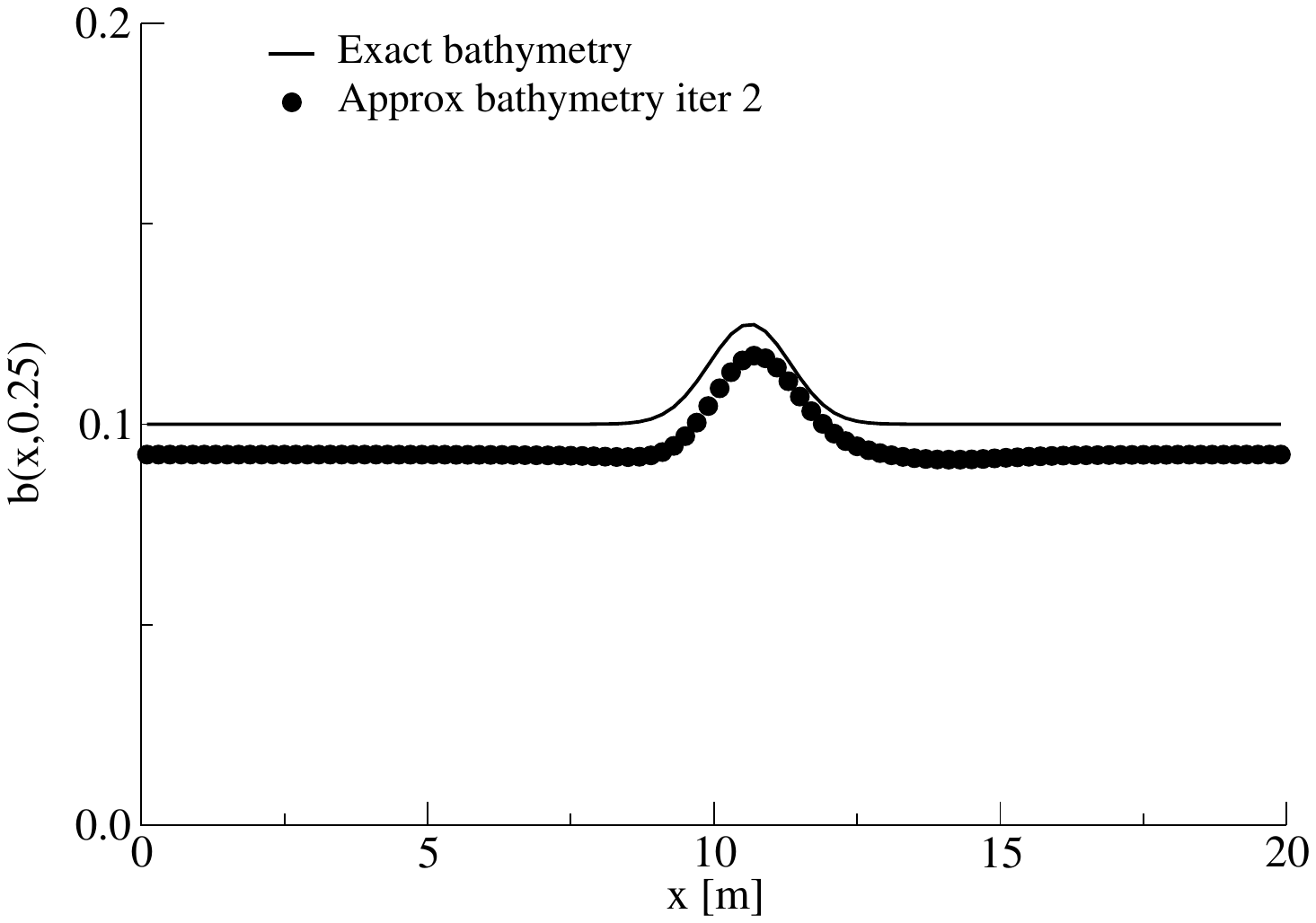} \quad
	\includegraphics[width=0.3\textwidth]{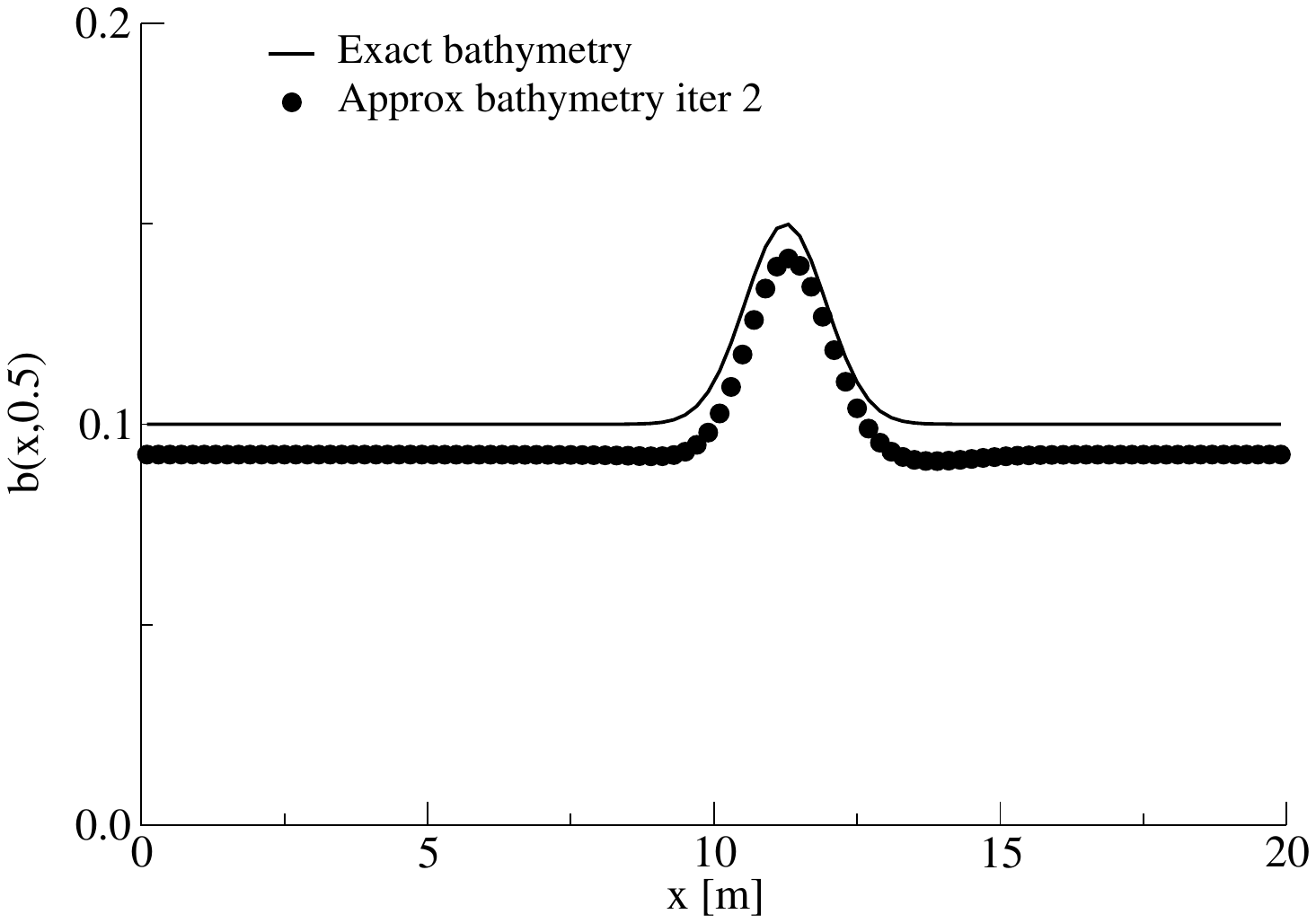} \quad
	\includegraphics[width=0.3\textwidth]{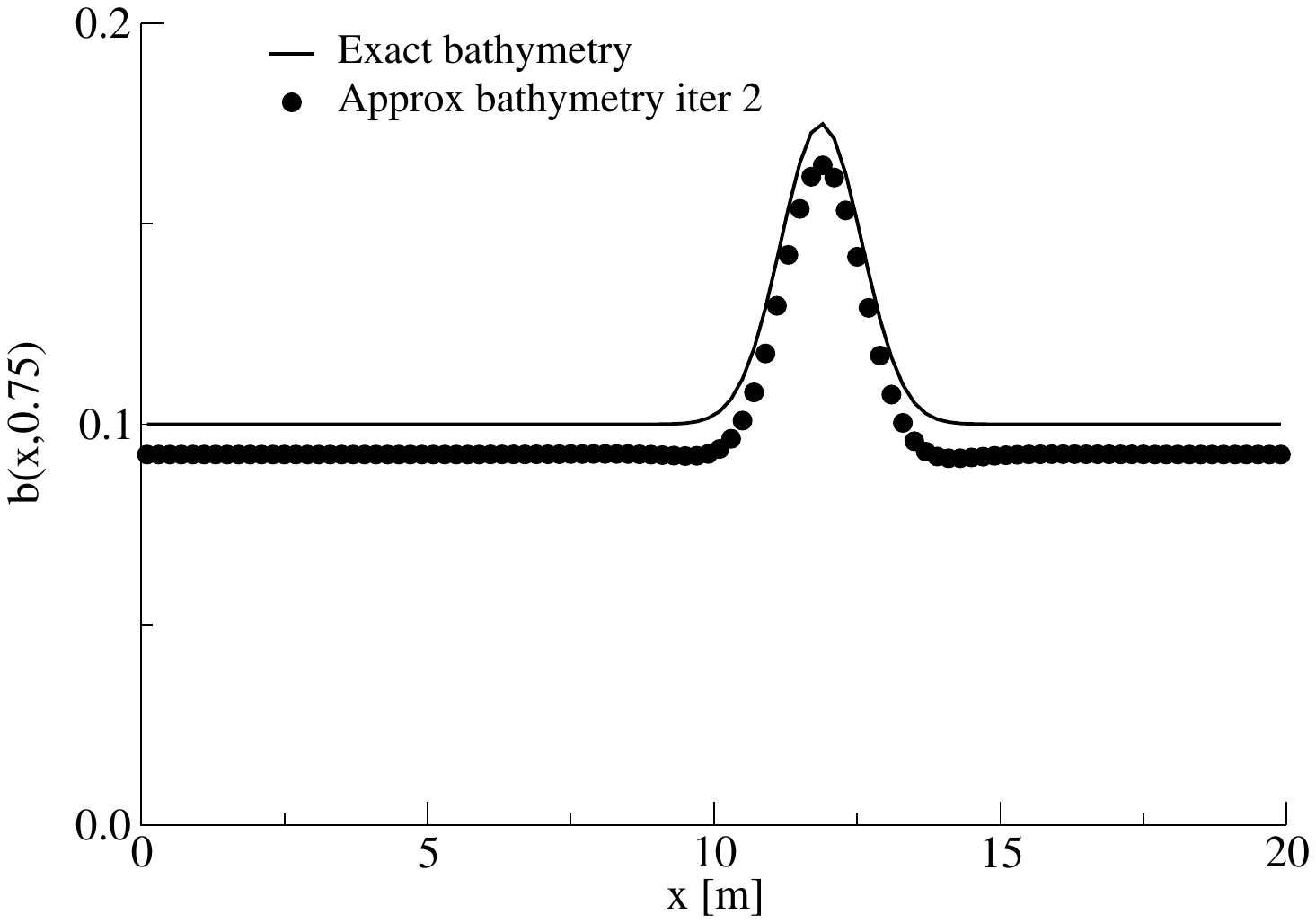} \\

	\includegraphics[width=0.3\textwidth]{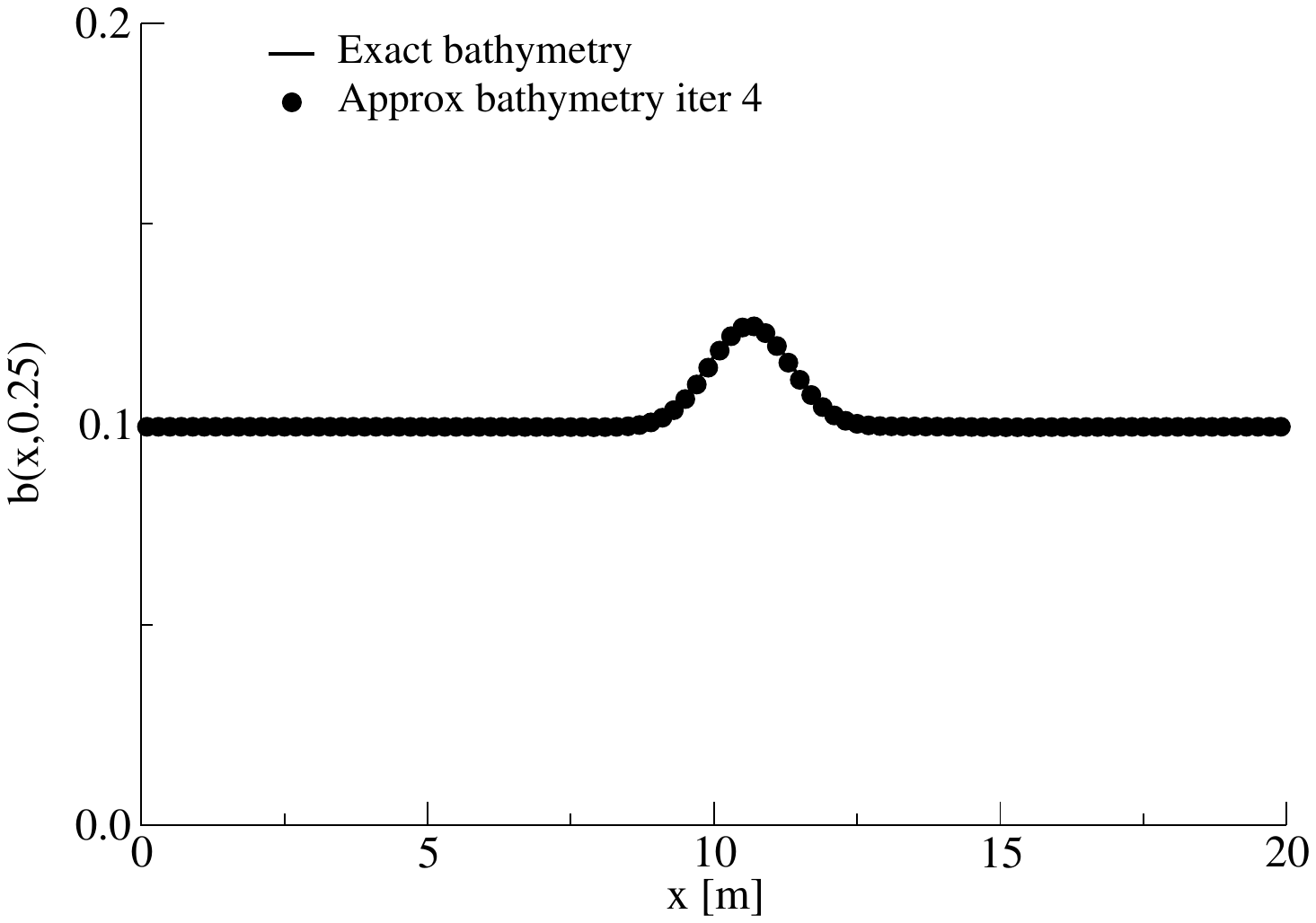} \quad
	\includegraphics[width=0.3\textwidth]{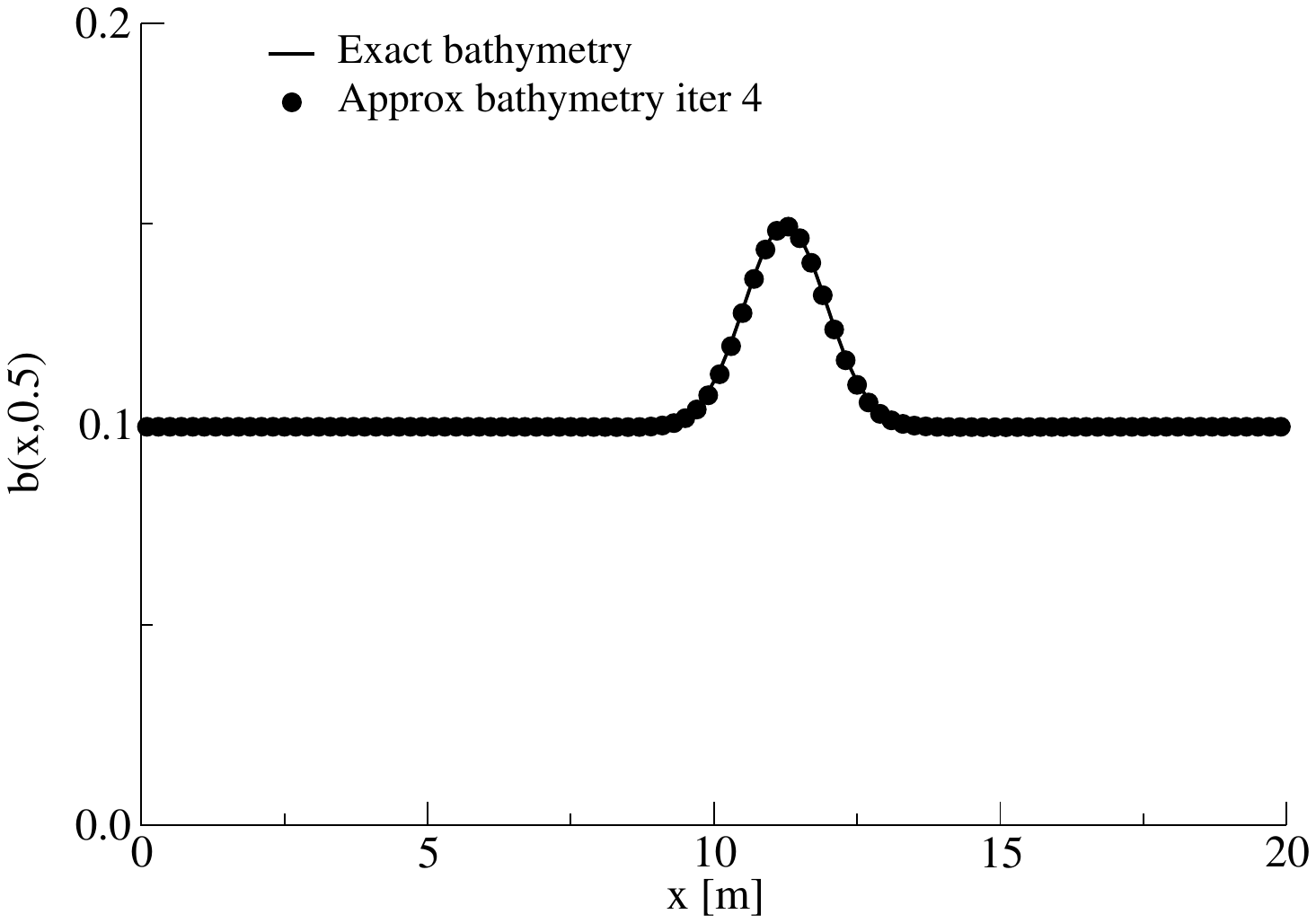} \quad
	\includegraphics[width=0.3\textwidth]{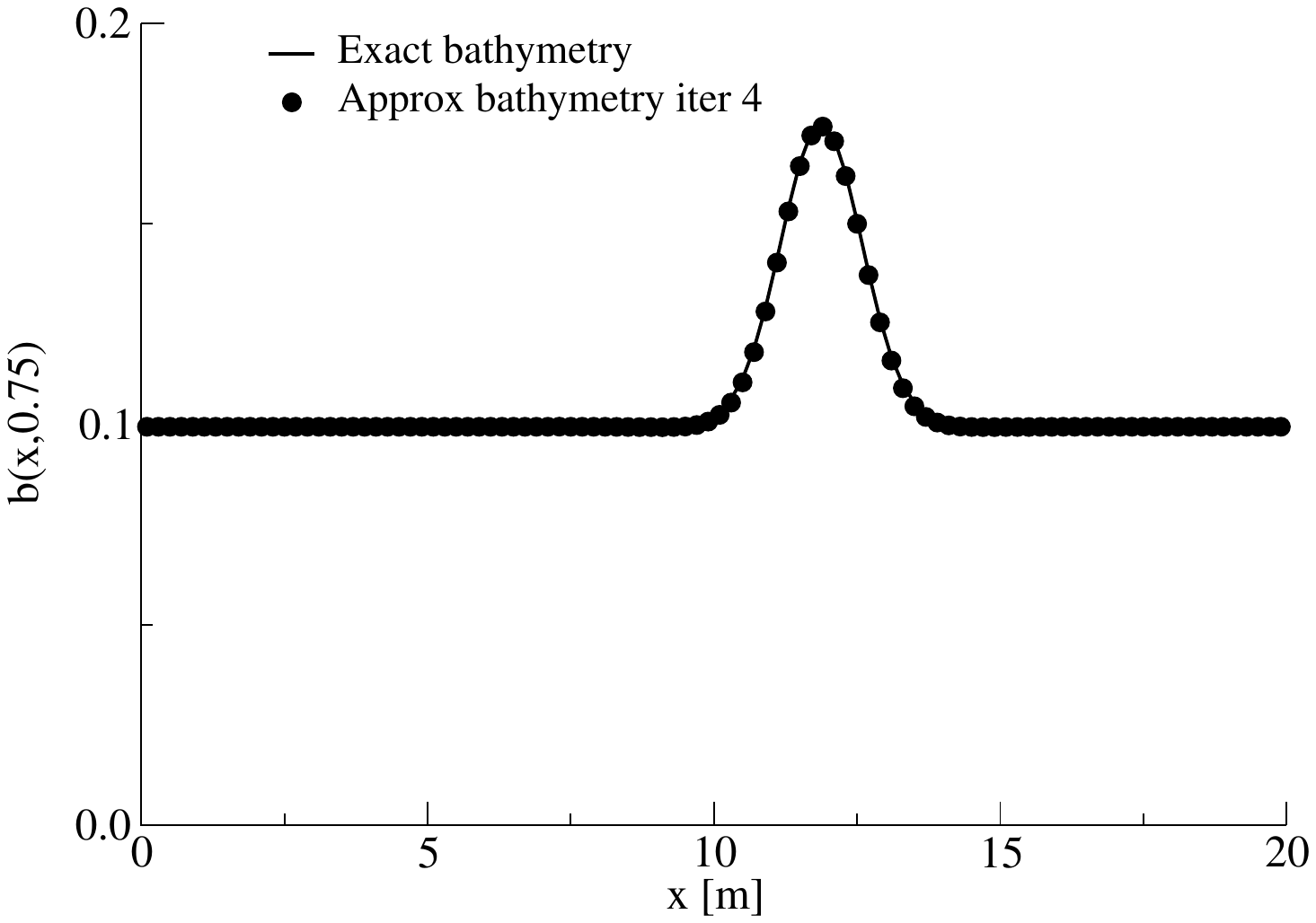} \\
	
	\includegraphics[width=0.3\textwidth]{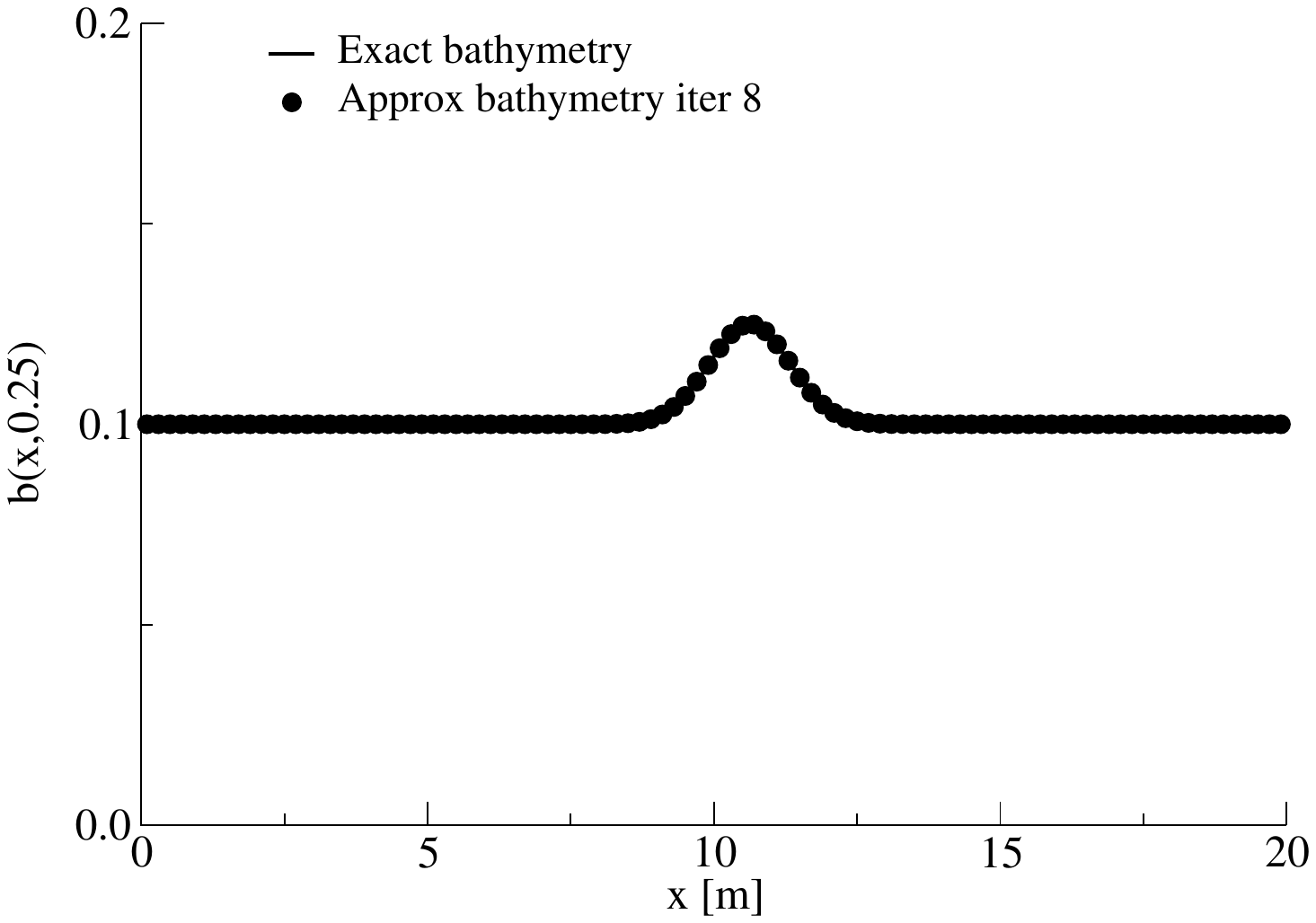} \quad
	\includegraphics[width=0.3\textwidth]{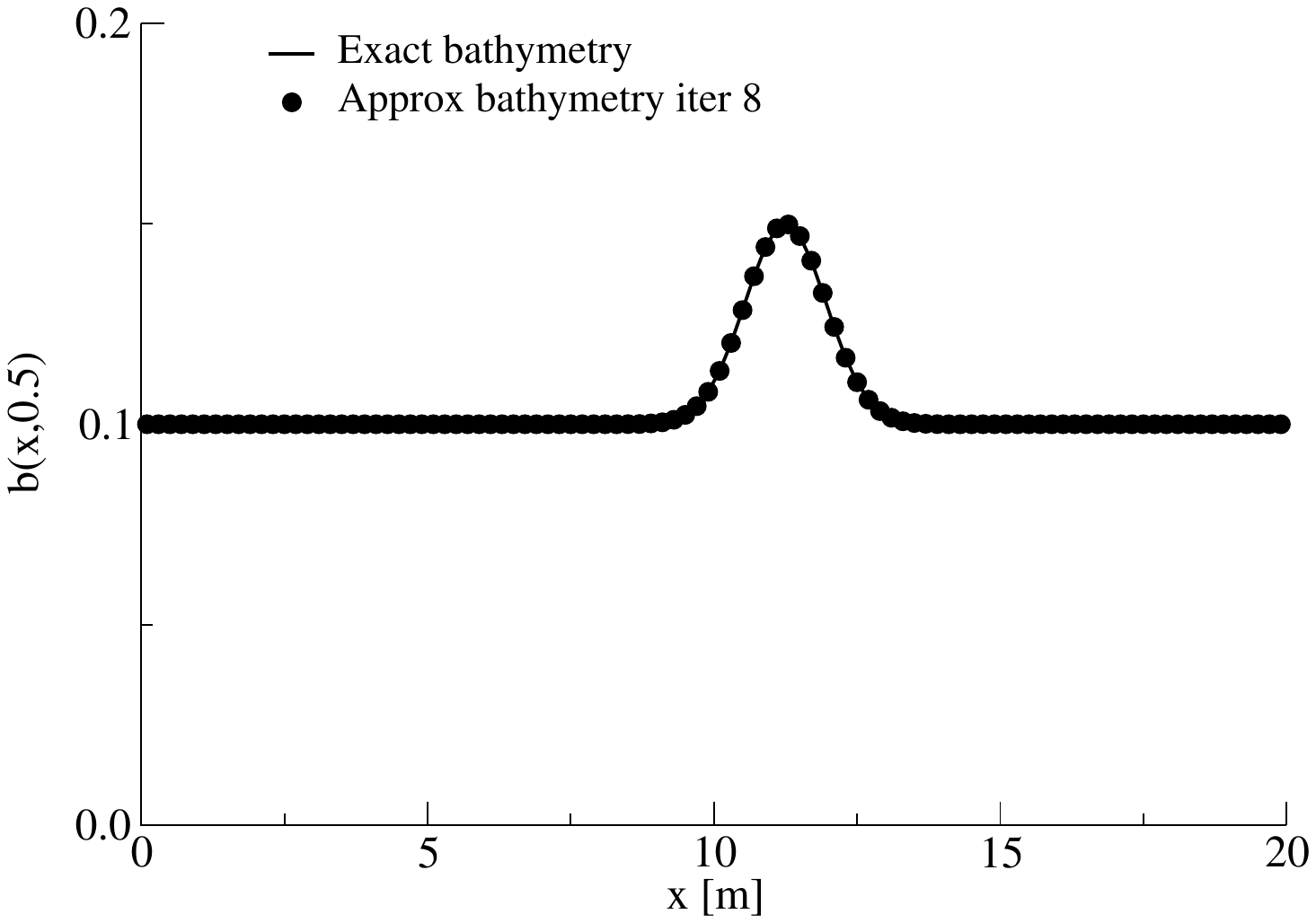} \quad
	\includegraphics[width=0.3\textwidth]{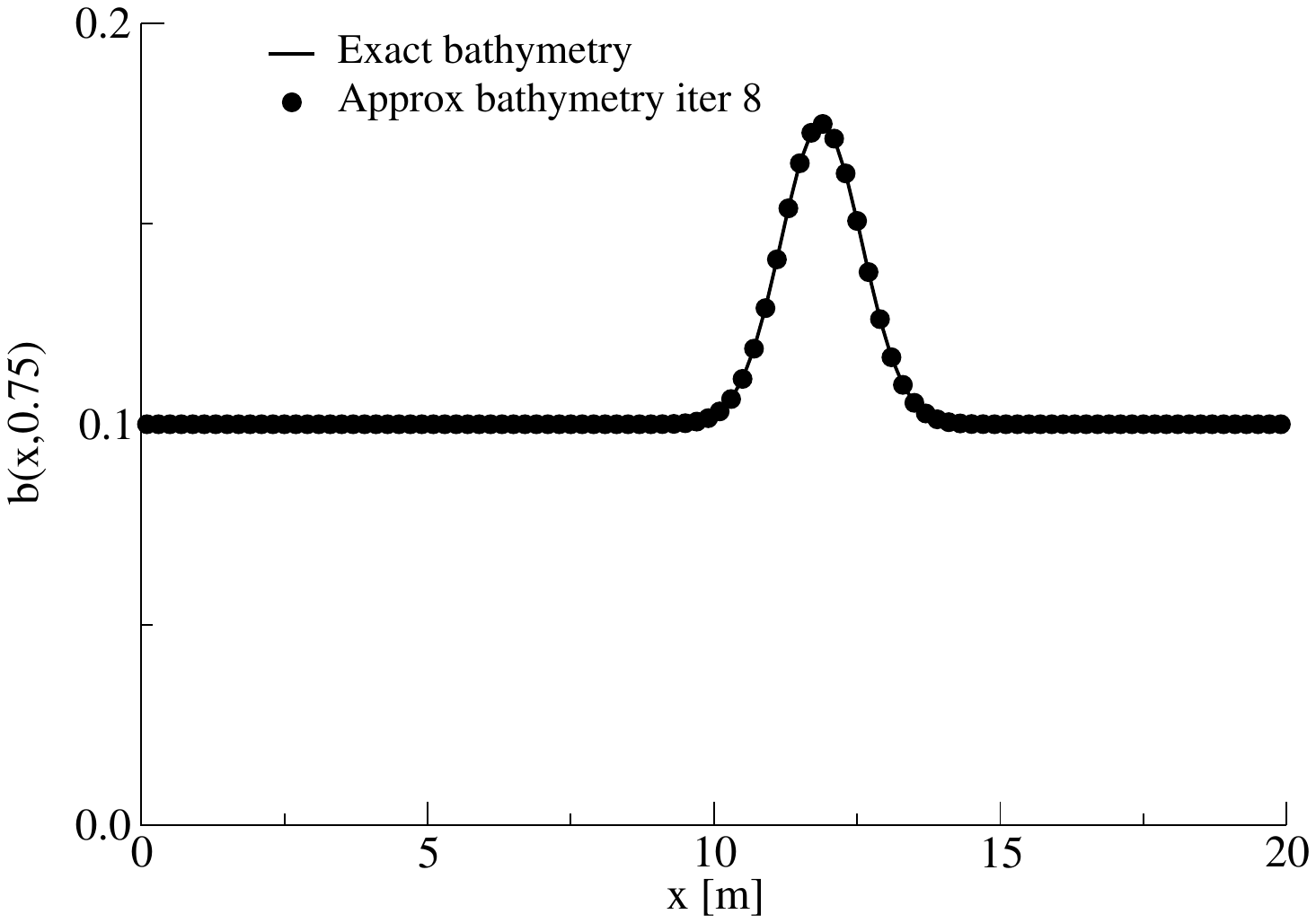} \\
	
	\caption{Smooth bottom profile (\ref{eq:b-smooth}): Result for the reconstruction procedure resulting from {\bf FORCE-$\alpha$+CSF}. Parameters $\Delta t = 0.01$,  $\alpha_F = 2$, $\varepsilon = 0.001$, $\lambda_b = 0.71$, $100$ cells.
		{\bf Feft:} $t=0.25$, {\bf centered} $t=0.5$, {\bf right:} $t=0.75$.}
	\label{fig:b-for-iter-and-times:test-0:NonCons-Force-alpha}
\end{figure}

\FloatBarrier

\subsection{Discontinuous bottom profile}

This test aims at recovering the discontinuous bottom profile

\begin{eqnarray}
	\label{eq:b-discontinuous}
	\begin{array}{c}
		
		\bar{b}(t,x) = \left\{
		\begin{array}{cc}
			\frac{1}{4}, & 5 < x< 7, \\
			0.3 t, & 7 < x< 10+4t, \\
			0.1,   & otherwise.
		\end{array}
		\right.
	\end{array}
\end{eqnarray}

Systems (\ref{eq:direct-system}) and (\ref{eq:dual-system}) are solved with transmissible boundary conditions.   This profile consists of two square waves, one is kept fix while the other increases in amplitude and moves on the right as the time advances.

Figure \ref{fig:b-for-iter-and-times:test-1-ConsRusanov} shows the results for the {\bf Rusanov+FD} scheme. Despite the locations of discontinuities are recovered, the right amplitude of the waves are not recovered. Furthermore, spurious oscillation appear in both the direct and inverse problems. This approach is not suitable for recovering this type of bottom profile.  It is well known that Rusanov scheme introduces large numerical diffusion when CFL coefficients are small. Figure \ref{fig:b-for-iter-and-times:test-1:ForceAlphaCons} shows the results for the {\bf FORCE-$\alpha$+FD}. It can be seen that the procedure recovers the correct profile. This highlights the fact that a low dissipation scheme for the state system is suitable for achieving convergence of the global procedure.  Figure \ref{fig:b-for-iter-and-times:test-1:NonCons-Force-alpha} shows the results for the {\bf FORCE-$\alpha$+CSF} scheme. We note that global convergence is achieved. We observe that already in the second iteration of the descent step procedure the scheme is able to recover the main features of the profile.  Figure \ref{fig:comp-error-test-1} shows the $L_\infty$ norm of $\nabla J$ against the number of iterations, to facilitate the visualization the plot is depicted in logarithmic scale. So, this measures the error between $b^k $ and $b^{k+1}$, which is the empirical convergence of the global algorithm.  These results show that to obtain convergence, not only becomes important the form in which the state system is discretized but also how the adjoint system does so.  As before, a low-dissipation scheme is beneficial in the PDE-constraint optimization context and discontinuous variables.

\begin{figure}[h]
	\begin{center}
		\includegraphics[scale=0.45]{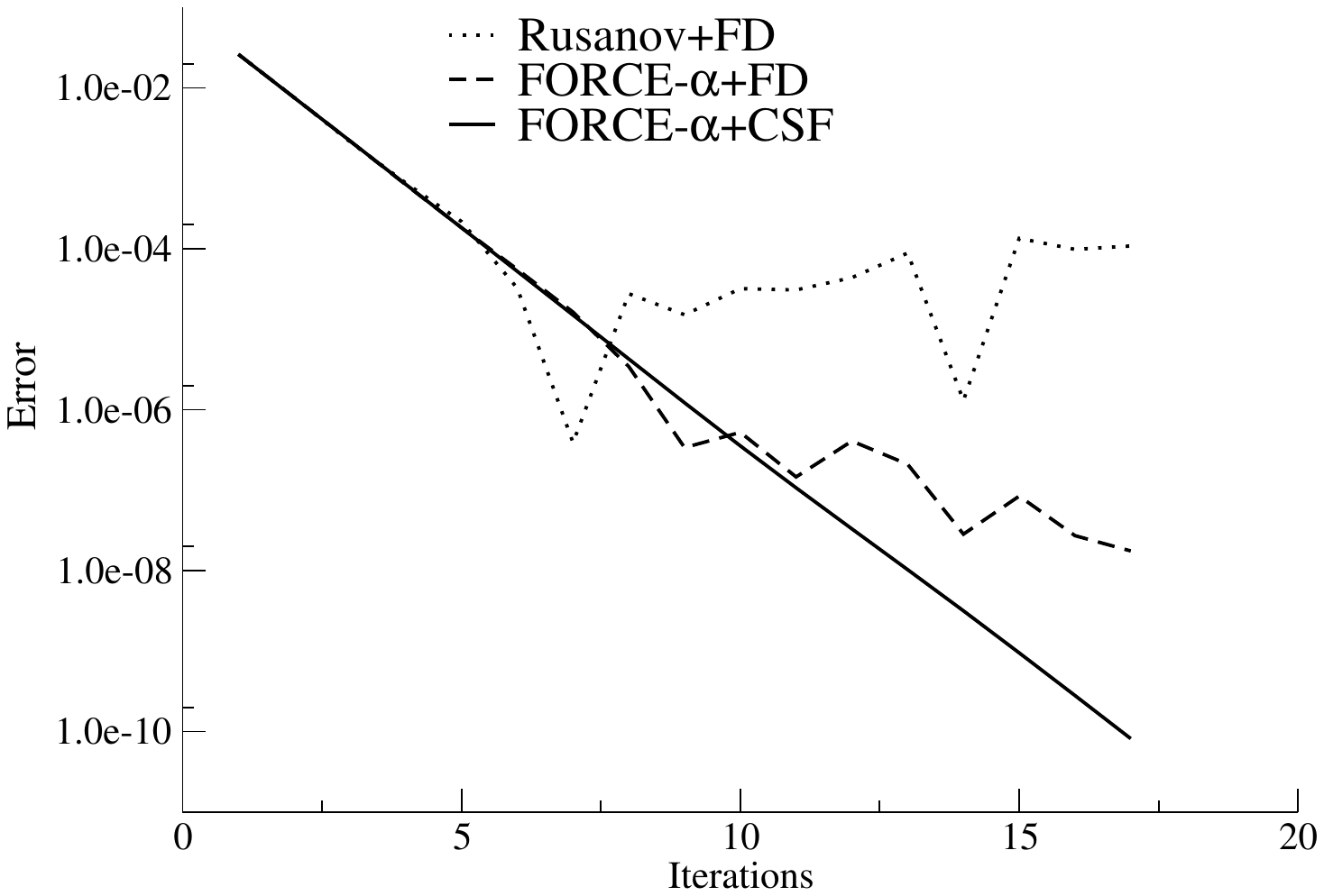}
	\end{center}
	\caption{Discontinuous bottom profile (\ref{eq:b-discontinuous}): The $L_\infty$ norm of $\nabla J $ for $\varepsilon = 0.001$, $\lambda_b = 0.71$, $100$ cells at $t = 1$, $\alpha_F = 2$. 
		(Dot line) {\bf Rusanov+FD} scheme.
		(Dash line) {\bf FORCE-$\alpha$+FD} scheme. 
		(Full line) {\bf FORCE-$\alpha$+CSF} scheme.
	}\label{fig:comp-error-test-1}
\end{figure}

\begin{figure}   
	\centering
	\includegraphics[width=0.3\textwidth]{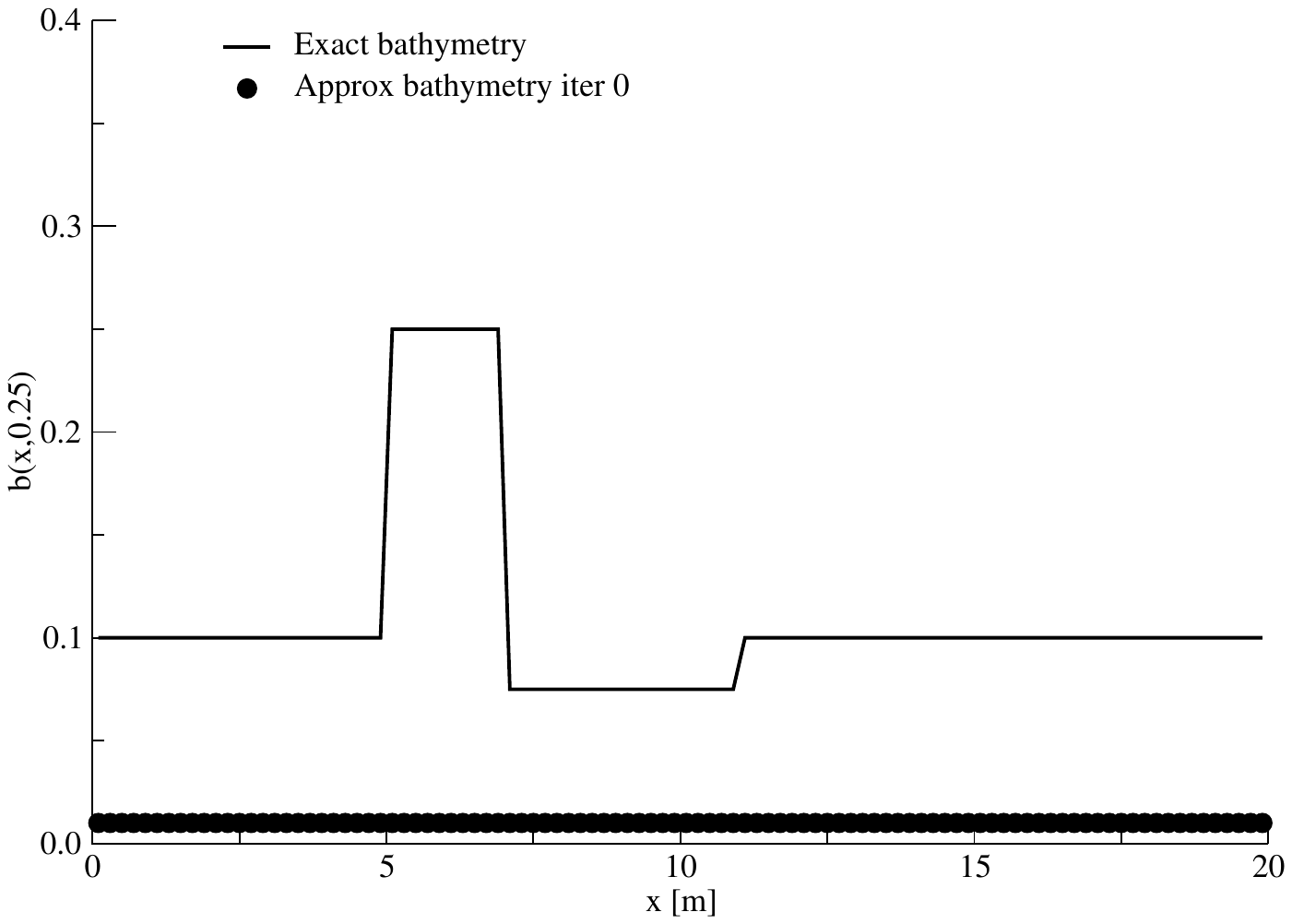} \quad
	\includegraphics[width=0.3\textwidth]{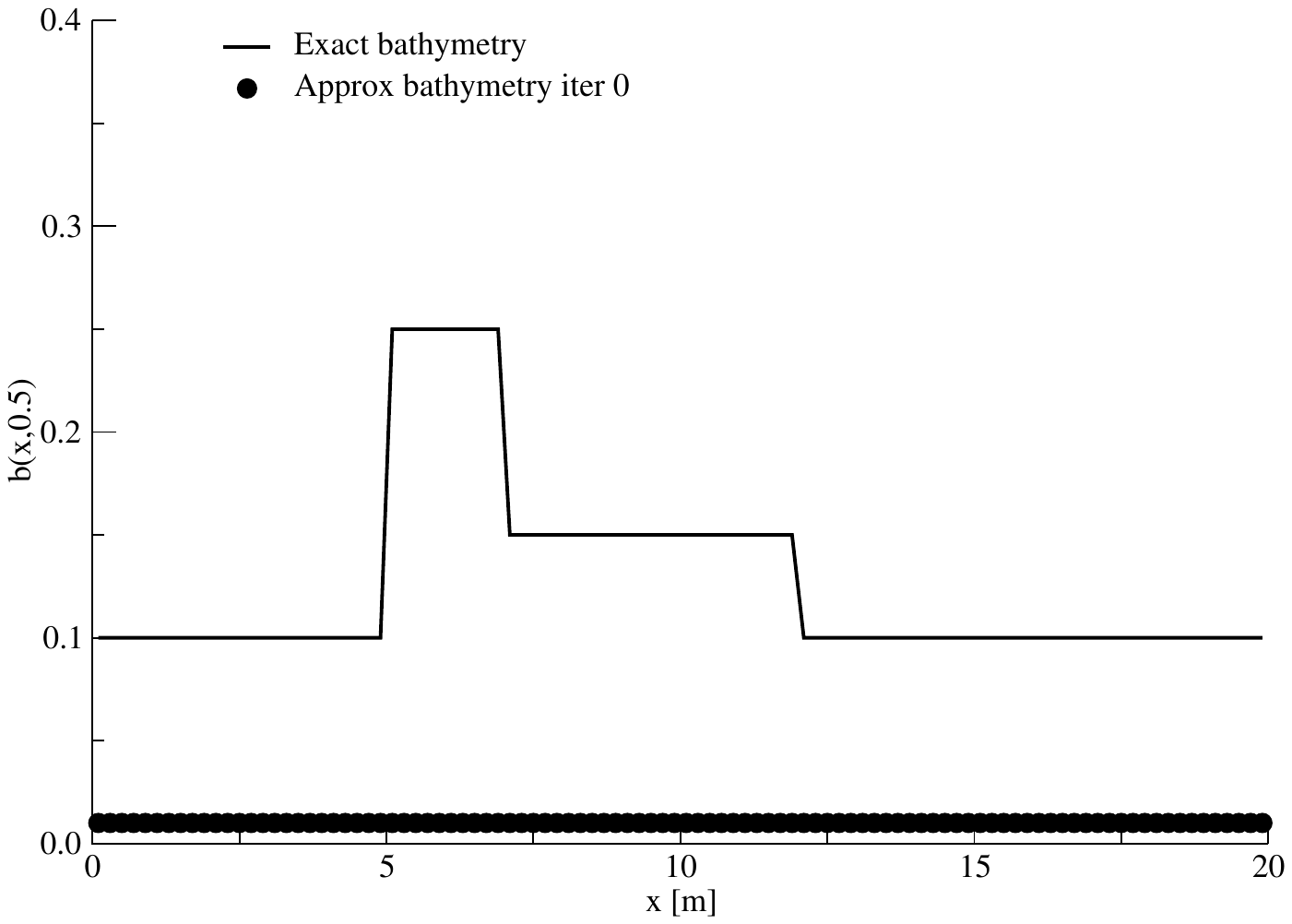} \quad
	\includegraphics[width=0.3\textwidth]{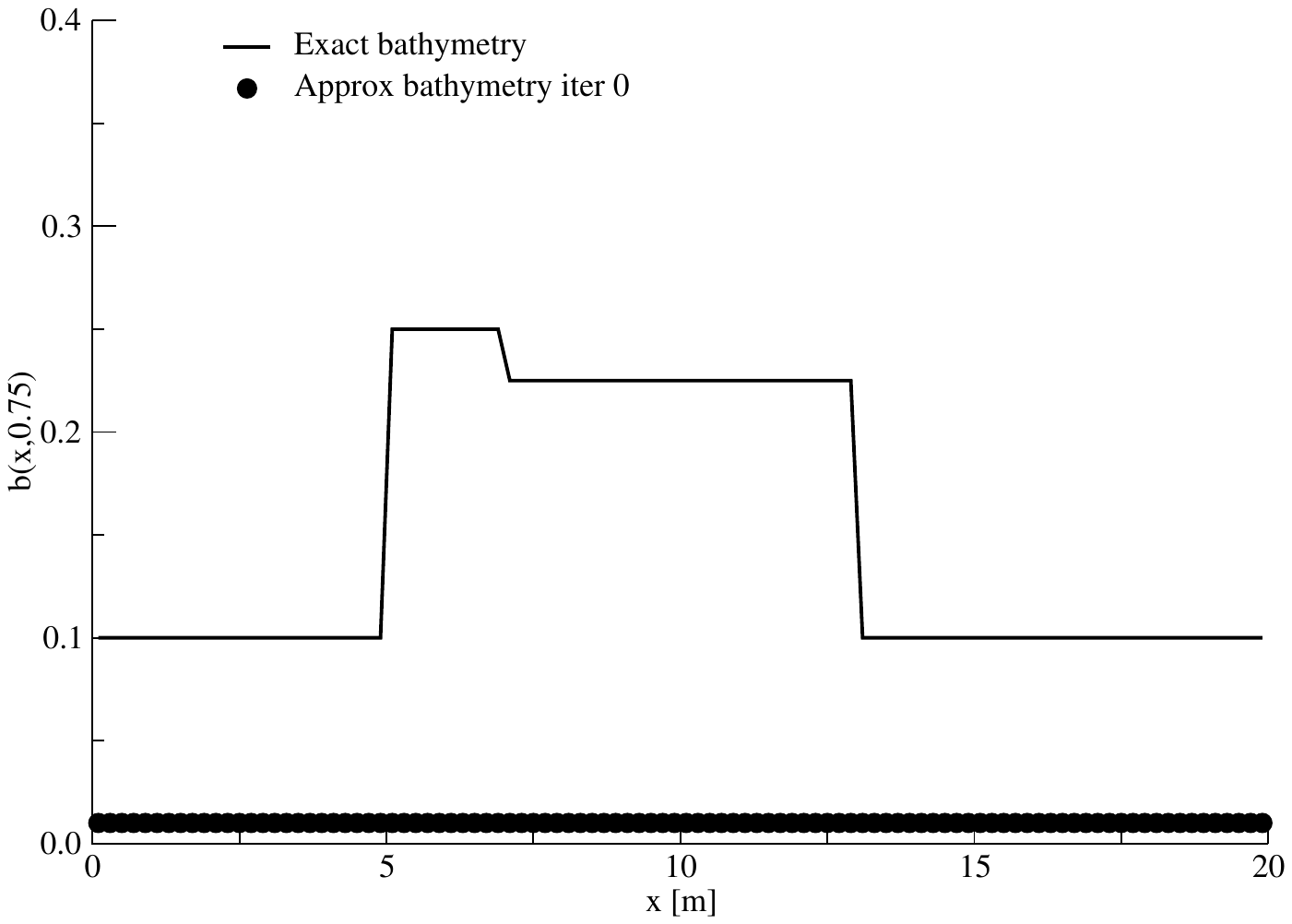} 
	\\
	
	\includegraphics[width=0.3\textwidth]{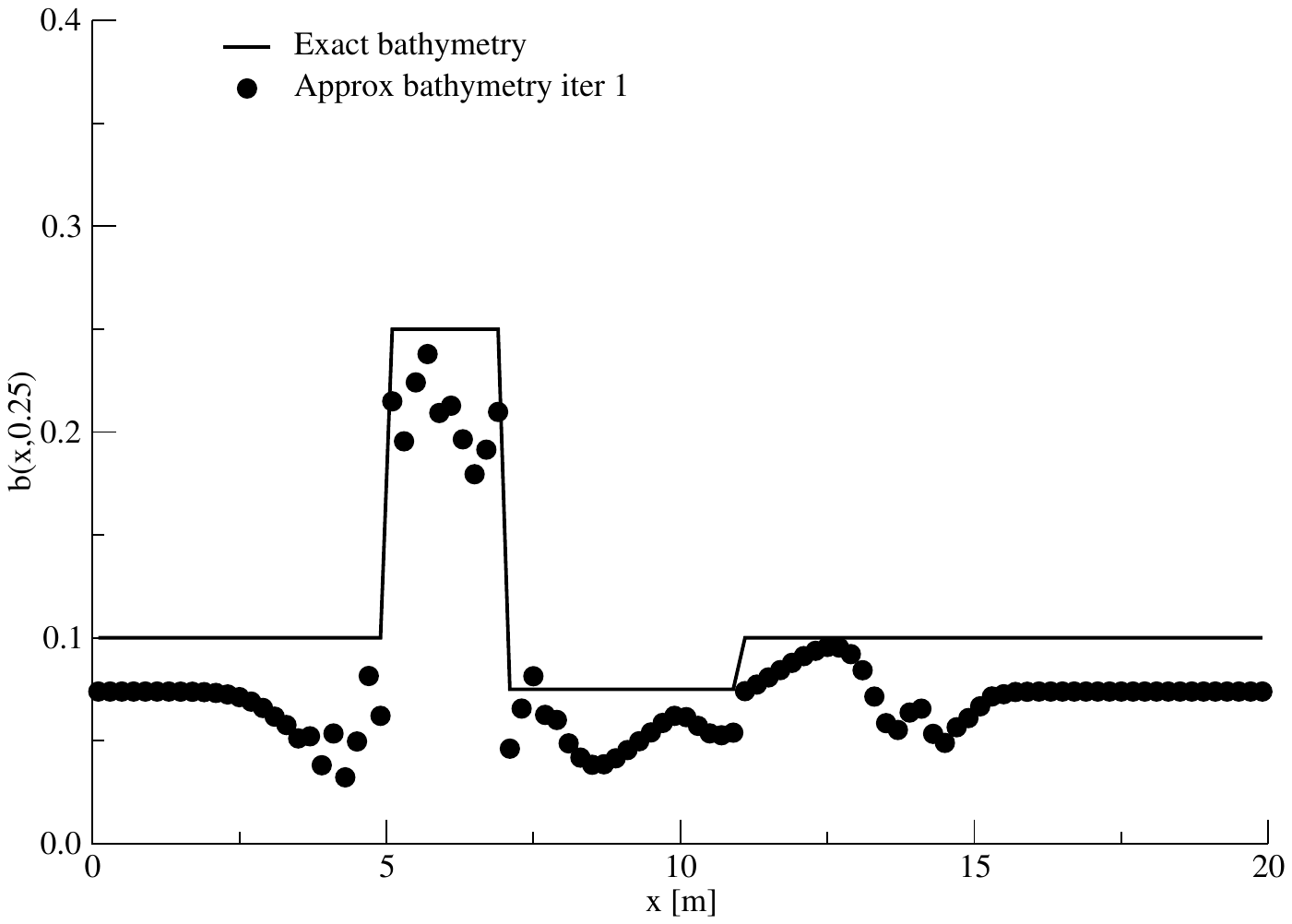} \quad
	\includegraphics[width=0.3\textwidth]{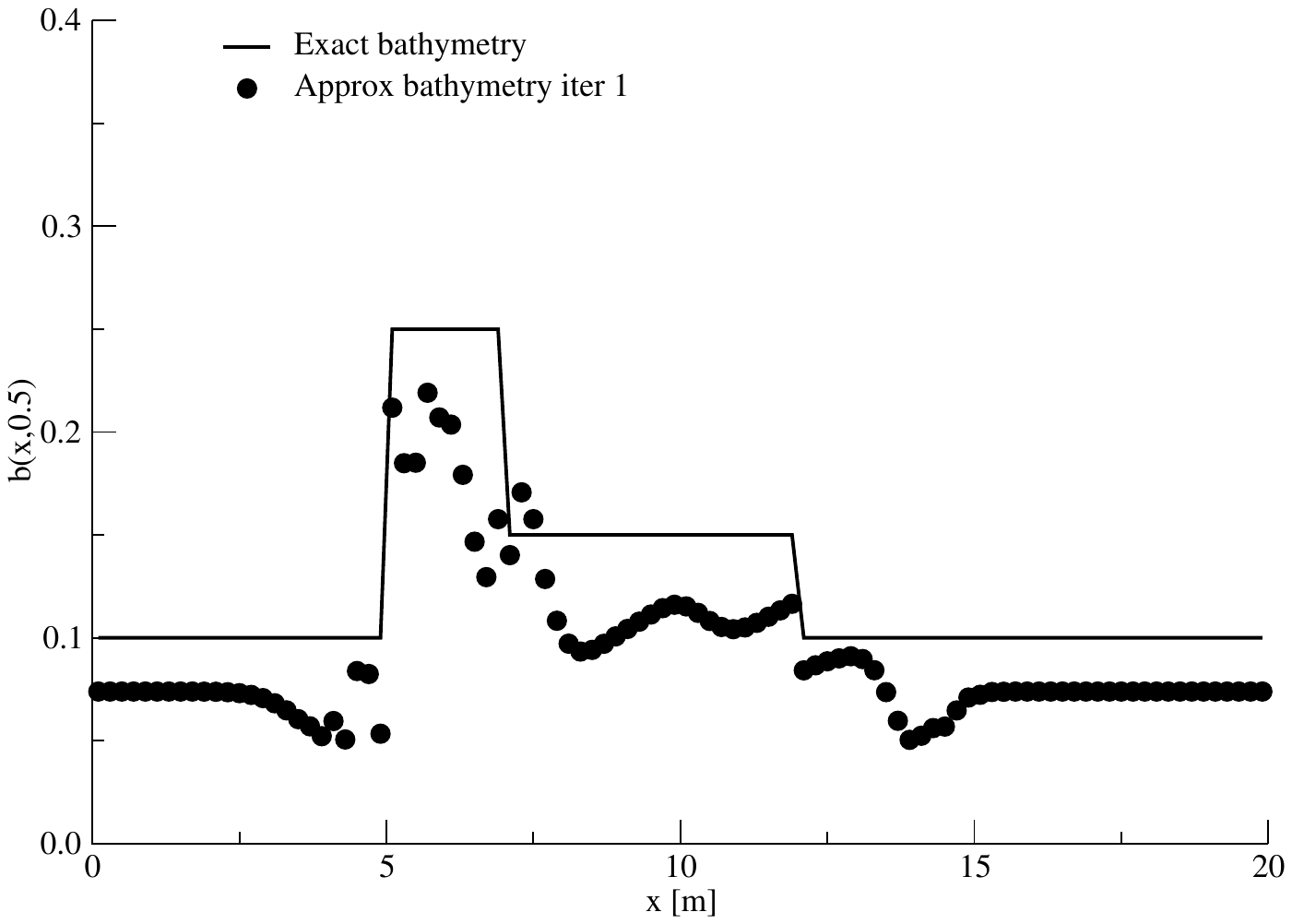} \quad
	\includegraphics[width=0.3\textwidth]{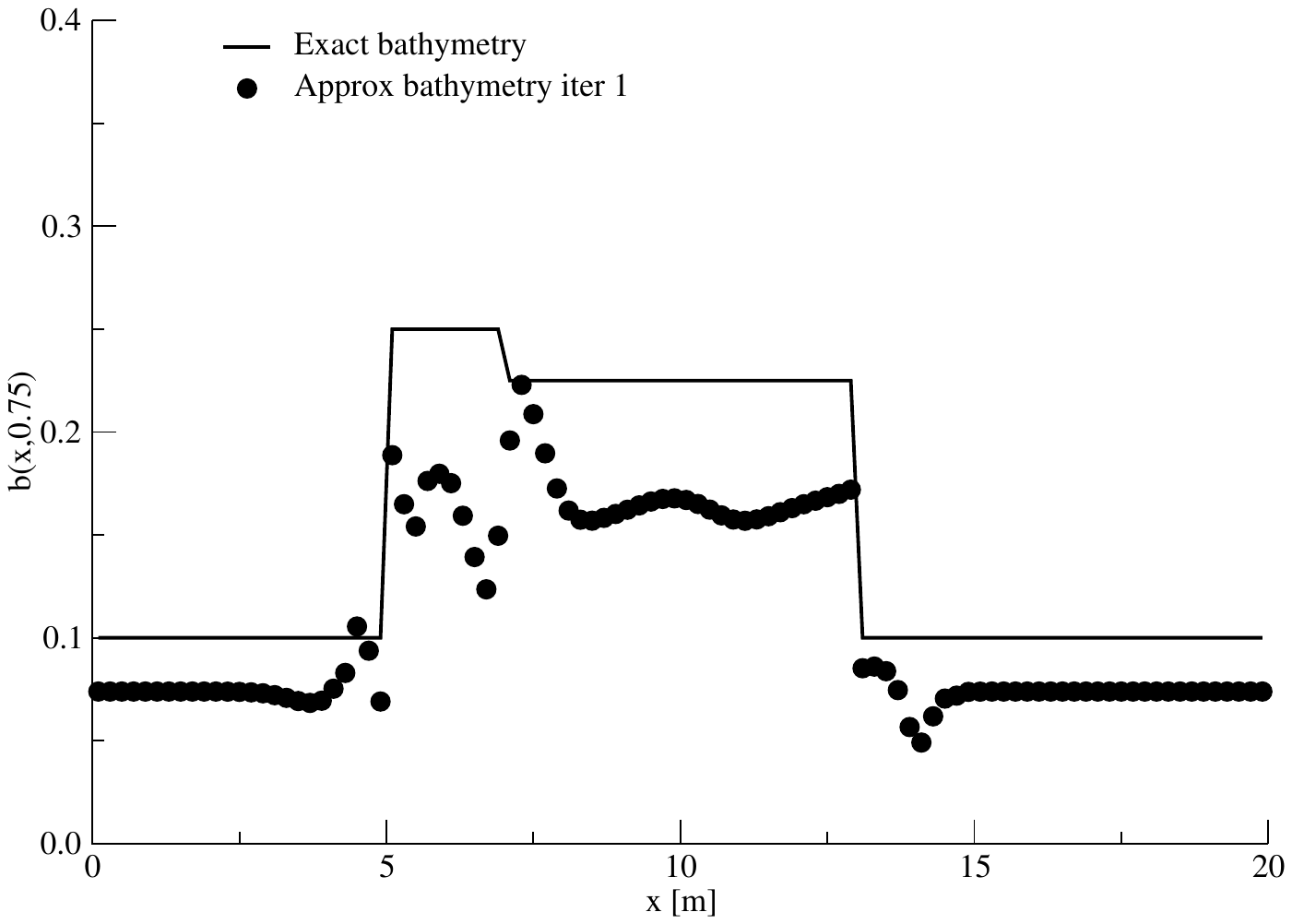} \\
	
	\includegraphics[width=0.3\textwidth]{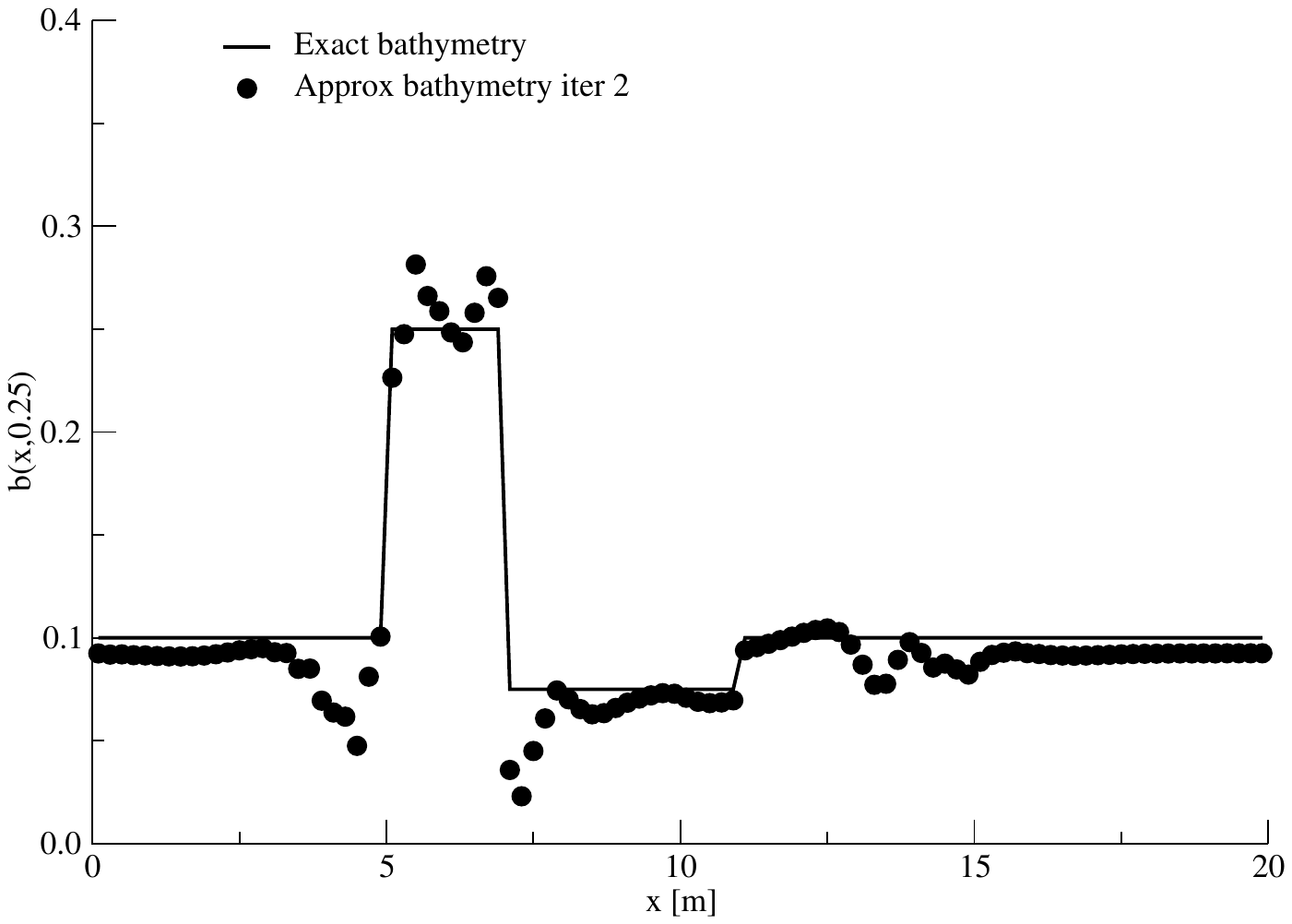} \quad
	\includegraphics[width=0.3\textwidth]{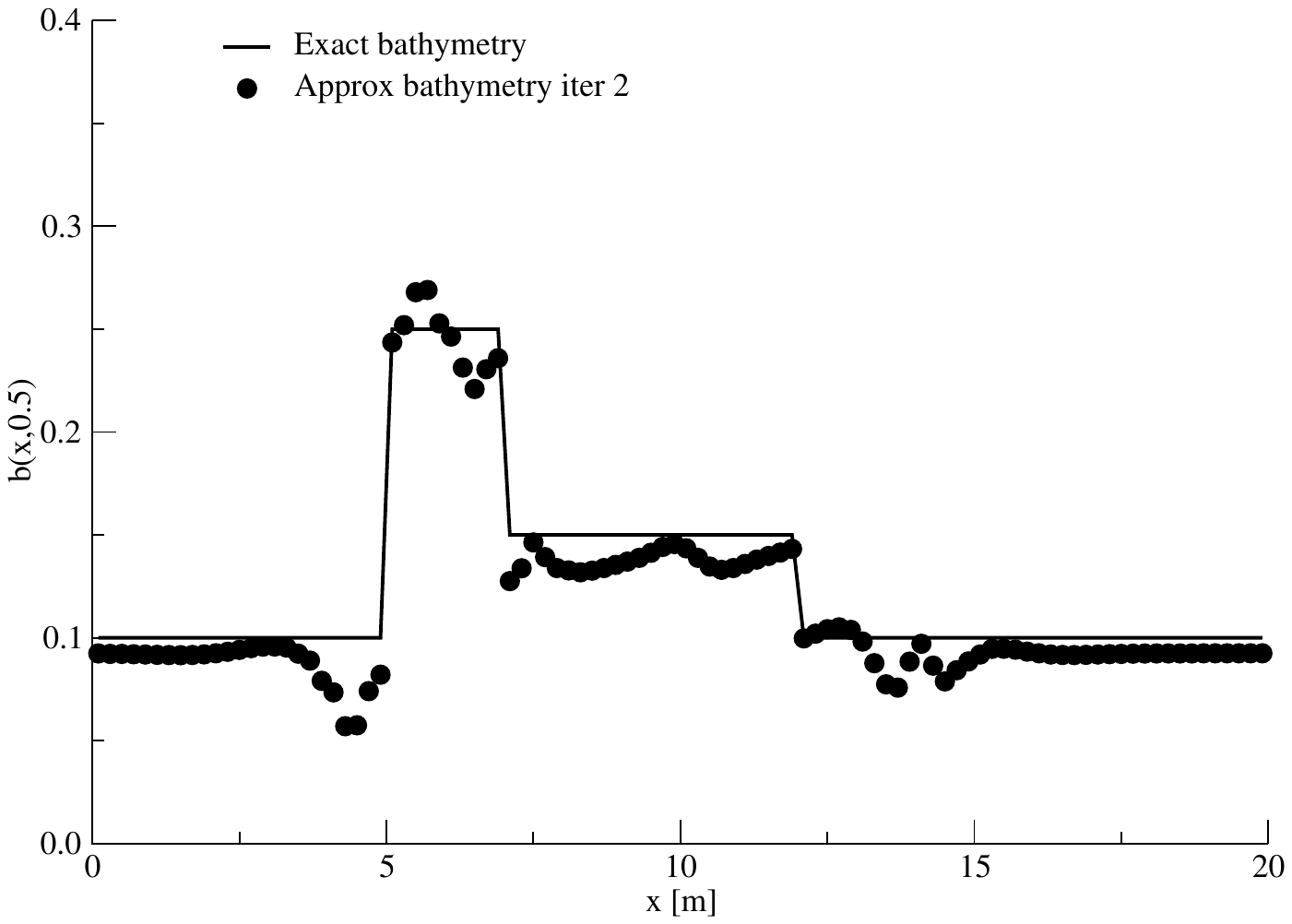} \quad
	\includegraphics[width=0.3\textwidth]{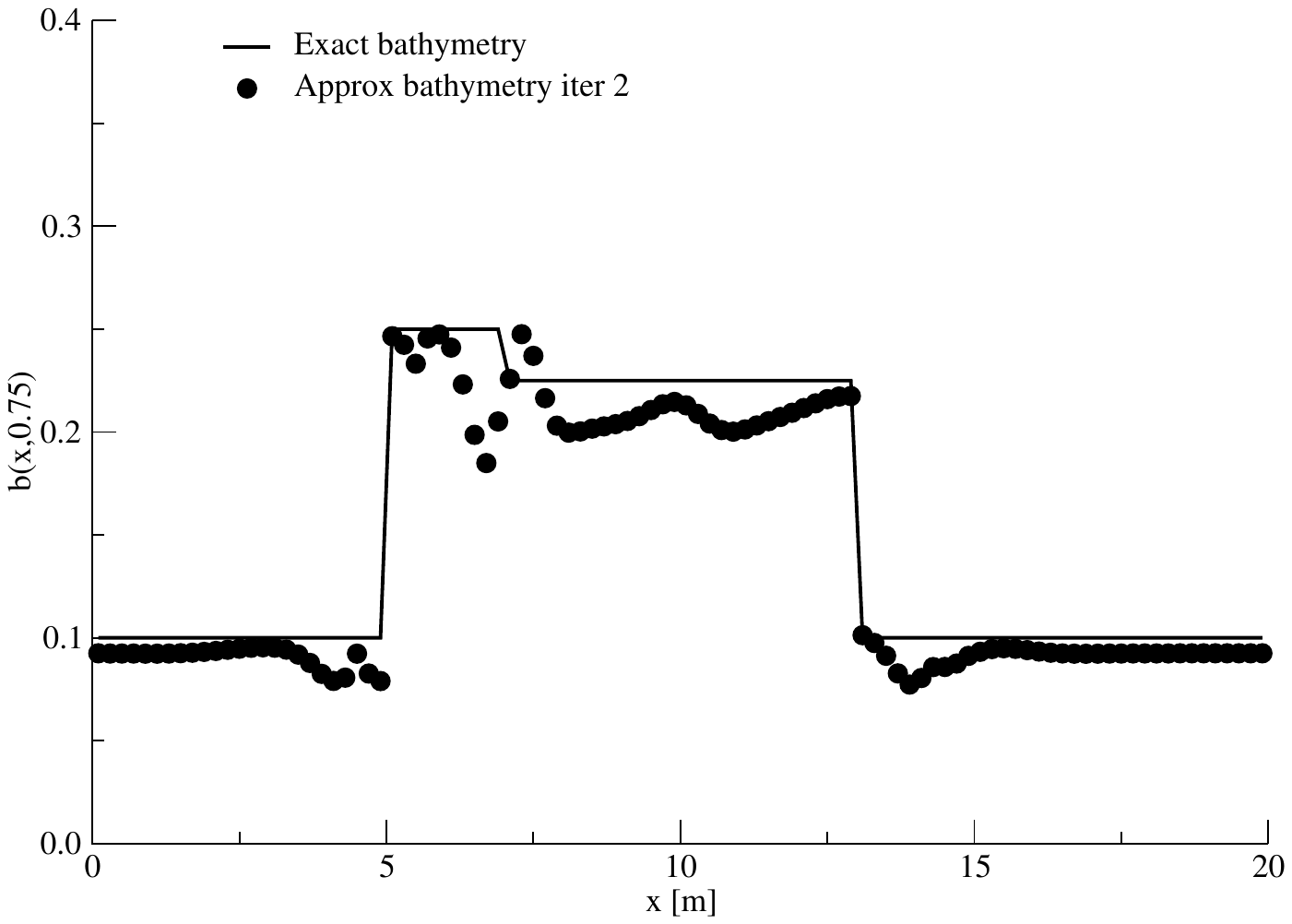} \\

	\includegraphics[width=0.3\textwidth]{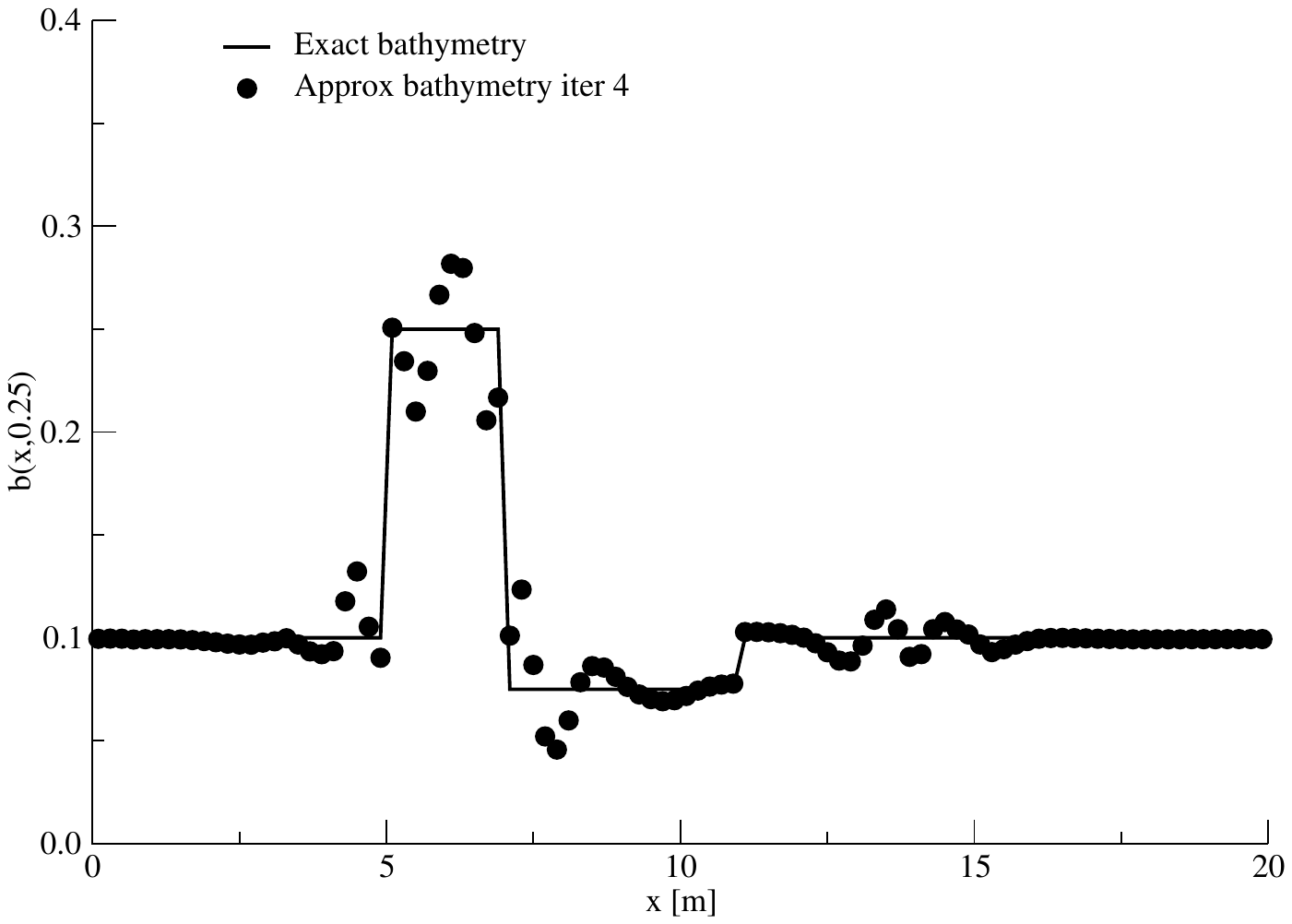} \quad
	\includegraphics[width=0.3\textwidth]{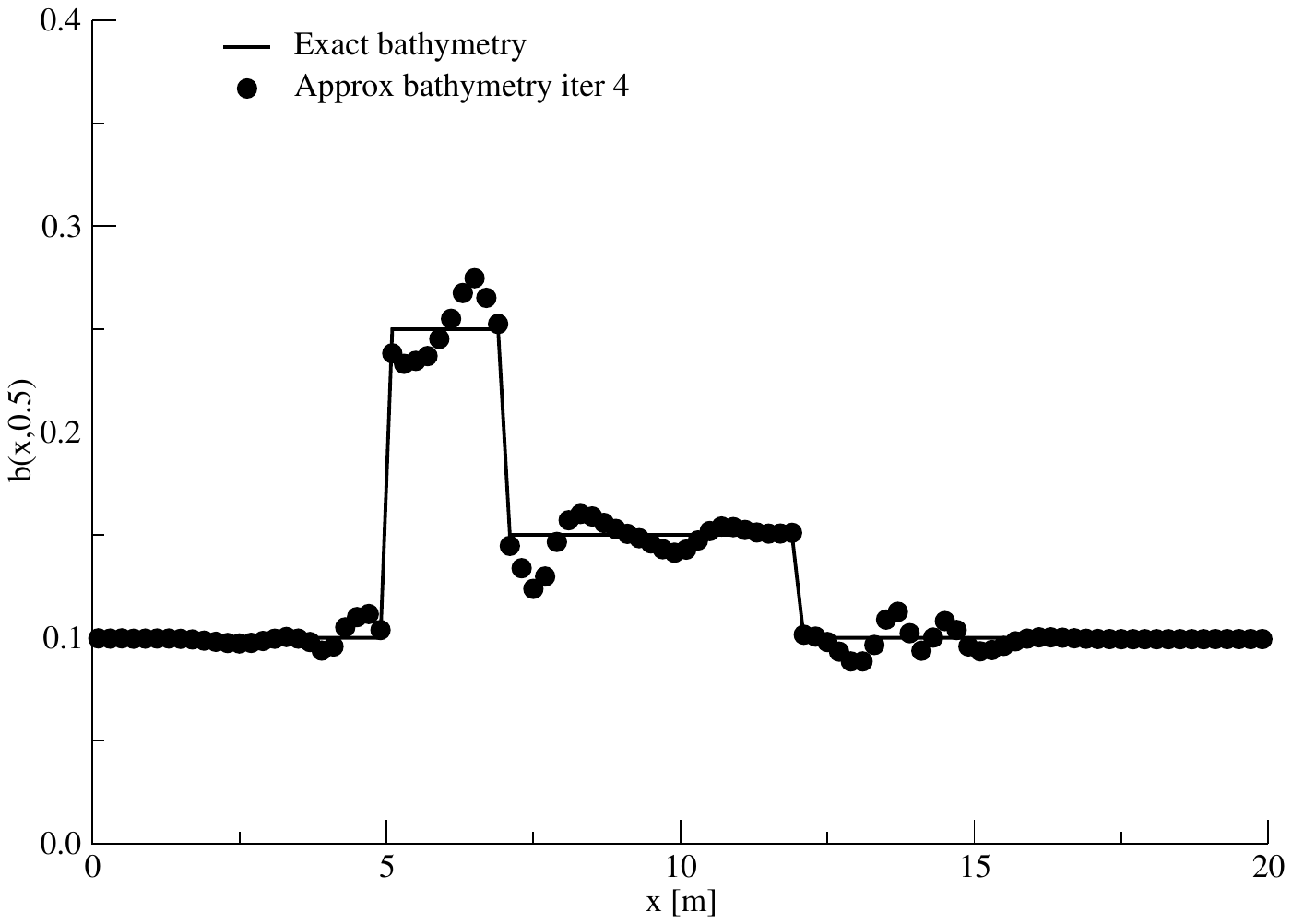} \quad
	\includegraphics[width=0.3\textwidth]{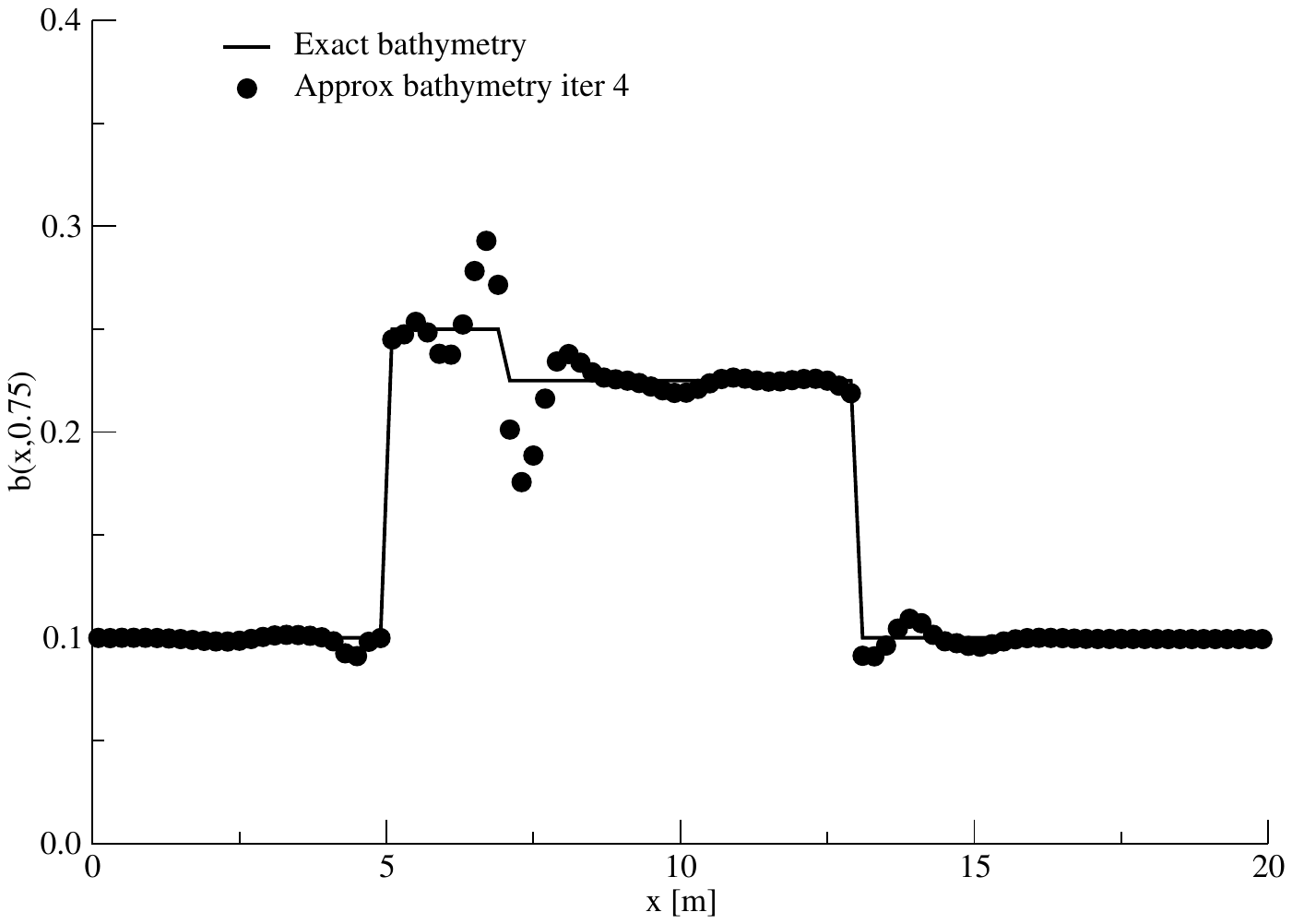} \\
	
	\includegraphics[width=0.3\textwidth]{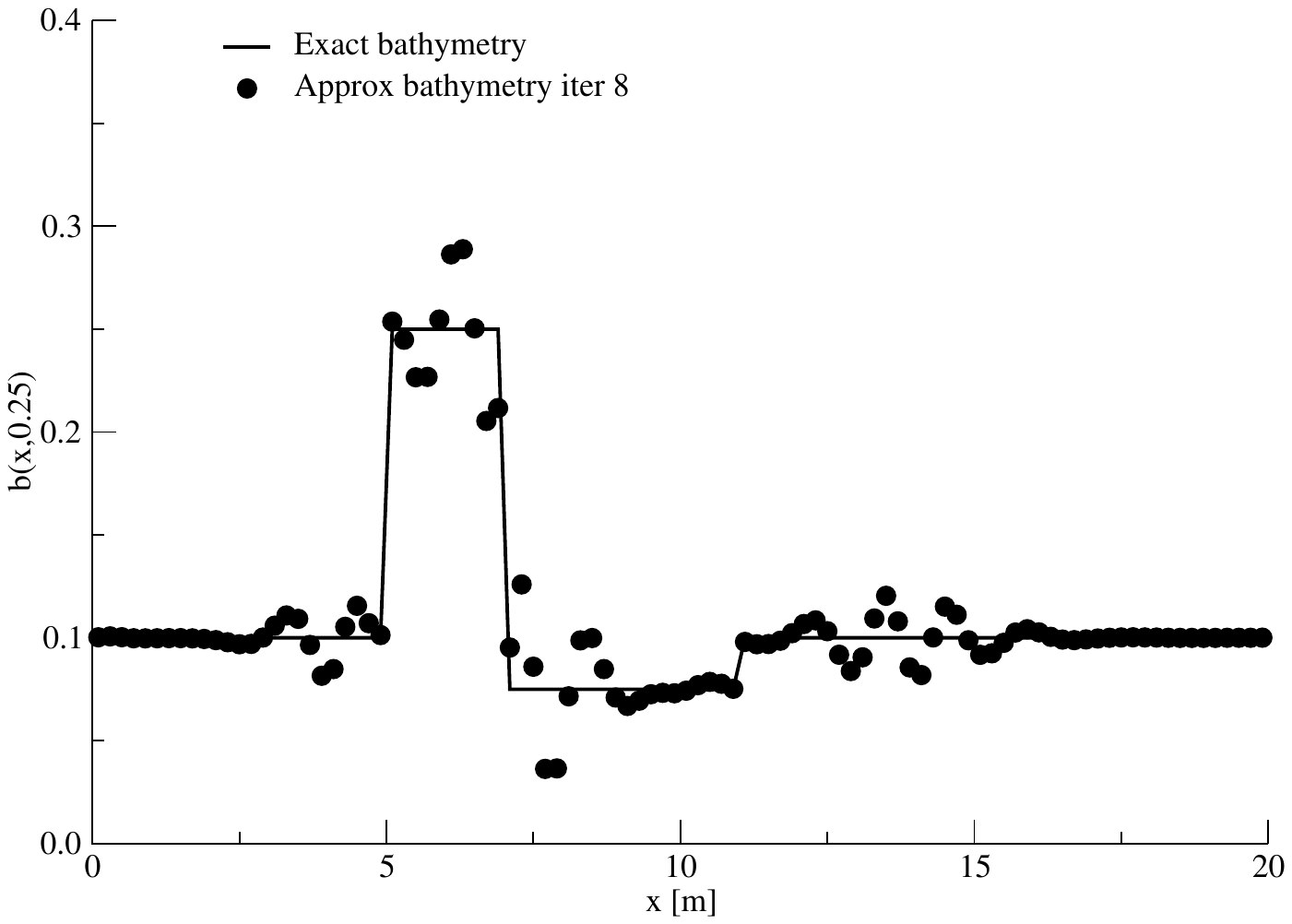} \quad
	\includegraphics[width=0.3\textwidth]{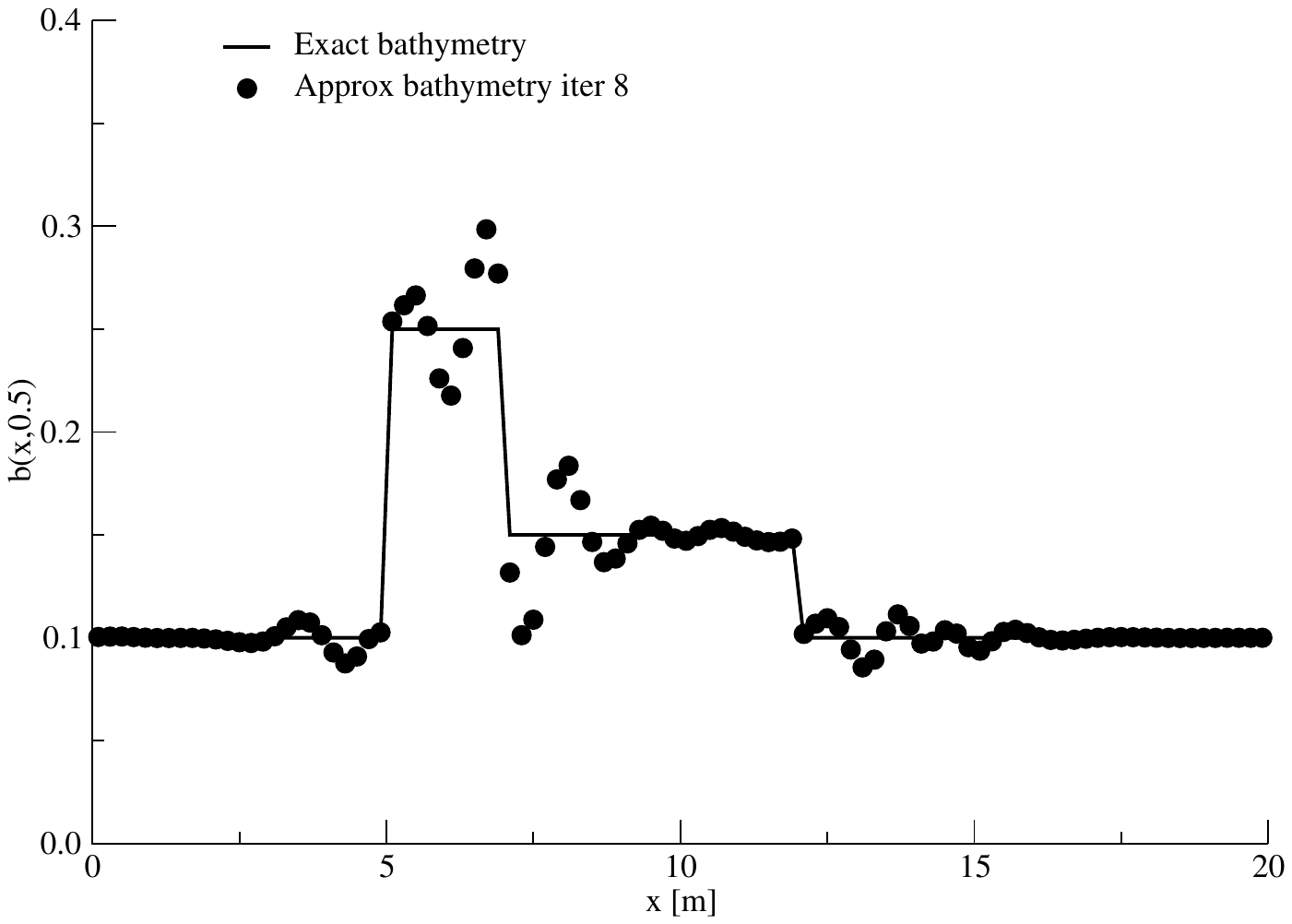} \quad
	\includegraphics[width=0.3\textwidth]{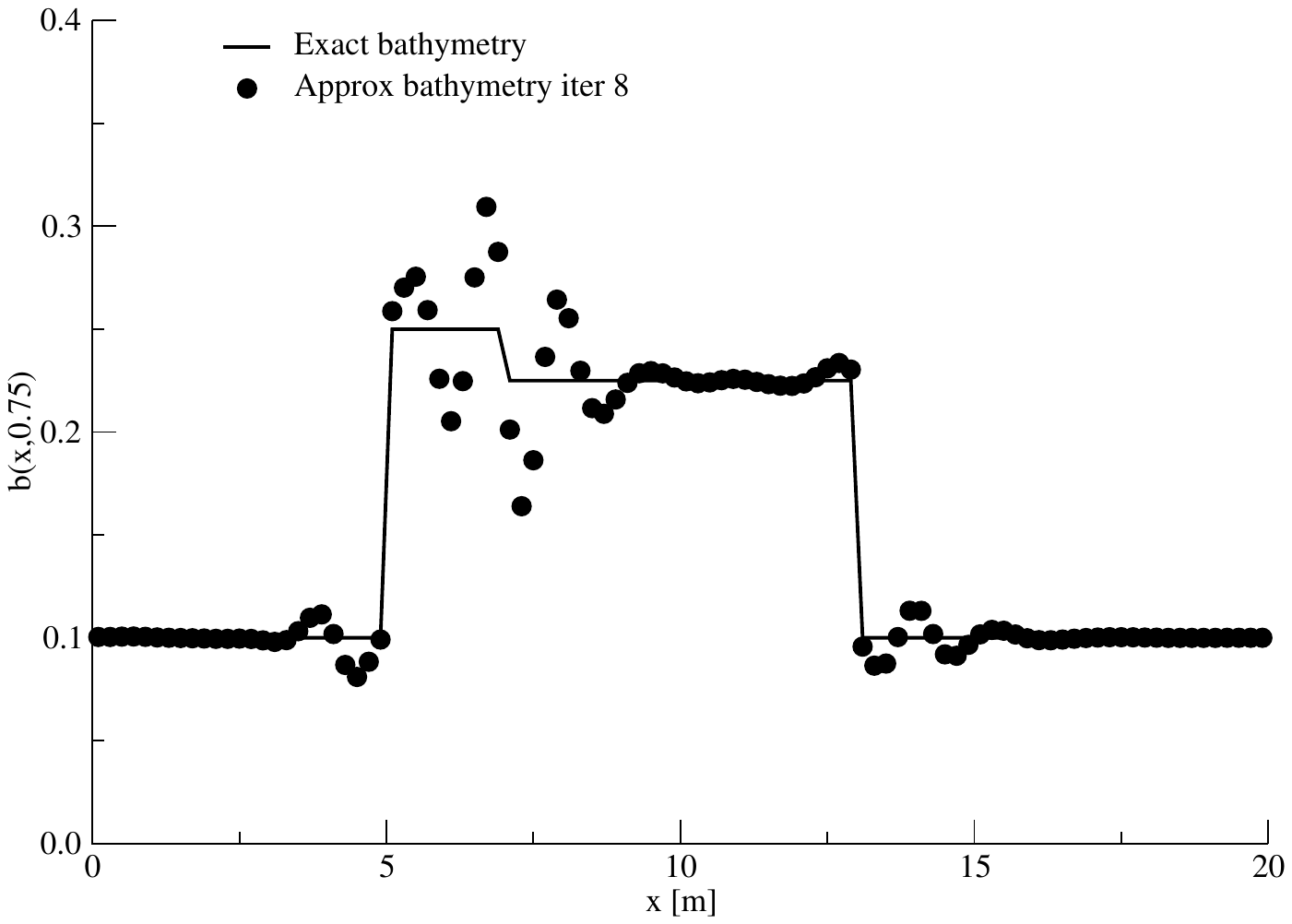} \\
	
	\caption{Discontinuous bottom profile (\ref{eq:b-discontinuous}): Result for the reconstruction procedure resulting from {\bf Rusanov+FD}. Parameters $\Delta t = 0.01$,  $\alpha_F = 2$, $\varepsilon = 0.001$, $\lambda_b = 0.71$, $100$ cells.
		{\bf Feft:} $t=0.25$, {\bf centered} $t=0.5$, {\bf right:} $t=0.75$.}
	\label{fig:b-for-iter-and-times:test-1-ConsRusanov}
\end{figure}
\FloatBarrier

\begin{figure}   
	\centering
	\includegraphics[width=0.3\textwidth]{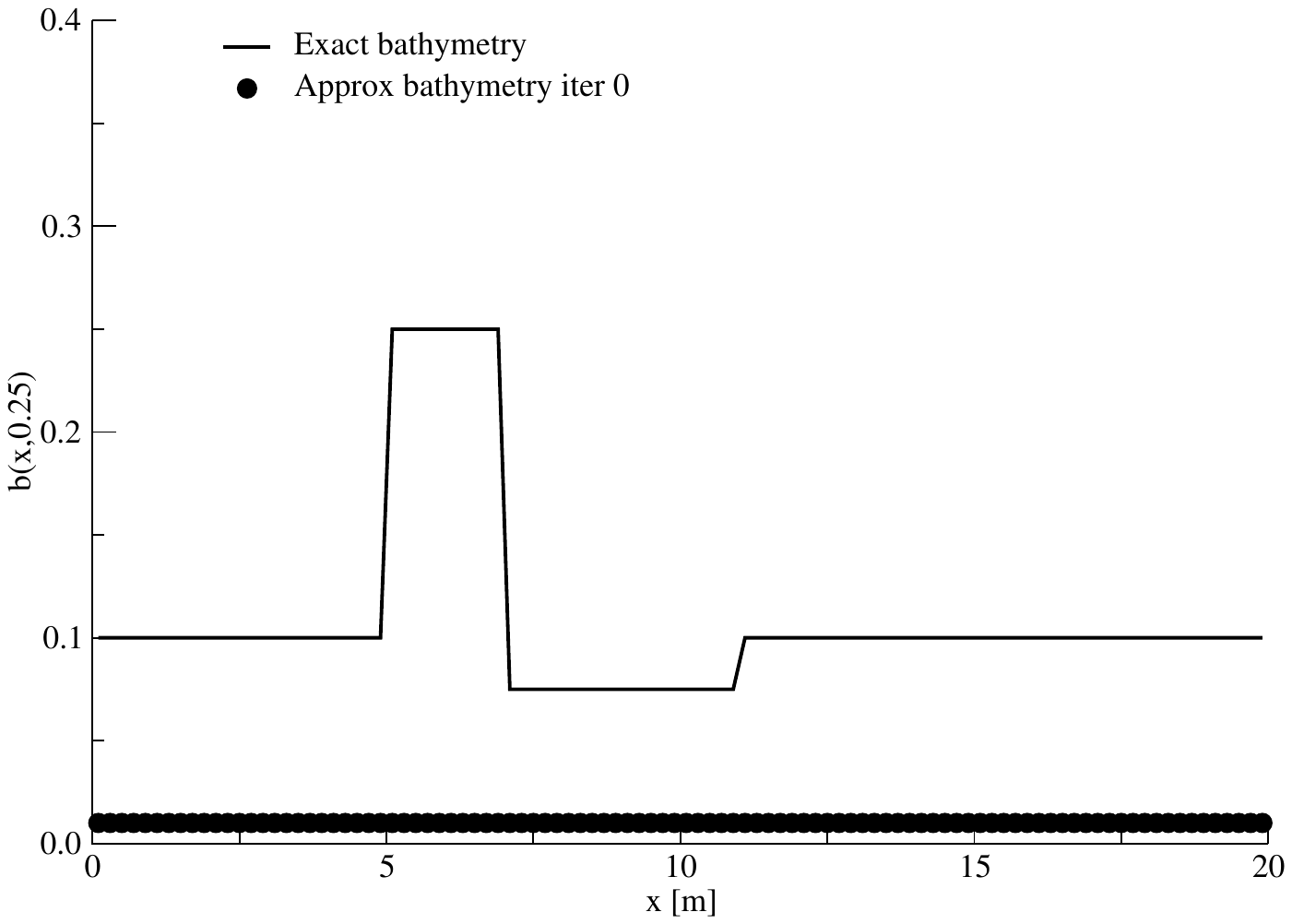} \quad
	\includegraphics[width=0.3\textwidth]{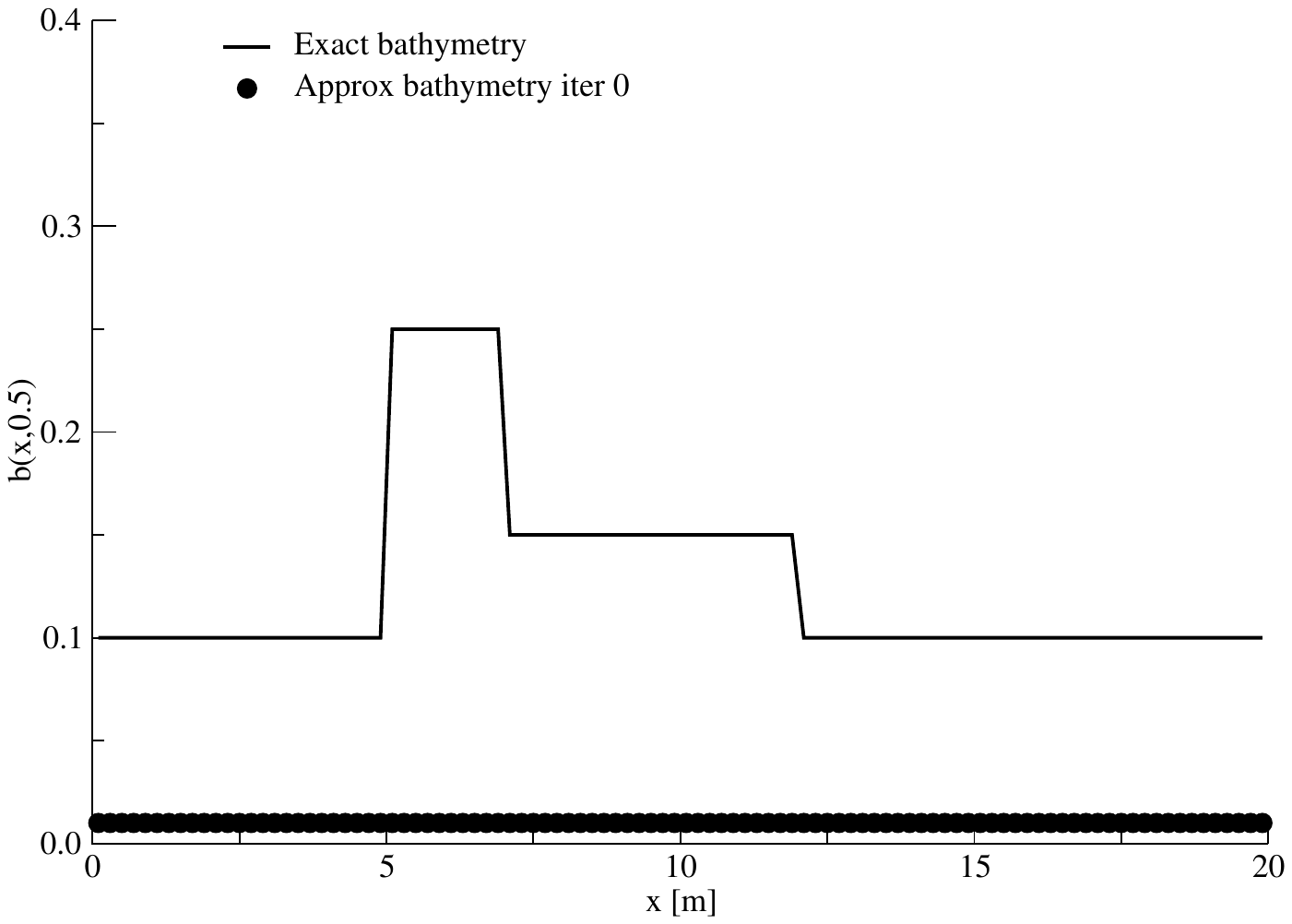} \quad
	\includegraphics[width=0.3\textwidth]{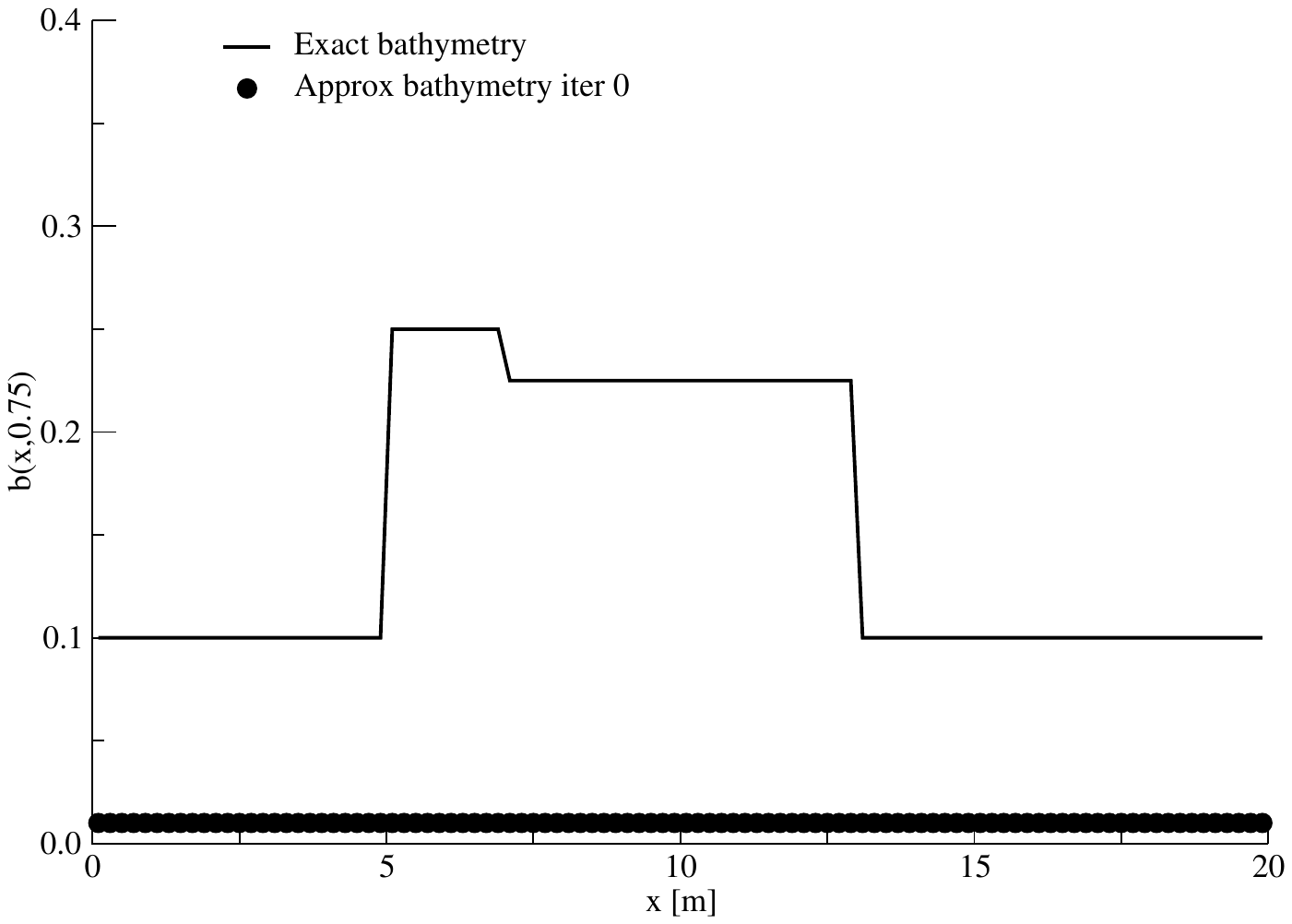} 
	\\
	
	\includegraphics[width=0.3\textwidth]{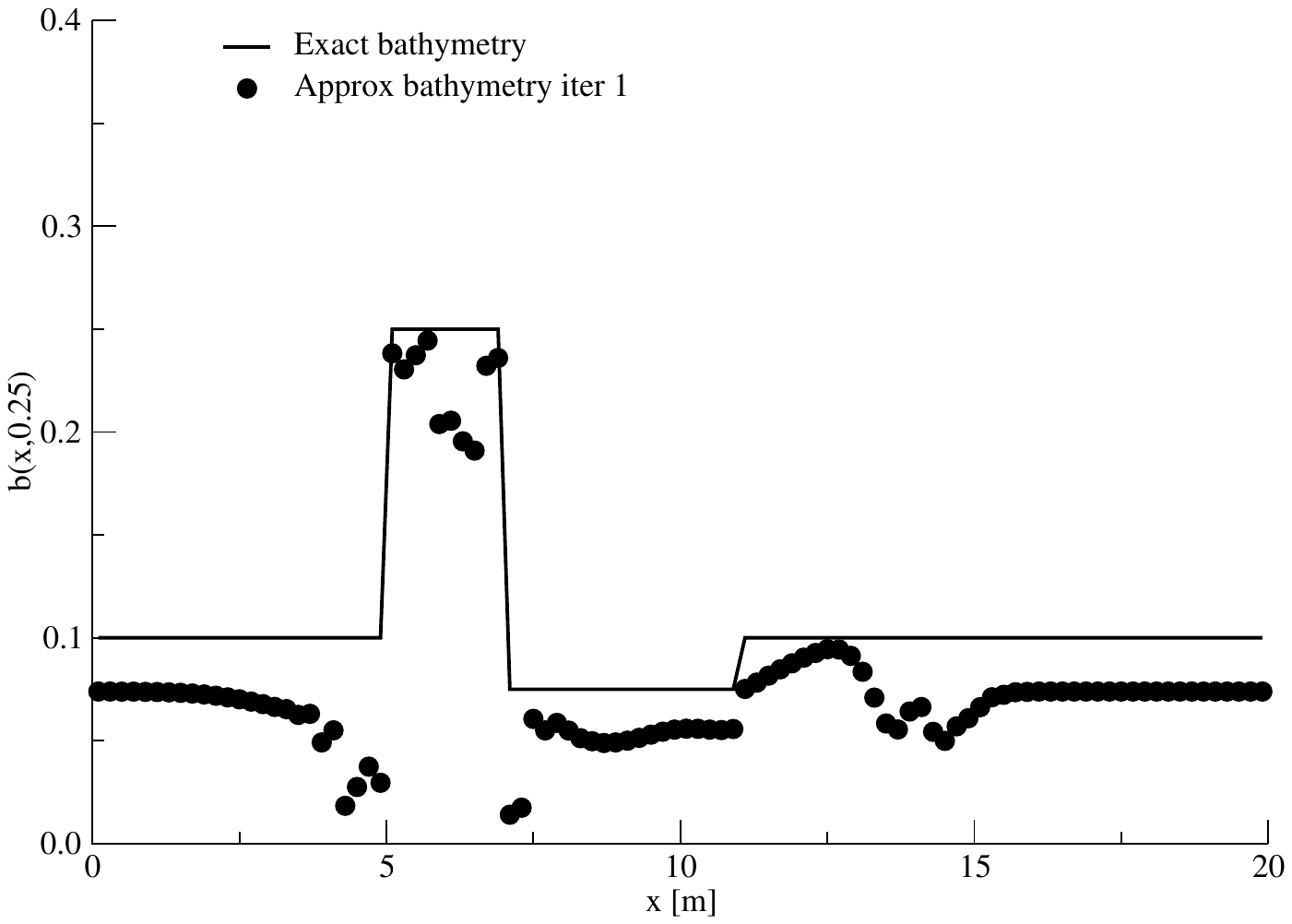} \quad
	\includegraphics[width=0.3\textwidth]{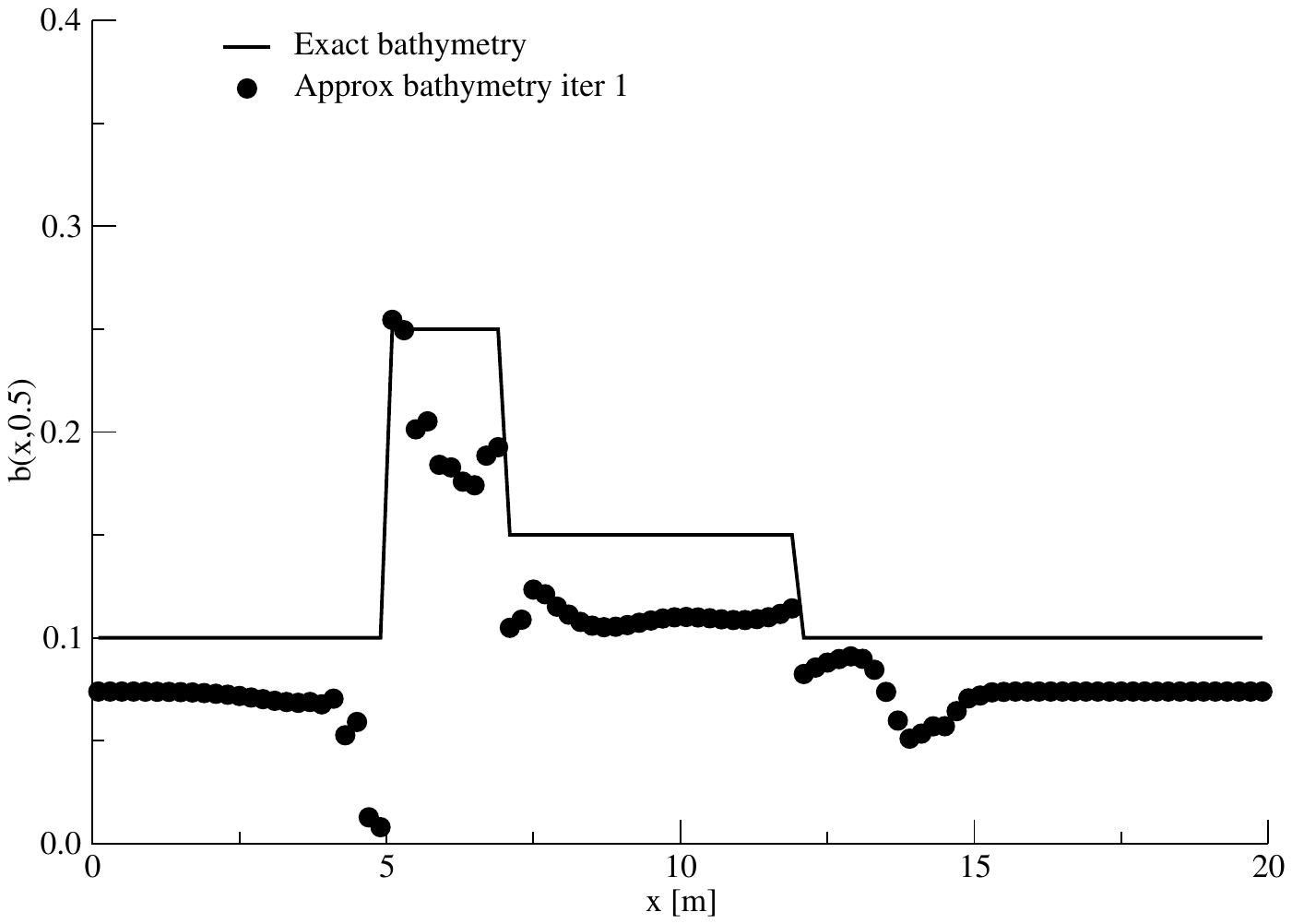} \quad
	\includegraphics[width=0.3\textwidth]{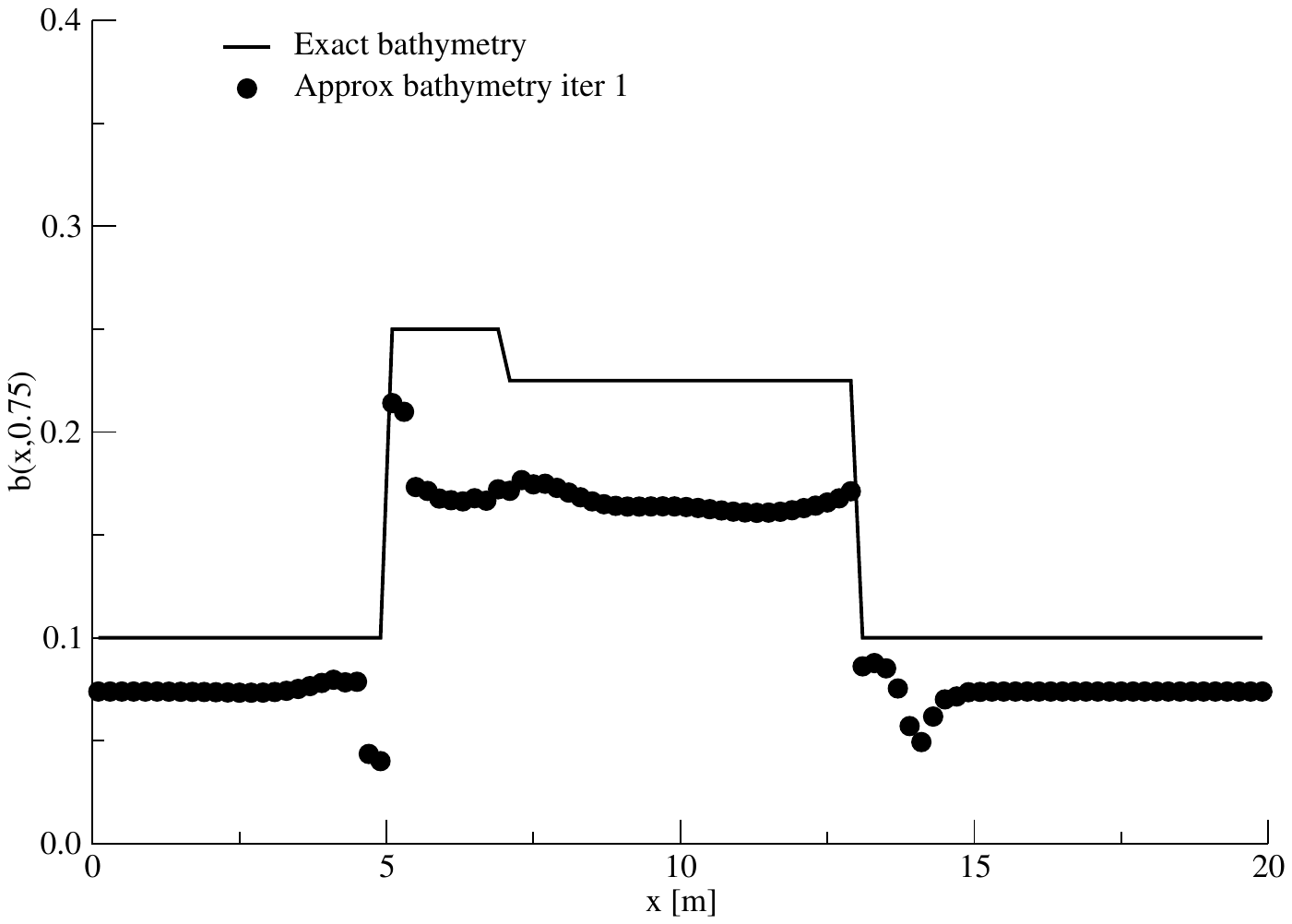} \\
	
	\includegraphics[width=0.3\textwidth]{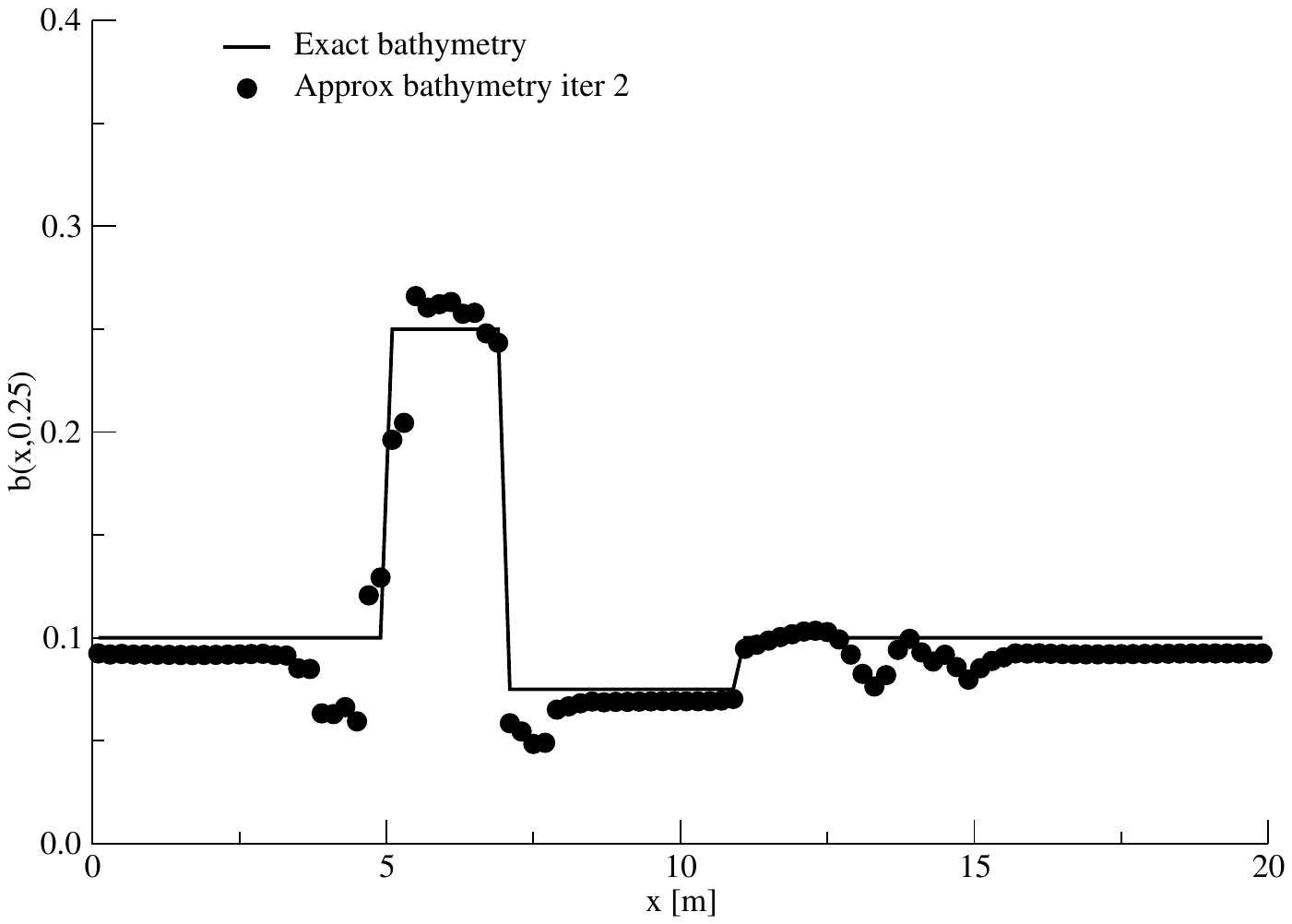} \quad
	\includegraphics[width=0.3\textwidth]{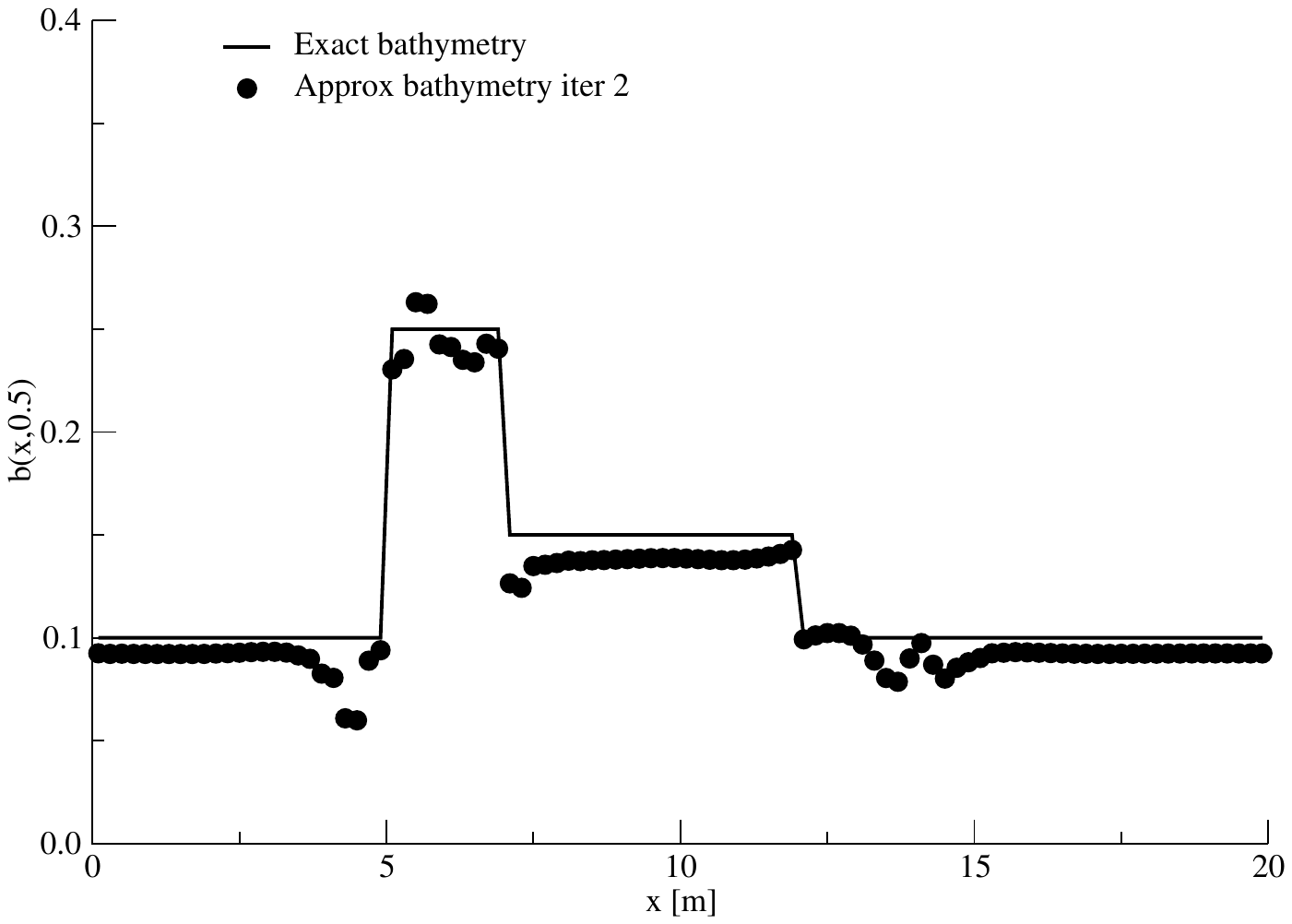} \quad
	\includegraphics[width=0.3\textwidth]{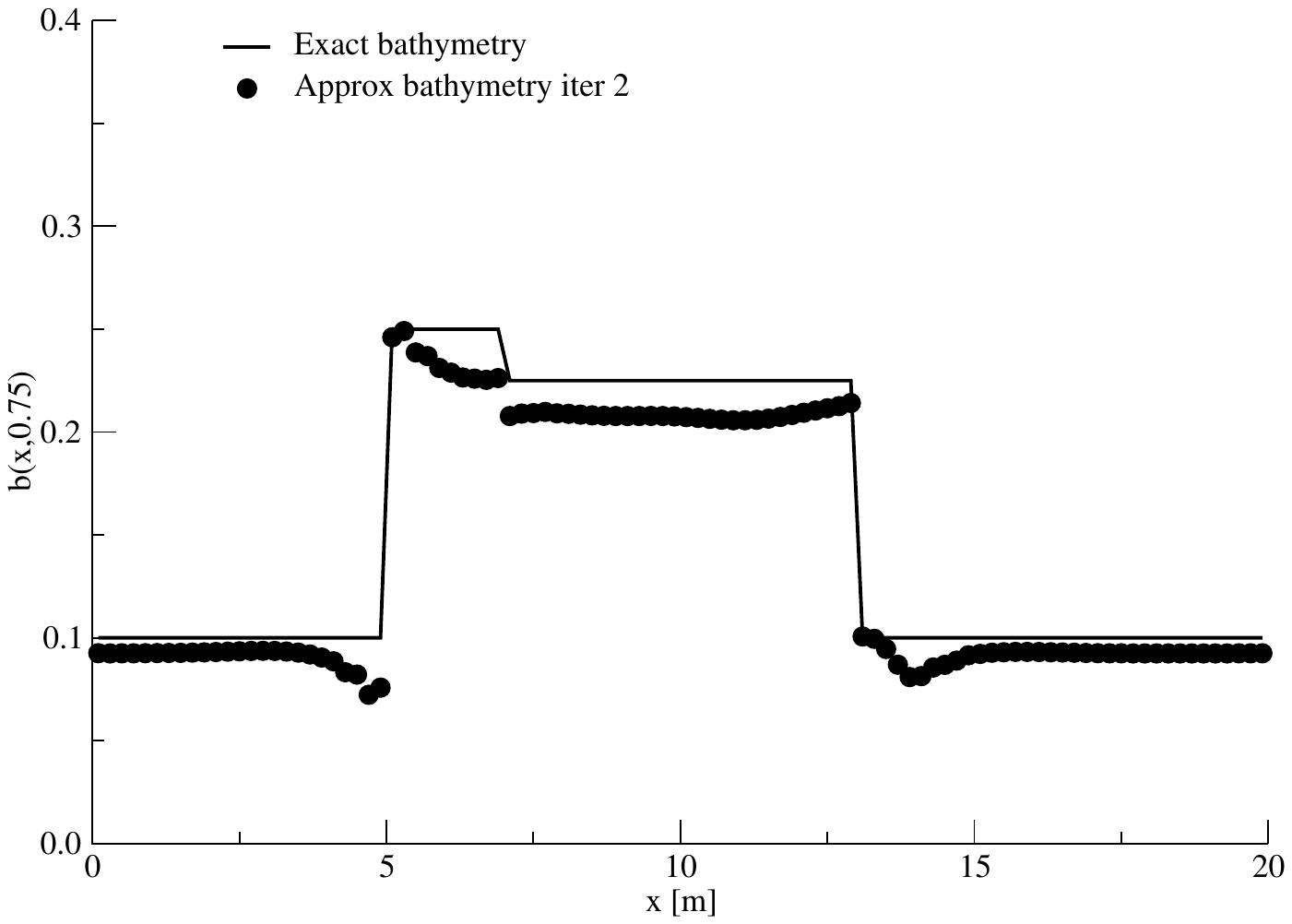} \\

	\includegraphics[width=0.3\textwidth]{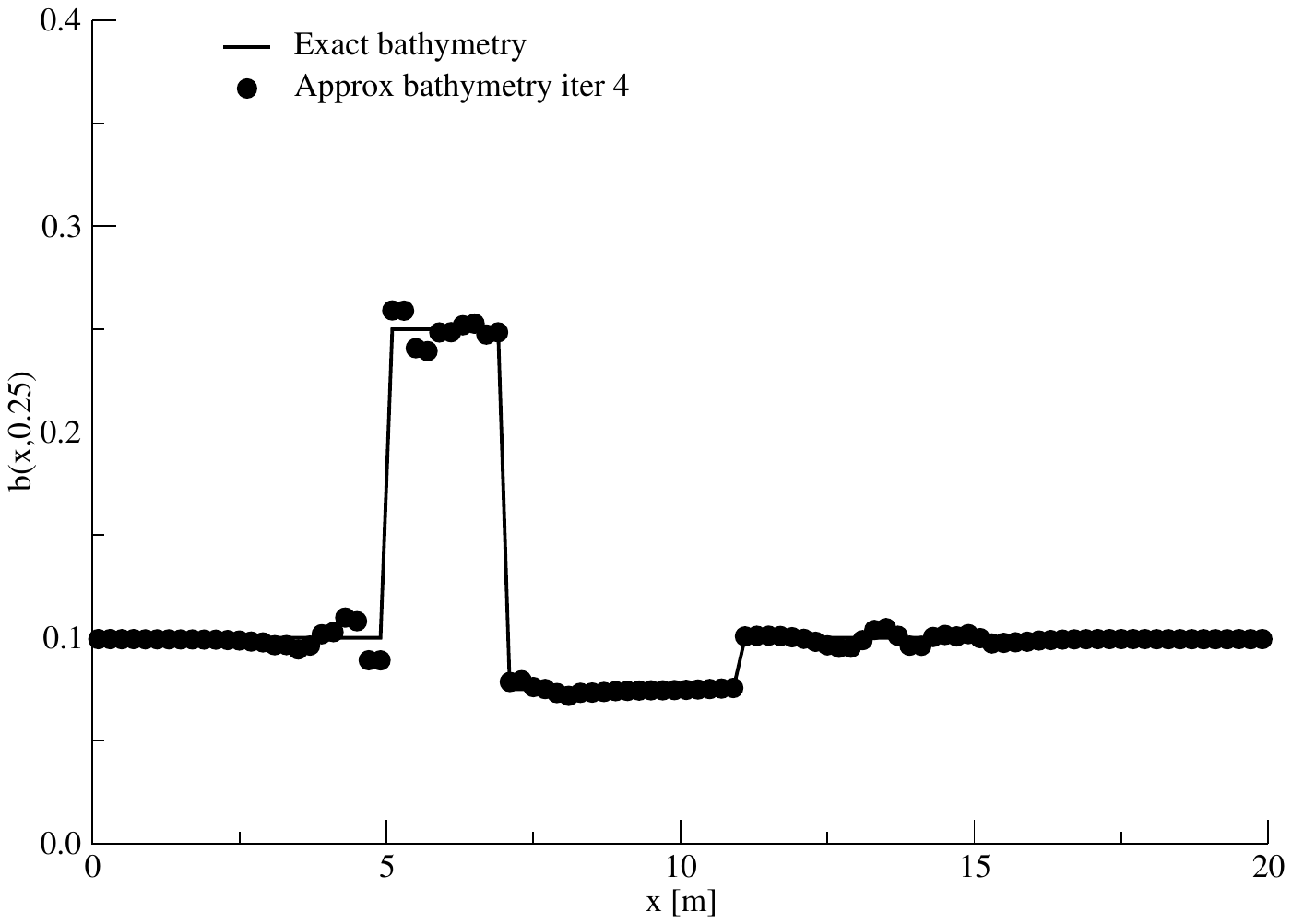} \quad
	\includegraphics[width=0.3\textwidth]{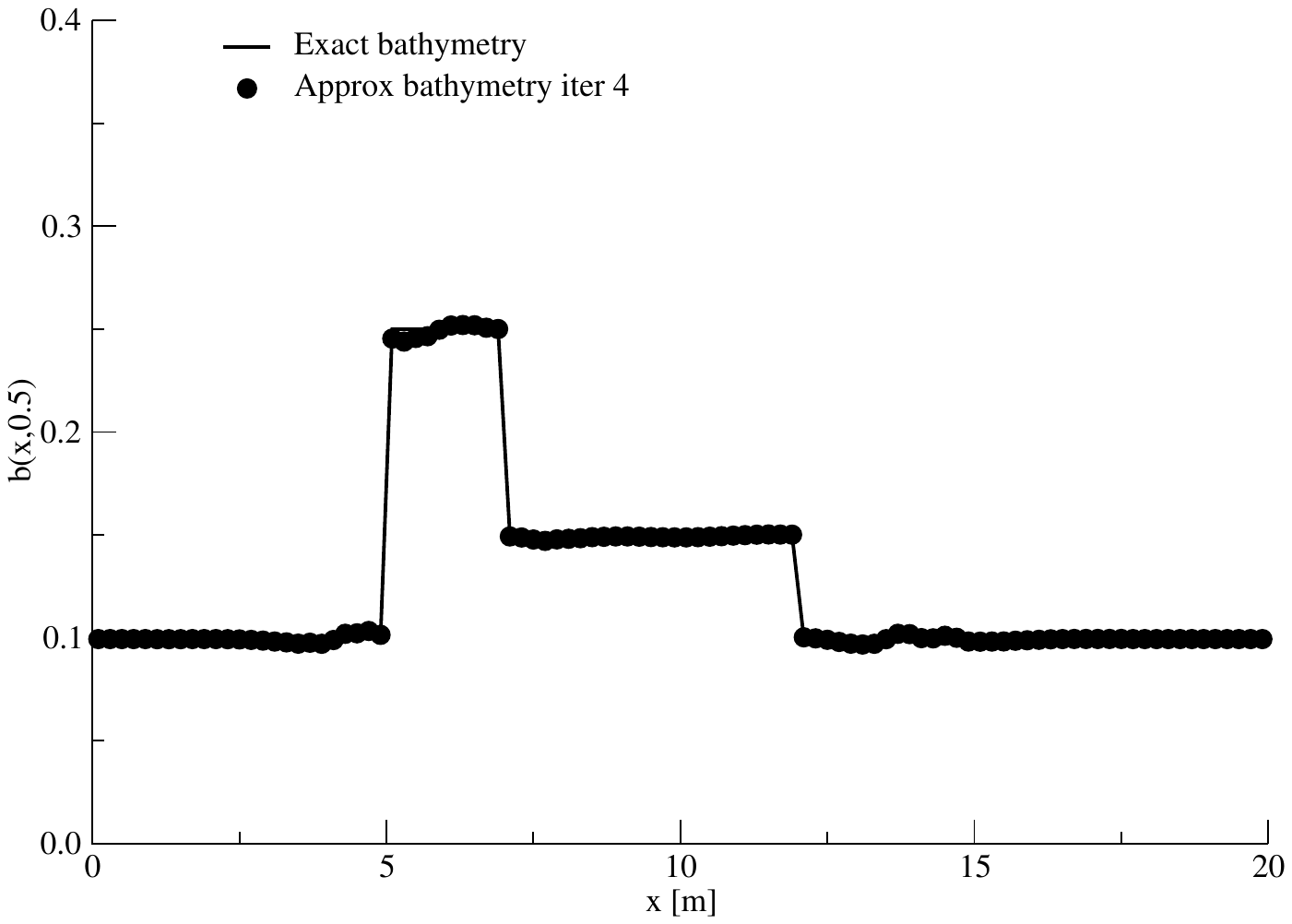} \quad
	\includegraphics[width=0.3\textwidth]{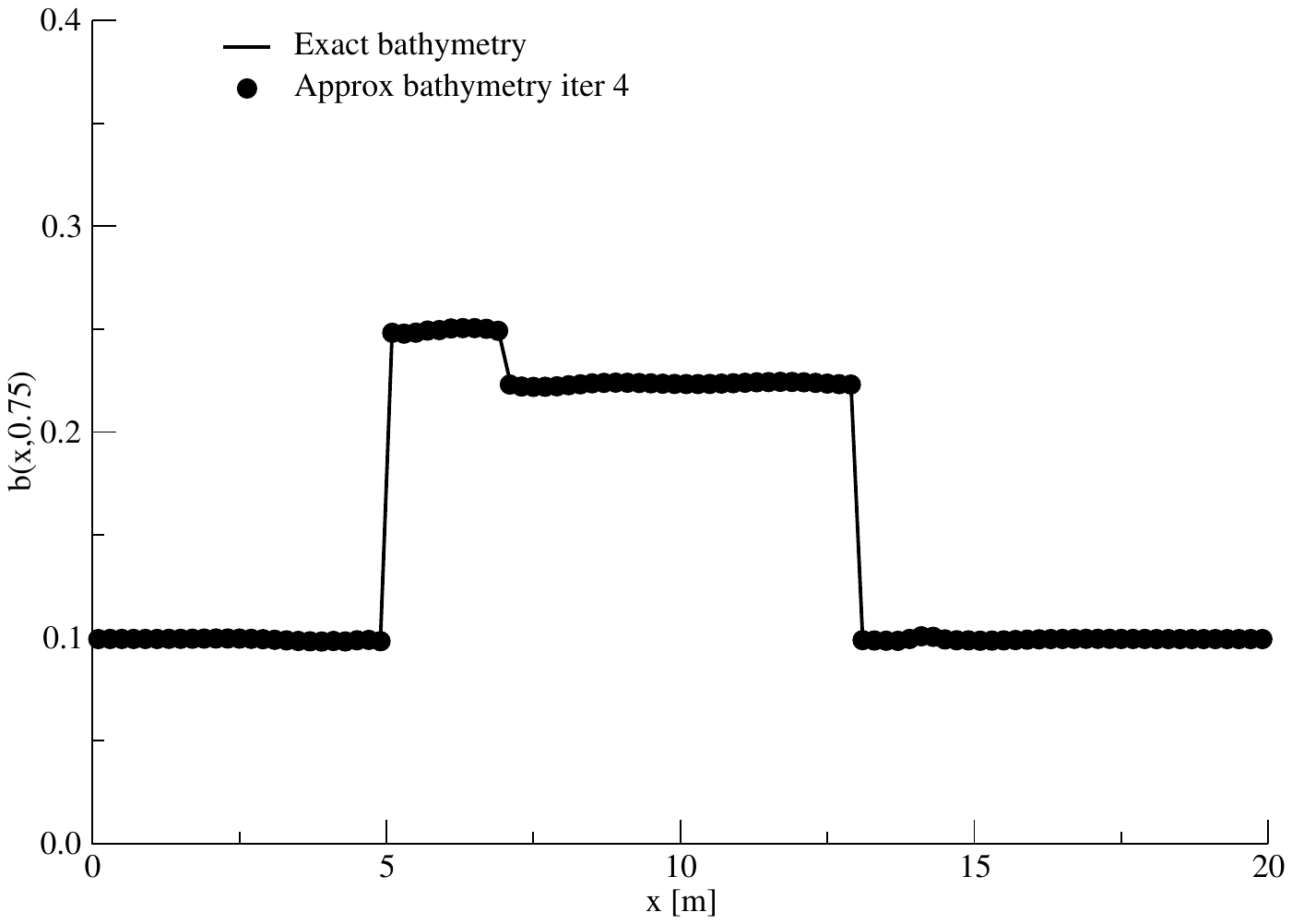} \\
	
	\includegraphics[width=0.3\textwidth]{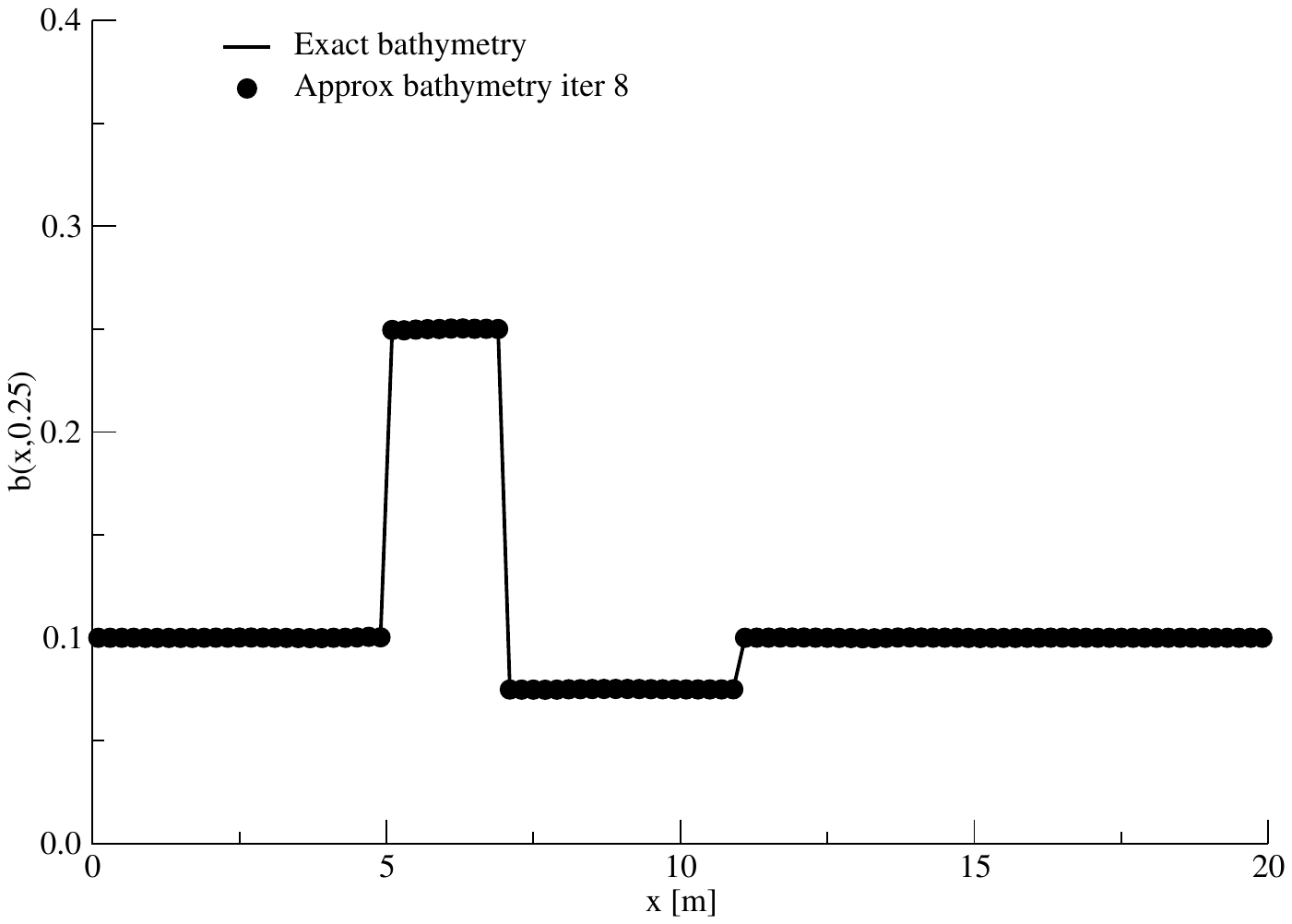} \quad
	\includegraphics[width=0.3\textwidth]{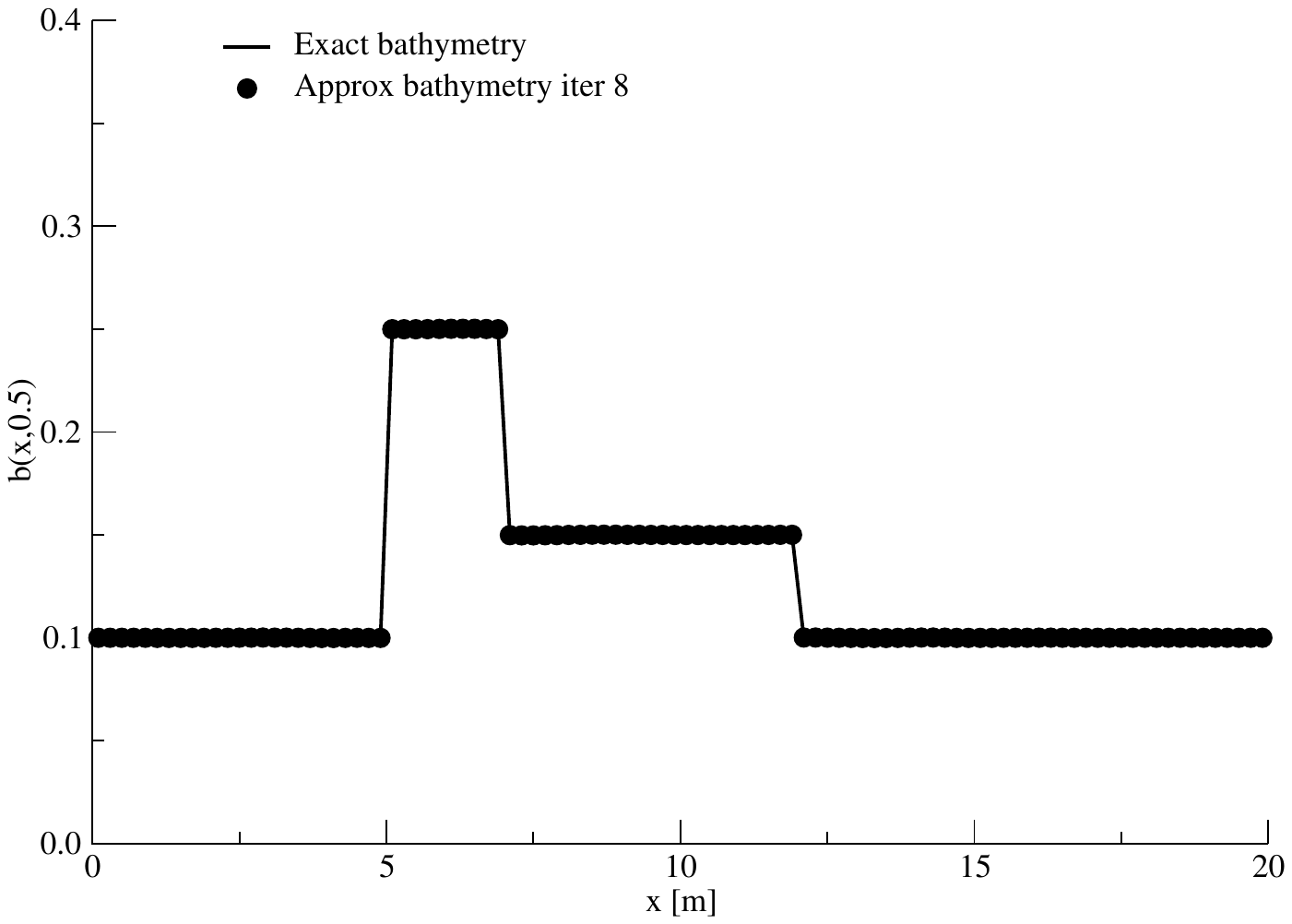} \quad
	\includegraphics[width=0.3\textwidth]{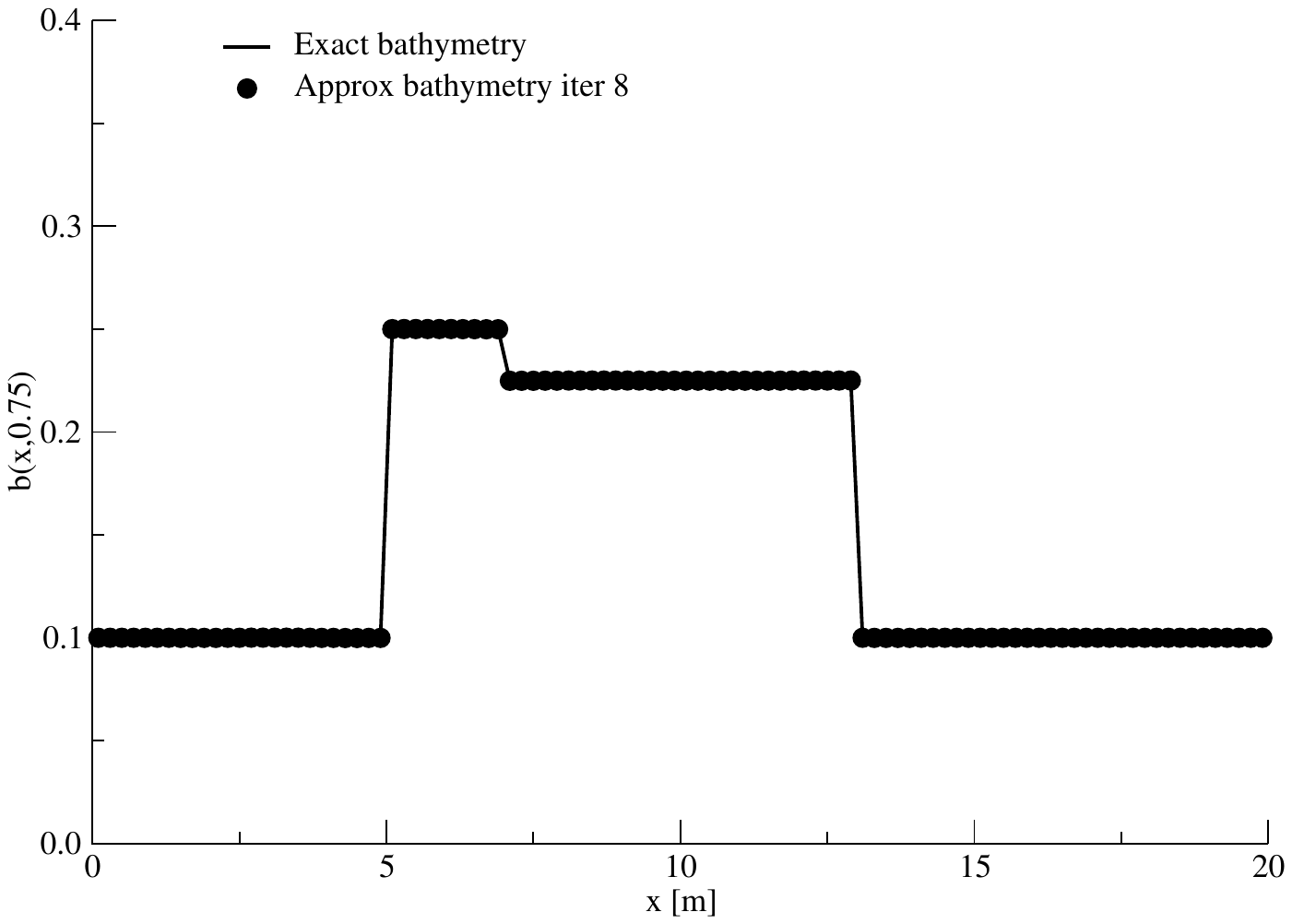} \\
	
	\caption{Discontinuous bottom profile (\ref{eq:b-discontinuous}): Result for the reconstruction procedure resulting from non-conservative {\bf FORCE-$\alpha$+FD}. Parameters $\Delta t = 0.01$,  $\alpha_F = 2$, $\varepsilon = 0.001$, $\lambda_b = 0.71$, $100$ cells.
		{\bf Feft:} $t=0.25$, {\bf centered} $t=0.5$, {\bf right:} $t=0.75$.}
	\label{fig:b-for-iter-and-times:test-1:ForceAlphaCons}
\end{figure}
\FloatBarrier

\begin{figure}   
	\centering
	\includegraphics[width=0.3\textwidth]{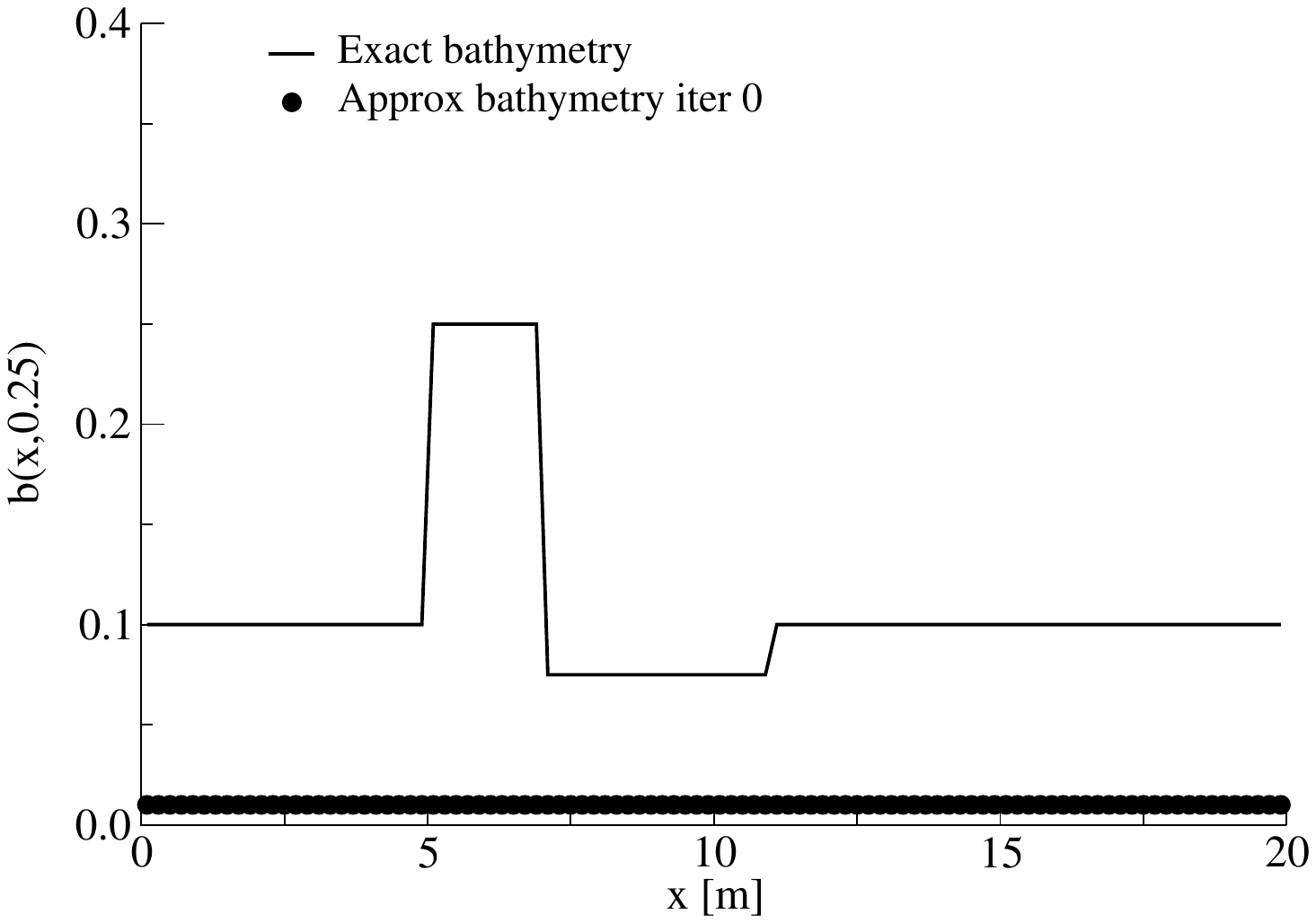} \quad
	\includegraphics[width=0.3\textwidth]{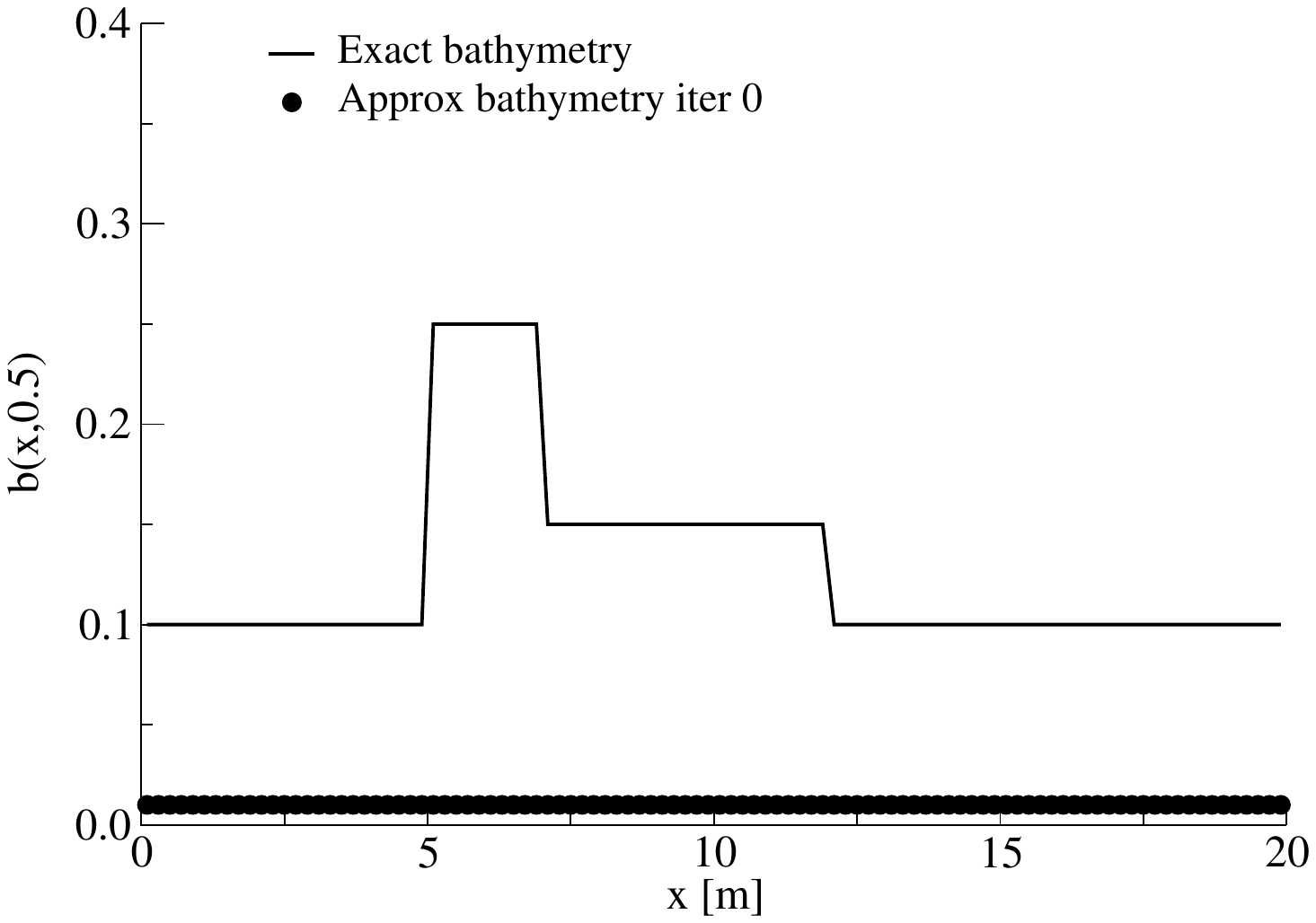} \quad
	\includegraphics[width=0.3\textwidth]{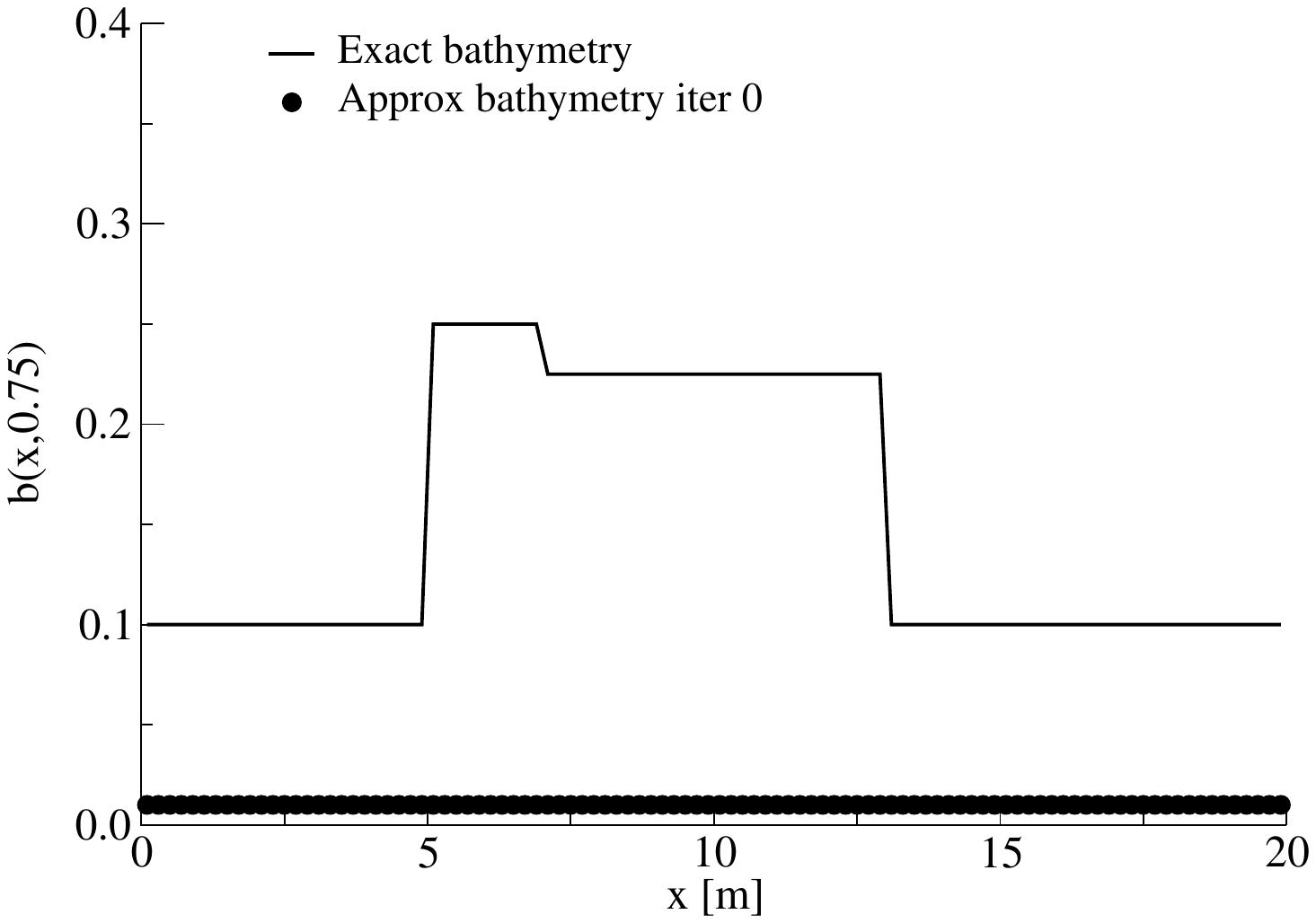} 
	\\
	
	\includegraphics[width=0.3\textwidth]{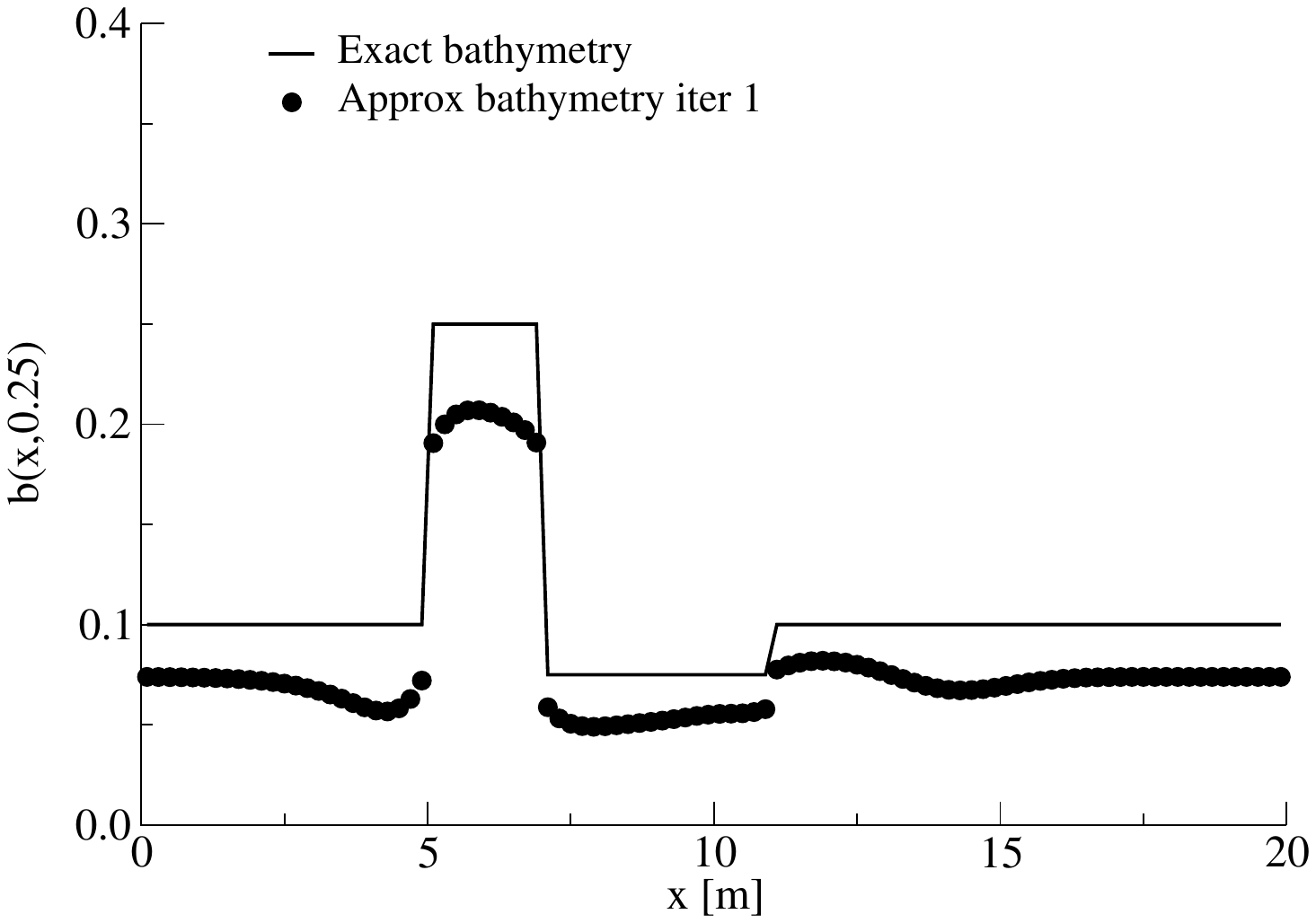} \quad
	\includegraphics[width=0.3\textwidth]{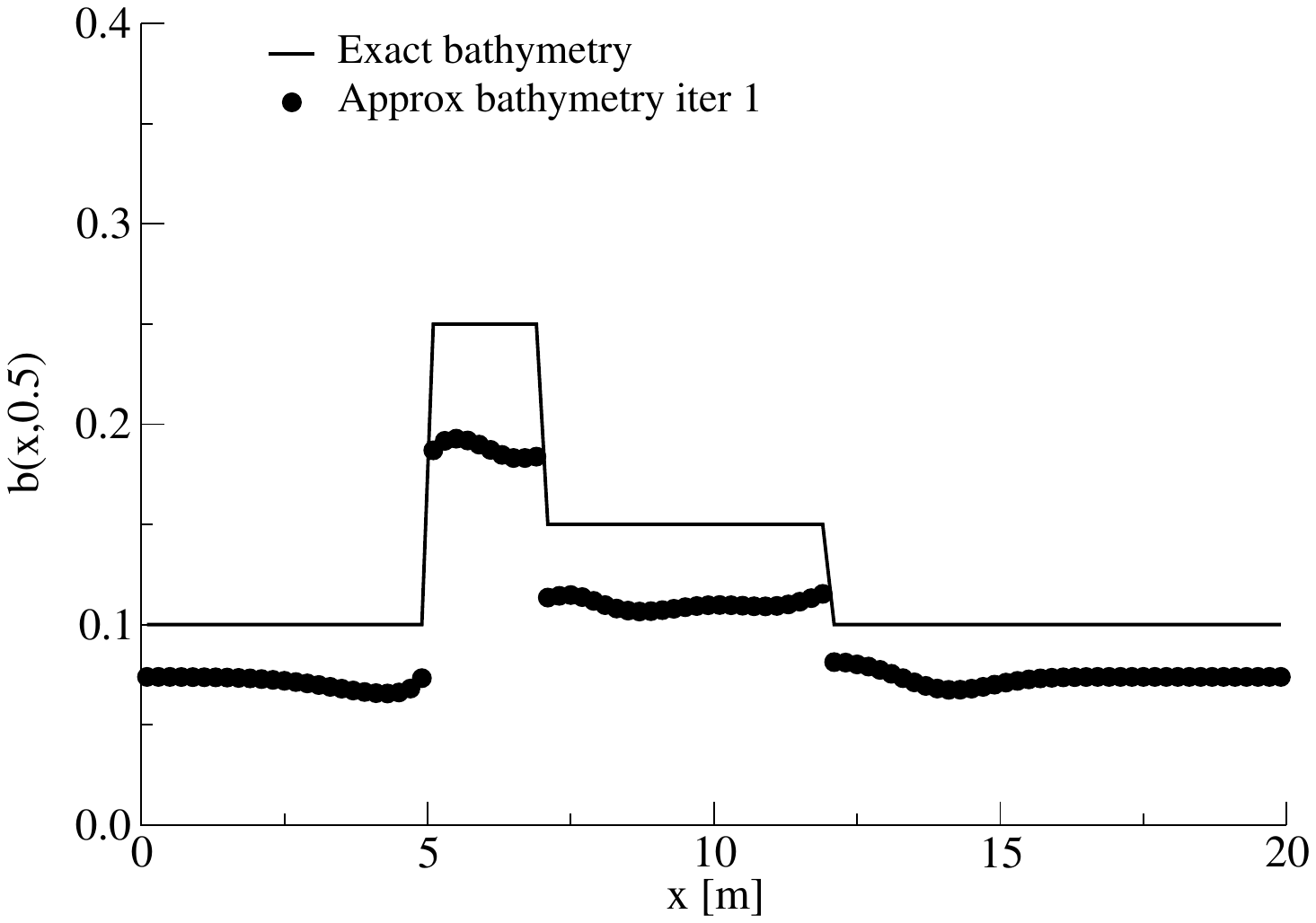} \quad
	\includegraphics[width=0.3\textwidth]{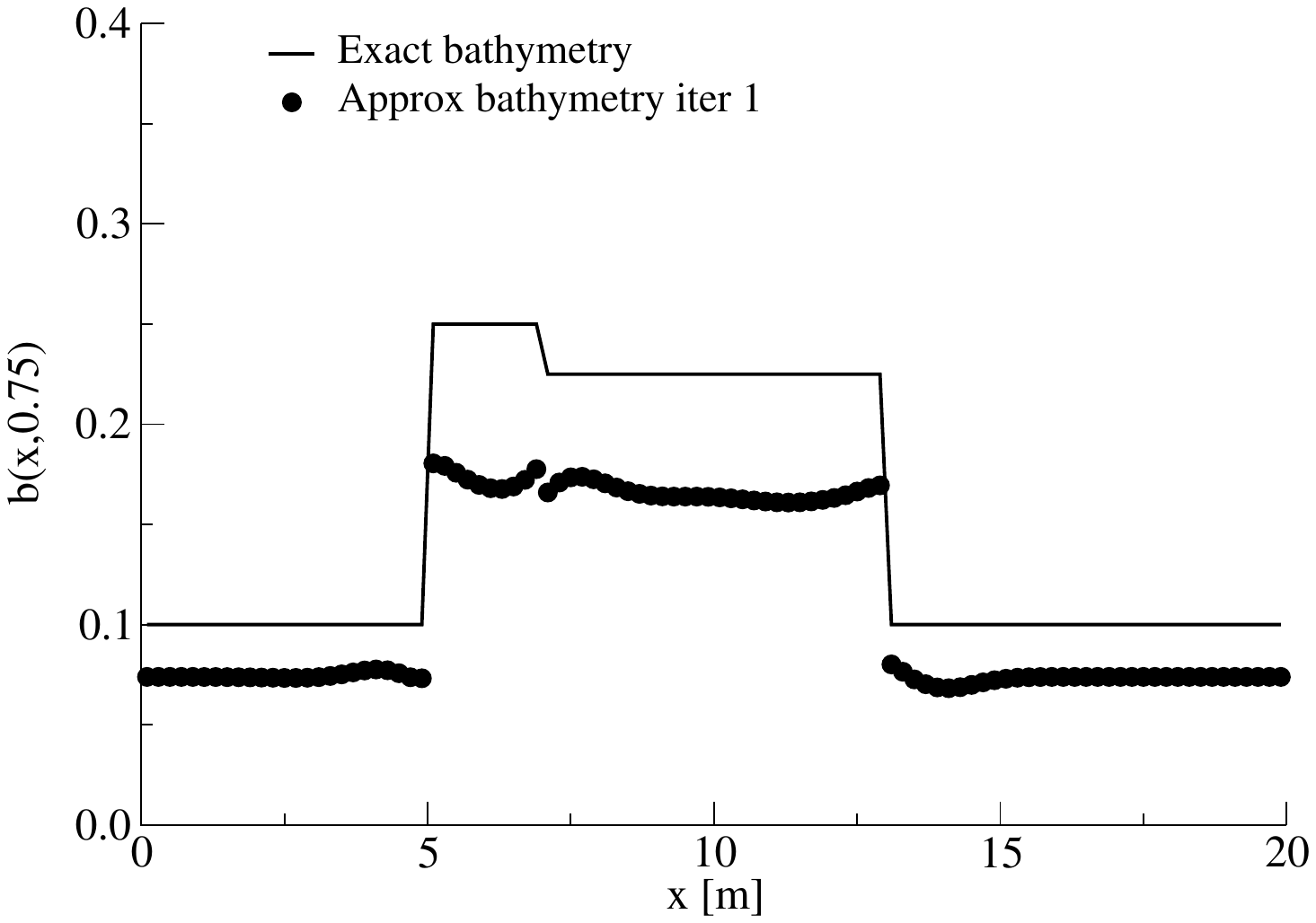} \\
	
	\includegraphics[width=0.3\textwidth]{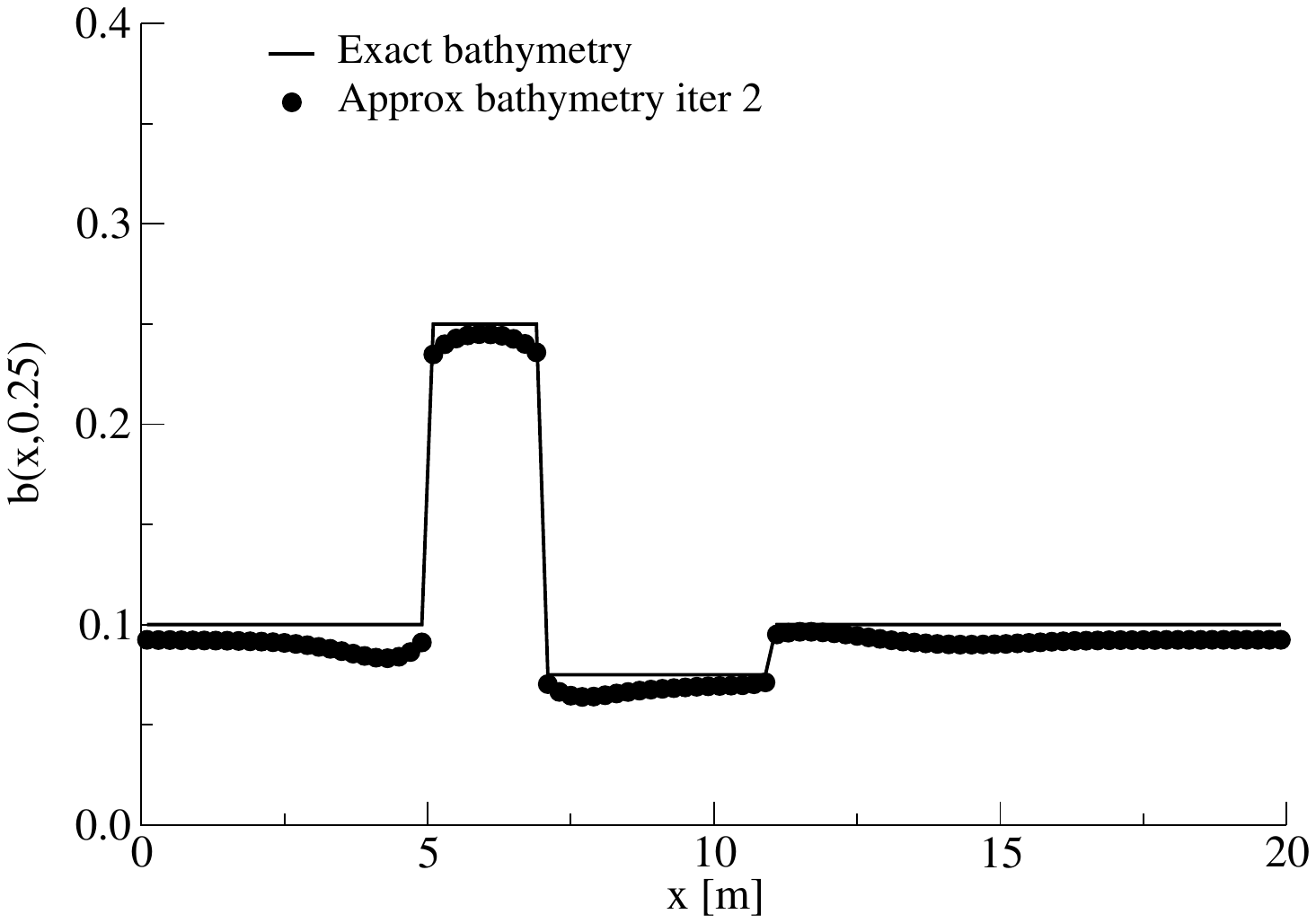} \quad
	\includegraphics[width=0.3\textwidth]{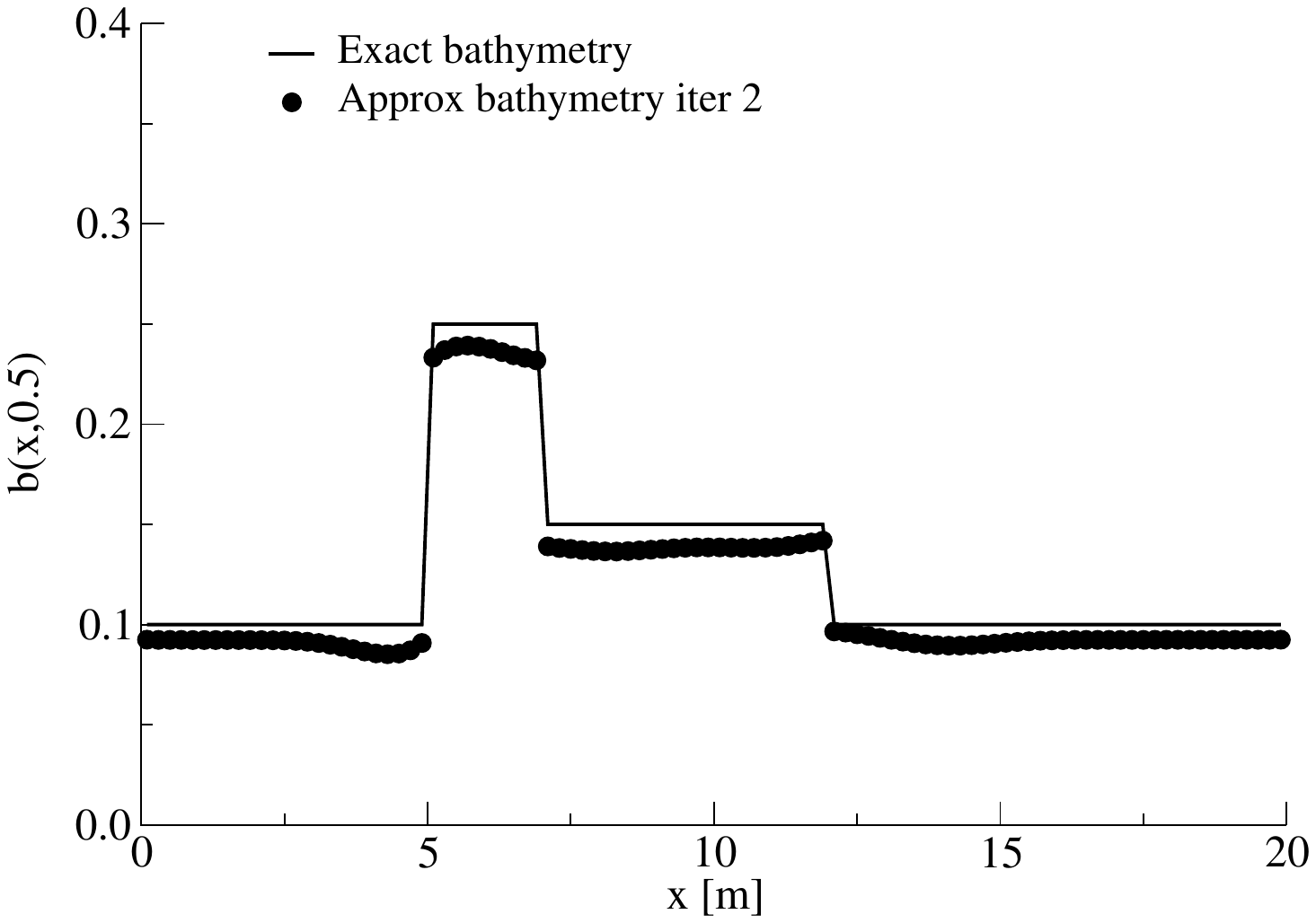} \quad
	\includegraphics[width=0.3\textwidth]{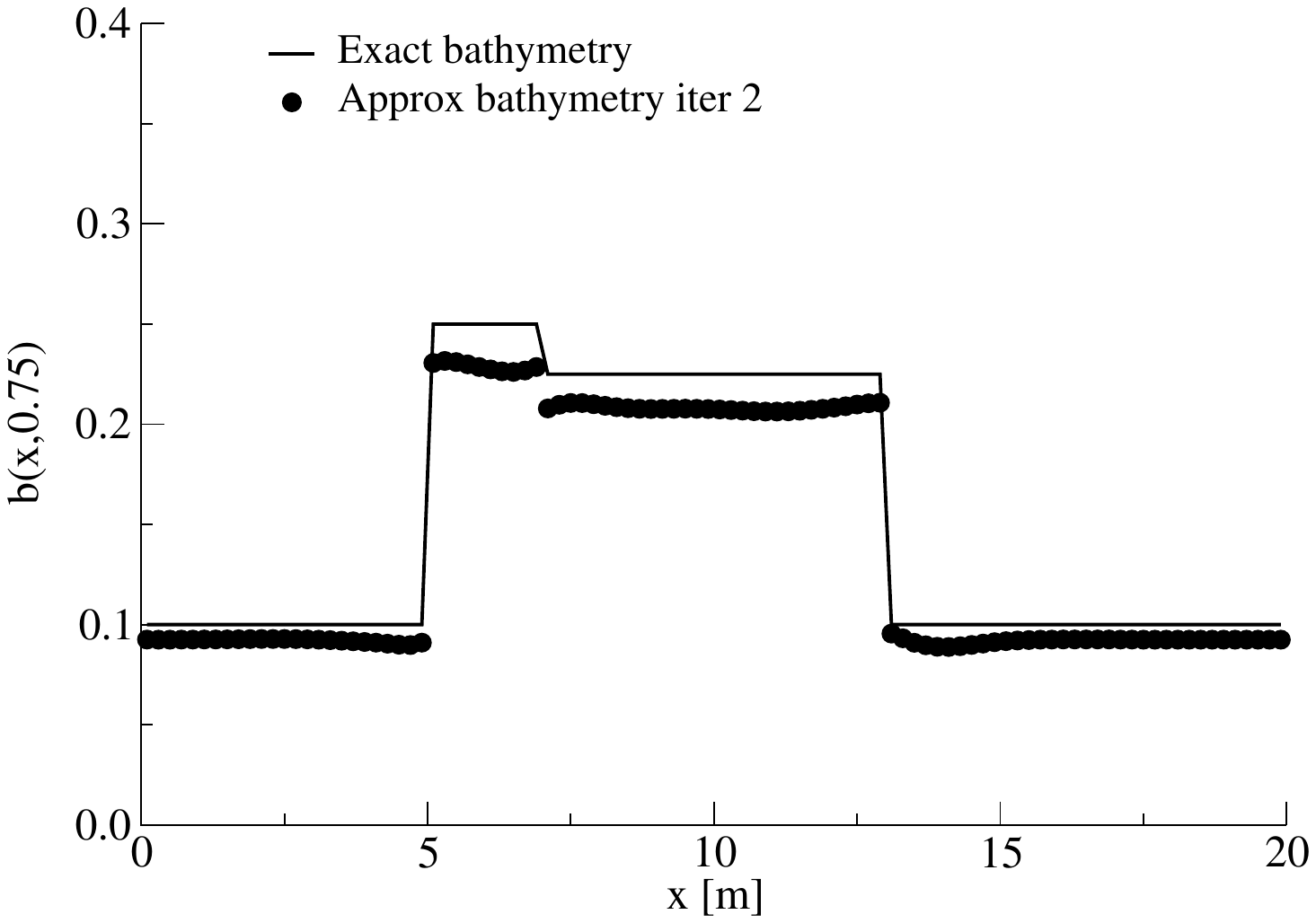} \\

	\includegraphics[width=0.3\textwidth]{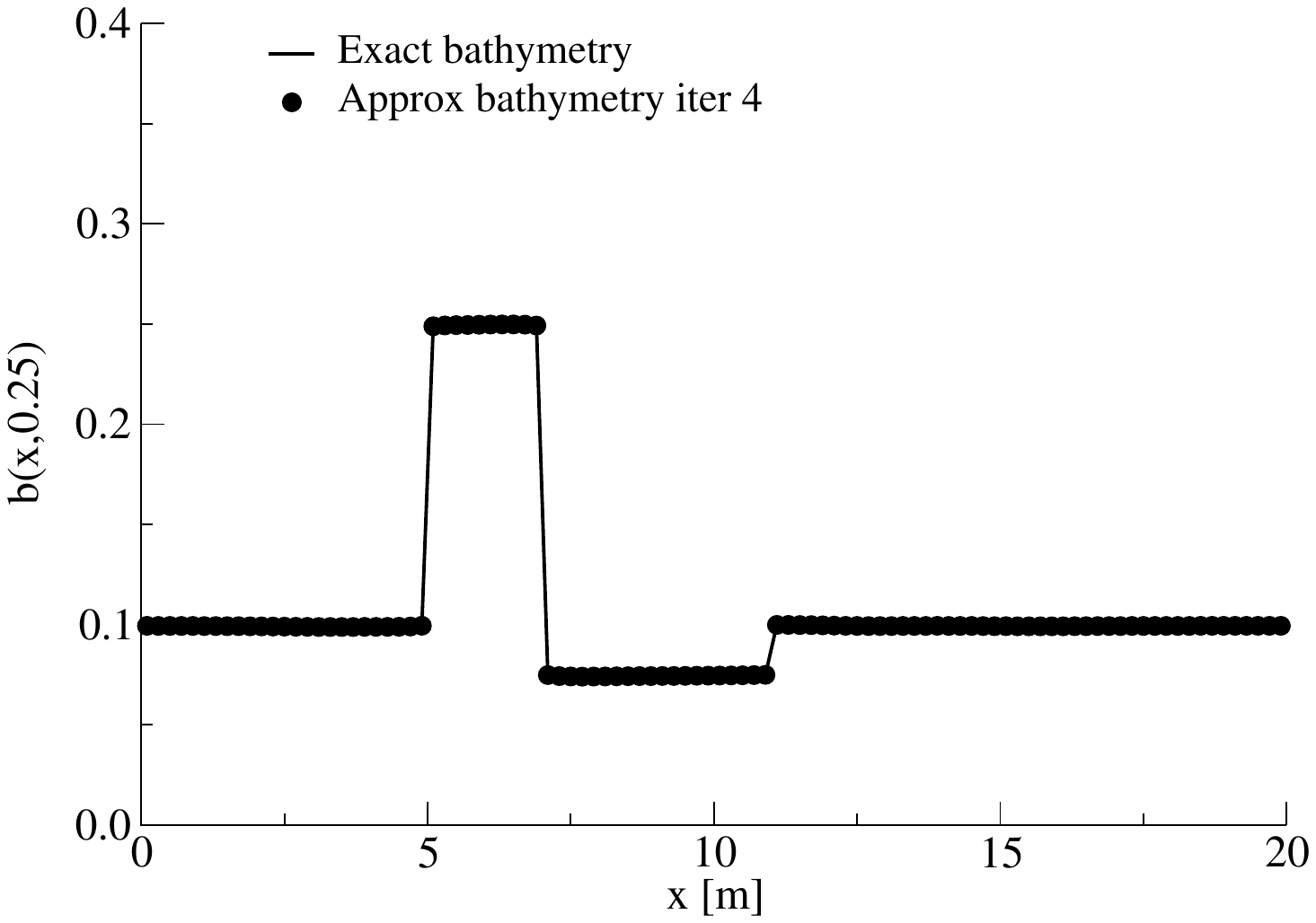} \quad
	\includegraphics[width=0.3\textwidth]{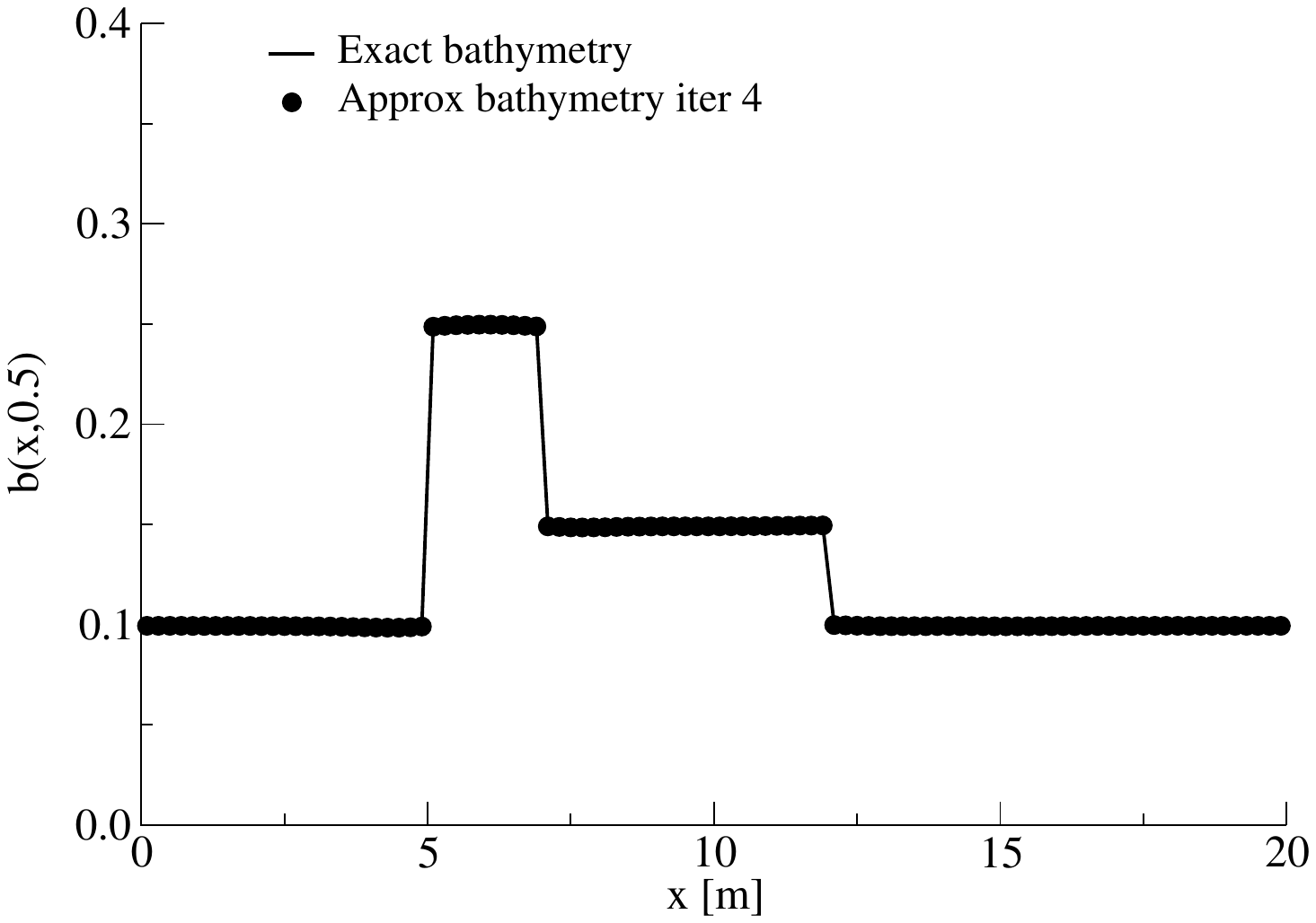} \quad
	\includegraphics[width=0.3\textwidth]{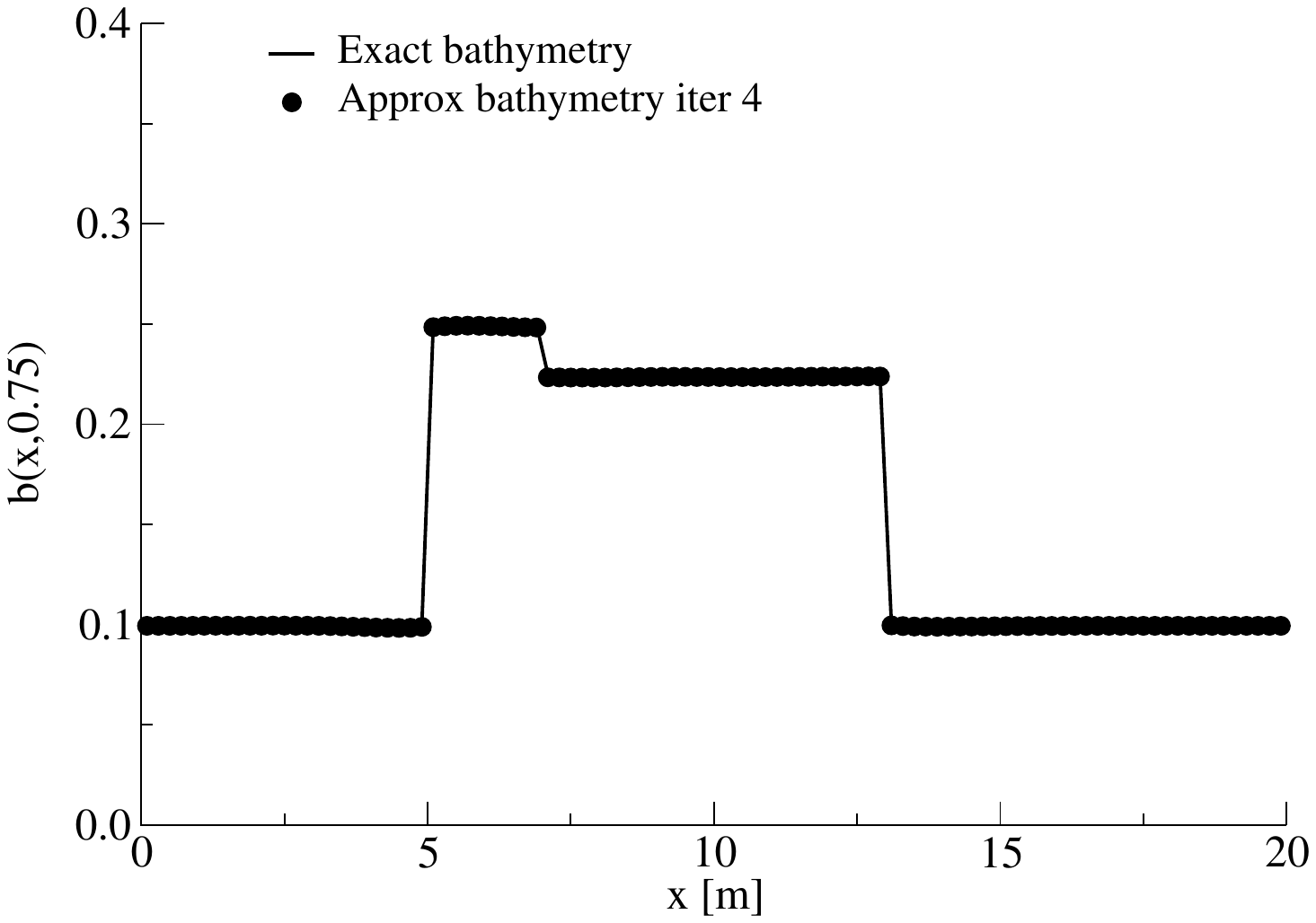} \\
	
	\includegraphics[width=0.3\textwidth]{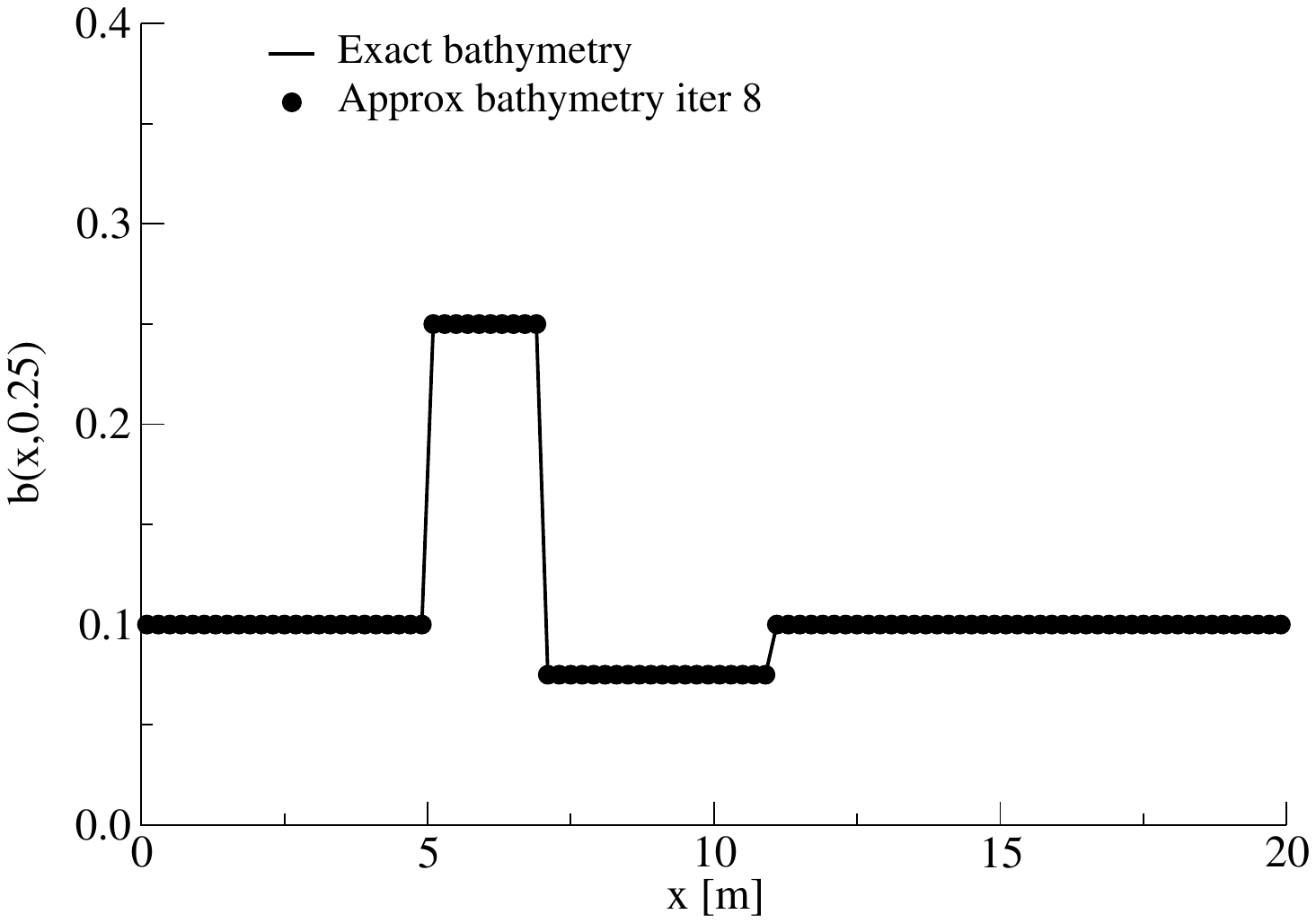} \quad
	\includegraphics[width=0.3\textwidth]{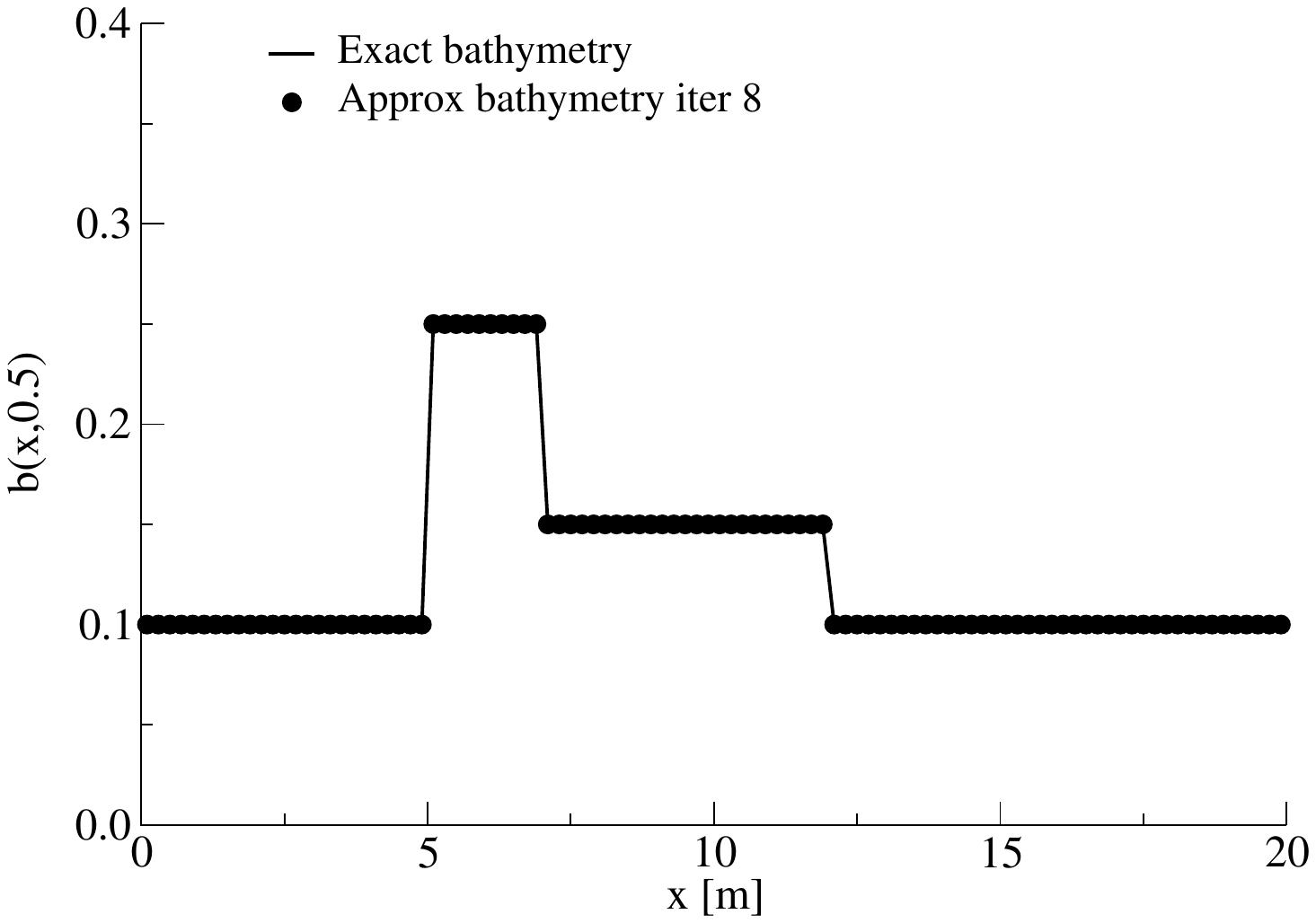} \quad
	\includegraphics[width=0.3\textwidth]{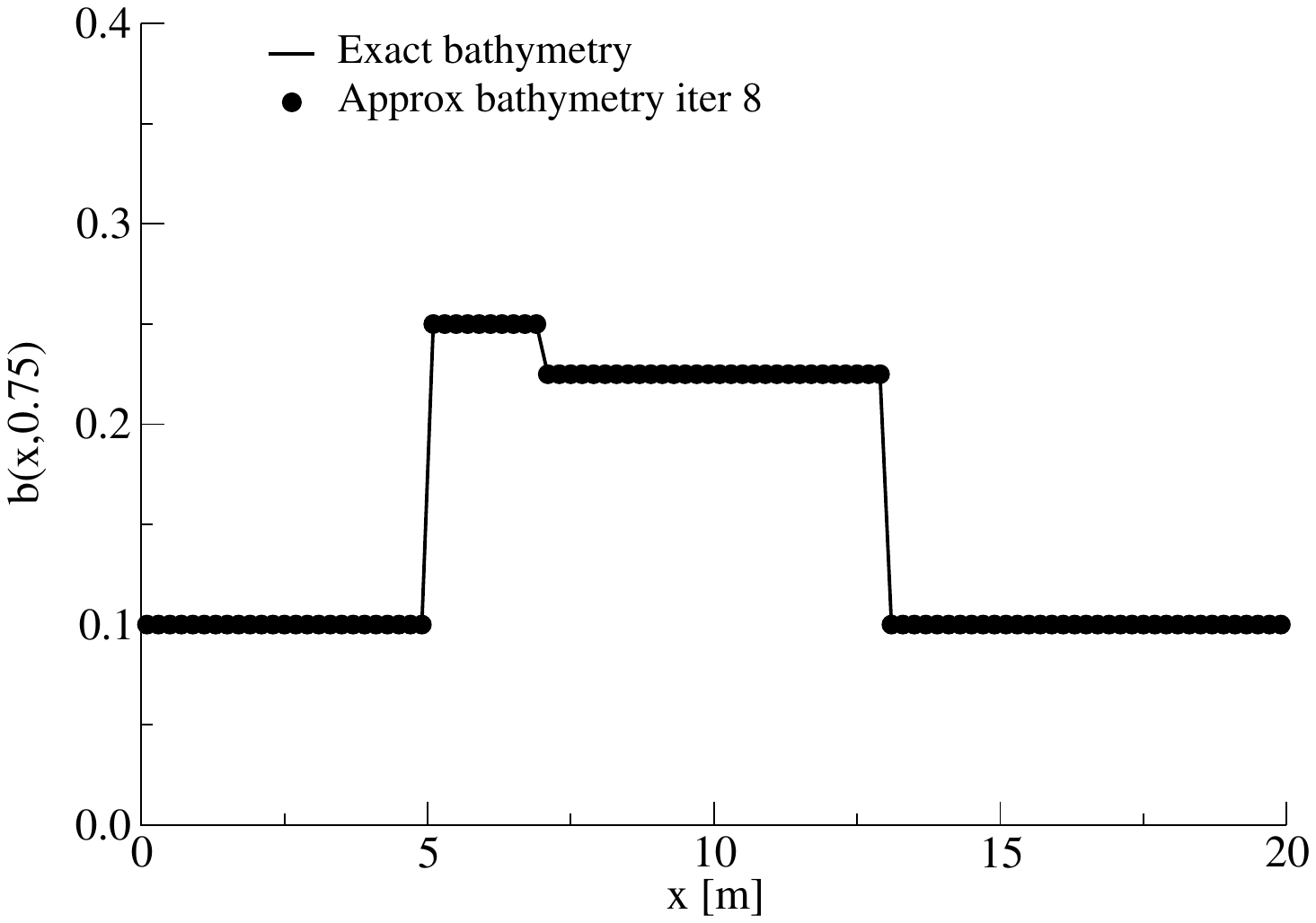} \\
	
	\caption{Discontinuous bottom profile (\ref{eq:b-discontinuous}): Result for the {\bf FORCE-$\alpha$+CSF}. Parameters $\Delta t = 0.01$,  $\alpha_F = 2$, $\varepsilon = 0.001$, $\lambda_b = 0.71$, $100$ cells.
		{\bf Feft:} $t=0.25$, {\bf centered} $t=0.5$, {\bf right:} $t=0.75$.}
	\label{fig:b-for-iter-and-times:test-1:NonCons-Force-alpha}
\end{figure}
\FloatBarrier

\subsection{Smooth bottom profile with large gradients}

This test aims at recovering the smooth bottom profile with a large gradient
\begin{eqnarray}
	\label{eq:b-large-gradient}
	\begin{array}{c}
		\bar{b}(t,x) =  0.15  t\left( 1 + \sin^4\left(  \frac{ \pi  ( x-10 - 4.5  t)}{5}  \right) \right).
	\end{array}
\end{eqnarray}

Systems (\ref{eq:direct-system}) and (\ref{eq:dual-system}) are solved with periodic boundary conditions. The profile consists of a triggered of pulses that moves periodically to the right, elevates with respect to an equilibrium $b=0$ and increments its amplitude as the time advances.

These type of tests are a challenge for the balance law (\ref{eq:direct-system}). They introduce stiffness on the source term and may induce large gradient on the other variables.    

Figure \ref{fig:b-for-iter-and-times:test-2-ConsRusanov} shows the results for the {\bf Rusanov+FD} scheme. Although the overall procedure converges, a phase mismatch of the amplitudes and asymmetries are observed in the first few iterations. Asymmetry is still present in the $8th$ iteration at $t = 0.5$. Figure \ref{fig:b-for-iter-and-times:test-2:ForceAlphaCons} shows the results for the {\bf FORCE-$\alpha$+FD}. Asymmetries are observed in the initial iterations that disappear by the eighth iteration. Figure \ref{fig:b-for-iter-and-times:test-2:NonCons-Force-alpha} shows the results for the {\bf FORCE-$\alpha$+CSF} scheme. In the first iterations, errors appear on the boundaries. In addition, asymmetries appear in the first iterations, which are not perceived after the fourth iteration. Figure \ref{fig:comp-error-Test-2} shows the $L_\infty$ norm of $\nabla J$ against the number of iterations, to facilitate the visualization the plot is depicted in logarithmic scale. So, this measures the error between $b^k $ and $b^{k+1}$, which is the empirical convergence of the global algorithm.  We observe that for smooth bottom profiles with large gradients, all methods have oscillations during the iterative process and asymmetries in the first iterations.  By comparing with reference schemes, {\bf FORCE-$\alpha$+CSF} depicts the best convergence. This test evidences that the methodology also applies to other types of boundary conditions than transmissive.

\begin{figure}[h]
	\begin{center}
		\includegraphics[scale=0.45]{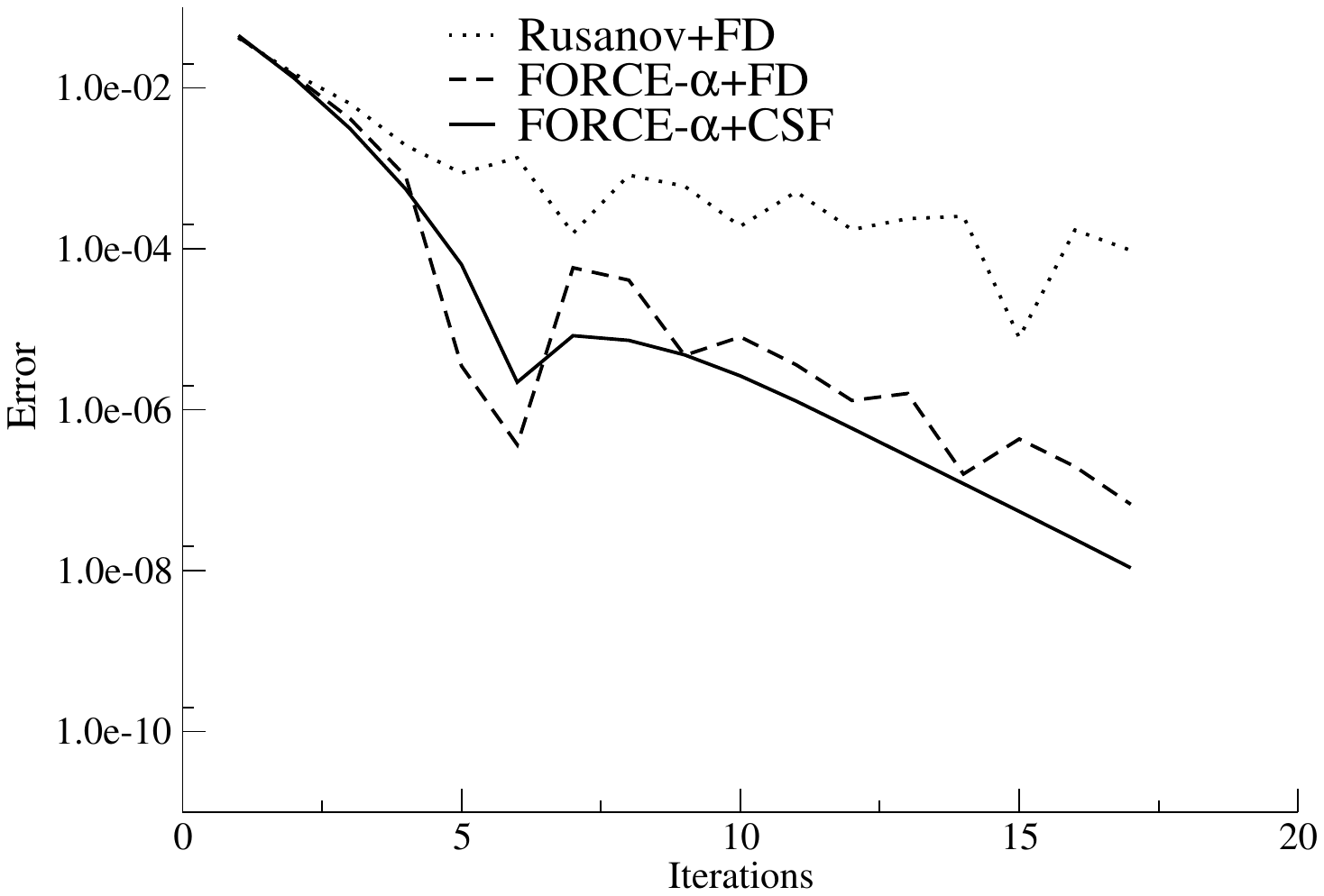}
	\end{center}
	\caption{Smooth bottom profile with large gradient (\ref{eq:b-large-gradient}): The $L_\infty$ norm of $\nabla J $ for $\varepsilon = 0.001$, $\lambda_b = 0.71$, $100$ cells at $t = 1$, $\alpha_F = 2$. 
		(Dot line) {\bf Rusanov+FD} scheme.
		(Dash line) {\bf FORCE-$\alpha$+FD} scheme. 
		(Full line) {\bf FORCE-$\alpha$+CSF} scheme.
	}\label{fig:comp-error-Test-2}
\end{figure}

\begin{figure}   
	\centering
	\includegraphics[width=0.3\textwidth]{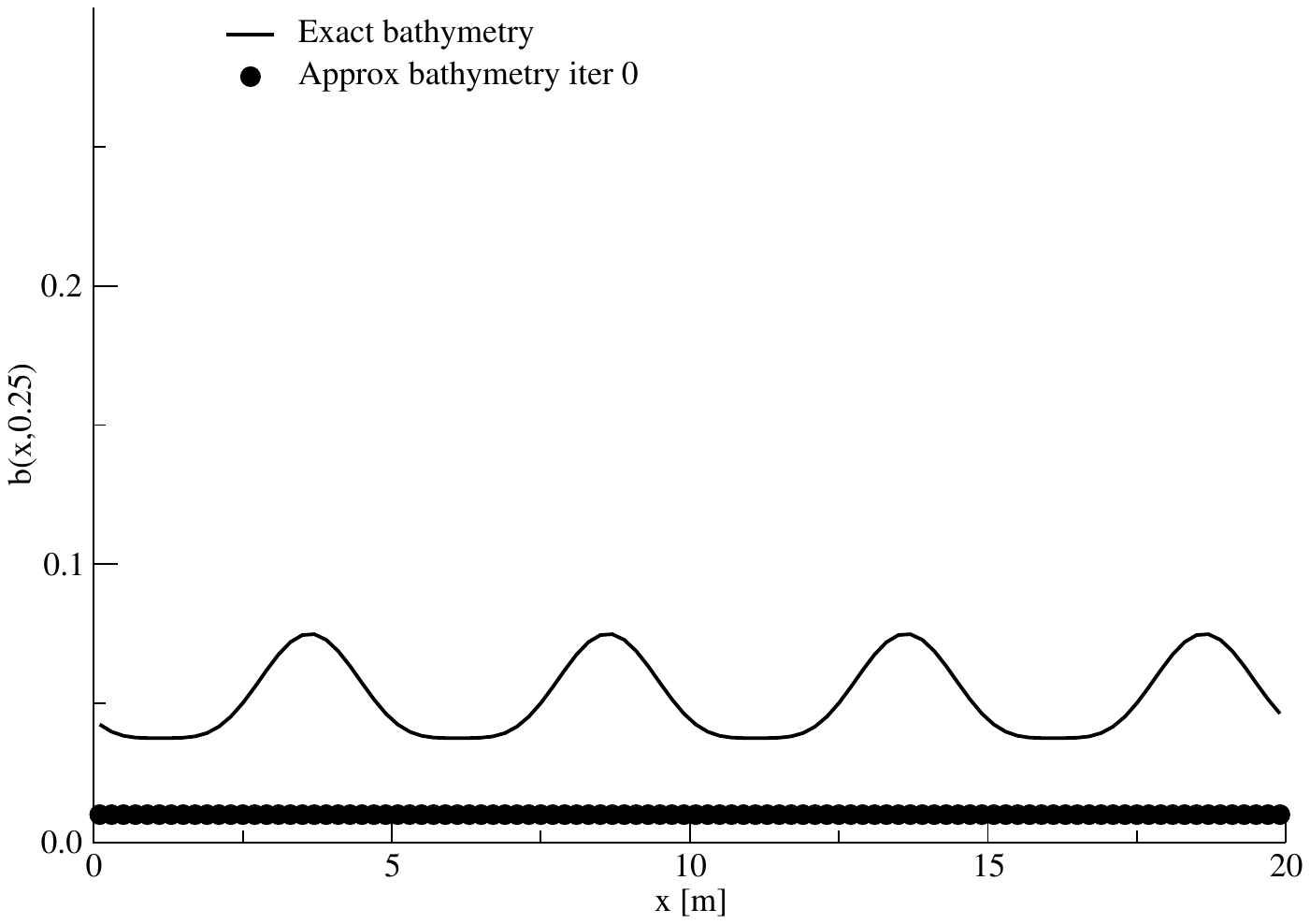} \quad
	\includegraphics[width=0.3\textwidth]{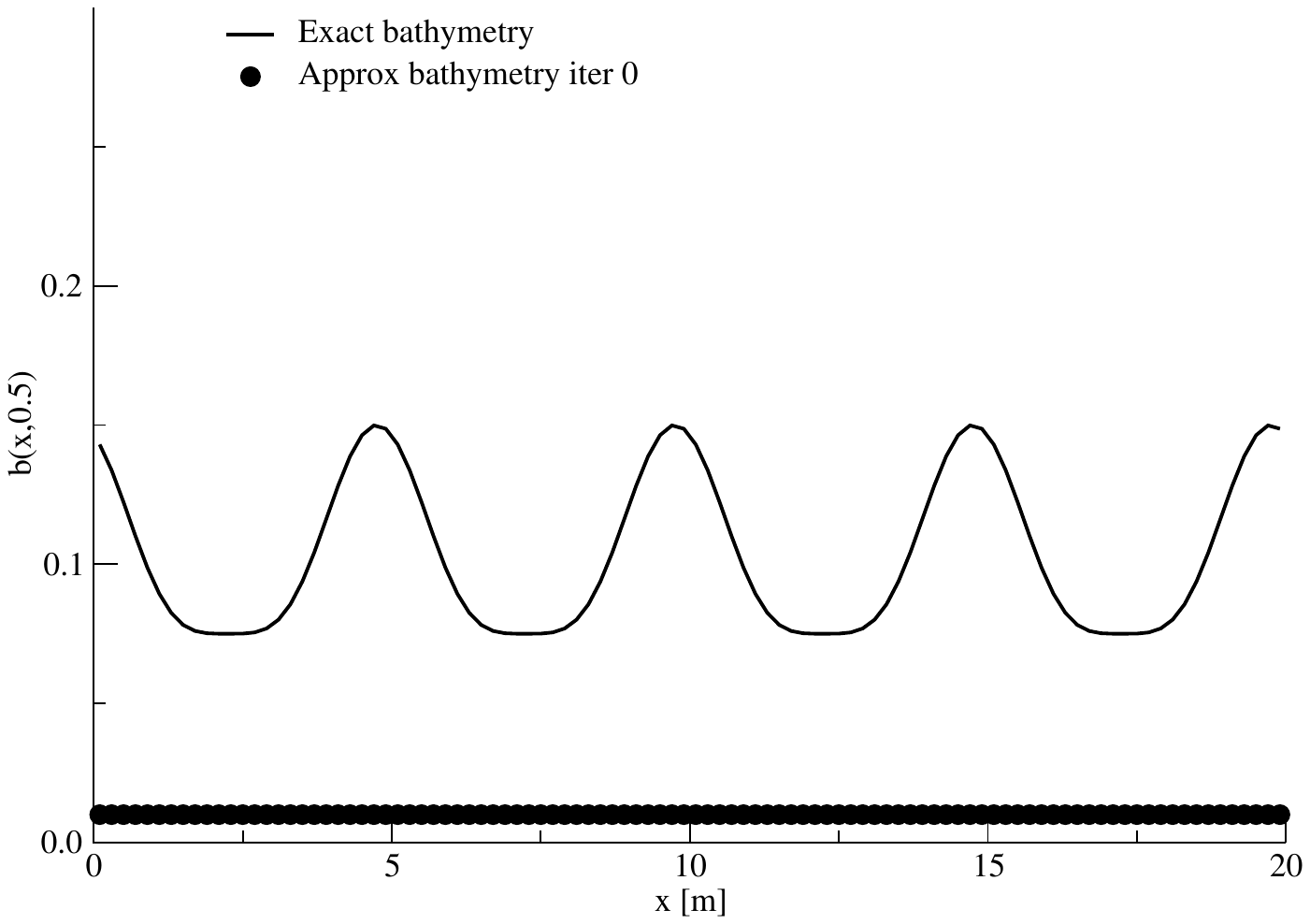} \quad
	\includegraphics[width=0.3\textwidth]{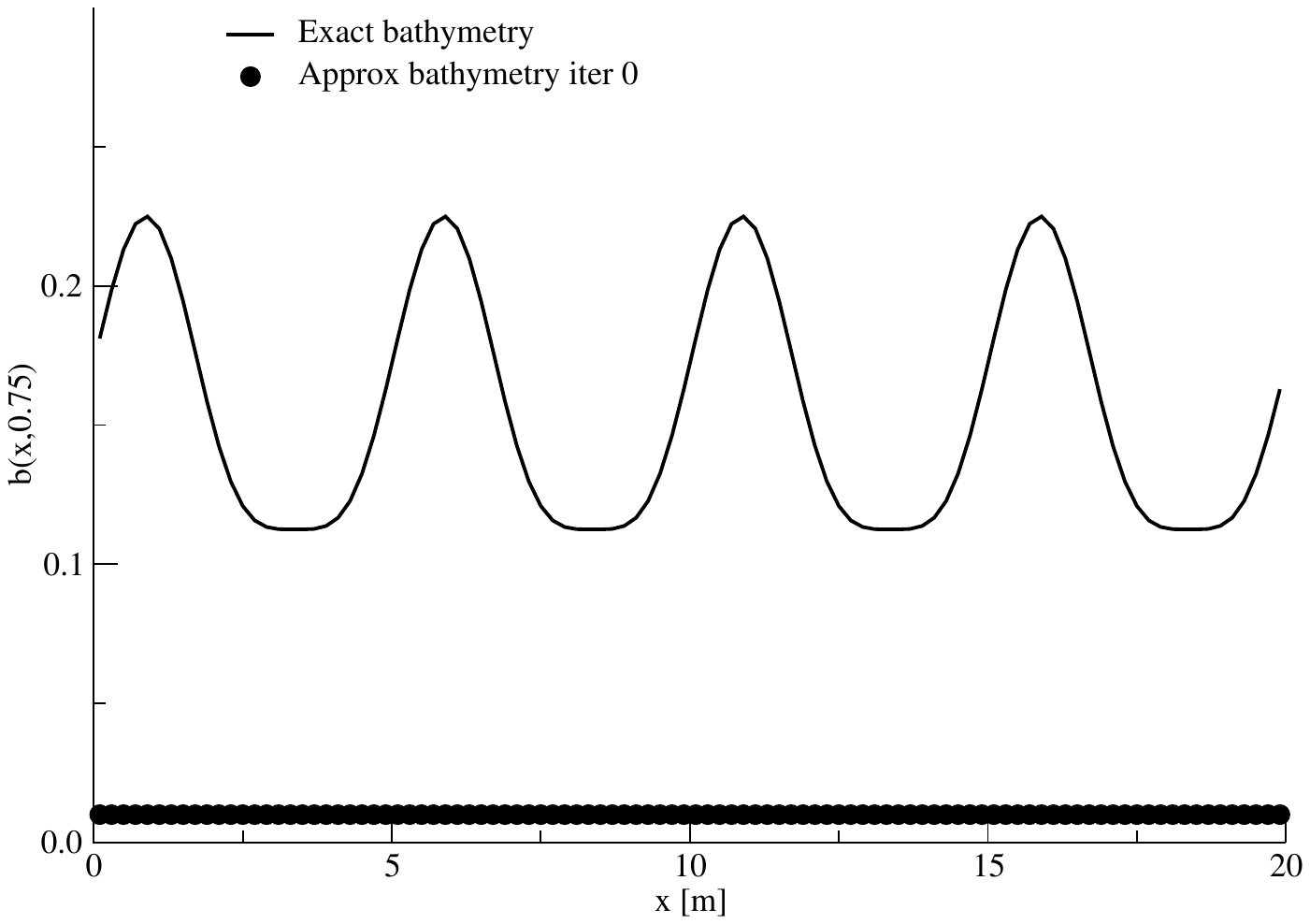} 
	\\
	
	\includegraphics[width=0.3\textwidth]{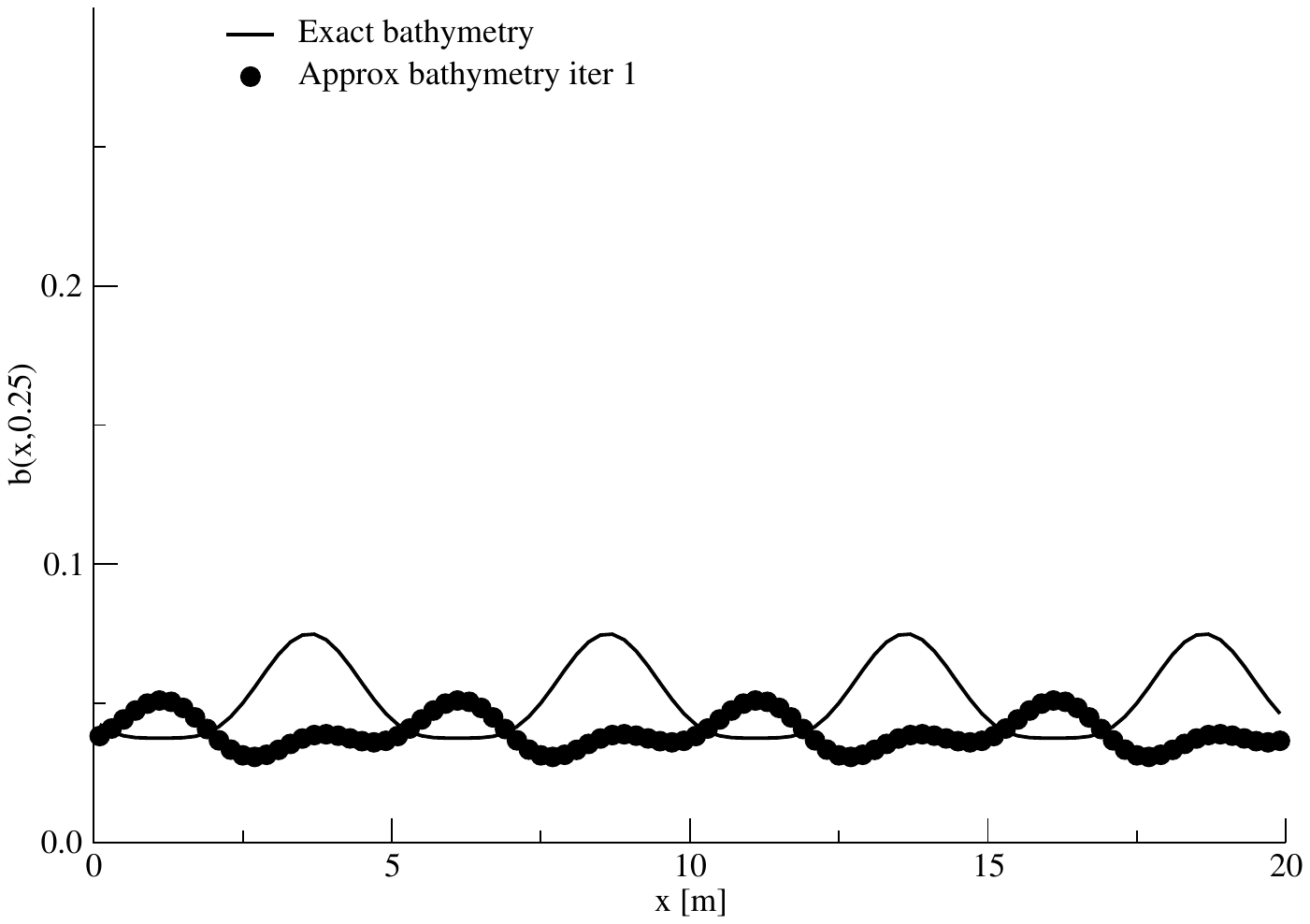} \quad
	\includegraphics[width=0.3\textwidth]{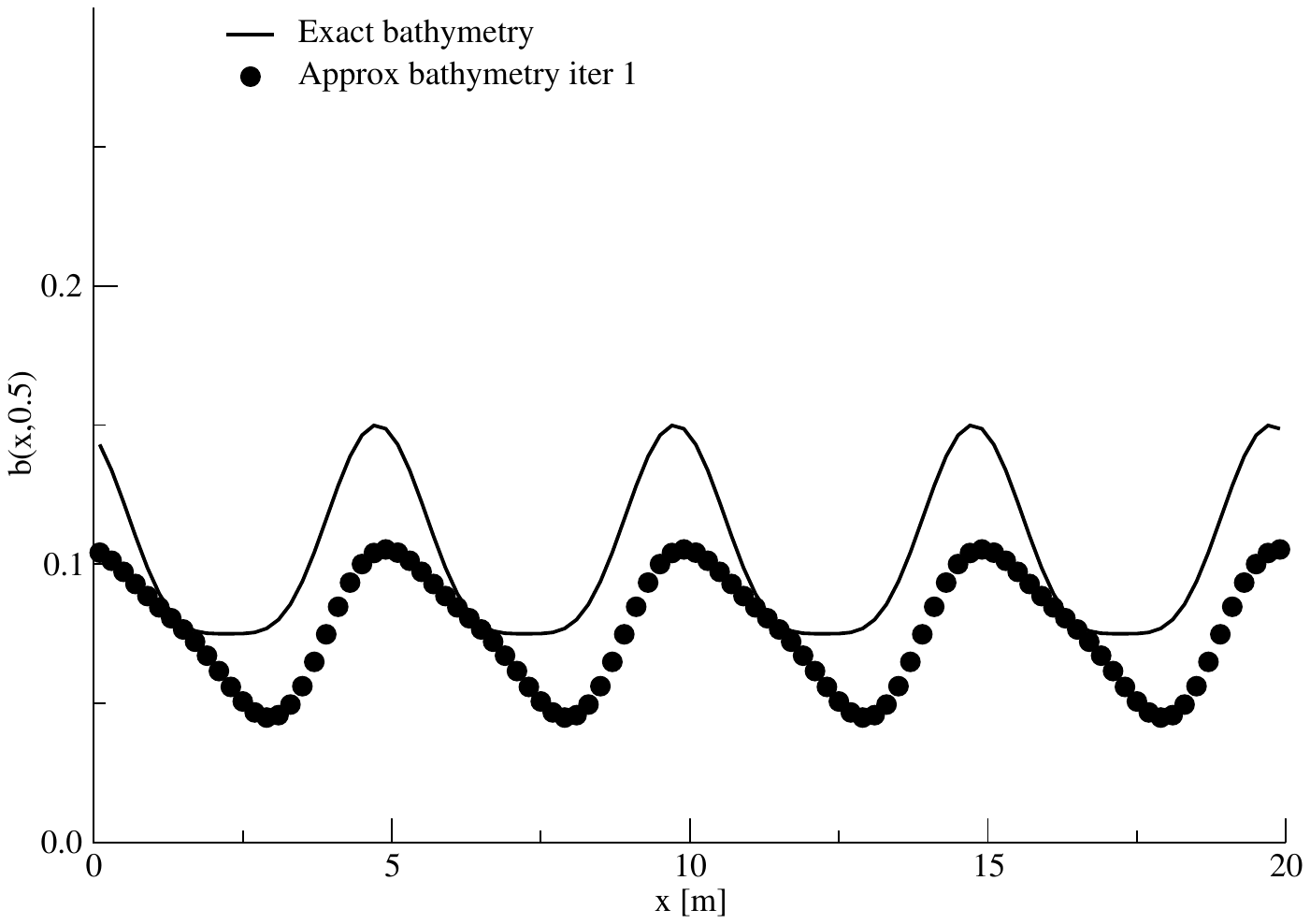} \quad
	\includegraphics[width=0.3\textwidth]{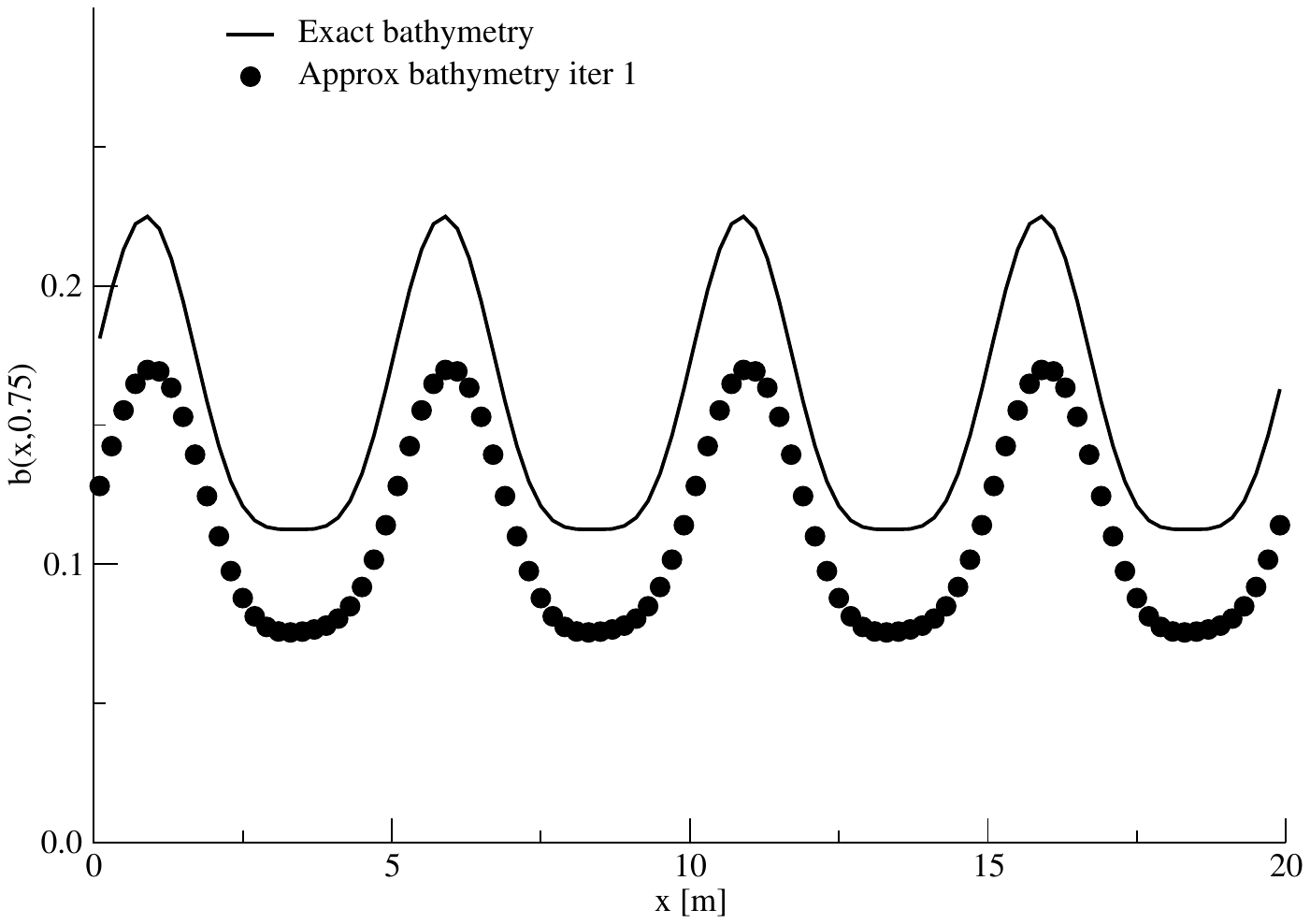} \\
	
	\includegraphics[width=0.3\textwidth]{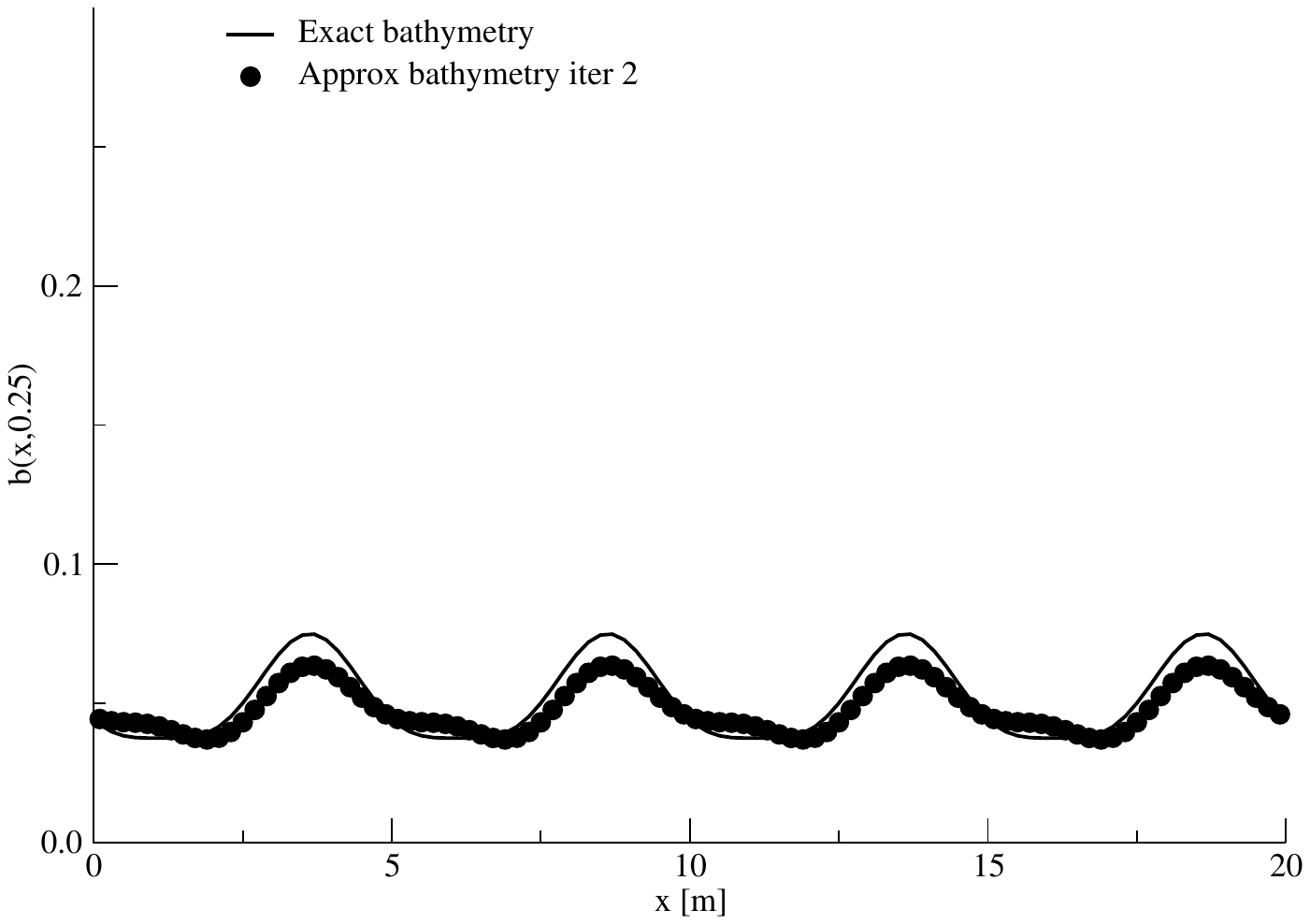} \quad
	\includegraphics[width=0.3\textwidth]{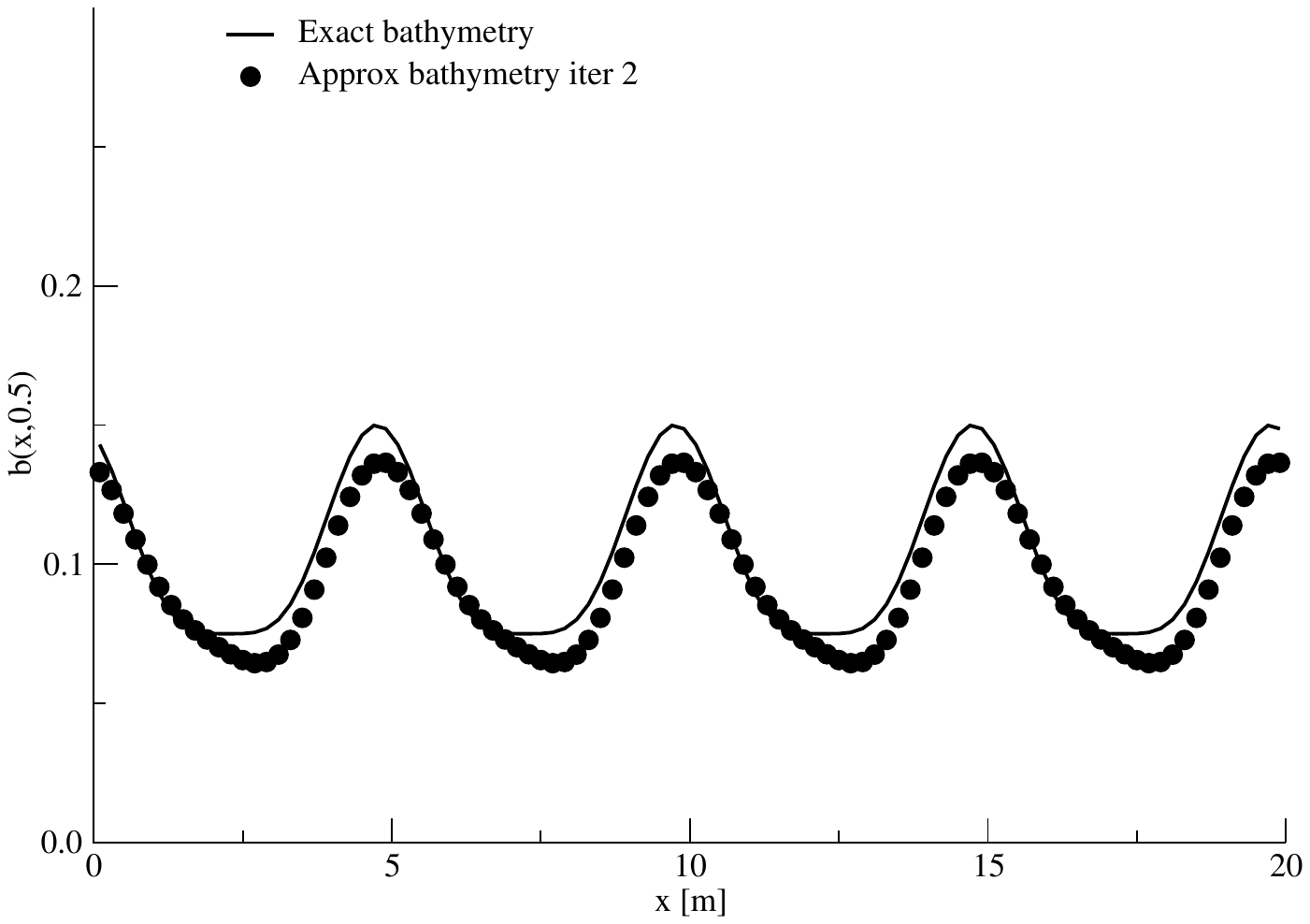} \quad
	\includegraphics[width=0.3\textwidth]{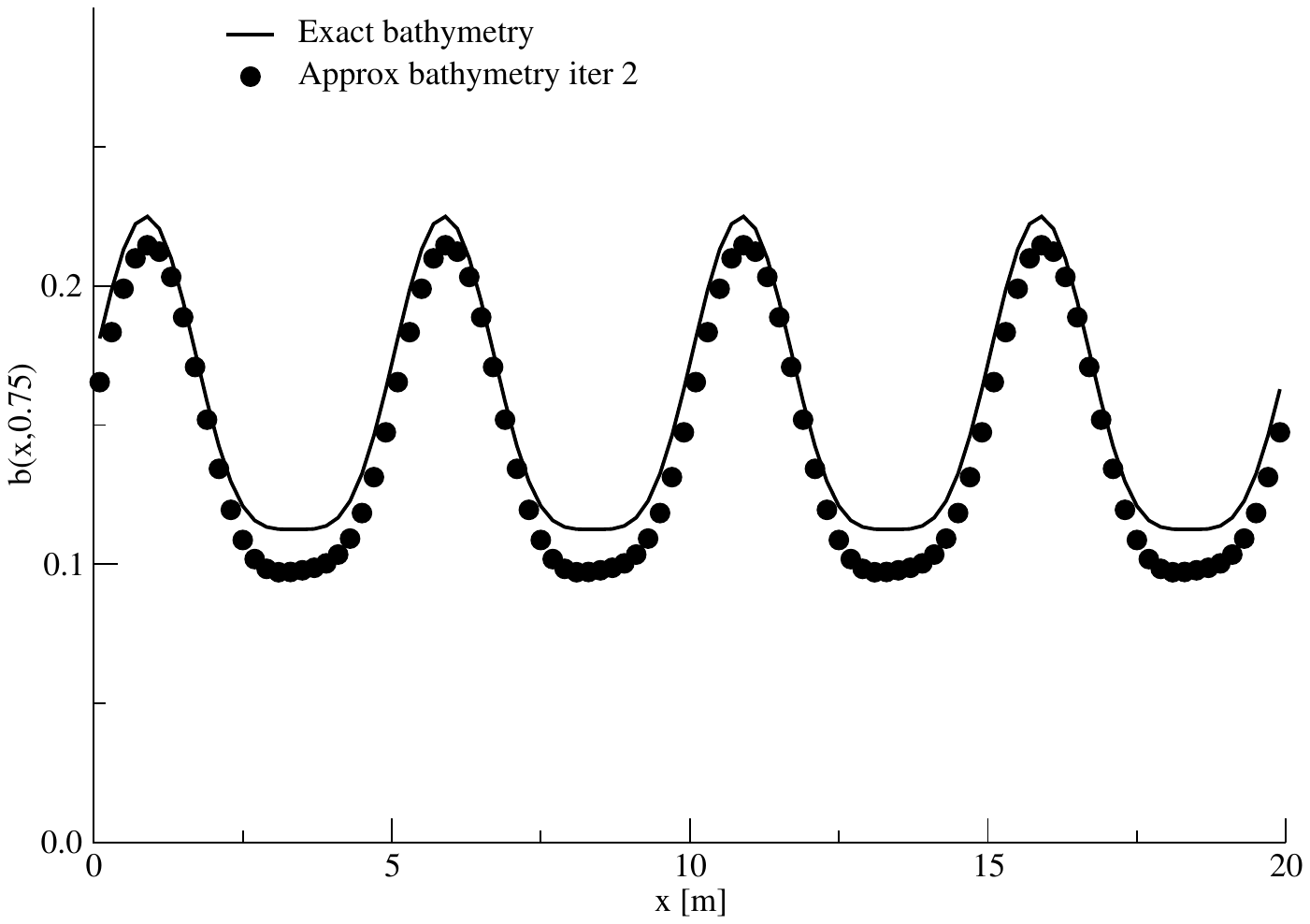} \\

	\includegraphics[width=0.3\textwidth]{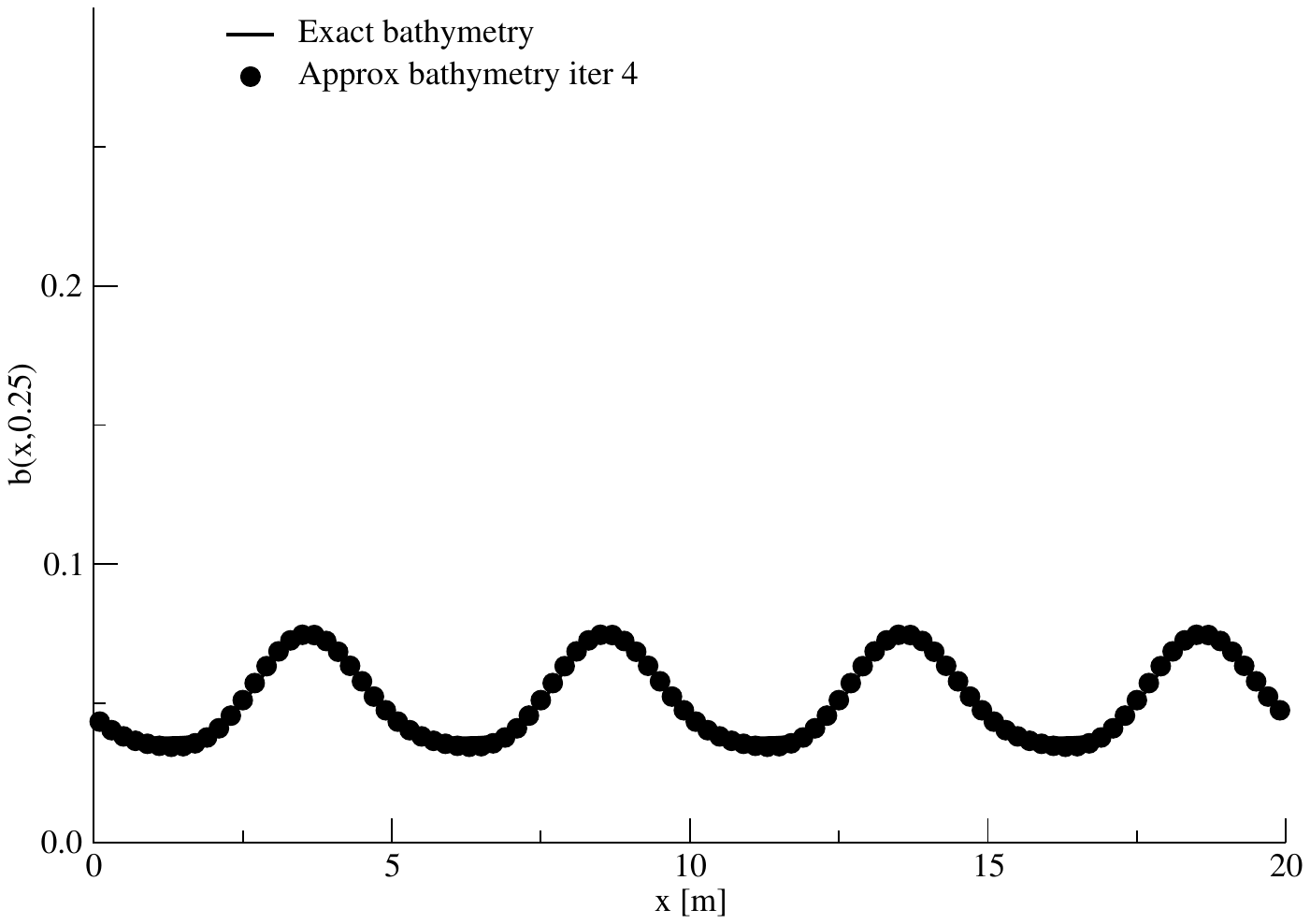} \quad
	\includegraphics[width=0.3\textwidth]{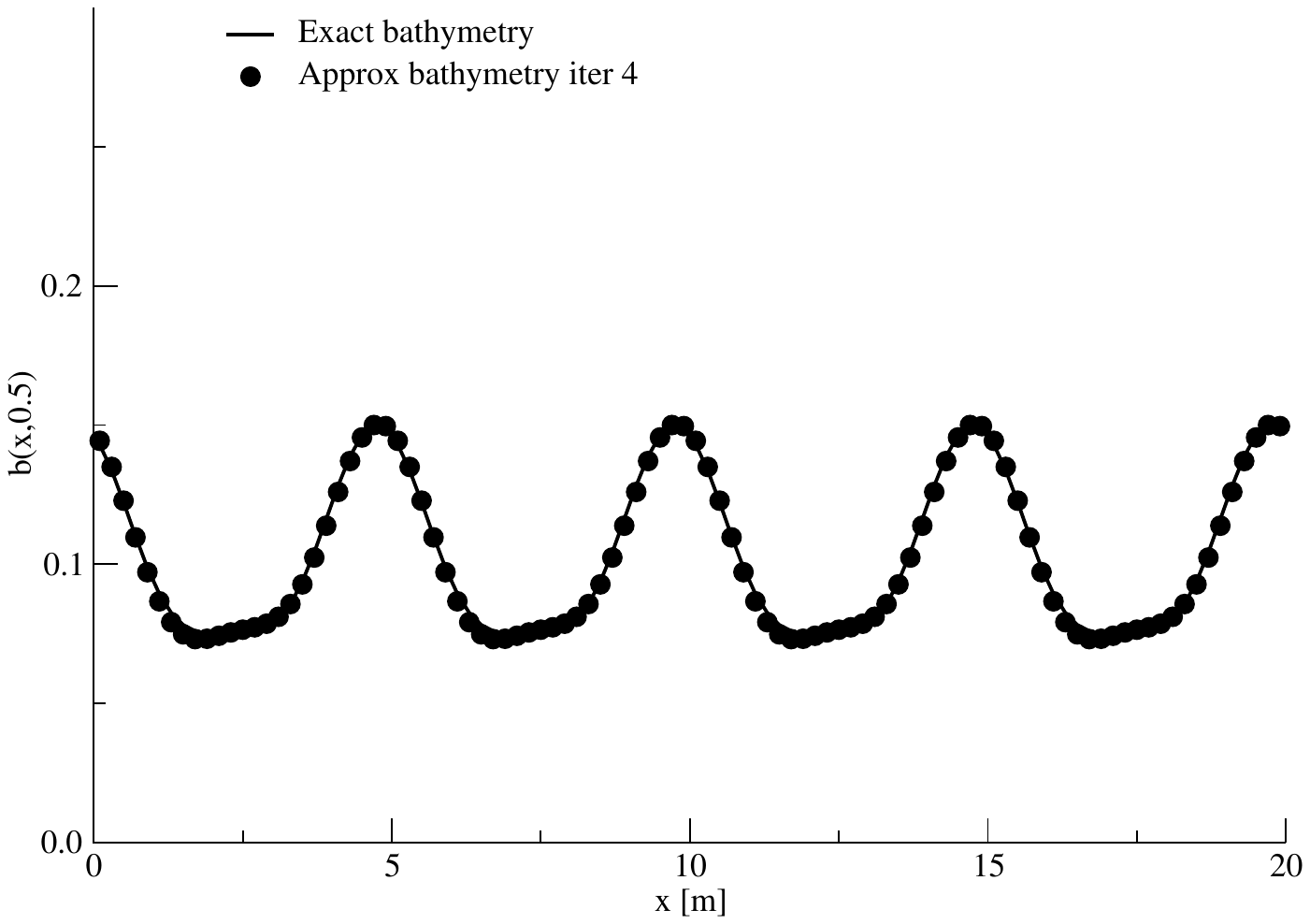} \quad
	\includegraphics[width=0.3\textwidth]{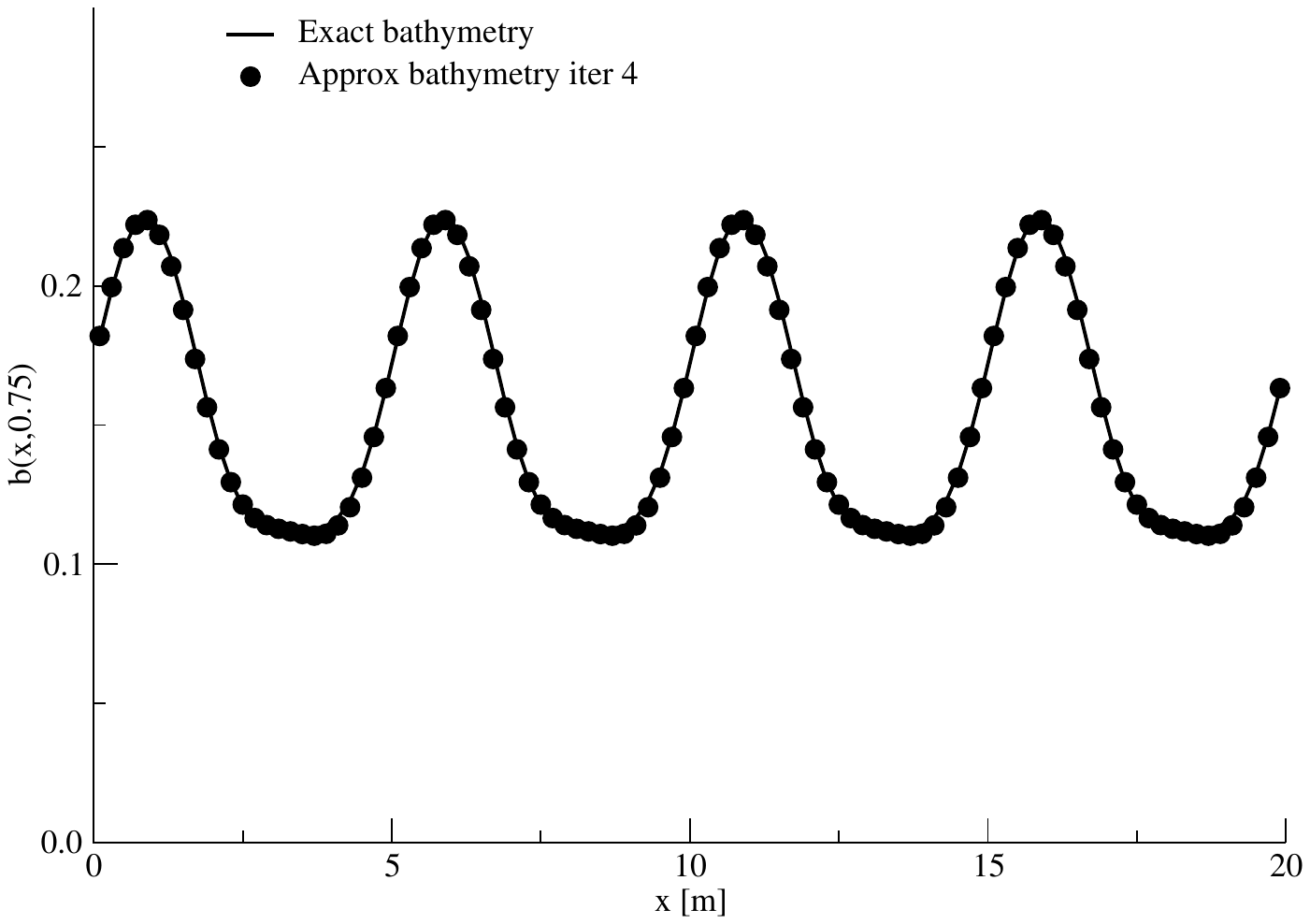} \\
	
	\includegraphics[width=0.3\textwidth]{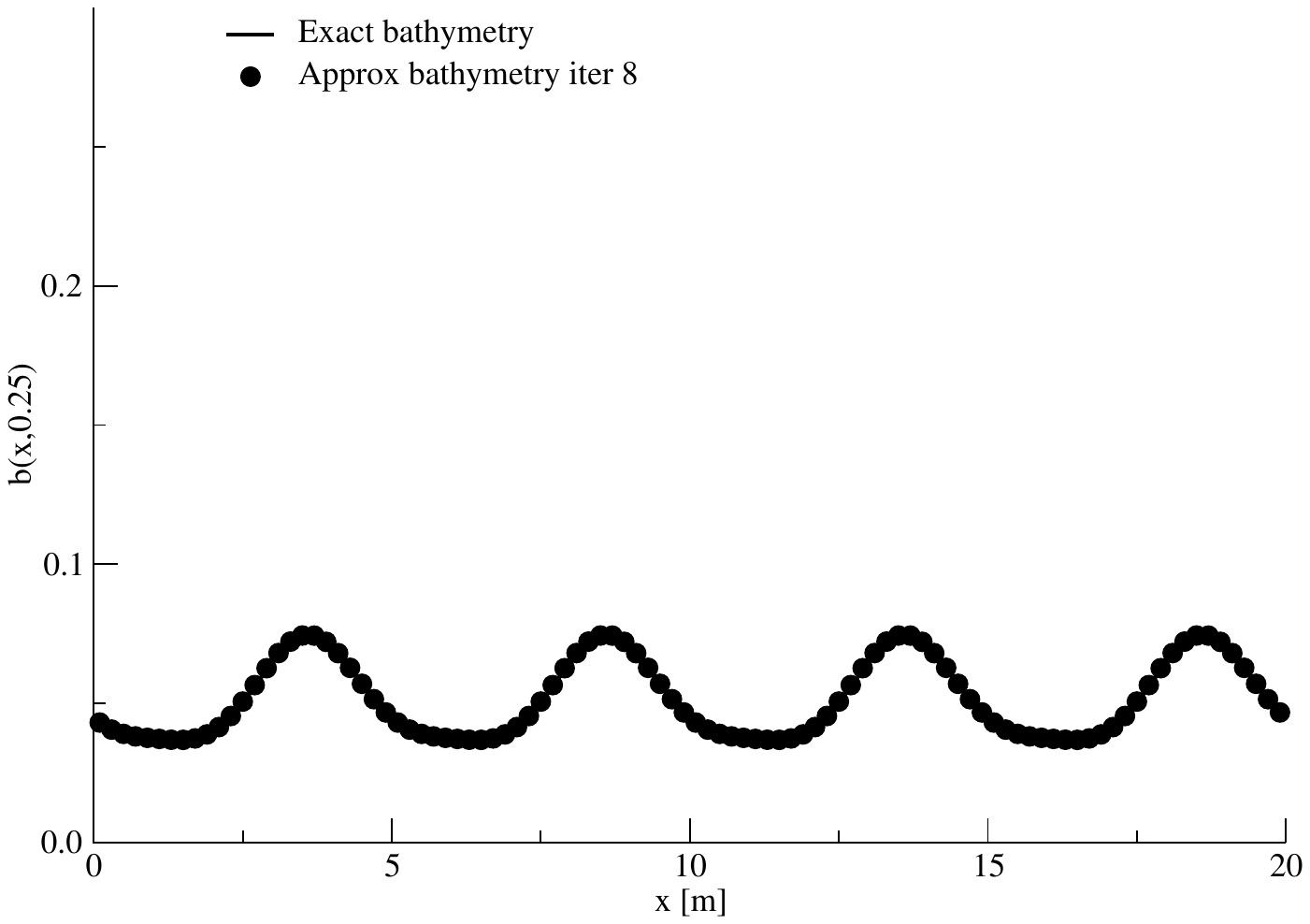} \quad
	\includegraphics[width=0.3\textwidth]{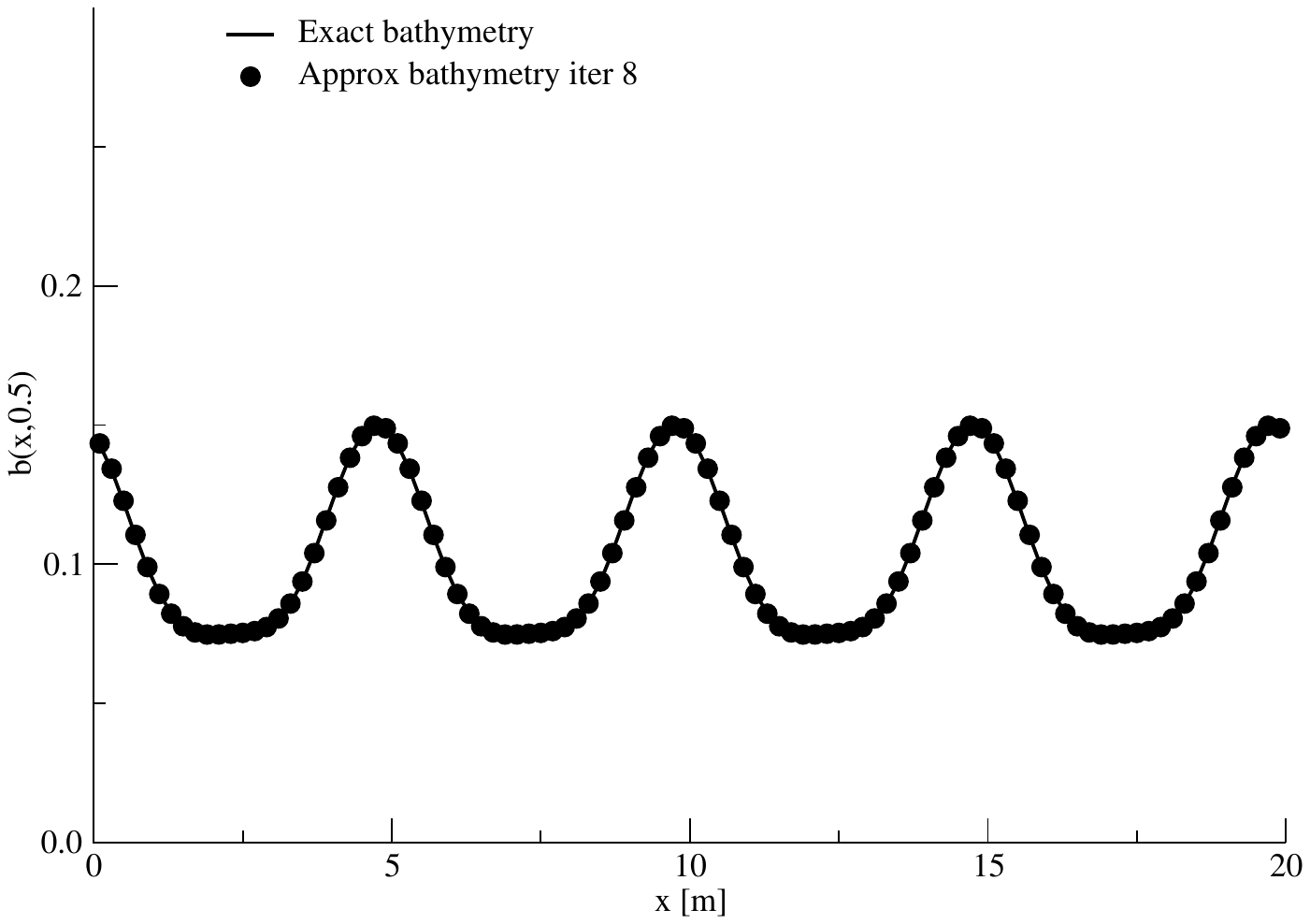} \quad
	\includegraphics[width=0.3\textwidth]{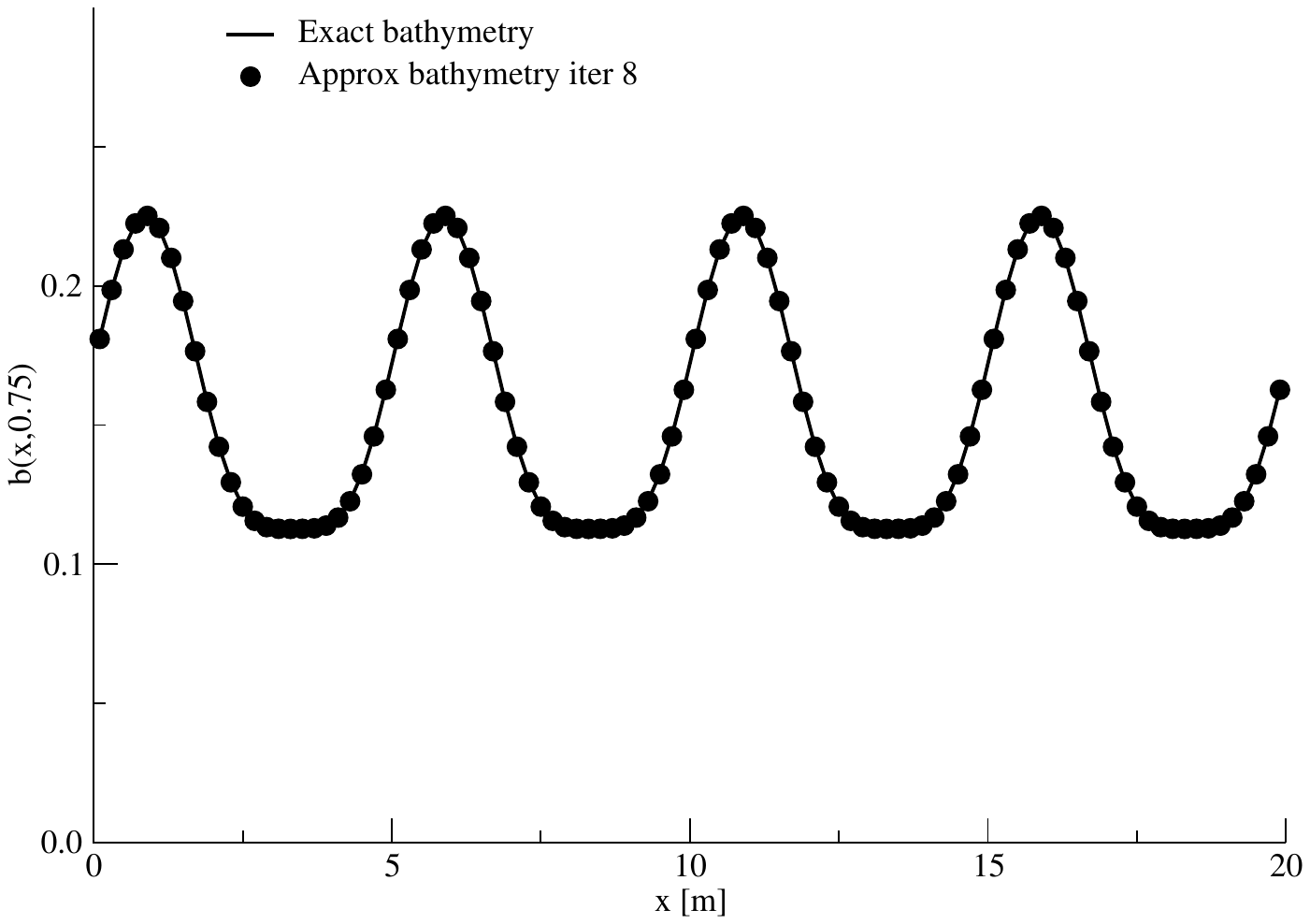} \\
	
	\caption{Smooth bottom profile with large gradient (\ref{eq:b-large-gradient}): Result for the reconstruction procedure resulting from {\bf Rusanov+FD}. Parameters $\Delta t = 0.01$,  $\alpha_F = 2$, $\varepsilon = 0.001$, $\lambda_b = 0.71$, $100$ cells.
		{\bf Feft:} $t=0.25$, {\bf centered} $t=0.5$, {\bf right:} $t=0.75$.}
	\label{fig:b-for-iter-and-times:test-2-ConsRusanov}
\end{figure}
\FloatBarrier

\begin{figure}   
	\centering
	\includegraphics[width=0.3\textwidth]{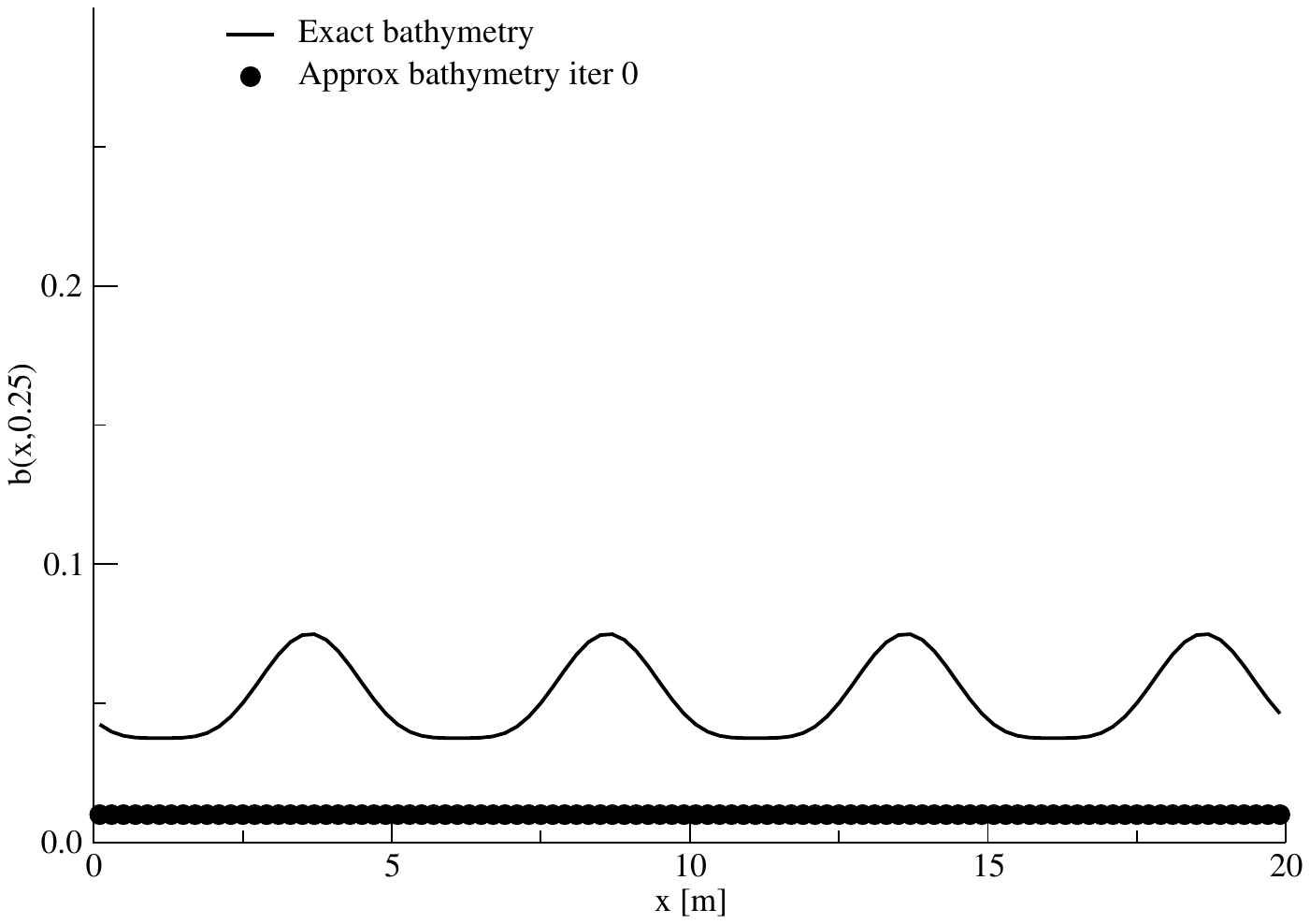} \quad
	\includegraphics[width=0.3\textwidth]{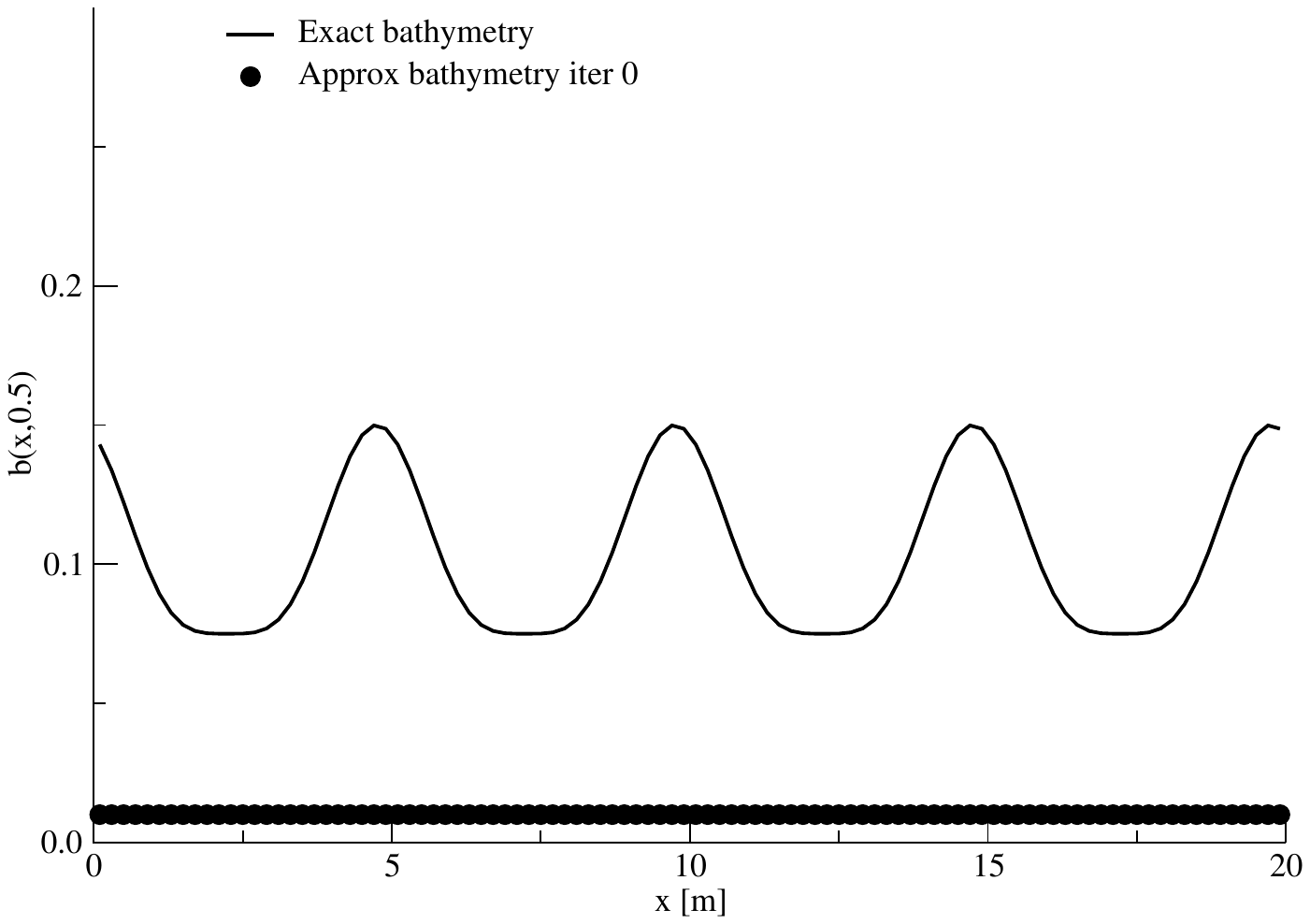} \quad
	\includegraphics[width=0.3\textwidth]{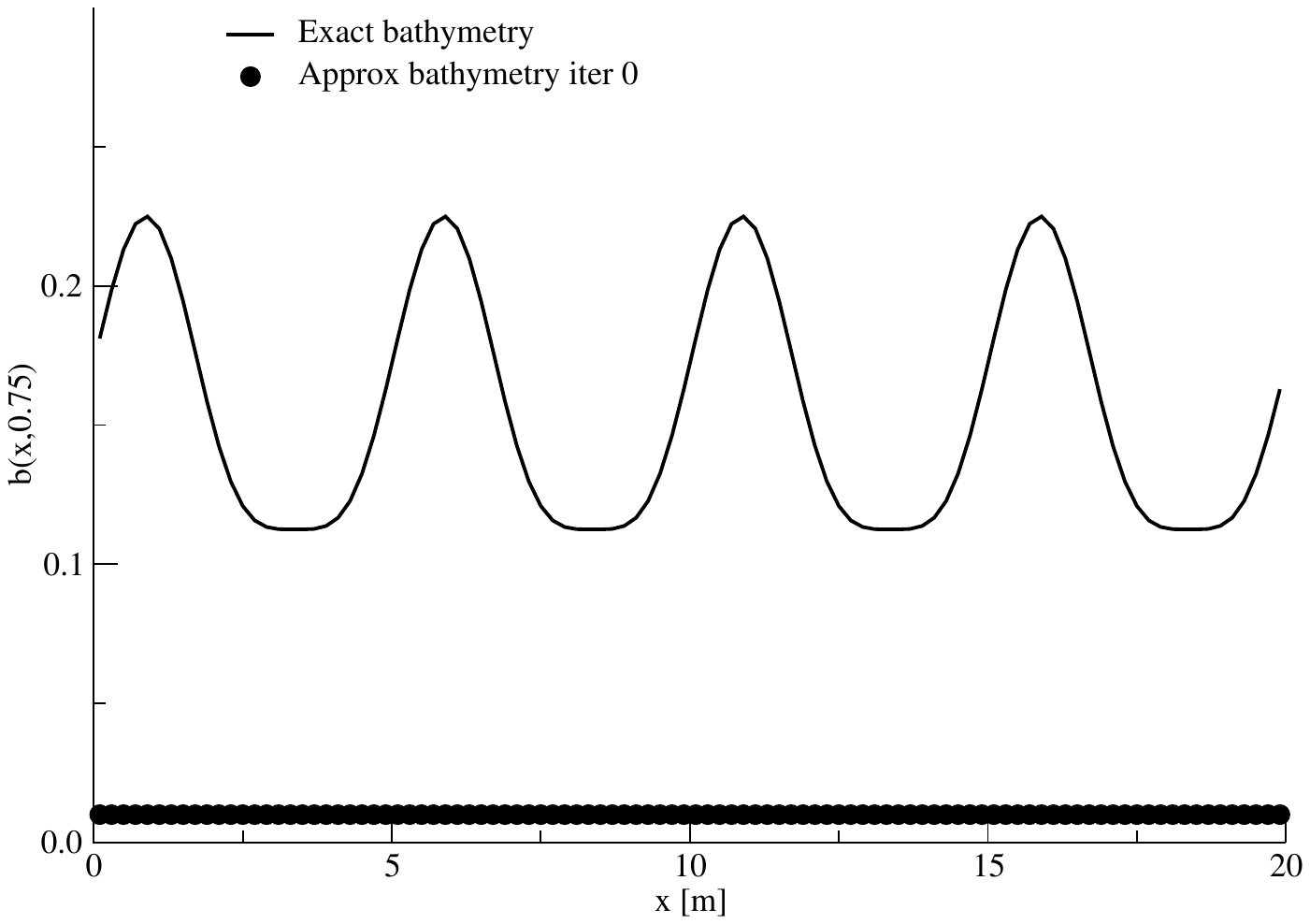} 
	\\
	
	\includegraphics[width=0.3\textwidth]{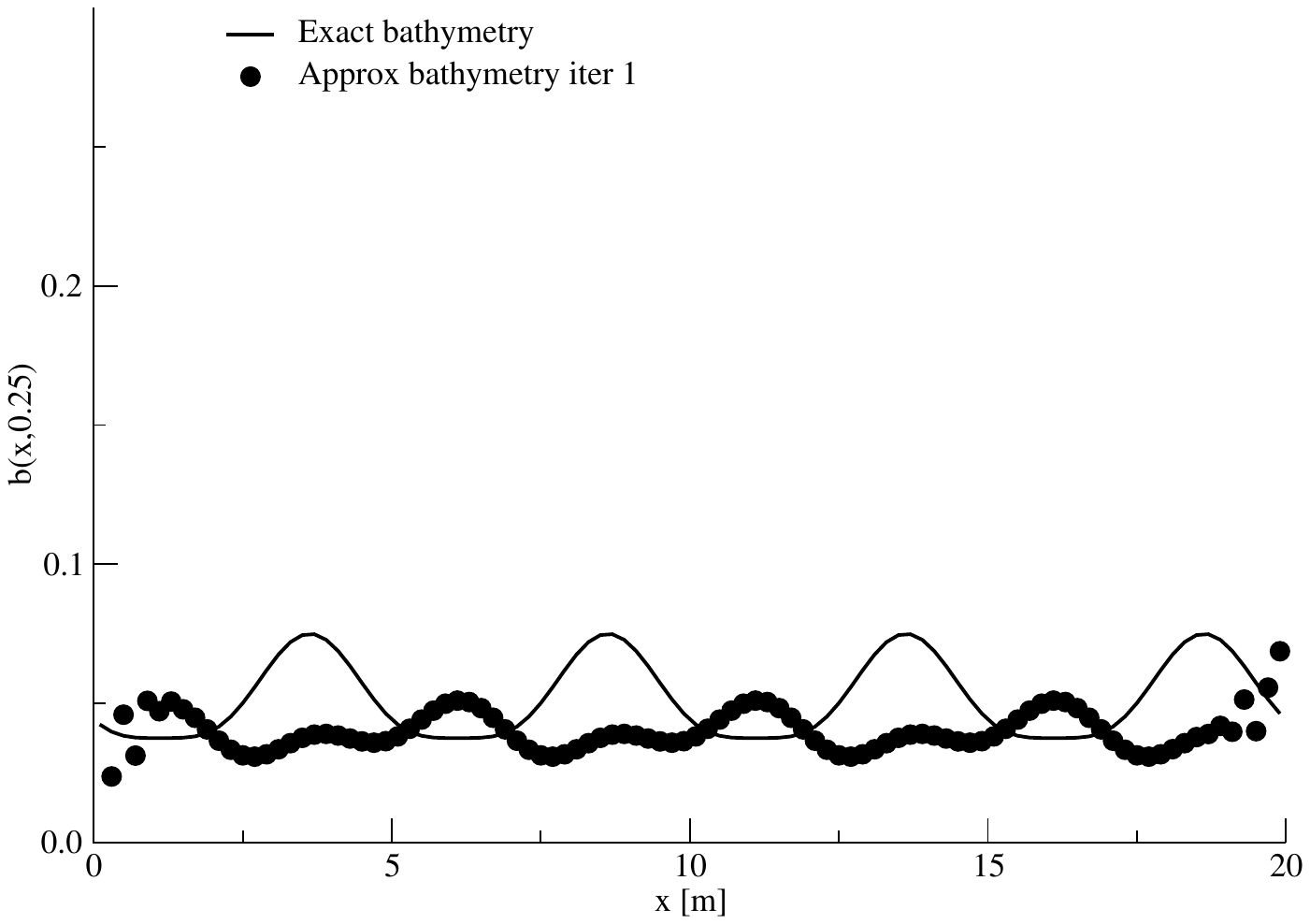} \quad
	\includegraphics[width=0.3\textwidth]{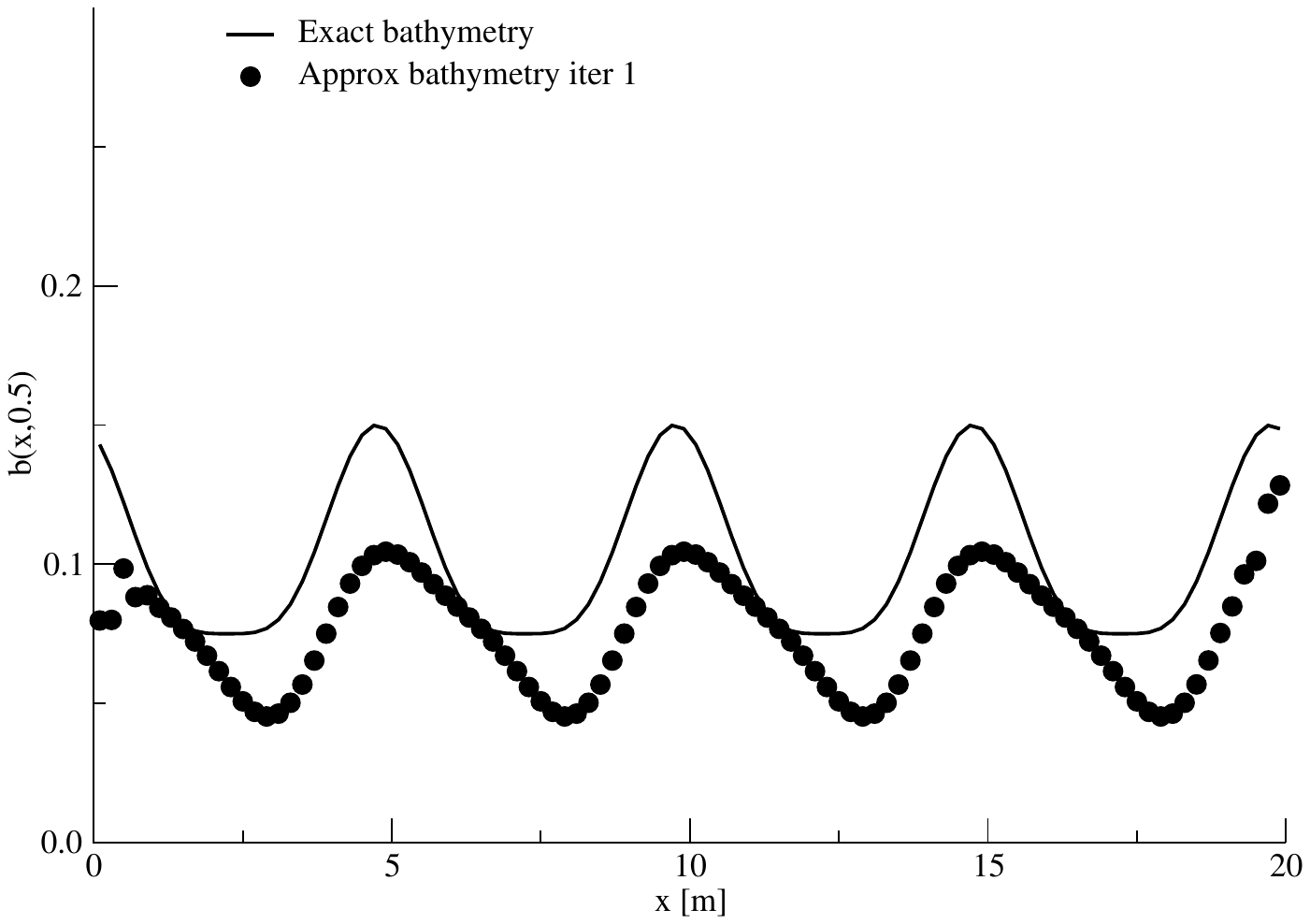} \quad
	\includegraphics[width=0.3\textwidth]{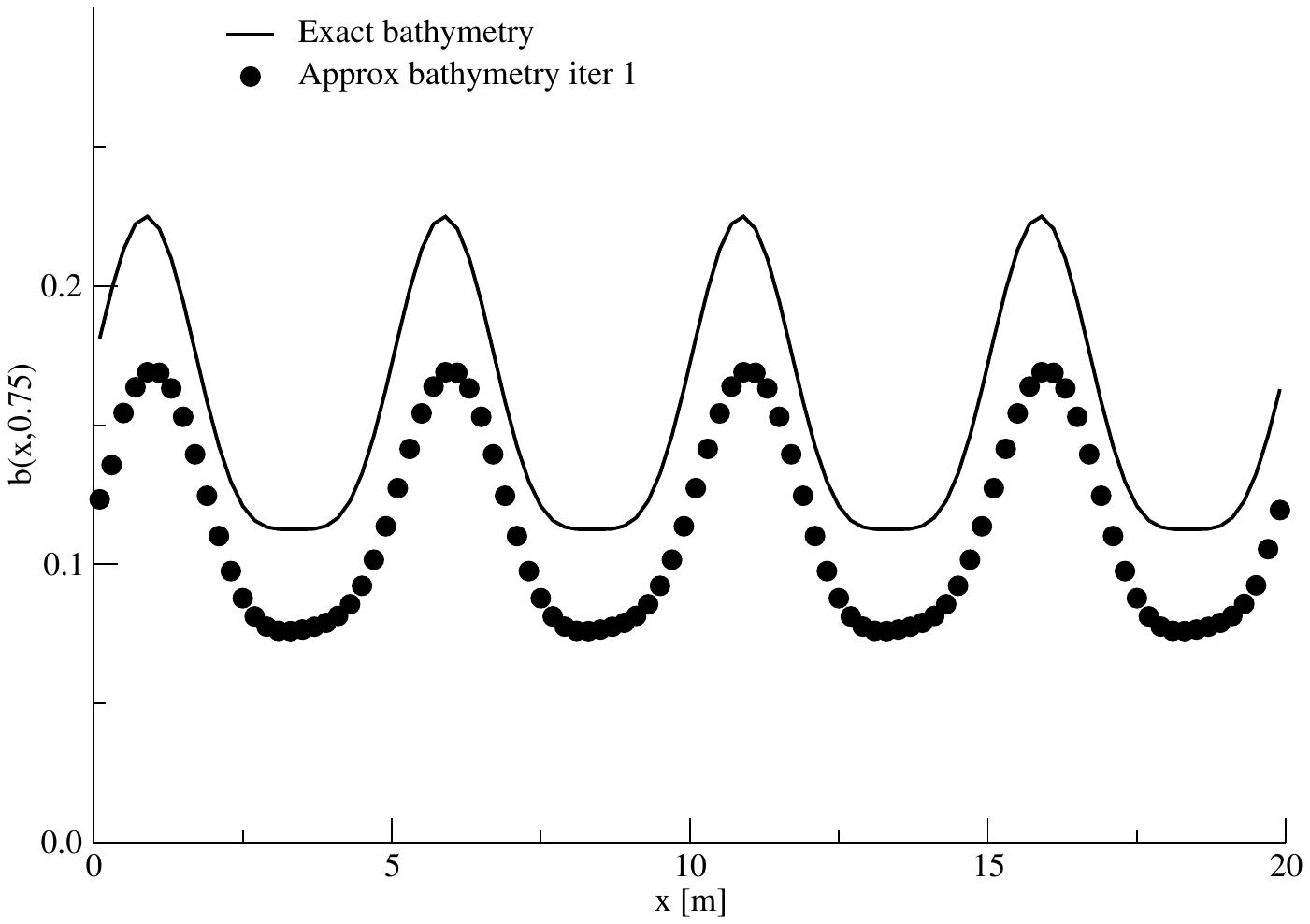} \\
	
	\includegraphics[width=0.3\textwidth]{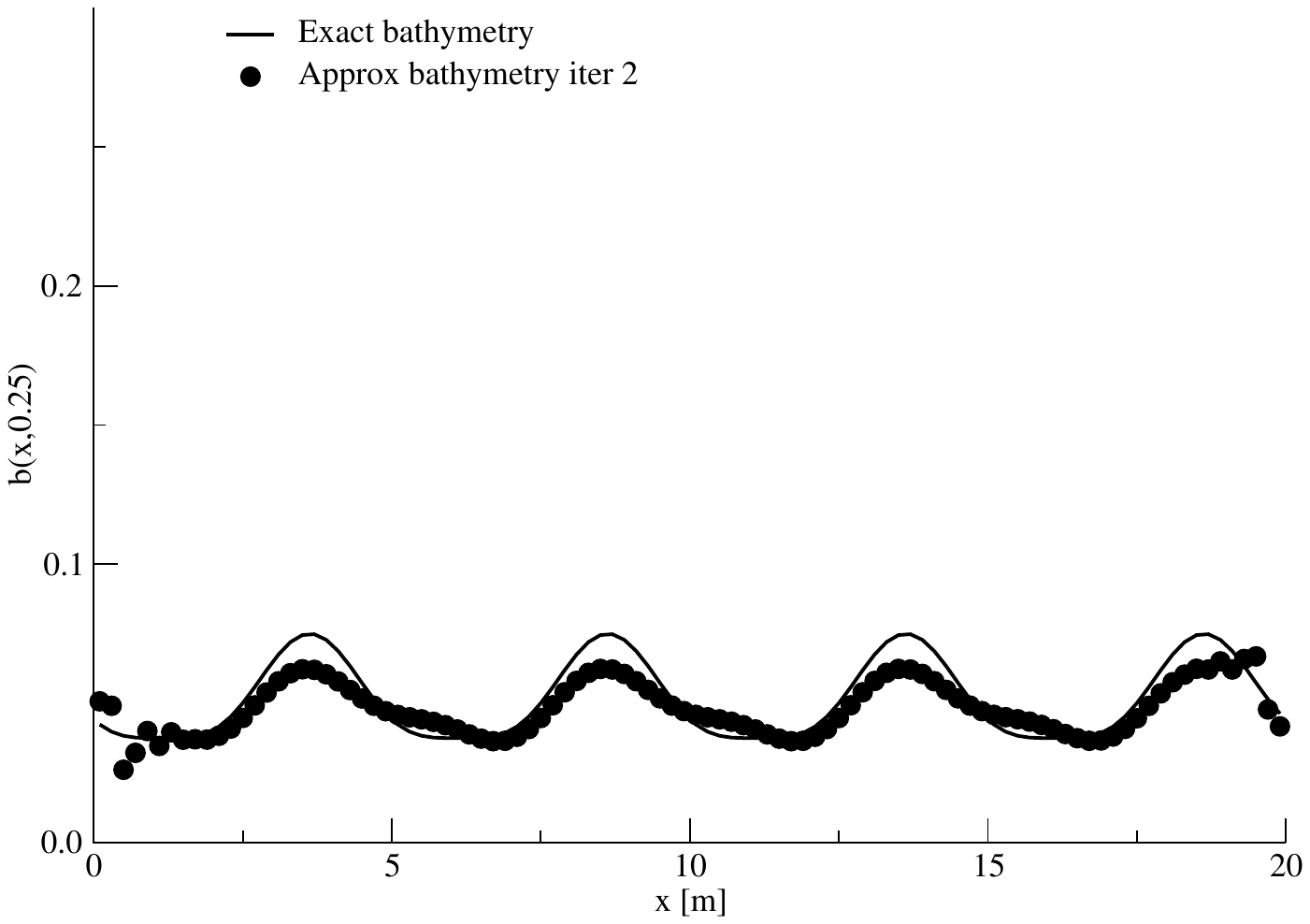} \quad
	\includegraphics[width=0.3\textwidth]{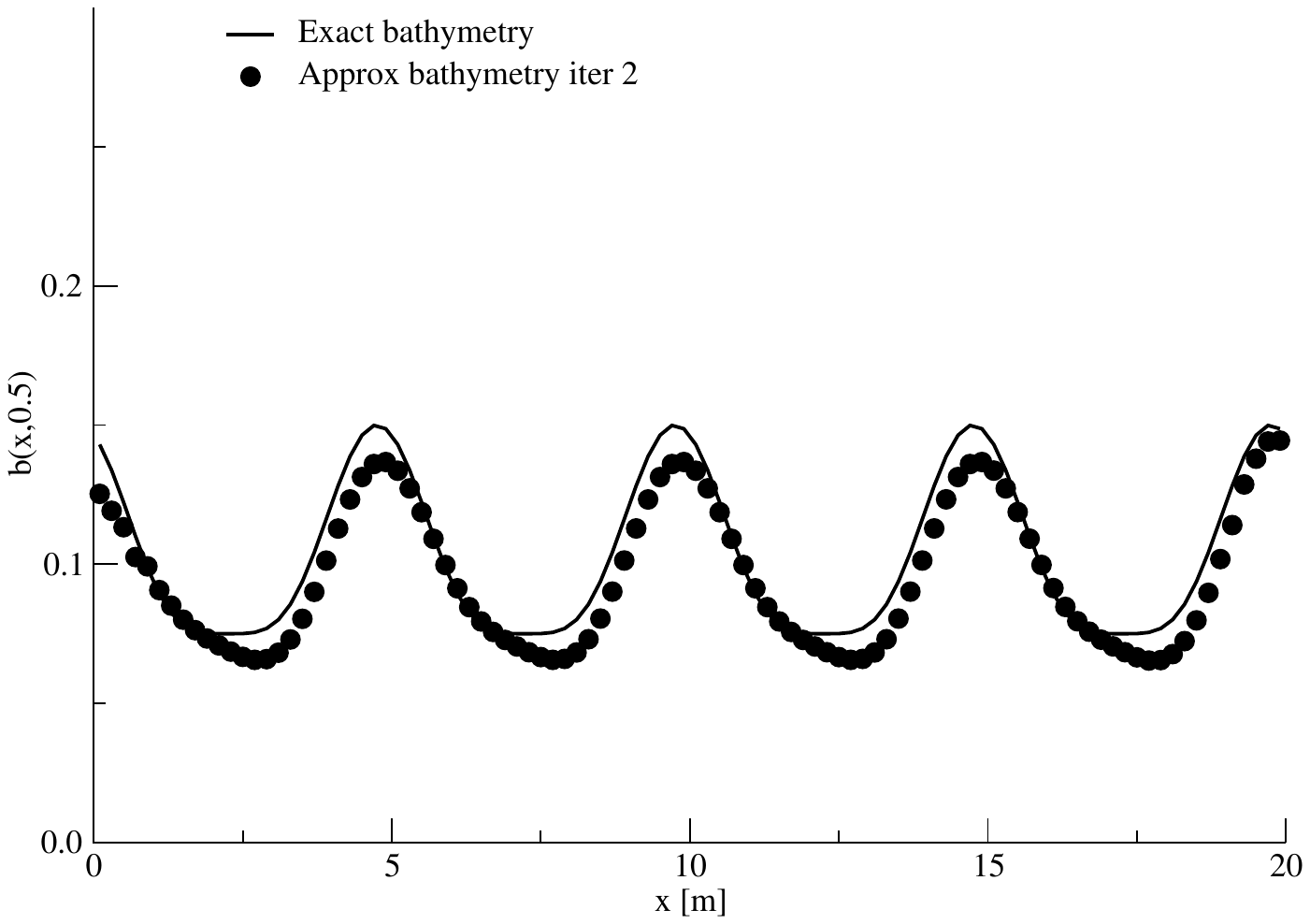} \quad
	\includegraphics[width=0.3\textwidth]{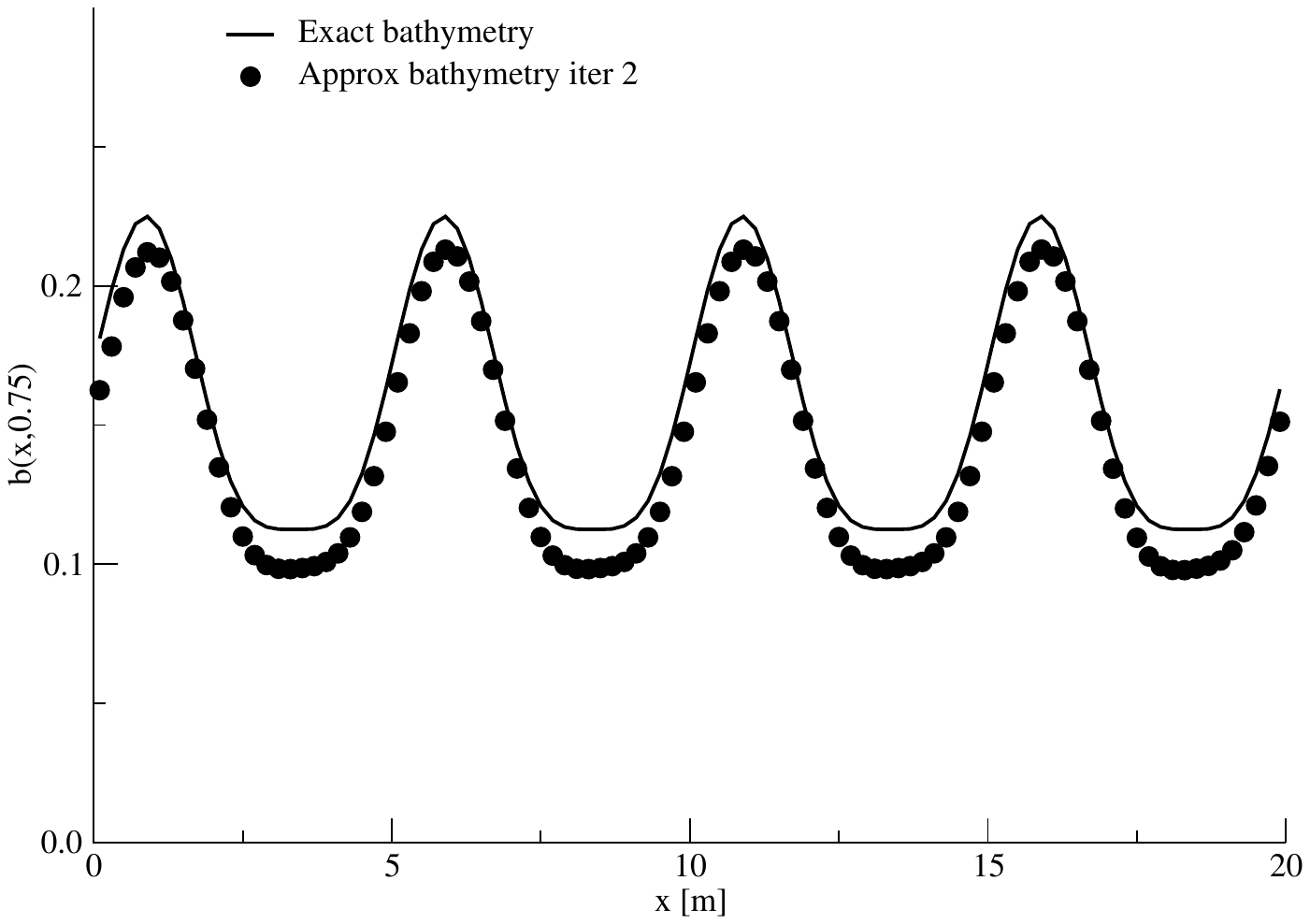} \\

	\includegraphics[width=0.3\textwidth]{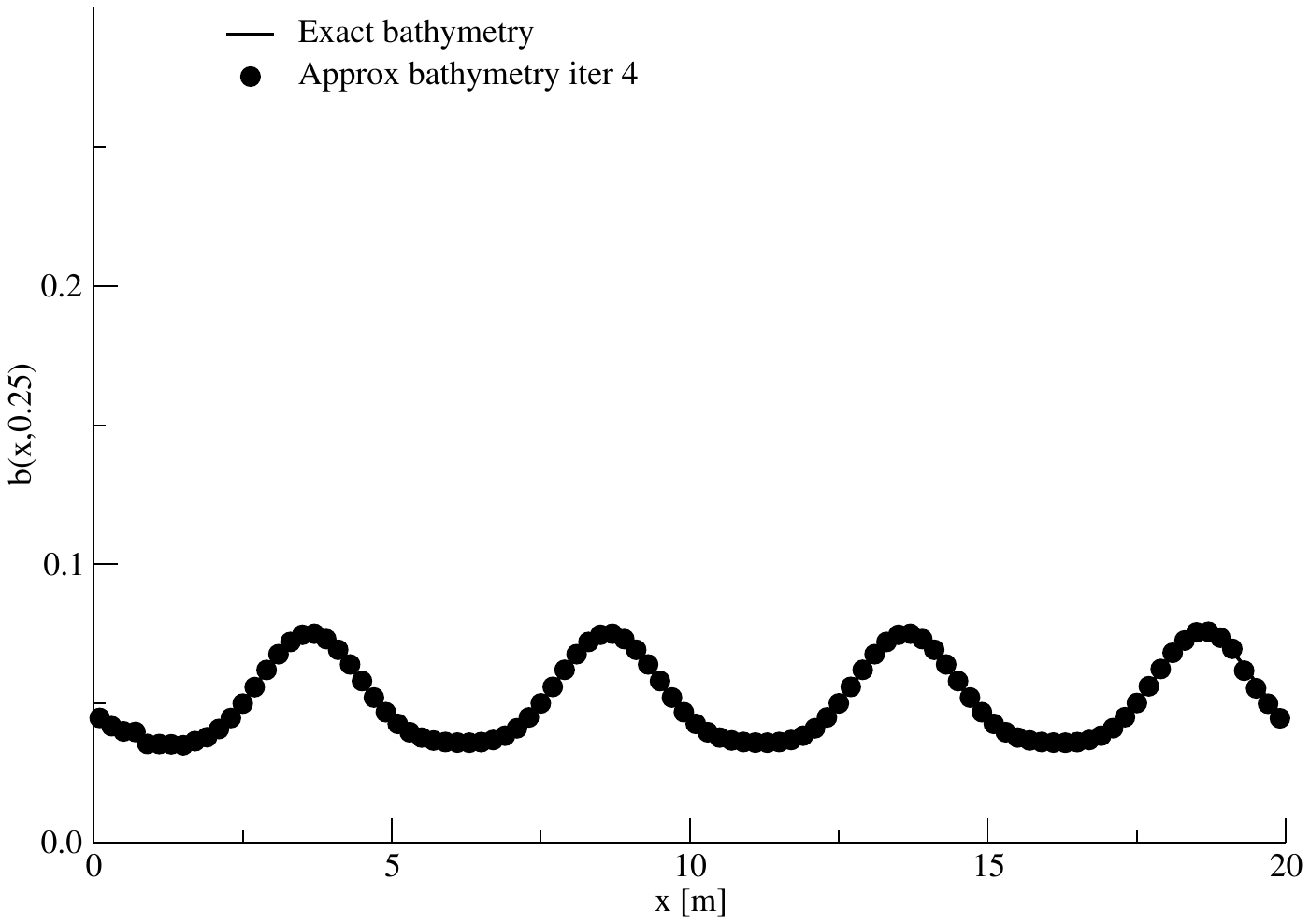} \quad
	\includegraphics[width=0.3\textwidth]{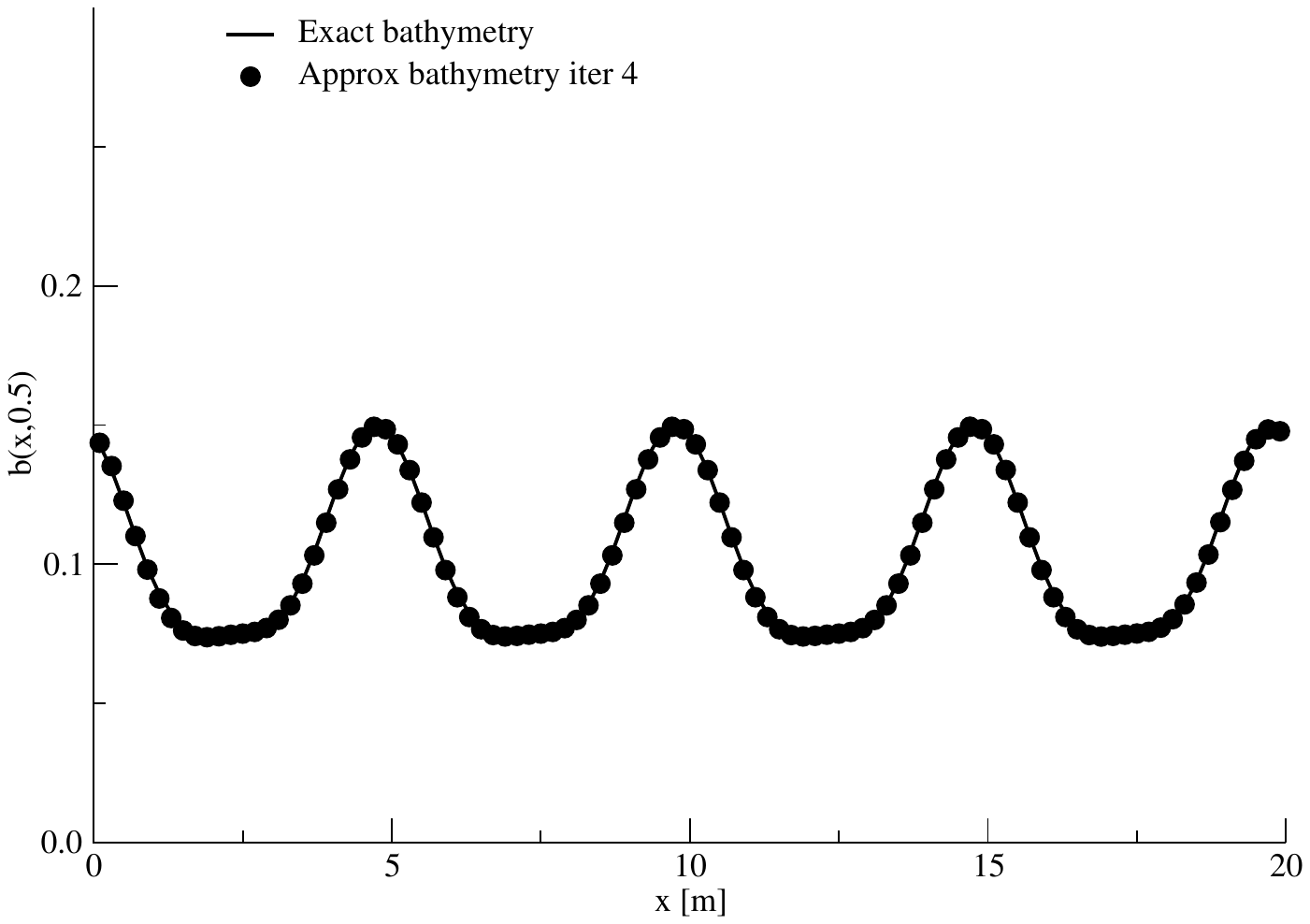} \quad
	\includegraphics[width=0.3\textwidth]{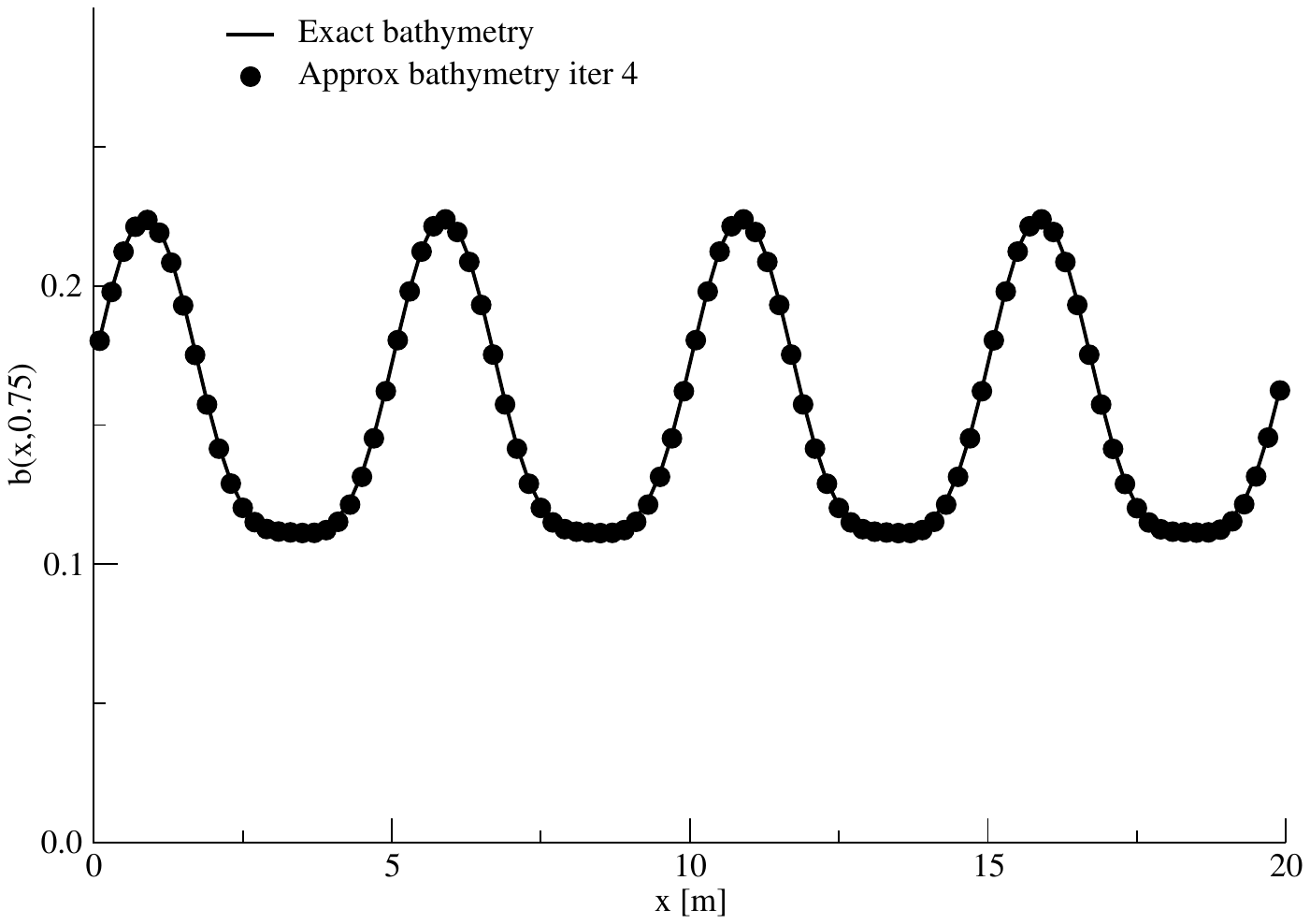} \\
	
	\includegraphics[width=0.3\textwidth]{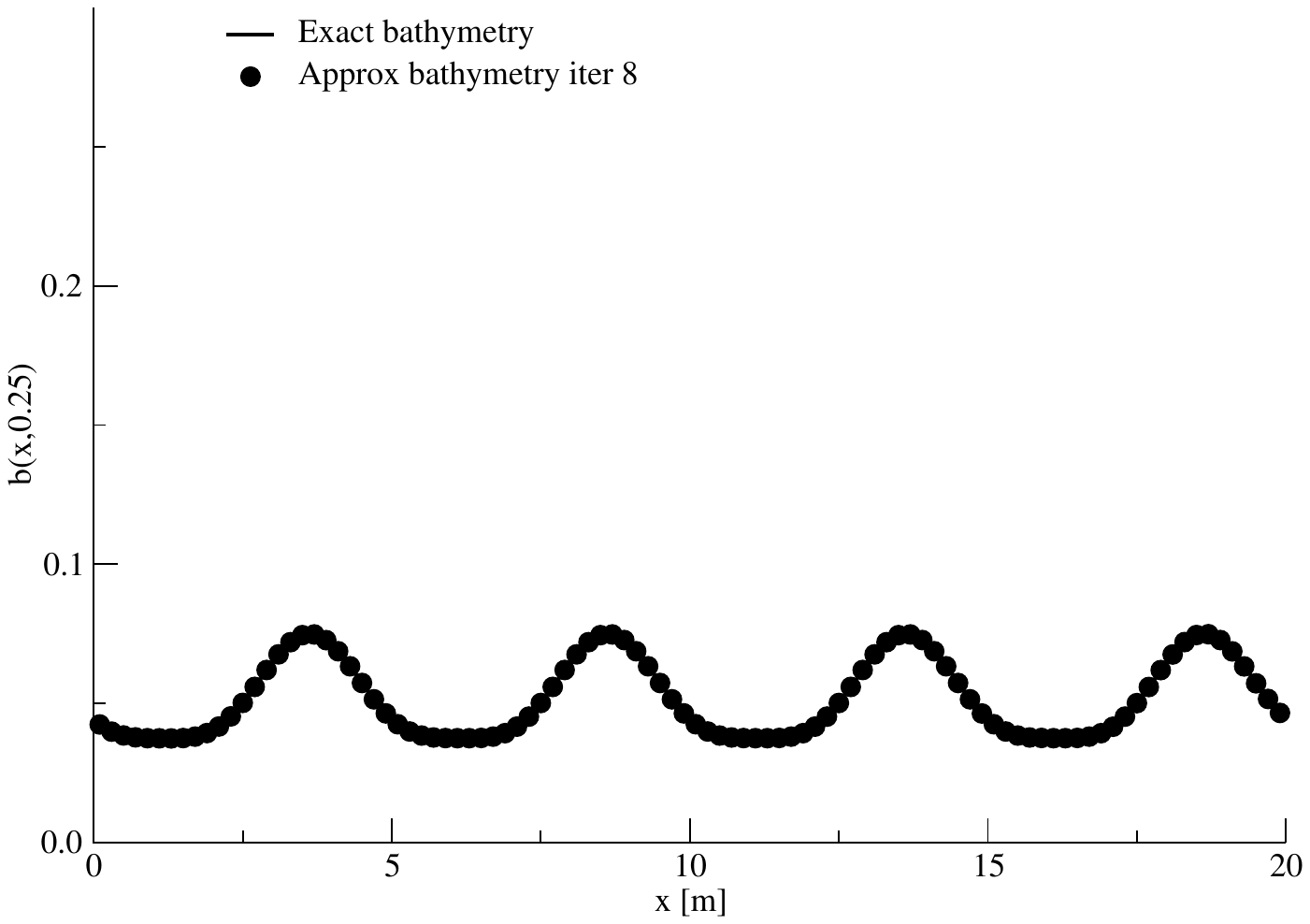} \quad
	\includegraphics[width=0.3\textwidth]{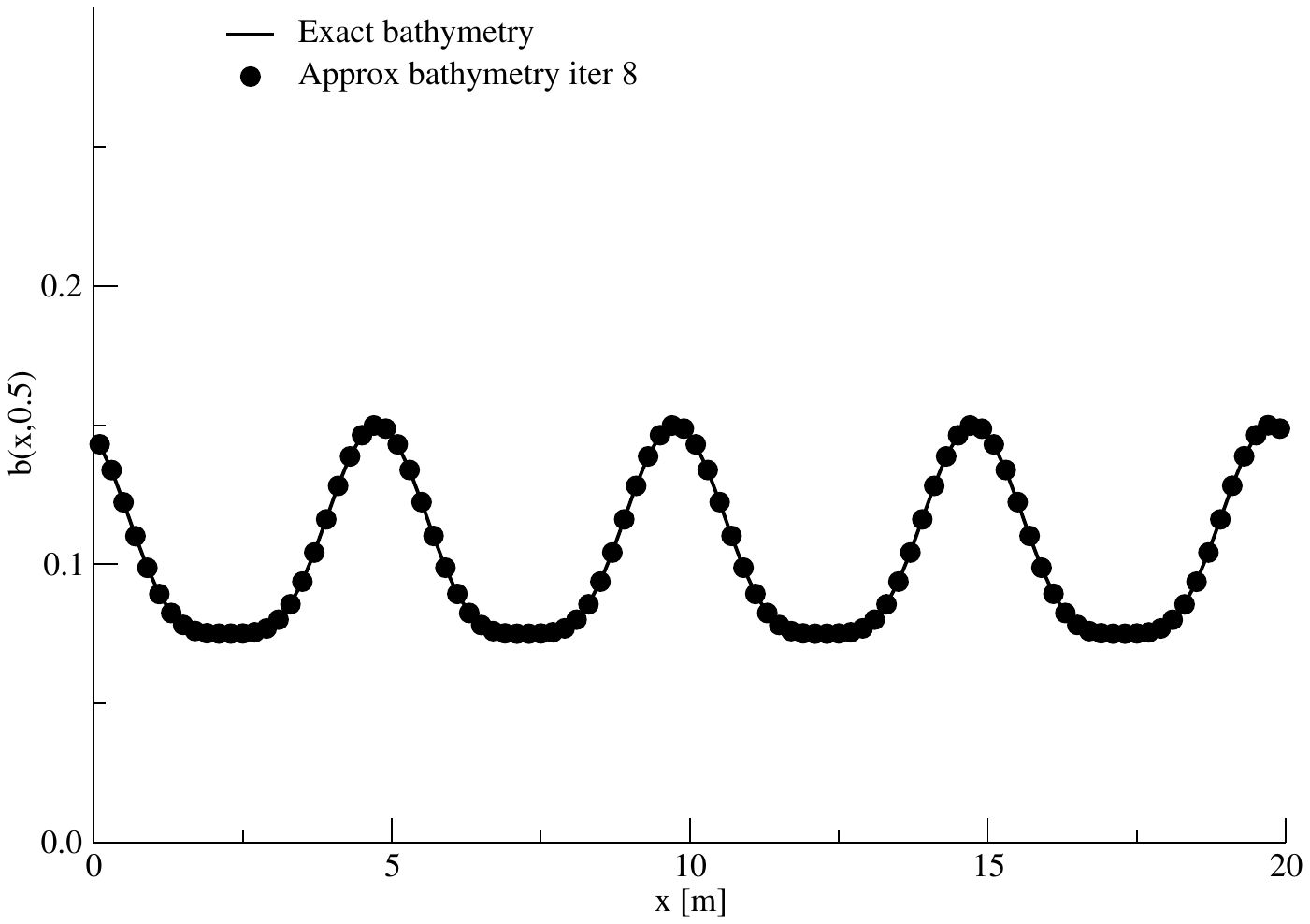} \quad
	\includegraphics[width=0.3\textwidth]{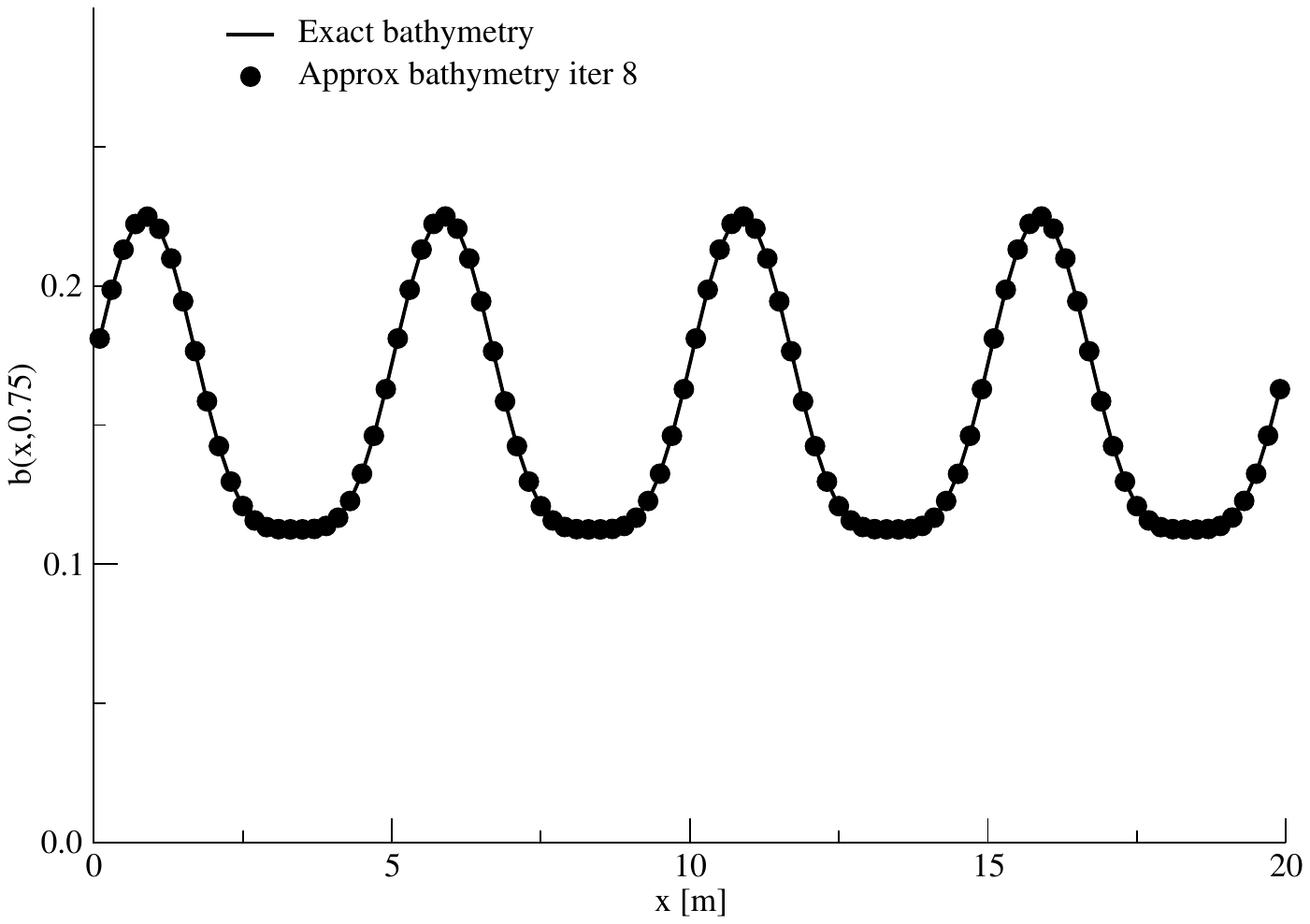} \\
	
	\caption{Smooth bottom profile with large gradient (\ref{eq:b-large-gradient}): Result for the reconstruction procedure resulting from {\bf FORCE-$\alpha$+FD}. Parameters $\Delta t = 0.01$,  $\alpha_F = 2$, $\varepsilon = 0.001$, $\lambda_b = 0.71$, $100$ cells.
		{\bf Feft:} $t=0.25$, {\bf centered} $t=0.5$, {\bf right:} $t=0.75$.}
	\label{fig:b-for-iter-and-times:test-2:ForceAlphaCons}
\end{figure}
\FloatBarrier

\begin{figure}   
	\centering
	\includegraphics[width=0.3\textwidth]{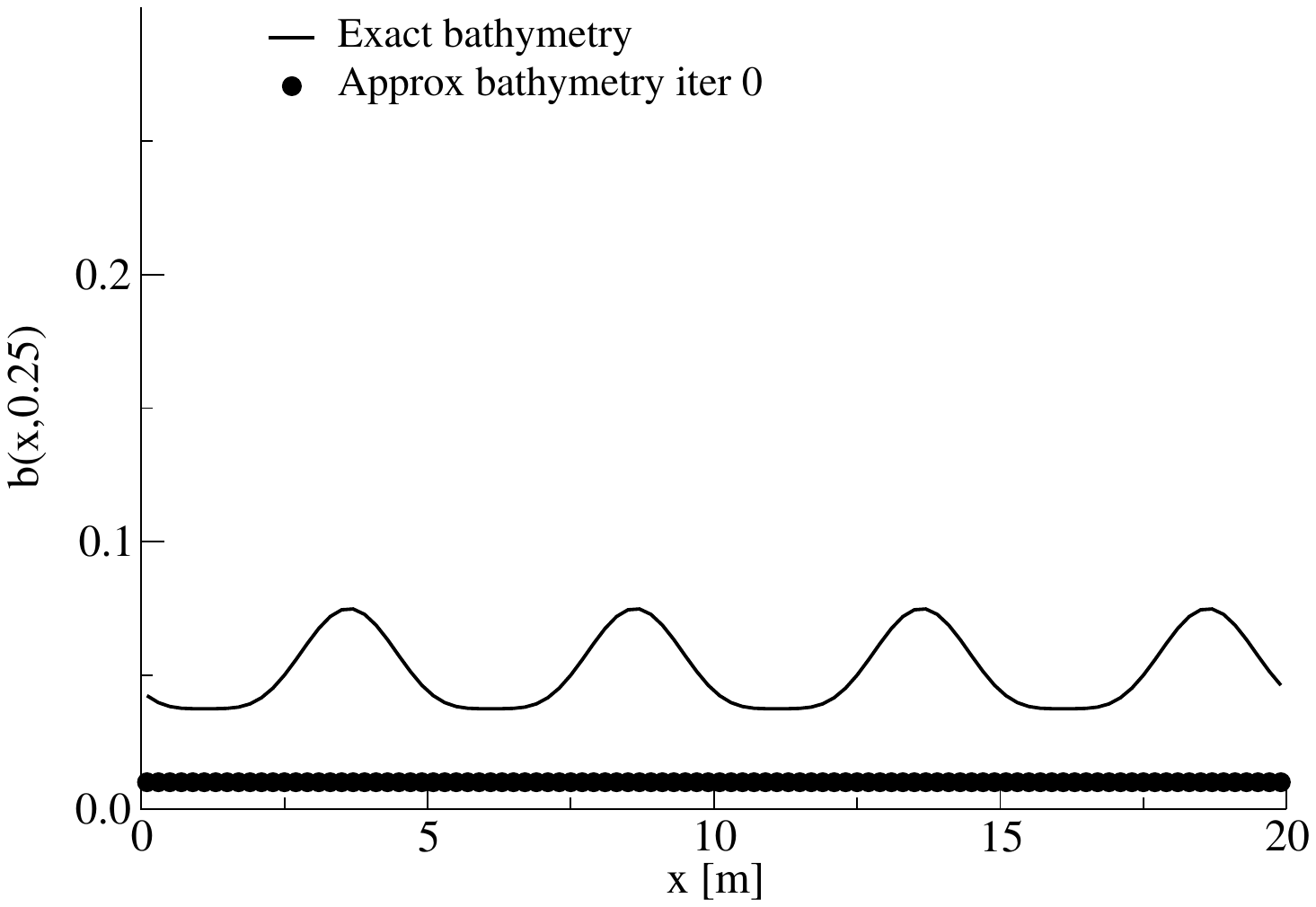} \quad
	\includegraphics[width=0.3\textwidth]{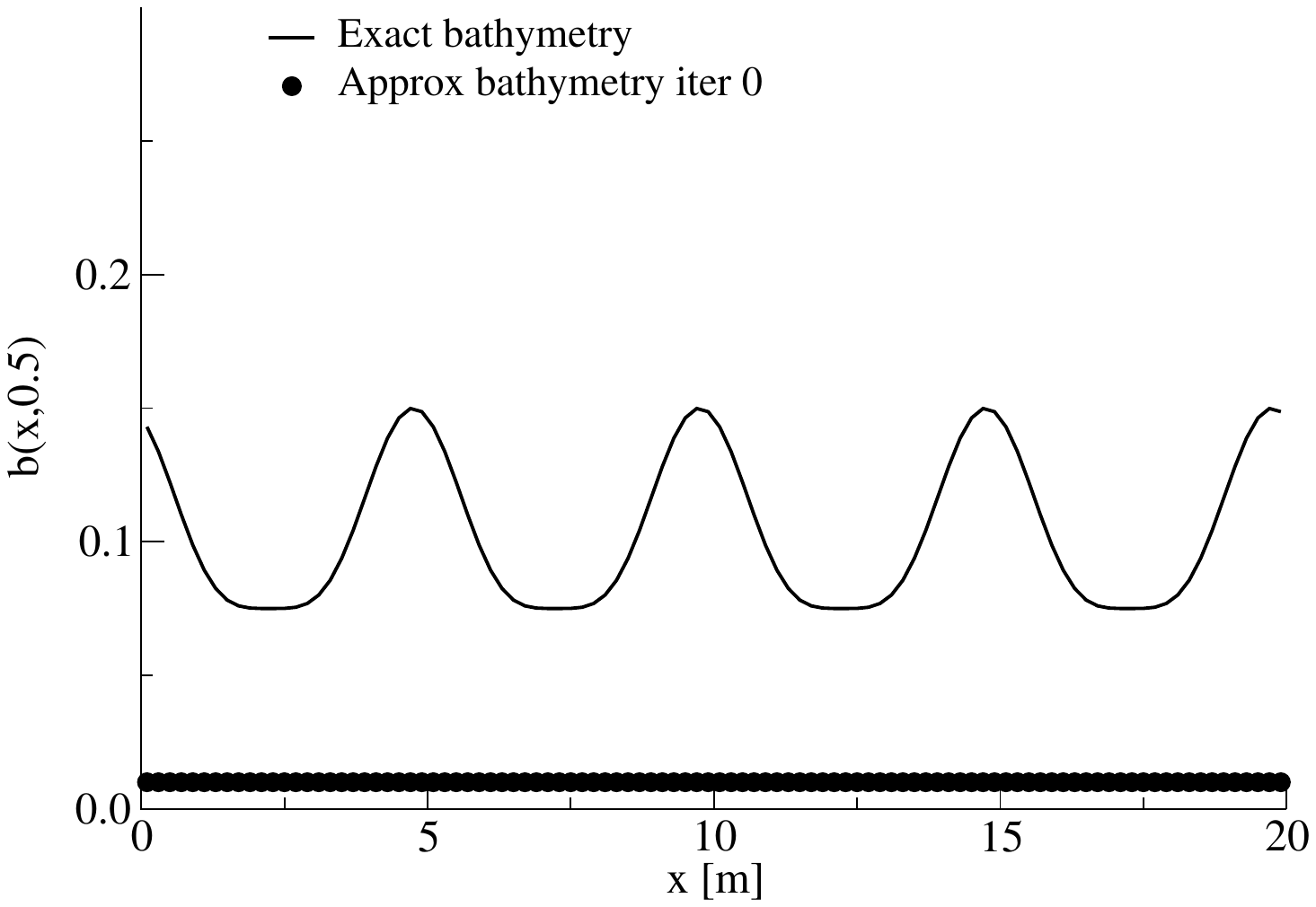} \quad
	\includegraphics[width=0.3\textwidth]{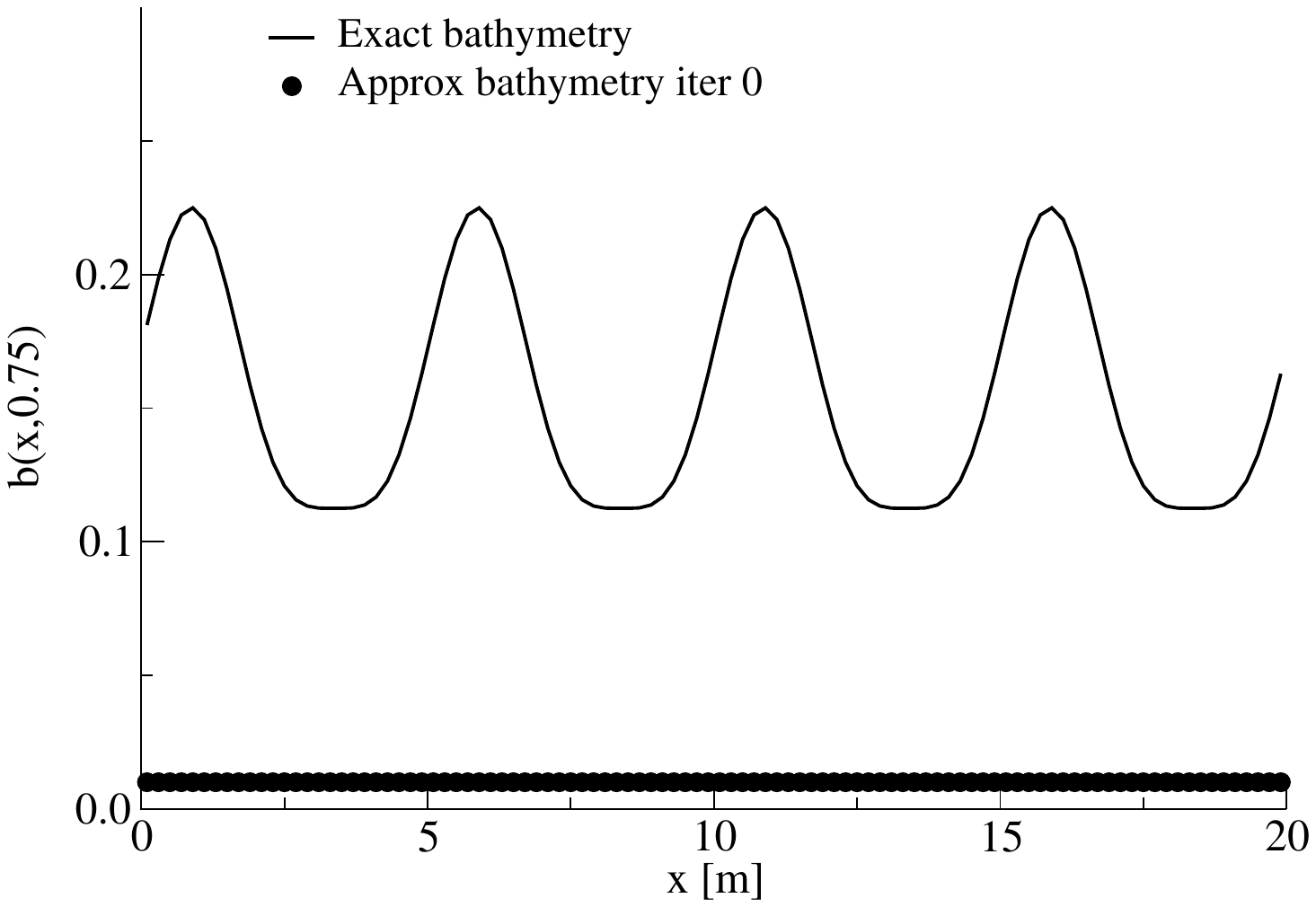} 
	\\
	
	\includegraphics[width=0.3\textwidth]{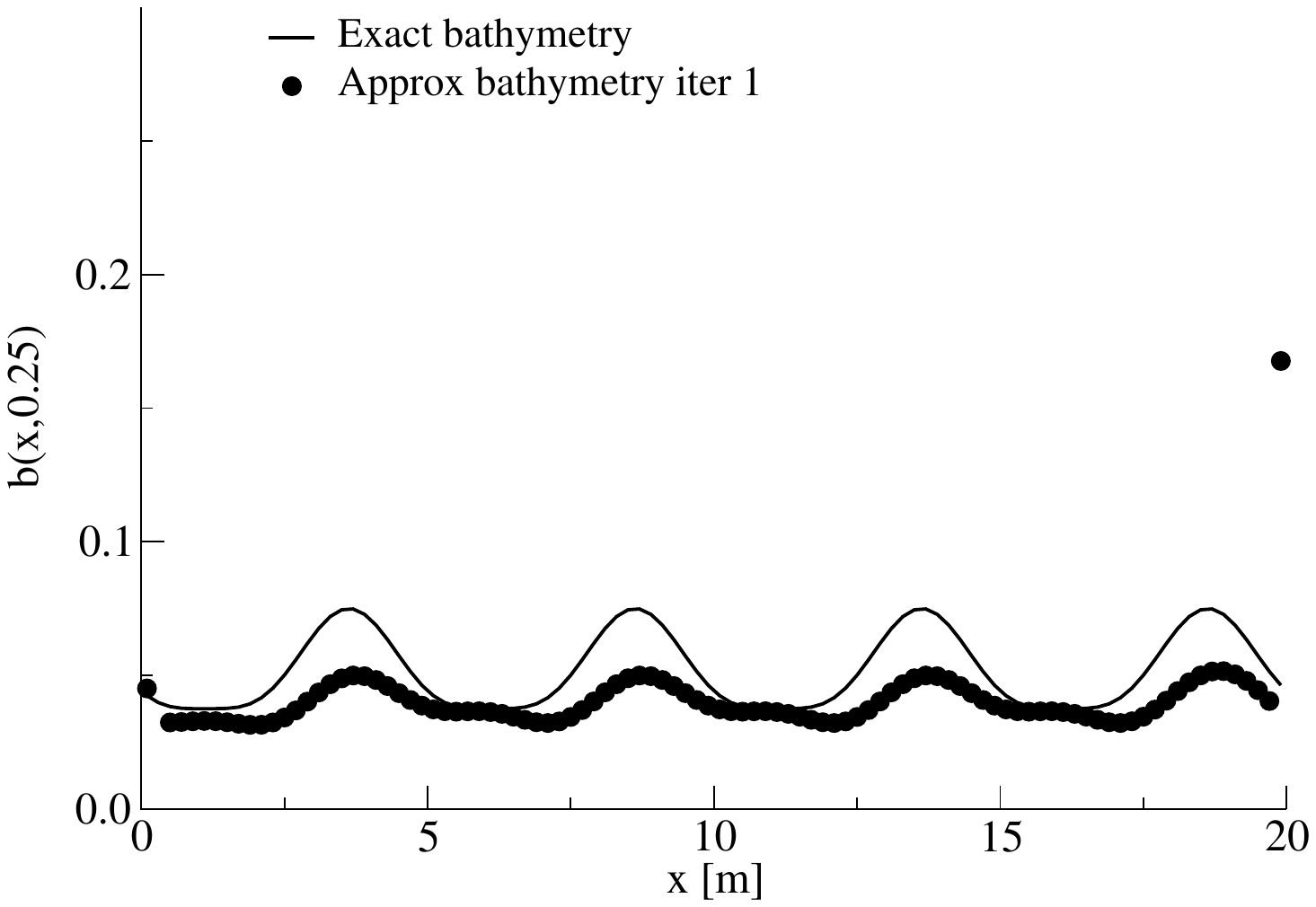} \quad
	\includegraphics[width=0.3\textwidth]{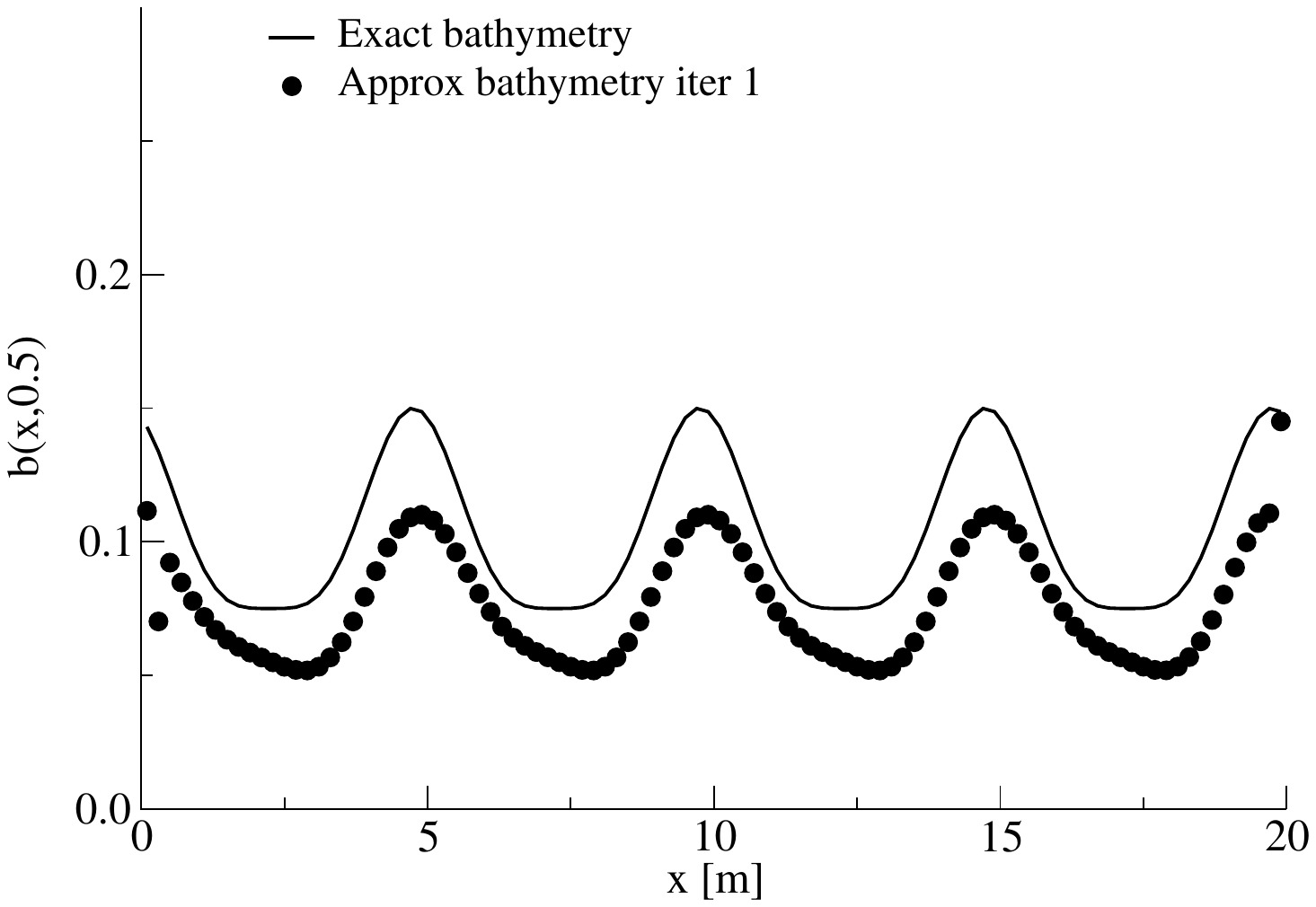} \quad
	\includegraphics[width=0.3\textwidth]{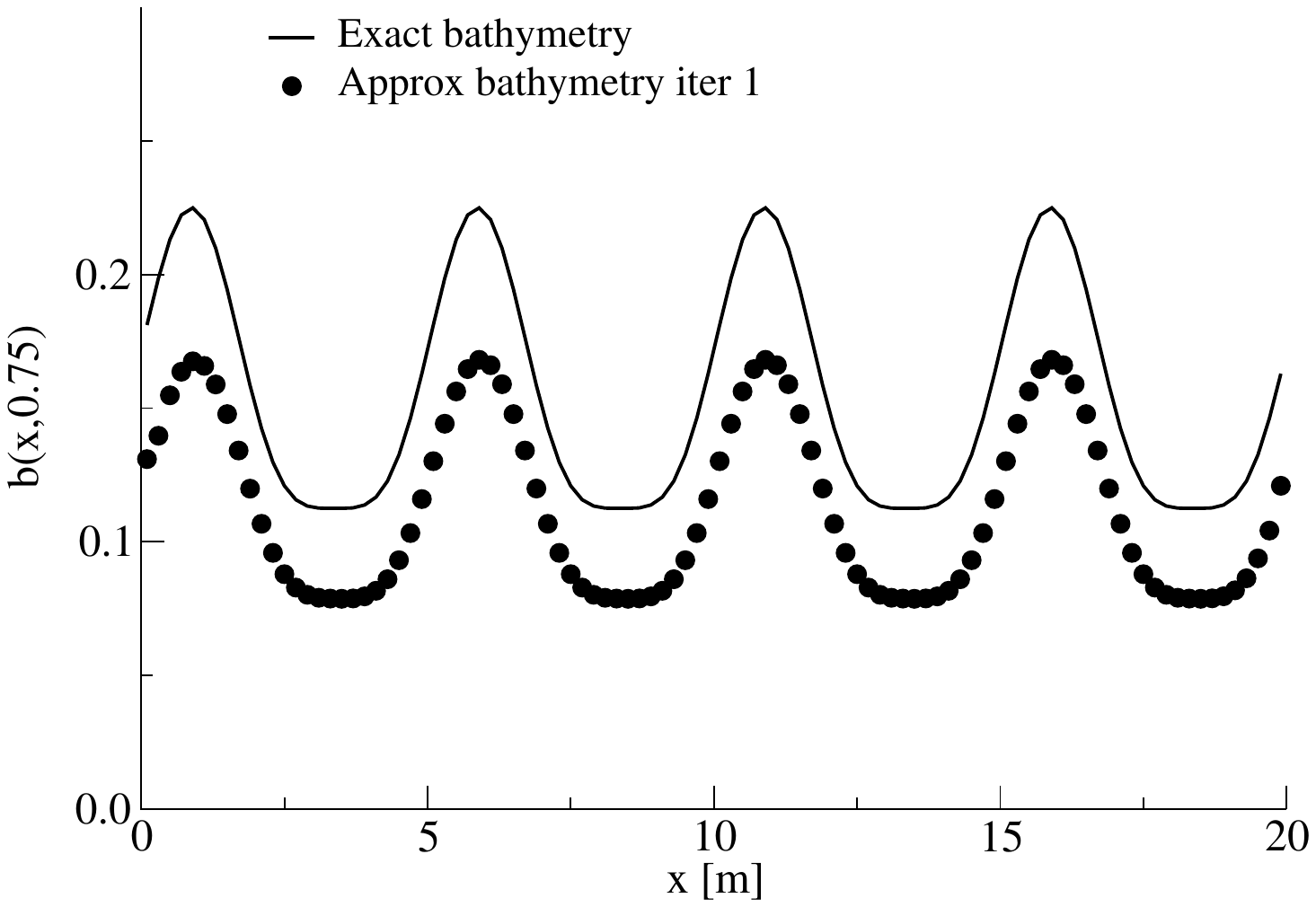} \\
	
	\includegraphics[width=0.3\textwidth]{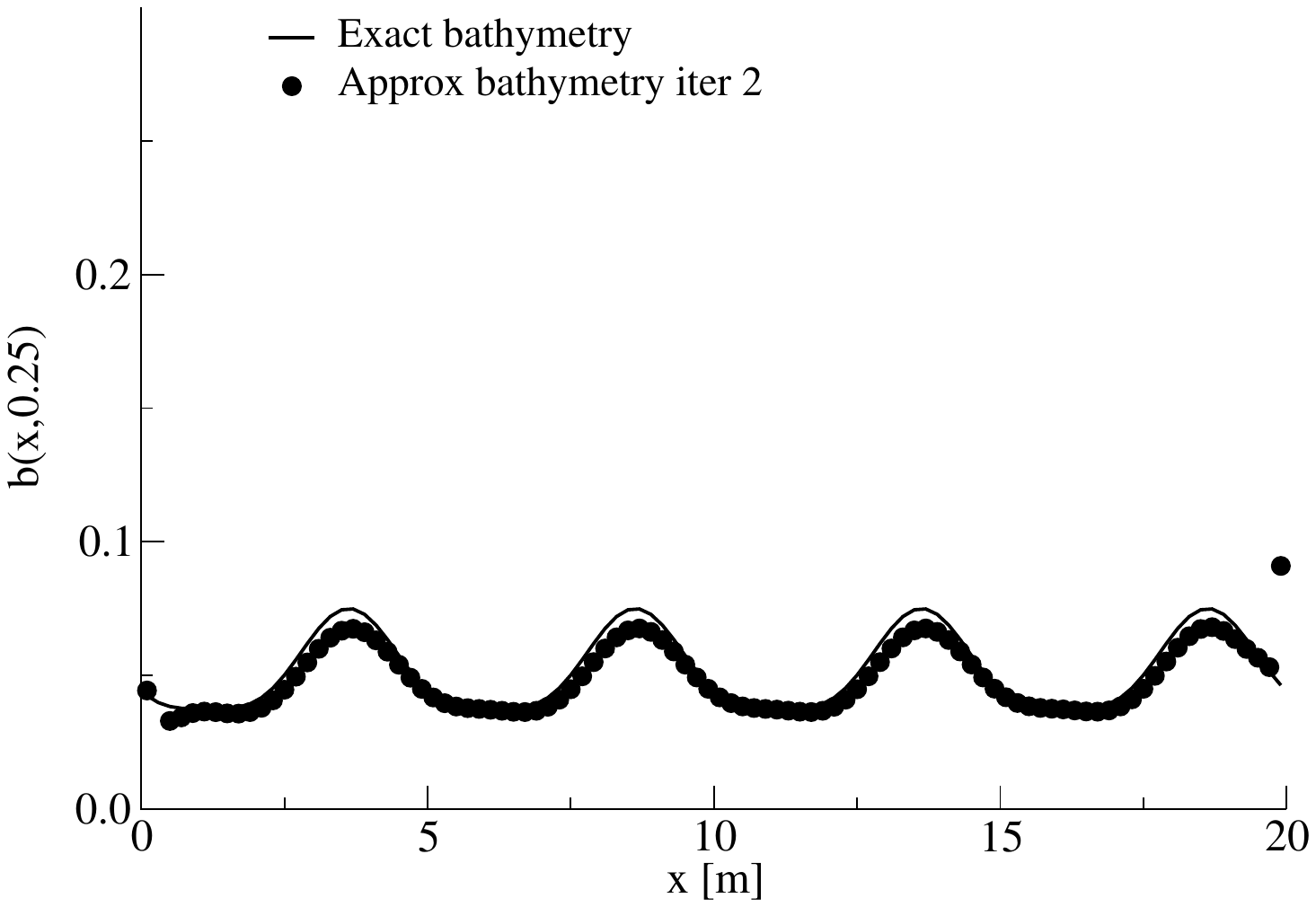} \quad
	\includegraphics[width=0.3\textwidth]{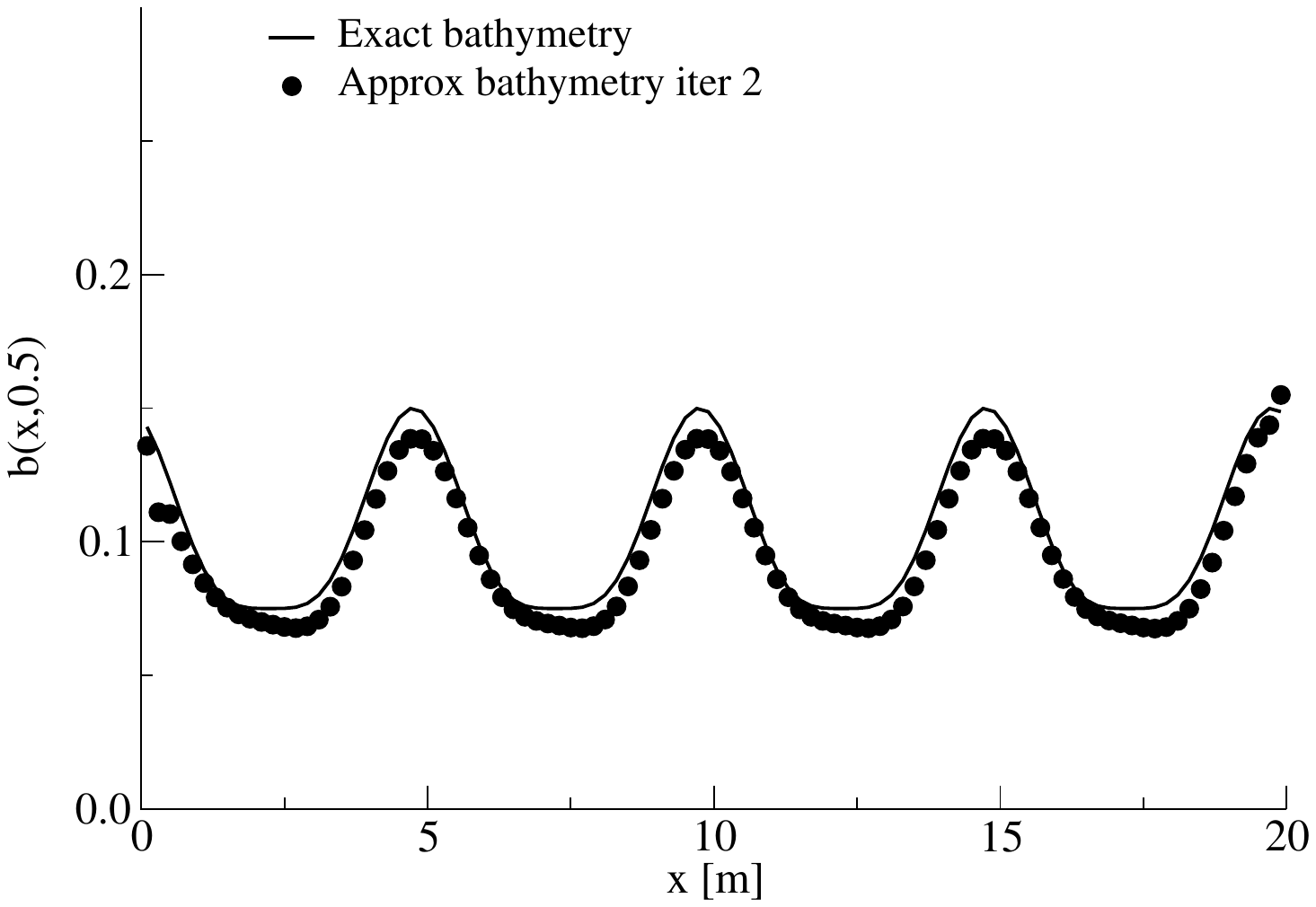} \quad
	\includegraphics[width=0.3\textwidth]{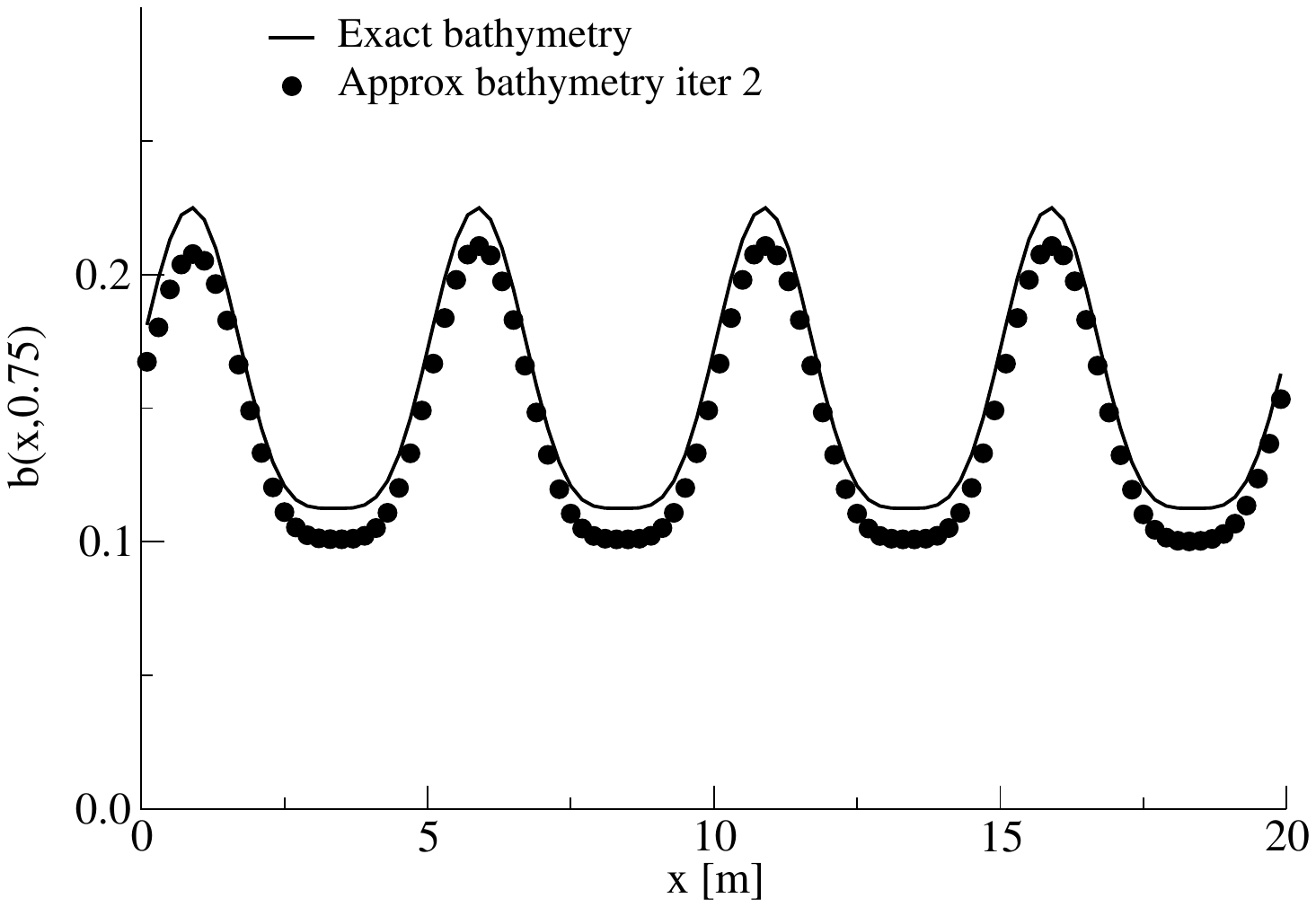} \\

	\includegraphics[width=0.3\textwidth]{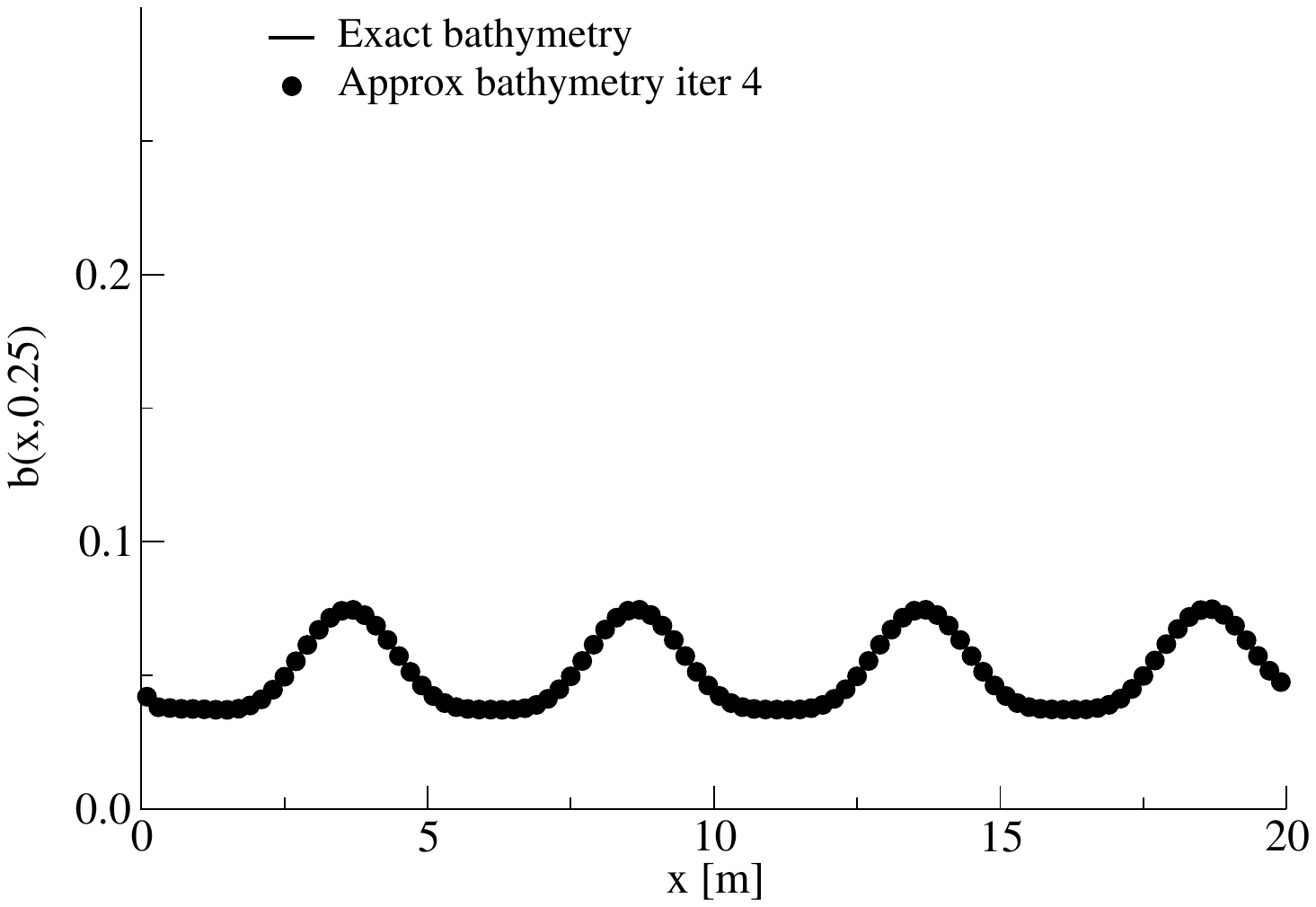} \quad
	\includegraphics[width=0.3\textwidth]{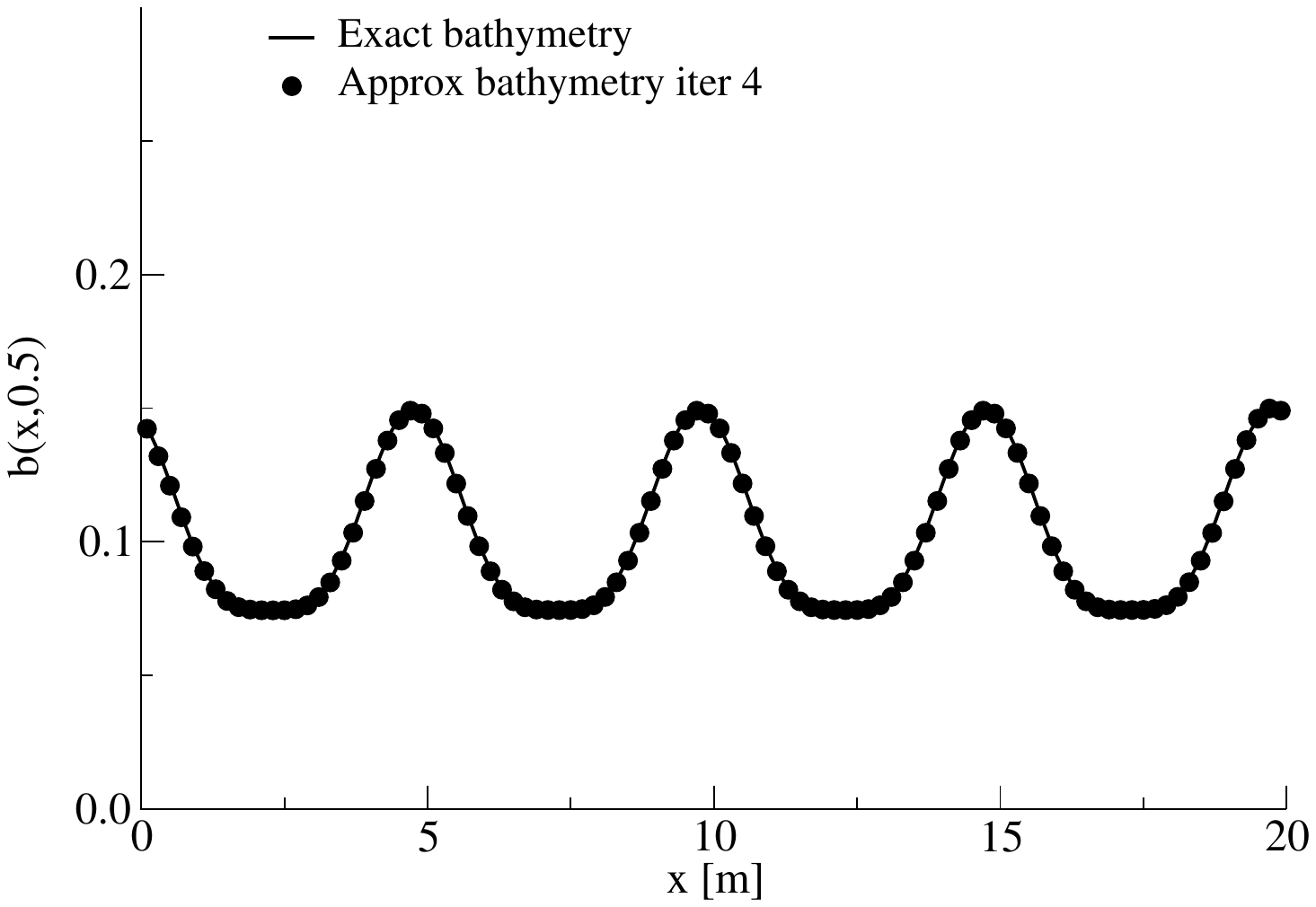} \quad
	\includegraphics[width=0.3\textwidth]{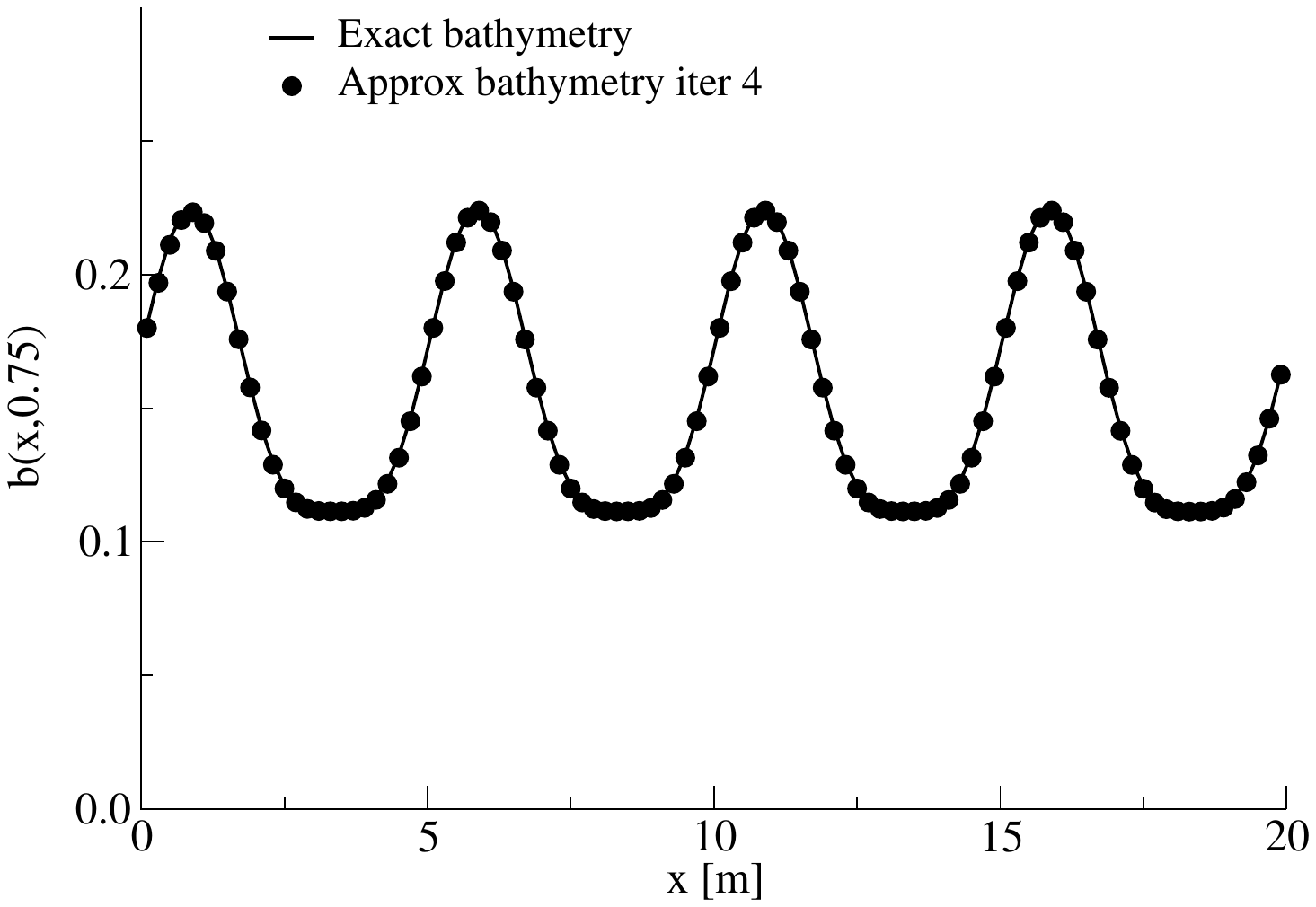} \\
	
	\includegraphics[width=0.3\textwidth]{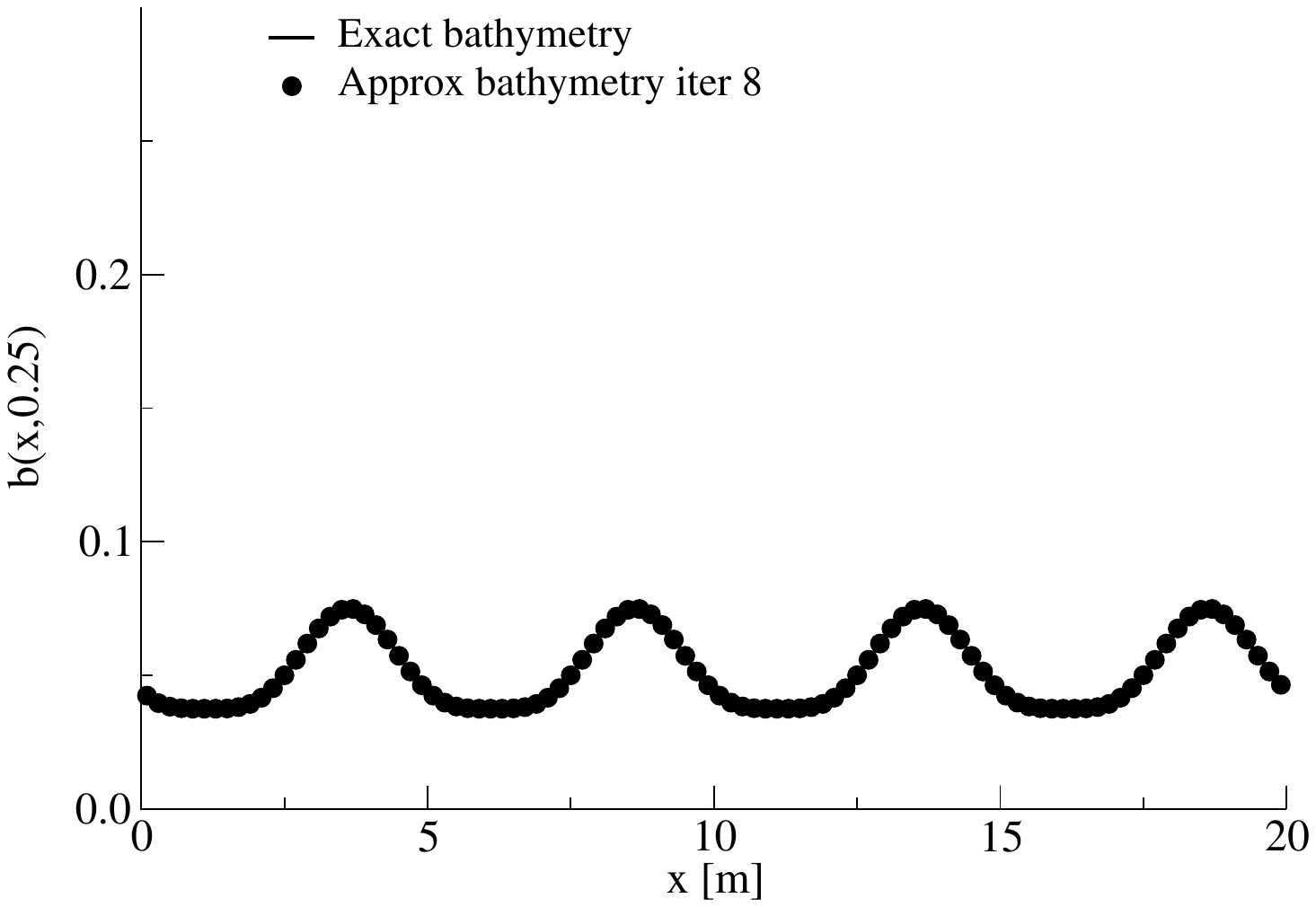} \quad
	\includegraphics[width=0.3\textwidth]{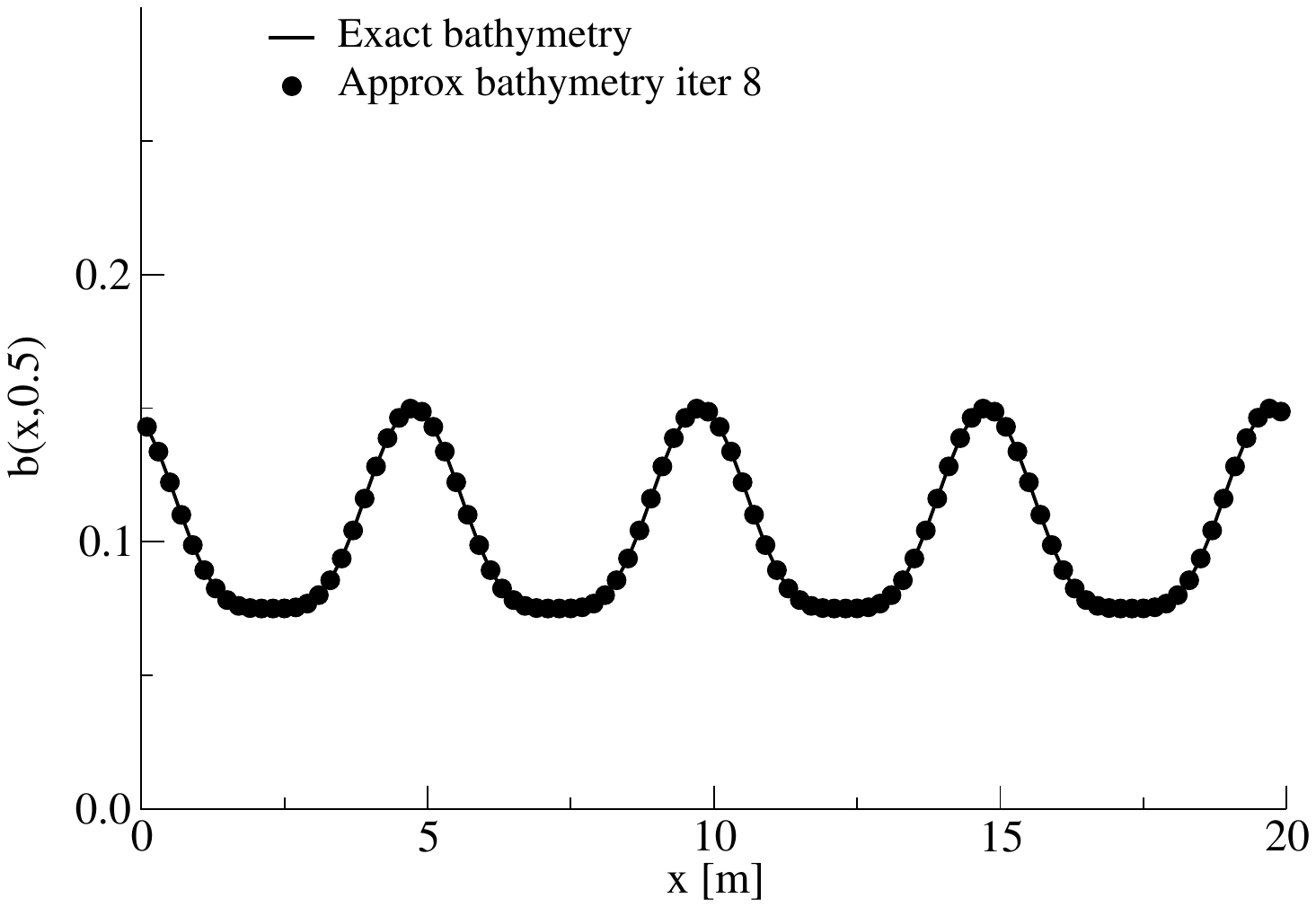} \quad
	\includegraphics[width=0.3\textwidth]{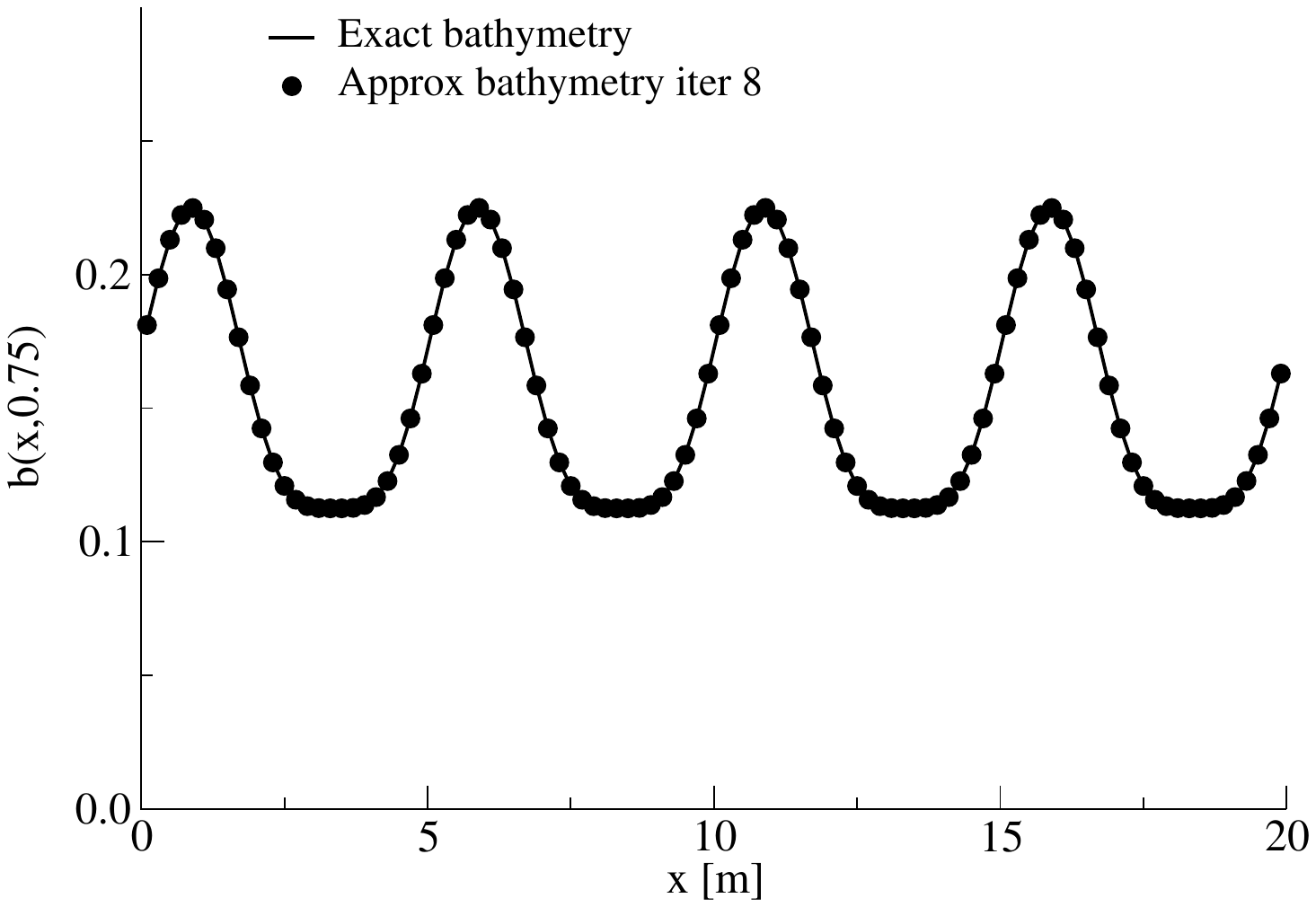} \\
	
	\caption{Smooth bottom profile with large gradient (\ref{eq:b-large-gradient}): Result for the {\bf FORCE-$\alpha$+CSF}. Parameters $\Delta t = 0.01$,  $\alpha_F = 2$, $\varepsilon = 0.001$, $\lambda_b = 0.71$, $100$ cells.
		{\bf Feft:} $t=0.25$, {\bf centered} $t=0.5$, {\bf right:} $t=0.75$.}
	\label{fig:b-for-iter-and-times:test-2:NonCons-Force-alpha}
\end{figure}
\FloatBarrier

\section{Conclusions}\label{S8}

In this paper we have considered the inverse problem of detecting a time-varying bottom through free surface measurements. This problem has applications in oceanography, in the study of tsunami generation by vertical and horizontal displacements of underwater plates. Also in man-made facilities to design surf pools.

To study the problem we used the system of conservation laws for an incompressible, non-viscous fluid, in the shallow water regime with time-varying bottom, known as the Boussinesq-Peregrine system (\ref{BP}) (see \cite{lannes2013water}). The solution to (\ref{BP}) is obtained through an iterative Picard scheme that generates a sequence of linearized systems. For each of them, energy estimates are obtained, where the symmetrization of (\ref{BP}) is fundamental. Well-posedness follows classical ideas from energy estimates as in \cite{israwi2011large,taylor1997partial}.

The identification of the bottom was obtained as the minimum of a functional involving the $L^2$ norm of the solution of (\ref{BP}). Existence of an optimal solution and formulation of a first order necessary condition, involving the adjoint system, were obtained.

The numerical identification of the bottom followed the methodology described above, through a descent method. Both, the state and adjoint equations, were discretized by using a unified finite volume scheme for non-conservative systems. Namely, we implemented the FORCE-$\alpha$ a universal, low-dissipation scheme applied to the coupled system (\ref{eq:unified-system}), in the classical fashion of forward and backward evolution that we called {\bf FORCE-$\alpha$+CSF}, and contrasted with two reference solutions, namely {\bf Rusanov+FD} and {\bf FORCE-$\alpha$+FD}, assumed as standard discretizations. 

Let us mention a few things in this regard. First, for a stable computation of (\ref{BP}) and its adjoint system, small CFL coefficients are required. Second, the reference scheme {\bf Rusanov+FD} uses the Rusanov method (see \cite{rusanov1961}) to solve BP and it is the one with the worst results. Possibly due to the fact that Rusanov scheme introduces the largest numerical diffusion among the three methods. Third, the universal, low-dissipation scheme for both the state and adjoint systems yields the best results, making clear the importance of discretizing the adjoint system in an appropriate manner.

Finally, test problems presented in this paper show that the bottoms have been successfully recovered, and therefore, low-dissipation schemes work as a correct numerical methodology in this type of inverse problems.

In the future, we will implement this strategy in more sophisticated models, such as the Green-Naghdi system (see \cite{israwi2011large}). Also, it would be interesting to understand how much of this strategy can be replicated for the class of equations coming from the water-waves theory in \cite{lannes2013water}. Finally, the current strategy leads us to the question of the exact or approximate controllability of these systems through a source term in the equation of state.

\appendix

\section{Existence and uniqueness of solutions}
\label{S3}

In this section we prove a well-posedness theorem for the Boussinesq--Peregrine system \eqref{BP}. The idea is to write \eqref{BP} as a symmetrizable hyperbolic system (see \cite{taylor1997partial}). Then, an iterative Picard scheme is carried out. For the resulting linear systems, an energy estimate is constructed which yields a solution to (\ref{BP}). This procedure has been done for the Nonlinear Shallow Water \cite{lannes2013water} and Green--Naghdi equations \cite{israwi2011large}, in the case of a stationary bottom.

After multiplying second equation in \eqref{BP} by $h$, the system can be written as
\begin{equation}
	\label{ES31}
	\left\{
	\begin{aligned}
		&\zeta_t+\epsilon(\zeta-b)_xV+hV_x=b_t, \\
		&\left(h+\mu h\left(-\frac{1}{3h}\partial_x(h^3\partial_x\cdot)\right)\right)(V_t+\epsilon VV_x)+h\zeta_x=-\frac{\epsilon h}{2}b_{ttx}.
	\end{aligned}
	\right.
\end{equation}

Then, if one defines the operator
\begin{equation}
	\label{Tgeneral}
	\mt w=hw+\mu h\left(-\frac{1}{3h}\partial_x(h^3\partial_xw)\right),
\end{equation}
system \eqref{ES31} becomes

\begin{equation*}
	\begin{pmatrix}
		1 & 0 \\ 0 & \mt\cdot
	\end{pmatrix}
	\begin{pmatrix}
		\zeta_t \\ V_t
	\end{pmatrix}
	+
	\begin{pmatrix}
		\epsilon V & h \\ h & \mt(\epsilon V\cdot)
	\end{pmatrix}
	\begin{pmatrix}
		\zeta_x \\ V_x
	\end{pmatrix}
	=
	\begin{pmatrix}
		b_t+\epsilon b_xV \\ -\frac{\epsilon}{2}hb_{ttx}
	\end{pmatrix}.
\end{equation*}

Namely,
\begin{equation}
	\label{ES32}
	\mathbf{S}(\mathbf{U})\partial_t\mathbf{U}+\mathbf{A}(\mathbf{U})\partial_x\mathbf{U}=\mathbf{B}(\mathbf{U}),
\end{equation}
with $\mathbf{U}=(\zeta,V)^T$,
\begin{equation*}
	\mathbf{B}(\mathbf{U})=
	\begin{pmatrix}
		b_t+\epsilon b_xV \\ -\frac{\epsilon}{2}hb_{ttx}
	\end{pmatrix}
\end{equation*}
and
\begin{equation}
	\label{ES33}
	\mathbf{S}(\mathbf{U})=
	\begin{pmatrix}
		1 & 0 \\ 0 & \mt\cdot
	\end{pmatrix},
	\quad
	\mathbf{A}(\mathbf{U})=
	\begin{pmatrix}
		\epsilon V & h \\ h & \mt(\epsilon V\cdot)
	\end{pmatrix},
\end{equation}
be two symmetric operators. Existence and uniqueness of a solution for system (\ref{ES32}) is guaranteed by the following theorem.

\begin{theorem}
	\label{Theorem Existencia}
	Let $s>3/2$ and $(\zeta_0,V_0)\in \mathbf{X}^s$. Let $b\in W^{2,\infty}([0,\infty);H^{s+1})$ and assume \eqref{ES34} is valid. Then, there exists $T_{BP}>0$, uniformly bounded from below with respect to $\epsilon$, such that system \eqref{ES32} admits a unique solution $(\zeta,V)^T\in \mathbf{X}_{T_{BP}}^s$ with initial condition $(\zeta_0,V_0)^T$.
\end{theorem}

This result is a consequence of applying a strategy similar to the one described in \cite{taylor1997partial}, Chapter 16. Namely, the solution is built as the limit of the following Picard iterative scheme
\begin{equation}\label{picard}
	\left\{
	\begin{aligned}
		& \mathbf{S}(\mathbf{U}^n)\partial_t\mathbf{U}^{n+1}+\mathbf{A}(\mathbf{U}^n)\partial_x\mathbf{U}^{n+1}=\mathbf{B}(\mathbf{U}^n), \\
		& \mathbf{U}^{n+1}|_{t=0}=\mathbf{U}_0.
	\end{aligned}
	\right.
\end{equation}

Now, to study (\ref{picard}), we consider the initial value problem
	\begin{equation}
		\label{ES35}
		\left\{
		\begin{aligned}
			& \mathbf{S}(\mathbf{\uu})\partial_t\mathbf{U}+\mathbf{A}(\mathbf{\uu})\partial_x\mathbf{U}=\mathbf{B}(\textbf{\uu}), \\
			& \mathbf{U}|_{t=0}=\mathbf{U}_0,
		\end{aligned}
		\right.
	\end{equation}
where $\mathbf{\uu}=(\zetau,\vu)^T$ is a reference state. For this system we provide an energy estimate. Before giving the proof, let us recall two lemmas describing some important properties of the operator $\mt$. Proofs can be found in \cite{israwi2011large}.


\begin{lemma}
	Let $\zeta,b\in W^{1,\infty}(\R)$ such that \eqref{ES34} is satisfied. Then the operator
	\[ \mt:H^2(\R)\rightarrow L^2(\R) \]
	is well defined, one-to-one and onto.
\end{lemma}

\begin{lemma}
	\label{Lemma 2Prop T}
	Let $d_0>\frac{1}{2}$ and $\zeta,b\in H^{d_0+1}(\R)$ be such that \eqref{ES34} is satisfied. 
	
	Then
	\begin{enumerate}
		\item $\forall s\ge0$, $|\mt^{-1}f|_{H^s}+\sqrt{\mu}|\partial_x\mt^{-1}f|_{H^s}\le C(\frac{1}{h_{min}},|\zeta,b|_{H^{s\vee d_0+1}})|f|_{H^s}$.
		\item $\forall s\ge0$, $\sqrt{\mu}|\mt^{-1}\partial_xf|_{H^s}\le C(\frac{1}{h_{min}},|\zeta,b|_{H^{s\vee d_0+1}})|f|_{H^s}$.
	\end{enumerate}
\end{lemma}

Finally, we present the energy estimate for the linearized system (\ref{ES35}). The proof follows the ideas presented in \cite{lannes2013water} for the Saint-Venant equation, but in our case, the bottom also depends on the time variable.

\begin{proposition}\label{pe}
	\label{Energy estimate}
	Let $s>3/2$, $T>0$, and $b\in W^{2,\infty}((0,\infty);H^{s+1})$. Let also $\mathbf{\uu}=(\zetau,\vu)^T\in \mathbf{X}_{T}^s$ be such that $\partial_t\mathbf{\uu}\in \mathbf{X}_{T}^{s-1}$ satisfying the condition \eqref{ES34} on $[0,T/\epsilon]$. Then for any $\mathbf{U}_0\in \mathbf{X}^s$, there exists a unique solution $\mathbf{U}=(\zeta,V)^T\in \mathbf{X}_{T}^s$ to \eqref{ES35} such that for any $t\in[0,T/\epsilon]$
	\begin{align}
		\label{ES35_1}
		|\mathbf{U}|_{\mathbf{X}^s}^2\le e^{\epsilon tC(\mathbf{\uu})}|\mathbf{U}_0|_{\mathbf{X}^s}^2+\epsilon C(\mathbf{\uu},T)\int_0^t(|b_t|_{H^s}^2+|b_{tt}|_{H^{s+1}}^2)d\tau.
	\end{align}
\end{proposition}

%

\begin{remark}\label{r2}
Due to the approach of the inverse problem, as an optimal control problem, it is important to note that no assumption has been made on the size of $\epsilon$. On the contrary, thanks to the fact that $T_{BP}$ is uniformly bounded from below with respect to $\epsilon$, it allows one to conclude that if any smallness assumption is made on $\epsilon$, the existence time is large, of order $O(1/\epsilon)$.
\end{remark}

\begin{proof}[Proof of Proposition \ref{pe}]
	Let
	\[ \mathbf{\uu}\in C([0,T/\epsilon];\mathbf{X}^s), \ \partial_t\mathbf{\uu}\in C([0,T/\epsilon];\mathbf{X}^{s-1}). \]
	
	Let 
	\[ \Lambda=(1-\partial_x^2)^{1/2}=(1+|D|^2)^{1/2} \]
	be the fractional derivative where $D=\frac{1}{i}\partial_x$ stands for the Fourier multiplier. We use notation $\tu=\mathcal{T}(\mathbf{\uu})$.
	
	By taking the $\Lambda^s$ operator in \eqref{ES35} and using the definition of the commutator we obtain
	\begin{align}
		\label{ES36}
		S(\mathbf{\uu})\partial_t\ls \mathbf{U}+\mathbf{A}(\mathbf{\uu})\partial_x\ls \mathbf{U}&=\ls B(\mathbf{\uu})-[\ls,\mathbf{S}(\mathbf{\uu})]\partial_t\mathbf{U}-[\ls,\mathbf{A}(\mathbf{\uu})]\partial_x\mathbf{U} \\
		&=\ls \mathbf{B}(\mathbf{\uu})-[\ls,\mathbf{S}(\mathbf{\uu})]\mathbf{S}^{-1}(\mathbf{\uu})\mathbf{B}(\mathbf{\uu}) \notag\\
		&\phantom{=}+[\ls,\mathbf{S}(\mathbf{\uu})]\mathbf{S}^{-1}(\mathbf{\uu})\mathbf{A}(\mathbf{\uu})\partial_x\mathbf{U}-[\ls,\mathbf{A}(\mathbf{\uu})]\partial_x\mathbf{U}. \notag
	\end{align}
	
	Since $\mathbf{S}$ is symmetric, we have
	\begin{align*}
		\frac{1}{2}\frac{d}{dt}(\mathbf{S}(\mathbf{\uu})\ls \mathbf{U},\ls \mathbf{U})&=\frac{1}{2}([\partial_t,\mathbf{S}(\mathbf{\uu})]\ls \mathbf{U},\ls \mathbf{U})+(\mathbf{S}(\mathbf{\uu})\partial_t\ls\mathbf{U},\ls \mathbf{U}) \\
		&=\frac{1}{2}([\partial_t,\tu]\ls V,\ls V)+(\mathbf{S}(\mathbf{\uu})\partial_t\ls \mathbf{U},\ls \mathbf{U}).
	\end{align*}
	
	Then, by replacing \eqref{ES36}, from last equality we obtain
	\begin{multline*}
		\frac{1}{2}\frac{d}{dt}(\mathbf{S}(\mathbf{\uu})\ls \mathbf{U},\ls \mathbf{U})=\frac{1}{2}([\partial_t,\tu]\ls V,\ls V)+(\ls \mathbf{B}(\mathbf{\uu}),\ls \mathbf{U}) \\
		-([\ls,\mathbf{S}(\mathbf{\uu})]\mathbf{S}^{-1}(\mathbf{\uu})\mathbf{B}(\mathbf{\uu}),\ls \mathbf{U})+([\ls,\mathbf{S}(\mathbf{\uu})]\mathbf{S}^{-1}(\mathbf{\uu})\mathbf{A}(\mathbf{\uu})\partial_x\mathbf{U},\ls \mathbf{U}) \\
		-([\ls,\mathbf{A}(\mathbf{\uu})]\partial_x\mathbf{U},\ls \mathbf{U})-(\mathbf{A}(\mathbf{\uu})\partial_x\ls \mathbf{U},\ls \mathbf{U}).
	\end{multline*}
	
	Thus, by the symmetry of $\mathbf{A}$,
	\begin{multline*}
		\frac{1}{2}\frac{d}{dt}(\mathbf{S}(\mathbf{\uu})\ls \mathbf{U},\ls \mathbf{U})=\frac{1}{2}([\partial_t,\tu]\ls V,\ls V)+(\ls \mathbf{B}(\mathbf{\uu}),\ls \mathbf{U}) \\
		-([\ls,\mathbf{S}(\mathbf{\uu})]\mathbf{S}^{-1}(\mathbf{\uu})\mathbf{B}(\mathbf{\uu}),\ls \mathbf{U})+([\ls,\mathbf{S}(\mathbf{\uu})]\mathbf{S}^{-1}(\mathbf{\uu})\mathbf{A}(\mathbf{\uu})\partial_x\mathbf{U},\ls \mathbf{U}) \\
		-([\ls,\mathbf{A}(\mathbf{\uu})]\partial_x\mathbf{U},\ls \mathbf{U})+\frac{1}{2}([\partial_x,\mathbf{A}(\mathbf{\uu})]\ls \mathbf{U},\ls \mathbf{U}).
	\end{multline*}
	
	We are going to estimate now each term of the right--hand side of this last equality in terms of $|\mathbf{U}|_{\mathbf{X}^s}$, taking into account the expressions for $\mathbf{S}$ and $\mathbf{A}$ from \eqref{ES33}.
	
	\subsection*{Estimate of $([\partial_x,\mathbf{A}(\mathbf{\uu})]\ls \mathbf{U},\ls \mathbf{U})$}
	
	\begin{align*}
		([\partial_x,\mathbf{A}(\mathbf{\uu})]\ls \mathbf{U},\ls \mathbf{U})&=(\epsilon\vu_x\ls\zeta,\ls\zeta)+(\hu_x\ls\zeta,\ls\zeta) \\
		&\phantom{=}+(\hu_x\ls V,\ls V)+([\partial_x,\tu(\epsilon\vu\cdot)]\ls V,\ls V),
	\end{align*}
	with $\hu=1+\epsilon(\zetau-b)$. 
	
	From the definition of the operator $\mt$, one has
	\begin{align*}
		([\partial_x,\tu(\epsilon\vu\cdot)]\ls V,\ls V)&=([\partial_x,(\epsilon\vu\hu)]\ls V,\ls V)-\frac{\mu}{3}([\partial_x,\partial_x(\hu^3\partial_x(\epsilon\vu\cdot))]\ls V,\ls V) \\
		&=([\partial_x,(\epsilon\vu\hu)]\ls V,\ls V)+\frac{\mu}{3}([\partial_x,\hu^3]\partial_x(\epsilon\vu\ls V),\ls V_x).
	\end{align*}
	
	Therefore
	\begin{equation*}
		([\partial_x,\mathbf{A}(\mathbf{\uu})]\ls \mathbf{U},\ls \mathbf{U})\le \epsilon C(|\mathbf{\uu},b|_{H^{d_0}})|\mathbf{U}|_{\mathbf{X}^s}^2.
	\end{equation*}
	
	\subsection*{Estimate of $([\ls,\mathbf{A}(\mathbf{\uu})]\partial_x\mathbf{U},\ls \mathbf{U})$}
	
	\begin{align*}
		([\ls,\mathbf{A}(\mathbf{\uu})]\partial_x\mathbf{U},\ls \mathbf{U})&=([\ls,\epsilon\vu]\zeta_x,\ls\zeta)+([\ls,\hu]V_x,\ls\zeta) \\
		&\phantom{=}+([\ls,\hu]\zeta_x,\ls V)+([\ls,\tu(\epsilon\vu)]V_x,\ls V).
	\end{align*}
	
	From the definition of $\mt$
	\begin{align*}
		([\ls,\tu(\epsilon\vu)]V_x,\ls V)&=([\ls,\epsilon\hu\vu]V_x,\ls V)-\frac{\mu}{3}([\ls,\partial_x(\hu^3\partial_x(\epsilon\vu\cdot))]V_x,\ls V).
	\end{align*}
	
	Since that
	\begin{align*}
		[\ls,\partial_x(\hu^3\partial_x(\epsilon\vu\cdot))]V_x&=\ls(\partial_x(\hu^3\partial_x(\epsilon\vu V_x)))-\partial_x(\hu^3\partial_x(\epsilon\vu\ls V_x)) \\
		&=\partial_x\left(\ls(\hu^3\partial_x(\epsilon\vu V_x))-\hu^3\partial_x\ls(\epsilon\vu V_x)+\hu^3\partial_x([\ls,\epsilon\vu]V_x)\right) \\
		&=\partial_x\left([\ls,\hu^3]\partial_x(\epsilon\vu V_x)+\hu^3\partial_x([\ls,\epsilon\vu]V_x)\right),
	\end{align*}
	we have
	\begin{align*}
		-\frac{\mu}{3}([\ls,\partial_x(\hu^3\partial_x&(\epsilon\vu\cdot))]V_x,\ls V) \\
		&=\frac{\mu}{3}([\ls,\hu^3]\partial_x(\epsilon\vu V_x),\ls V_x)+\frac{\mu}{3}(\partial_x([\ls,\epsilon\vu]V_x),\hu^3\ls V_x) \\
		&=\frac{\mu}{3}([\ls,\hu^3]\partial_x(\epsilon\vu V_x),\ls V_x)+\frac{\mu}{3}([\ls,\epsilon\vu_x]V_x,\hu^3\ls V_x) \\
		&\phantom{=}+\frac{\mu}{3}([\ls,\epsilon\vu]V_{xx},\hu^3\ls V_x).
	\end{align*}
	
	Then one finally has
	\begin{equation*}
		([\ls,\mathbf{A}(\mathbf{\uu})]\partial_x\mathbf{U},\ls\mathbf{U})\le \epsilon C(|\mathbf{\uu},b|_{H^{s\vee d_0+1}},\mu|\vu_x|_{H^s})|\mathbf{U}|_{\mathbf{X}^s}^2.
	\end{equation*}
	
	\subsection*{Estimate of $([\ls,\mathbf{S}(\mathbf{\uu})]\mathbf{S}^{-1}(\mathbf{\uu})\mathbf{A}(\mathbf{\uu})\partial_x\mathbf{U},\ls \mathbf{U})$}
	
	We have that
	\[ 
	\mathbf{S}^{-1}
	=
	\begin{pmatrix}
		1 & 0 \\ 0 & \mt^{-1}
	\end{pmatrix}.
	\]
	
	Then using the definition of $\mathbf{A}$ given in \eqref{ES33}
	\begin{align*}
		([\ls,\mathbf{S}(\mathbf{\uu})]\mathbf{S}^{-1}(\mathbf{\uu})\mathbf{A}(\mathbf{\uu})\partial_x\mathbf{U},\ls \mathbf{U})&=([\ls,\tu]\tu^{-1}(\hu\zeta_x),\ls V) \\
		&\phantom{=}+([\ls,\tu]\epsilon \vu V_x,\ls V).
	\end{align*}
	
	By the definition of the operator $\mt$
	\begin{align*}
		([\ls,\tu]\tu^{-1}\hu\zeta_x,\ls V)&=([\ls,\hu]\tu^{-1}\hu\zeta_x,\ls V)-\frac{\mu}{3}([\ls,\partial_x(\hu^3\partial_x\cdot)]\tu^{-1}\hu\zeta_x,\ls V) \\
		&=([\ls,\hu]\tu^{-1}\hu\zeta_x,\ls V)-\frac{\mu}{3}(\partial_x[\ls,\hu^3]\partial_x\tu^{-1}\hu\zeta_x,\ls V) \\
		&=([\ls,\hu]\tu^{-1}\hu\zeta_x,\ls V)+\frac{\mu}{3}([\ls,\hu^3]\partial_x\tu^{-1}\hu\zeta_x,\ls V_x).
	\end{align*}
	
	On the other hand
	\begin{align*}
		([\ls,\tu]\epsilon \vu V_x,\ls V)&=([\ls,\hu]\epsilon \vu V_x,\ls V)-\frac{\mu}{3}([\ls,\partial_x(\hu^3\partial_x\cdot)]\epsilon \vu V_x,\ls V) \\
		&=([\ls,\hu]\epsilon \vu V_x,\ls V)+\frac{\mu}{3}([\ls,\hu^3]\partial_x(\epsilon \vu V_x),\ls V_x).
	\end{align*}
	
	Therefore by Lemma \ref{Lemma 2Prop T}
	\begin{equation*} ([\ls,\mathbf{S}(\mathbf{\uu})]\mathbf{S}^{-1}(\mathbf{\uu})\mathbf{A}(\mathbf{\uu})\partial_x\mathbf{U},\ls \mathbf{U})\le \epsilon C(|\mathbf{\uu},b|_{H^{s\vee d_0+1}})|\mathbf{U}|_{\mathbf{X}^s}^2.
	\end{equation*}
	
	\subsection*{Estimate of $([\ls,\mathbf{S}(\mathbf{\uu})]\mathbf{S}^{-1}(\mathbf{\uu})\mathbf{B}(\mathbf{\uu}),\ls \mathbf{U})$}
	
	\begin{align*}
		([\ls,\mathbf{S}(\mathbf{\uu})]\mathbf{S}^{-1}(\mathbf{\uu})\mathbf{B}(\mathbf{\uu}),\ls \mathbf{U})=-\frac{\epsilon}{2}([\ls,\tu]\tu^{-1}(\hu b_{ttx}),\ls V).
	\end{align*}
	
	Then, by the definition of the operator $\mt$
	\begin{align*}
		-\frac{\epsilon}{2}([\ls,\tu]\tu^{-1}(\hu b_{ttx}),\ls V)&=-\frac{\epsilon}{2}([\ls,\hu]\tu^{-1}(\hu b_{ttx}),\ls V) \\
		&\phantom{=}+\frac{\epsilon\mu}{6}([\ls,\partial_x(\hu^3\partial_x\cdot)]\tu^{-1}(\hu b_{ttx}),\ls V) \\
		&=-\frac{\epsilon}{2}([\ls,\hu]\tu^{-1}(\hu b_{ttx}),\ls V) \\
		&\phantom{=}-\frac{\epsilon\mu}{6}([\ls,\hu^3]\partial_x\tu^{-1}(\hu b_{ttx}),\ls V_x).
	\end{align*}
	
	Therefore from Lemma \ref{Lemma 2Prop T} and the commutator properties (\cite[Lemma 4.6]{alvarez2008nash})
	\begin{align*}
		|[\ls,\hu^3]\partial_x\tu^{-1}(\hu b_{ttx})|_{L^2}&\le C|\hu-1|_{H^s}|\partial_x\tu^{-1}(\hu b_{ttx})|_{H^{s-1}} \\
		&\le C(|\hu-1|_{H^{s\vee d_0+1}})|\hu b_{ttx}|_{H^s} \\
		&\le C(|\uu,b|_{H^{s\vee d_0+1}})|b_{tt}|_{H^{s+1}},
	\end{align*}
	namely
	\begin{equation*}
		([\ls,\mathbf{S}(\mathbf{\uu})]\mathbf{S}^{-1}(\mathbf{\uu})\mathbf{B}(\mathbf{\uu}),\ls \mathbf{U})\le \epsilon C(|\mathbf{\uu},b|_{H^{s\vee d_0+1}})(|\mathbf{U}|_{\mathbf{X}^s}^2+|b_{tt}|_{H^{s+1}}^2).
	\end{equation*}
	
	\subsection*{Estimate of $(\ls \mathbf{B}(\mathbf{\uu}),\ls \mathbf{U})$}
	
	\begin{align*}
		(\ls \mathbf{B}(\mathbf{\uu}),\ls \mathbf{U})&=([\ls,\epsilon b_x]\vu+\epsilon b_x\ls\vu,\ls\zeta)+(\ls b_t,\ls\zeta) \\
		&\phantom{=}-\frac{\epsilon}{2}([\ls,\hu]b_{ttx}+\hu\ls b_{ttx},\ls V).
	\end{align*}
	
	Then
	\begin{equation*}
		(\ls \mathbf{B}(\mathbf{\uu}),\ls \mathbf{U})\le \epsilon C(|\mathbf{\uu}|_{H^{s\vee d_0+1}},|b|_{H^{s+1}})(|\mathbf{U}|_{\mathbf{X}^s}^2+|b_{tt}|_{H^{s+1}}^2+\frac{1}{\epsilon}|b_t|_{H^s}^2).
	\end{equation*}
	
	\subsection*{Estimate of $([\partial_t,\tu]\ls V,\ls V)$}
	
	By the definition of the operator $\mt$
	\begin{align*}
		([\partial_t,\tu]\ls V,\ls V)&=(\partial_t\hu\ls V,\ls V)-\frac{\mu}{3}(\partial_x(\partial_t\hu^3\ls V_x),\ls V) \\
		&=(\partial_t\hu\ls V,\ls V)+\frac{\mu}{3}(\partial_t\hu^3\ls V_x,\ls V_x).
	\end{align*}
	
	Therefore
	\begin{equation*}
		([\partial_t,\tu]\ls V,\ls V)\le \epsilon C(|\mathbf{\uu},b|_{W^{1,\infty}(H^{d_0})})|\mathbf{U}|_{\mathbf{X}^s}^2.
	\end{equation*}
	
	From the definition of $\mathbf{S}$ given in \eqref{ES33} and by \eqref{ES34}
	\begin{equation}
		\label{ES37}
		|\mathbf{U}|_{\mathbf{X}^s}^2\le C\left(\frac{1}{\hm}\right)(\mathbf{S}(\mathbf{\uu})\ls \mathbf{U},\ls \mathbf{U})
	\end{equation}
	and
	\begin{equation}
		\label{ES38}
		(\mathbf{S}(\mathbf{\uu})\ls \mathbf{U},\ls \mathbf{U})\le C(|\hu|_{L^\infty},|\hu_x|_{L^\infty})|\mathbf{U}|_{\mathbf{X}^s}^2.
	\end{equation}
	
	Then, summarizing all the estimates above together with \eqref{ES37} 
	\begin{equation}
		\label{ES39}
		\frac{d}{dt}(\mathbf{S}(\mathbf{\uu})\ls \mathbf{U},\ls \mathbf{U})\le\epsilon c_1(\mathbf{\uu})((\mathbf{S}\ls \mathbf{U},\ls \mathbf{U})+\frac{1}{\epsilon}|b_t|_{H^s}^2+|b_{tt}|_{H^{s+1}}^2),
	\end{equation}
	with $c_1(\mathbf{\uu})=C(\frac{1}{\hm},|\mathbf{\uu}|_{\mathbf{X}_{T}^s},|\partial_t\mathbf{\uu}|_{\mathbf{X}_{T}^{s-1}},|b|_{L^\infty(H^{s+1})},|b|_{W^{1,\infty}(H^{d_0})})$.
	
	By \eqref{ES39}, for any $\lambda\in\R$
	\begin{align*}
		e^{\epsilon\lambda t}&\frac{d}{dt}(e^{-\epsilon\lambda t}(\mathbf{S}(\mathbf{\uu})\ls \mathbf{U},\ls \mathbf{U})) \\
		&=-\epsilon\lambda(\mathbf{S}(\mathbf{\uu})\ls \mathbf{U},\ls \mathbf{U})+\frac{d}{dt}(\mathbf{S}(\mathbf{\uu})\ls \mathbf{U},\ls \mathbf{U}) \\
		&\le (\epsilon c_1(\mathbf{\uu})-\epsilon\lambda)(\mathbf{S}(\mathbf{\uu})\ls \mathbf{U},\ls \mathbf{U})+\epsilon c_1(\mathbf{\uu})(\frac{1}{\epsilon}|b_t|_{H^s}^2+|b_{tt}|_{H^{s+1}}^2).
	\end{align*}
	
	Then, if $\lambda=\lambda(c_1(\mathbf{\uu}))$ is sufficiently large, by \eqref{ES38} and the last inequality
	\begin{align*}
		e^{\epsilon\lambda t}\frac{d}{dt}(e^{-\epsilon\lambda t}(\mathbf{S}(\mathbf{\uu})\ls \mathbf{U},\ls \mathbf{U}))\le\epsilon c_1(\mathbf{\uu})(\frac{1}{\epsilon}|b_t|_{H^s}^2+|b_{tt}|_{H^{s+1}}^2).
	\end{align*}
	
	Integrating in $0\le t\le T/\epsilon$
	\begin{equation*}
		(\mathbf{S}(\mathbf{\uu})\ls \mathbf{U},\ls \mathbf{U})
		\le e^{\epsilon\lambda t}(\mathbf{S}(\mathbf{\uu})\ls \mathbf{U},\ls \mathbf{U})|_{t=0}+\epsilon c_1(\mathbf{\uu})\int_0^te^{\epsilon\lambda(t-\tau)}(\frac{1}{\epsilon}|b_t|_{H^s}^2+|b_{tt}|_{H^{s+1}}^2)d\tau,
	\end{equation*}
	or equivalently, by \eqref{ES37}, \eqref{ES38}
	\begin{align}
		\label{ES310}
		|\mathbf{U}|_{\mathbf{X}^s}^2\le e^{\epsilon\lambda t}|\mathbf{U}_0|_{\mathbf{X}^s}^2+\epsilon c_2(\mathbf{\uu})\int_0^t(\frac{1}{\epsilon}|b_t|_{H^s}^2+|b_{tt}|_{H^{s+1}}^2)d\tau,
	\end{align}
	with $c_2(\mathbf{\uu})=C(c_1(\mathbf{\uu}),T)$.
	
The fact that $T_{BP}$ is bounded from below by some $T>0$ independent of $\epsilon\in(0,1)$, follows from the analysis above (see \cite{alinhac2022operateurs}).
\end{proof}

As was mentioned at the beginning of this section, existence and uniqueness of solutions for system \eqref{BP}, follow classical ideas based on the energy estimate provided in Proposition \ref{Energy estimate}. Indeed, to prove existence of solutions for system \eqref{ES35} one regularizes the operators $\mathbf{S}$ and $\mathbf{A}$, to apply the Cauchy--Lipschitz Theorem for ODE. To prove existence of solutions for \eqref{BP} one uses the above energy estimate and proceed as in \cite{taylor1997partial}, chapter 16 (see also \cite{israwi2011large}).

\section{First order necessary condition and optimality system}\label{a1}

We now proceed to derive a system of first order necessary optimality conditions for problem \eqref{P}. This is done in a straightforward manner by studying the Gateaux derivative of the functional $J(b)$. From \eqref{BP}, we are considering system:
\begin{equation}
	\label{s1}
	\left\{
	\begin{aligned}
		&r_t+(hV)_x=0, \\
		& V_t-\frac{\mu}{3h}\partial_x(h^3V_{tx})+r_x+\epsilon VV_x=-b_x-\frac{\epsilon}{2}b_{ttx},
	\end{aligned}
	\right.
\end{equation}
with $r=\zeta-b$, $h=1+\epsilon r$.

We denote $\Omega=\{(x,y)\in\R^2:\epsilon b\le y\le\epsilon\zeta\}$. Given $\lambda>0$ and $b\in C^1(W^{1,\infty}(\R))$, we define the domain $\Omega_\lambda=\Omega+\lambda b$. For $U:=(\zeta,V)^T$ we use notation $\delta U=\lim_{\lambda\rightarrow0}\frac{U^\lambda-U}{\lambda}$ where $U^\lambda$ is the solution of \eqref{s1} in the domain $\Omega_\lambda$.

Differentiating formally in system \eqref{s1} with respect to bottom variations, we obtain
\begin{equation}
	\left\{
	\label{s2}
	\begin{aligned}
		&\delta r_t+\partial_x(\delta hV+h\delta V)=0, \\
		& \delta V_t-\frac{\mu}{3}\delta\left(\frac{1}{h}\partial_x(h^3V_{tx})\right)+\delta r_x+\epsilon\delta(VV_x)=-\delta b_x-\frac{\epsilon}{2}\delta b_{ttx}.
	\end{aligned}
	\right.
\end{equation}

Multiplying first and second equations in \eqref{s2} by test functions $p=p(t,x)$ and $q=q(t,x)$ respectively, such that $\lim_{|x|\rightarrow\infty}p(t,x)=0=\lim_{|x|\rightarrow\infty}q(t,x)$, and integrating in $x$
\begin{equation}
	\left\{
	\label{s3}
	\begin{aligned}
		&\int_0^T\int_\R\delta r_tp-\int_0^T\int_\R(\delta hV+h\delta V)p_x=0, \\
		&\int_0^T\int_\R\delta V_tq-\frac{\mu}{3}\int_0^T\int_\R\delta\left(\frac{1}{h}\partial_x(h^3V_{tx})\right)q-\int_0^T\int_\R\delta r q_x \\
		&+\epsilon\int_0^T\int_\R\delta VV_xq-\epsilon\int_0^T\int_\R\delta V(Vq)_x=\int_0^T\int_\R\delta bq_x+\frac{\epsilon}{2}\int_0^T\int_\R\delta b_{tt}q_x.
	\end{aligned}
	\right.
\end{equation}

Since it is possible to drop the $O(\mu^2)$ terms, under the extra assumptions $\lim_{|x|\rightarrow\infty}q_x(t,x)=0$ and $q_{xx}(T,x)=0$, one has
\begin{align*}
	-\frac{\mu}{3}\int_0^T\int_\R\delta\left(\frac{1}{h}\partial_x(h^3V_{tx})\right)q=\frac{\mu}{3}\int_0^T\int_\R\delta V\partial_x(h^2q_{tx}).
\end{align*}

We have that $\delta r(0,x)=\delta V(0,x)=0$. On the control $b$ we are going to assume $\delta b(0,x)=\delta b(T,0)=0$ and $\delta b_t(0,x)=\delta b_t(T,x)=0$. Then \eqref{s3} becomes into

\begin{equation}
	\label{s4}
	\left\{
	\begin{aligned}
		&-\int_0^T\int_\R\delta r p_t+\int_\R\delta r(T)p(T)-\int_0^T\int_\R\delta hVp_x-\int_0^T\int_\R\delta Vhp_x=0, \\
		&-\int_0^T\int_\R\delta Vq_t+\int_\R\delta V(T)q(T)+\int_0^T\int_\R\delta V\frac{\mu}{3}\partial_x(h^2q_{tx})-\int_0^T\int_\R\delta r q_x \\
		&-\epsilon\int_0^T\int_\R\delta VVq_x=\int_0^T\int_\R\delta bq_x+\frac{\epsilon}{2}\int_0^T\int_\R\delta bq_{ttx}.
	\end{aligned}
	\right.
\end{equation}

Summing both equations in \eqref{s4} and having into account that $\delta h=\epsilon\delta r$, we have
\begin{multline}
	\label{s5}
	\int_0^T\int_\R\delta r(-p_t-\epsilon Vp_x-q_x) \\ 
	+\int_0^T\int_\R\delta V(-hp_x-q_t+\frac{\mu}{3}\partial_x(h^2q_{tx})-\epsilon Vq_x)+\int_\R\delta r(T)p(T)+\int_\R\delta V(T)q(T) \\
	=\int_0^T\int_\R\delta b(q_x+\frac{\epsilon}{2}q_{ttx}).
\end{multline}

\begin{remark}
	Note that we have made some assumptions on the behavior of the bottom and its first time derivative on $0$ and $T$. This is not an issue since we are thinking in an underwater device that is at rest at the beginning and end of the physical experiment.
\end{remark}

Considering the functional $J$ defined in \eqref{P}, with $\alpha=0$, we obtain
\begin{align}
	\label{s6}
	\delta J(b)[\delta b]&=\int_0^T\int_\R(\zeta-\bar{\zeta})\delta\zeta+\int_0^T\int_\R(V-\bar{V})\delta V \\
	&=\int_0^T\int_\R(\zeta-\bar{\zeta})(\delta r+\delta b)+\int_0^T\int_\R(V-\bar{V})\delta V, \notag
\end{align}
as long as $b(\cdot,x)$ be compact supported in $[0,T]$. 

Therefore, assuming
\begin{equation}
	\label{s7}
	\left\{
	\begin{aligned}
		&p_t+\epsilon Vp_x+q_x=\bar{\zeta}-\zeta, \\
		&q_t-\frac{\mu}{3}\partial_x(h^2q_{tx})+hp_x+\epsilon Vq_x=\bar{V}-V, \\
		&p(T)=0, \quad q(T)=0,
	\end{aligned}
	\right.
\end{equation}
from \eqref{s5} and \eqref{s6} we have
\begin{equation*}
	\delta J(b)[\delta b]=
	\int_0^T\int_\R\delta b(q_x+\frac{\epsilon}{2}q_{ttx}+(\zeta-\bar{\zeta}))dxdt.
\end{equation*}

On the other hand, by Theorem \ref{Theorem Existencia Minimo} we know that there exists an optimal pair $(U^*,b^*)$ of \eqref{P}. Then, necessarily we have $\nabla J(b^*)=0$. Namely, if we consider $\delta U=U+U^*$ and $\delta b=b+b^*$ then
\[ 0=\nabla J(b^*)=q_x^*+\frac{\epsilon}{2}q_{ttx}^*+(\zeta^*-\bar{\zeta}). \]


Observe that \eqref{s7} can be written as
\begin{equation*}
	\left(
	\begin{pmatrix}
		h & 0 \\ 0 & 1
	\end{pmatrix}
	+\mu
	\begin{pmatrix}
		0 \\ -\frac{1}{3}\partial_x(h^2\partial_x\cdot)
	\end{pmatrix}
	\right)
	\begin{pmatrix}
		p_t \\ q_t
	\end{pmatrix}
	+
	\begin{pmatrix}
		h\epsilon V & h \\ h & \epsilon V
	\end{pmatrix}
	\begin{pmatrix}
		p_x \\ q_x
	\end{pmatrix}
	=
	\begin{pmatrix}
		h(\bar{\zeta}-\zeta) \\ \bar{V}-V
	\end{pmatrix}.
\end{equation*}

From this last expression and proceeding as we did to prove Theorem \ref{Theorem Existencia}, it can be proved that there exists a $T>0$ such that $(p^*,q^*)\in C([0,T/\epsilon];X^s)$. Therefore, we can state now a theorem giving a system of first order necessary optimality conditions for problem \eqref{P}.
\begin{theorem}
	\label{t3}
	Let $(U^*,b^*)$ be an optimal solution for \eqref{P}. Then there exists $(p^*,q^*)\in C([0,T/\epsilon];X^s)$ satisfying the following optimality system, in a variational sense:
	\begin{equation*}
		\left\{
		\begin{aligned}
			&r_t^*+(h^*V^*)_x=0, \\
			& V_t^*-\frac{\mu}{3h^*}\partial_x((h^*)^3V_{tx}^*)+r_x^*+\epsilon V^*V_x^*=-b_x^*-\frac{\epsilon}{2}b_{ttx}^*, \\
			&\zeta(0)=\zeta_0, \ V(0)=V_0, \\
			&p_t^*+\epsilon V^*p_x^*+q_x^*=\bar{\zeta}-\zeta^*, \\
			& q_t^*-\frac{\mu}{3}\partial_x((h^*)^2\partial_xq_t^*)+h^*p_x^*+\epsilon V^*q_x^*=\bar{V}-V^*, \\
			&p^*(T)=0, \ q^*(T)=0, \\
			&q_x^*+\frac{\epsilon}{2}q_{ttx}^*+(\zeta^*-\bar{\zeta})=0.
		\end{aligned}
		\right.
	\end{equation*}
\end{theorem}

\section{The reference scheme {\bf Rusanov+FD}}\label{sec:ref-scheme}

In this section we build a solver for the PDE-constraint optimization problem in which system (\ref{eq:direct-system}) is solved with a conservative finite volume scheme, the Rusanov method, and system (\ref{eq:dual-system}) is discretized by a finite difference approach. Here, we propose the solver in a general approach, so we describe the method for (\ref{eq:general-state}) and (\ref{eq:general-dual}) which are more general than (\ref{eq:direct-system}) and (\ref{eq:dual-system}), respectively.

System (\ref{eq:general-state}) is solved by the finite volume formula

\begin{eqnarray}
	\label{eq:one-step-formula}
	\begin{array}{c}
		\mathbf{U}_i^{n+1}
		=
		\mathbf{U}_i^{n} - \frac{\Delta t}{\Delta x} ( \mathbf{F}_{i+\frac{1}{2}}-\mathbf{F}_{i-\frac{1}{2}}) + \Delta t \mathbf{B}_i \;,
	\end{array}
\end{eqnarray}
where $\mathbf{F}_{i+\frac{1}{2}}$ is a numerical flux, in this case the Rusanov flux is implemented, that corresponds to
\begin{eqnarray}
	\begin{array}{c}
		\mathbf{F}_{i+\frac{1}{2}}
		=\frac{1}{2} ( \mathbf{F}( \mathbf{U}_{i+1}^{n}) +  \mathbf{F}( \mathbf{U}_i^{n}))
		-\frac{\lambda_{i+\frac{1}{2}}}{2} ( \mathbf{U}_{i+1}^{n} - \mathbf{U}_{i}^{n} ) \;;
	\end{array}
\end{eqnarray}
here $\lambda_{i+\frac{1}{2}} = max(\lambda_L, \lambda_R)$ with 
$$ \lambda_L = max\{   | \epsilon V_i^n - \sqrt{r_i^n}   |, | \epsilon V_i^n + \sqrt{r_i^n}   | \} $$ 
and 
$$ \lambda_R  = max\{   | \epsilon V_{i+1}^n - \sqrt{r_{i+1}^n}   |, | \epsilon V_{i+1}^n + \sqrt{r_{i+1}^n}   | \} \;.$$

For the source term we use central finite difference scheme for approximating spatial and time derivatives, that is, a similar approach to that used in (\ref{eq:source-term-unified}).

To solve the adjoint system (\ref{eq:general-dual}) we use 
\begin{eqnarray}
	\label{eq:discretize-adjoint}
	\begin{array}{c}
		- \mathbf{P}_i^{n+1} + \mathbf{P}_i^{n} - \Delta t \mathbf{A}(\mathbf{U}_i^{n+1}, \mathbf{P}_i^{n+1}) \frac{ ( \mathbf{P}_{i+1}^{n+1} - \mathbf{P}_{i-1}^{n+1}   )}{2 \Delta x}  = \Delta t \mathbf{R}( x_{i}, \mathbf{P}_i^{n+1}, \mathbf{U}_i^{n+1} ) \;.
	\end{array}
\end{eqnarray}
Notice that the adjoint system is solved back from $ \mathbf{P}_i^{n_T} = \mathbf{0}$\;, where $ n_T$ is the number of time iterations to reach the output time, $T$, of the simulation.

As commented above, these methods are functional and simple enough for problems in this paper. To extend them to other systems only expressions for $\mathbf{B}(x,\mathbf{U})$, $\mathbf{F}(\mathbf{U})$, $\mathbf{A}(\mathbf{U}, \mathbf{P})$ and $\mathbf{R}(x,\mathbf{U}, \mathbf{P})$ must be provided.

This complete the description of the solvers for the state and adjoint systems. They need to be combined to generate the minimizer sequence for a cost functional $J(b)$. Algorithm \ref{alg:cap} summarizes the global procedure to obtain $b$, iteratively. The iterations are stopped by fixing the maximum number iterations to $IterTotal = 17$.

\begin{algorithm}
	\caption{Global algorithm for obtaining $b(t,x)$. Provide $b^0$, $it = 1$, $\lambda_b$, $IterTotal$ and $Tol$. }\label{alg:cap}
	\begin{algorithmic}
		\Ensure $b$ at each $x_i$ and $t^n$.
		
		
		\While{$it \leq IterTotal$}
		
		\State {\bf Step 1:} Solve eq. (\ref{eq:general-state}) using (\ref{eq:one-step-formula}).
		This generates $\{ \mathbf{U}_{i}^{n} \}_{n=1}^{n_T}$, $i=1,...,N_x$, where $N_x$ is the number of cells.

		\State {\bf Step 2:} Solve (\ref{eq:general-dual}) using (\ref{eq:discretize-adjoint}), in this step  $\{ \mathbf{P}_{i}^{n} \}_{n=1}^{n_T}$ is generated.

		\State {\bf Step 3:} Update $b^{k+1} = b^k -\lambda_b\nabla(b^k)$, where
		$  \nabla J = q_x +\frac{1}{2}\epsilon q_{xtt} - (\bar{\zeta} - \zeta) \;.$

		\State {\bf Step 4:}  Stop the algorithm once $ | \nabla (b^k,\mathbf{U}_i, \mathbf{P}_i)| < Tol$, where $Tol$ is some prescribed tolerance.

		\EndWhile
	\end{algorithmic}
\end{algorithm}

\bibliographystyle{abbrv}
\bibliography{MLLZ-preprint}

\begin{thebibliography}{10}

\bibitem{webpage}
S. l. instant sport.
\newblock \url{https://wavegarden.com/}.
\newblock Accessed: 2023-09-23.

\bibitem{alazard2018control}
T.~Alazard, P.~Baldi, and D.~Han-Kwan.
\newblock Control of water waves.
\newblock {\em Journal of the European Mathematical Society}, 20(3):657--745,
  2018.

\bibitem{alinhac2022operateurs}
S.~Alinhac and P.~G{\'e}rard.
\newblock Op{\'e}rateurs pseudo-diff{\'e}rentiels et th{\'e}oreme de
  nash-moser.
\newblock In {\em Op{\'e}rateurs pseudo-diff{\'e}rentiels et th{\'e}or{\`e}me
  de Nash-Moser}. EDP Sciences, 2022.

\bibitem{alvarez2008nash}
B.~Alvarez-Samaniego and D.~Lannes.
\newblock A nash-moser theorem for singular evolution equations. application to
  the serre and green-naghdi equations.
\newblock {\em Indiana University mathematics journal}, pages 97--131, 2008.

\bibitem{baudouin2014determination}
L.~Baudouin, E.~Cerpa, E.~Cr{\'e}peau, and A.~Mercado.
\newblock On the determination of the principal coefficient from boundary
  measurements in a kdv equation.
\newblock {\em Journal of Inverse and Ill-posed Problems}, 22(6):819--845,
  2014.

\bibitem{beizel2012simulation}
S.~Beizel, L.~Chubarov, D.~Dutykh, G.~Khakimzyanov, and N.~Y. Shokina.
\newblock Simulation of surface waves generated by an underwater landslide in a
  bounded reservoir.
\newblock {\em Russian Journal of Numerical Analysis and Mathematical
  Modelling}, 27(6):539--558, 2012.

\bibitem{constantin2011nonlinear}
A.~Constantin.
\newblock {\em Nonlinear water waves with applications to wave-current
  interactions and tsunamis}.
\newblock SIAM, 2011.

\bibitem{craig1985existence}
W.~Craig.
\newblock An existence theory for water waves and the boussinesq and
  korteweg-devries scaling limits.
\newblock {\em Communications in Partial Differential Equations},
  10(8):787--1003, 1985.

\bibitem{dutykh2007water}
D.~Dutykh and F.~Dias.
\newblock Water waves generated by a moving bottom.
\newblock In {\em Tsunami and Nonlinear waves}, pages 65--95. Springer, 2007.

\bibitem{dutykh2012contribution}
D.~Dutykh, D.~Mitsotakis, L.~B. Chubarov, and Y.~I. Shokin.
\newblock On the contribution of the horizontal sea-bed displacements into the
  tsunami generation process.
\newblock {\em Ocean Modelling}, 56:43--56, 2012.

\bibitem{fontelos2023controllability}
M.~Fontelos and J.~L{\'o}pez-R{\'\i}os.
\newblock Controllability of surface gravity waves and the sloshing problem.
\newblock {\em ESAIM: Control, Optimisation and Calculus of Variations}, 29:65,
  2023.

\bibitem{fontelos2017bottom}
M.~A. Fontelos, R.~Lecaros, J.~L\'opez-R\'ios, and J.~H. Ortega.
\newblock Bottom detection through surface measurements on water waves.
\newblock {\em SIAM Journal on Control and Optimization}, 55(6):3890--3907,
  2017.

\bibitem{ghidaglia2001numerical}
J.-M. Ghidaglia, A.~Kumbaro, and G.~Le~Coq.
\newblock On the numerical solution to two fluid models via a cell centered
  finite volume method.
\newblock {\em European Journal of Mechanics-B/Fluids}, 20(6):841--867, 2001.

\bibitem{iguchi2011mathematical}
T.~Iguchi.
\newblock A mathematical analysis of tsunami generation in shallow water due to
  seabed deformation.
\newblock In {\em Proceedings of the Royal Society of Edinburgh-A-Mathematics},
  volume 141, page 551. Cambridge Univ Press, 2011.

\bibitem{israwi2011large}
S.~Israwi.
\newblock Large time existence for 1d green-naghdi equations.
\newblock {\em Nonlinear Analysis: Theory, Methods \& Applications},
  74(1):81--93, 2011.

\bibitem{lannes2005well}
D.~Lannes.
\newblock Well-posedness of the water-waves equations.
\newblock {\em Journal of the American Mathematical Society}, 18(3):605--654,
  2005.

\bibitem{lannes2013water}
D.~Lannes.
\newblock {\em The water waves problem: mathematical analysis and asymptotics},
  volume 188.
\newblock American Mathematical Soc., 2013.

\bibitem{lannes2009derivation}
D.~Lannes and P.~Bonneton.
\newblock Derivation of asymptotic two-dimensional time-dependent equations for
  surface water wave propagation.
\newblock {\em Physics of Fluids (1994-present)}, 21(1):016601, 2009.

\bibitem{lecaros2020stability}
R.~Lecaros, J.~L{\'o}pez-R{\'\i}os, J.~Ortega, and S.~Zamorano.
\newblock The stability for an inverse problem of bottom recovering in
  water-waves.
\newblock {\em Inverse Problems}, 36(11):115002, 2020.

\bibitem{lellouche1994boundary}
J.-M. Lellouche, J.-L. Devenon, and I.~Dekeyser.
\newblock Boundary control of burgers' equation—a numerical approach.
\newblock {\em Computers \& Mathematics with Applications}, 28(5):33--44, 1994.

\bibitem{micu2001controllability}
S.~Micu.
\newblock On the controllability of the linearized benjamin--bona--mahony
  equation.
\newblock {\em SIAM Journal on Control and Optimization}, 39(6):1677--1696,
  2001.

\bibitem{micu2009control}
S.~Micu, J.~H. Ortega, L.~Rosier, and B.-Y. Zhang.
\newblock Control and stabilization of a family of boussinesq systems.
\newblock {\em Discrete and Continuous Dynamical Systems}, 24:273--313, 2009.

\bibitem{montecinos2021universal}
G.~I. Montecinos.
\newblock A universal centred high-order method based on implicit taylor series
  expansion with fast second order evolution of spatial derivatives.
\newblock {\em Journal of Computational Physics}, 443:110535, 2021.

\bibitem{montecinos2019numerical}
G.~I. Montecinos, J.~C. L{\'o}pez-R{\'\i}os, J.~H. Ortega, and R.~Lecaros.
\newblock A numerical procedure and coupled system formulation for the adjoint
  approach in hyperbolic pde-constrained optimization problems.
\newblock {\em IMA Journal of Applied Mathematics}, 84(3):483--516, 2019.

\bibitem{mottelet2000controllability}
S.~Mottelet.
\newblock Controllability and stabilization of a canal with wave generators.
\newblock {\em SIAM journal on control and optimization}, 38(3):711--735, 2000.

\bibitem{nalimov1974cauchy}
V.~Nalimov.
\newblock The cauchy-poisson problem.
\newblock {\em Dynamika Splosh. Sredy}, 18:104--210, 1974.

\bibitem{Nersisyan01082015}
H.~Nersisyan, D.~Dutykh, and E.~Zuazua.
\newblock Generation of 2d water waves by moving bottom disturbances.
\newblock {\em IMA Journal of Applied Mathematics}, 80(4):1235--1253, 2015.

\bibitem{nosov2009method}
M.~Nosov and S.~Kolesov.
\newblock Method of specification of the initial conditions for numerical
  tsunami modeling.
\newblock {\em Moscow University Physics Bulletin}, 64(2):208--213, 2009.

\bibitem{nosov2001nonlinear}
M.~Nosov and S.~Skachko.
\newblock Nonlinear tsunami generation mechanism.
\newblock {\em Natural Hazards and Earth System Sciences}, 1(4):251--253, 2001.

\bibitem{qi2017numerical}
M.~Qi, Y.~Kuai, and J.~Li.
\newblock Numerical simulation of water waves generated by seabed movement.
\newblock {\em Applied Ocean Research}, 65:302--314, 2017.

\bibitem{rosier2000exact}
L.~Rosier.
\newblock Exact boundary controllability for the linear korteweg--de vries
  equation on the half-line.
\newblock {\em SIAM Journal on Control and Optimization}, 39(2):331--351, 2000.

\bibitem{rusanov1961}
V.~Rusanov.
\newblock Calculation of interaction of non--steady shock waves with obstacles,
  ussr comp. math. math., 1, 267--279.
\newblock 1961.

\bibitem{shen2022interference}
Y.~Shen, C.~N. Whittaker, E.~M. Lane, W.~Power, and B.~W. Melville.
\newblock Interference effect on tsunami generation by segmented seafloor
  deformations.
\newblock {\em Ocean Engineering}, 245:110244, 2022.

\bibitem{su2020stabilizability}
P.~Su, M.~Tucsnak, and G.~Weiss.
\newblock Stabilizability properties of a linearized water waves system.
\newblock {\em Systems \& Control Letters}, 139:104672, 2020.

\bibitem{su2021strong}
P.~Su, M.~Tucsnak, and G.~Weiss.
\newblock Strong stabilization of small water waves in a pool.
\newblock {\em IFAC-PapersOnLine}, 54(9):378--383, 2021.

\bibitem{taylor1997partial}
M.~E. Taylor.
\newblock Partial differential equations. iii, volume 117 of applied
  mathematical sciences, 1997.

\bibitem{toro2020low}
E.~F. Toro, B.~Saggiorato, S.~Tokareva, and A.~Hidalgo.
\newblock Low-dissipation centred schemes for hyperbolic equations in
  conservative and non-conservative form.
\newblock {\em Journal of Computational Physics}, 416:109545, 2020.

\bibitem{wu1997well}
S.~Wu.
\newblock Well-posedness in sobolev spaces of the full water wave problem in
  2-d.
\newblock {\em Inventiones mathematicae}, 130(1):39--72, 1997.

\bibitem{wu1999well}
S.~Wu.
\newblock Well-posedness in sobolev spaces of the full water wave problem in
  3-d.
\newblock {\em Journal of the American Mathematical Society}, 12(2):445--495,
  1999.

\bibitem{yosihara1983capillary}
H.~Yosihara.
\newblock Capillary-gravity waves for an incompressible ideal fluid.
\newblock {\em Kyoto Journal of Mathematics}, 23(4):649--694, 1983.

\bibitem{zhang2003unique}
X.~Zhang and E.~Zuazua.
\newblock Unique continuation for the linearized benjamin-bona-mahony equation
  with space-dependent potential.
\newblock {\em Mathematische Annalen}, 325(3):543--582, 2003.

\end{thebibliography}

\end{document}